\documentclass[11pt]{amsart}

\usepackage{amssymb,amsfonts}
\usepackage{hyperref}
\usepackage{mdwlist}
\usepackage{amssymb,textcomp}
\usepackage[dvipsnames]{xcolor}
\usepackage{mathtools}
\usepackage{tikz}
\usetikzlibrary{positioning}
\usepackage{enumerate}
\usepackage{enumitem}

\usepackage[utf8]{inputenc}
\usepackage[T1]{fontenc}

\usepackage[mathscr]{eucal}

\usepackage{hyperref}
 \hypersetup{
     colorlinks=true,
     linktocpage=true,
     linkcolor=red,
     filecolor=blue,
     citecolor = blue,
     urlcolor=cyan,
     }

\usepackage[a4paper, twoside=false, vmargin={2cm,3cm}, includehead]{geometry}

\usepackage{comment}
\newtheorem{lemma}{Lemma}[section]
\newtheorem{theorem}[lemma]{Theorem}
\newtheorem*{theorem*}{Theorem}

\newtheorem{corollary}[lemma]{Corollary}

\newtheorem{proposition}[lemma]{Proposition}
\newtheorem*{proposition*}{Proposition}

\newtheorem{problem}[lemma]{Problem}
\newtheorem*{problem*}{Problem}

\theoremstyle{definition}
\newtheorem{claim}{Claim}
\newtheorem*{claim*}{Claim}

\newtheorem{definition}[lemma]{Definition}

\newtheorem{example}[lemma]{Example}

\DeclareMathOperator*{\E}{\mathbb{E}}

\newcommand{\C}{{\mathbb C}}

\newcommand{\N}{{\mathbb N}}
\renewcommand{\P}{{\mathbb P}}
\newcommand{\Q}{{\mathbb Q}}
\newcommand{\R}{{\mathbb R}}

\newcommand{\T}{{\mathbb T}}
\newcommand{\Z}{{\mathbb Z}}

\newcommand{\CA}{{\mathcal A}}
\newcommand{\CB}{{\mathcal B}}
\newcommand{\CC}{{\mathcal C}}
\newcommand{\CD}{{\mathcal D}}

\newcommand{\CH}{{\mathcal H}}
\newcommand{\CI}{{\mathcal I}}
\newcommand{\CK}{{\mathcal K}}

\newcommand{\CN}{{\mathcal N}}
\newcommand{\CP}{{\mathcal P}}
\newcommand{\CQ}{{\mathcal Q}}
\newcommand{\CX}{{\mathcal X}}

\newcommand{\CZ}{{\mathcal Z}}

\newcommand{\FD}{{\mathfrak D}}
\newcommand{\FL}{{\mathfrak L}}

\newcommand{\ba}{\mathbf{a}}
\renewcommand{\b}{{\mathbf{b}}}
\newcommand{\bb}{{\mathbf{b}}}
\newcommand{\bc}{{\mathbf{c}}}
\newcommand{\bm}{{\mathbf{m}}}
\newcommand{\bn}{{\mathbf{n}}}
\newcommand{\be}{{\mathbf{e}}}
\newcommand{\bg}{{\mathbf{g}}}

\newcommand{\bk}{{\mathbf{k}}}
\newcommand{\bp}{{\mathbf{p}}}
\newcommand{\p}{{\mathbf{p}}}
\newcommand{\bq}{{\mathbf{q}}}
\newcommand{\q}{{\mathbf{q}}}

\let\ringaccent\r
\def\aa{\ringaccent a}
\def\aa{\ringaccent a}

\renewcommand{\r}{{\mathbf{r}}}
\newcommand{\br}{{\mathbf{r}}}

\newcommand{\bu}{{\mathbf{u}}}
\newcommand{\bv}{{\mathbf{v}}}

\newcommand{\bx}{{\mathbf{x}}}

\newcommand{\balpha}{{\boldsymbol{\alpha}}}

\newcommand{\bgamma}{{\boldsymbol{\gamma}}}

\newcommand{\bbeta}{{\boldsymbol{\beta}}}

\newcommand{\uh}{{\underline{h}}}
\newcommand{\uk}{{\underline{k}}}

\newcommand{\un}{{\underline{n}}}
\newcommand{\um}{{\underline{m}}}
\newcommand{\ur}{{\underline{r}}}
\newcommand{\uu}{{\underline{u}}}
\newcommand{\uv}{{\underline{v}}}

\newcommand{\veps}{\varepsilon}

\newcommand{\eps}{\epsilon}
\newcommand{\ueps}{{\underline{\epsilon}}}

\newcommand{\supp}{\textrm{supp}}

\newcommand{\norm}[1]{\left\Vert #1\right\Vert}
\newcommand{\fnnorm}[1]{|\!|\!| #1|\!|\!|}
\newcommand{\nnorm}[1]{\left|\!\left|\!\left| #1\right|\!\right|\!\right|}
\newcommand{\bignnorm}[1]{\Big|\!\Big|\!\Big| #1\Big|\!\Big|\!\Big|}

\newcommand{\inv}{^{-1}}

\DeclareMathOperator{\Span}{Span}

\DeclareMathOperator{\nonpol}{nonpol}
\DeclareMathOperator{\pol}{pol}

\DeclareMathOperator{\fracdeg}{frac\; deg}
\newcommand{\Krat}{{\CK_{\text{\rm rat}}}}

\newcommand{\abs}[1]{\mathopen{}\left| #1\mathclose{}\right|}
\newcommand{\bigabs}[1]{\bigl| #1 \bigr|}
\newcommand{\Bigabs}[1]{\Bigl| #1 \Bigr|}
\newcommand{\brac}[1]{\mathopen{}\left( #1 \mathclose{}\right)}

\newcommand{\Bigbrac}[1]{\Bigl( #1 \Bigr)}
\newcommand{\sfloor}[1]{{\lfloor #1 \rfloor}}
\newcommand{\floor}[1]{{\left \lfloor #1 \right \rfloor}}
\newcommand{\ceil}[1]{{\left \lceil #1 \right \rceil}}
\newcommand{\rem}[1]{\left \{ #1 \right \}}
\usepackage[dvipsnames]{xcolor}

\title[]{Seminorm estimates and joint ergodicity for pairwise independent Hardy sequences}
\author{Sebasti\'an Donoso, Andreas Koutsogiannis, Borys Kuca,\\ Wenbo Sun, and Konstantinos Tsinas}

\address[Sebasti\'an Donoso]{Departamento de Ingenier\'{\i}a
	Matem\'atica and Centro de Modelamiento Ma\-te\-m\'a\-ti\-co, Universidad de Chile and IRL-CNRS 2807, Beauchef 851, Santiago,
	Chile.} 
\email{sdonosof@uchile.cl}

\address[Andreas Koutsogiannis]{
Department of Mathematics, Aristotle University of Thessaloniki, Thessaloniki 54124, Greece}
\email{akoutsogiannis@math.auth.gr}

\address[Borys Kuca]{Faculty of Mathematics and Computer Science, Jagiellonian University, 30-348 Krak\'ow, Poland}
\email{borys.kuca@uj.edu.pl}

\address[Wenbo Sun]{Department of Mathematics, Virginia Tech, 225 Stanger Street, Blacksburg, VA, 24061, USA}
\email{swenbo@vt.edu}

\address[Konstantinos Tsinas]{Institute of Mathematics, Ecole Polytechnique F\'{e}d\'{e}rale de Lausanne (EPFL), Lausanne 1015,
Switzerland}
\email{konstantinos.tsinas@epfl.ch}

\thanks{The first author was partially funded by Centro de Modelamiento Matemático (CMM) FB210005, BASAL funds for centers of excellence from ANID-Chile and ANID/Fondecyt/1241346. The second author was supported by the ``Excellence in Research'' program of the Special Account for Research Funds AUTh (Code 10316).  The third author was supported by the NCN Polonez Bis 3 grant No. 2022/47/P/ST1/00854 (H2020 MSCA GA No. 945339); for the initial stages of the project, the third author was supported by the ELIDEK Grant No: 1684. The fourth author was partially supported by the NSF Grant DMS-2247331. The fifth author was initially funded by ELIDEK--Fellowship number 5367 (3rd Call for HFRI Ph.D. Fellowships) and currently supported by the Swiss National Science Foundation grant TMSGI2-211214. For the purpose of Open Access, the authors have applied a CC-BY public copyright licence to any Author Accepted Manuscript (AAM) version arising from this submission.}



\begin{document}

\begin{abstract}

    We develop a robust structure theory for multiple ergodic averages of commuting transformations along Hardy sequences of polynomial growth. We then apply it to derive a number of novel results on joint ergodicity, recurrence and convergence. 
    In particular, we prove joint ergodicity for (a) pairwise independent Hardy sequences and weakly mixing transformations, (b) strongly independent Hardy sequences and ergodic transformations, (c) strongly irrationally independent Hardy sequences and totally ergodic transformations. 
     We use these joint ergodicity results to provide new recurrence results for multidimensional patterns along strongly independent Hardy sequences, showing for instance that all subsets of $\Z^2$ of positive upper density contain patterns of the form
    $$    (m_1, m_2),\; (m_1 + \sfloor{n^{\sqrt{2}}}, m_2),\; (m_1, m_2 + \sfloor{n^{\sqrt{2}} + n^{1/2}}).$$
    Last but not least, we positively resolve the joint ergodicity classification problem for pairwise independent Hardy sequences, of which the aforementioned families are special cases.
    
    
    While building on recent technical advances (e.g. PET coefficient tracking schemes and joint ergodicity criteria), our work introduces a number of technical developments of its own. We construct a suitable generalization of Host-Kra and box seminorms that quantitatively control ergodic averages along Hardy sequences. 
    We subsequently use them to obtain Host-Kra seminorm estimates for averages along all pairwise independent Hardy sequences. Furthermore, we develop an ergodic version of the quantitative concatenation argument that has recently found extensive use in combinatorics, number theory and harmonic analysis. Lastly, we obtain new simultaneous Taylor approximations for Hardy sequences, a crucial ingredient to deal with the aforementioned classes of Hardy sequences.

    
\end{abstract}

\maketitle

\tableofcontents

\section{Introduction}
Furstenberg's celebrated proof of the Szemer\'edi theorem \cite{Fu77} initiated the study of the limiting behavior of multiple ergodic averages, a rich field within ergodic theory with profound connections to combinatorics, number theory, Lie groups and harmonic analysis. Given a multiple ergodic average, one is interested in understanding whether the average converges (in norm or pointwise) and what recurrence properties it satisfies.  For either purpose, one usually needs to have some structural description of the average; that is, one needs to identify a (useful) seminorm and/or a factor that controls its limiting behavior.

In this paper, we develop a robust machinery to examine the $L^2(\mu)$-limiting behavior of multiple ergodic averages of the form  
\begin{align}\label{E: general average}
    \E_{n\in[N]}\Bigbrac{\prod_{i=1}^k T_i^{\floor{a_{1i}(n)}}}f_1\cdots \Bigbrac{\prod_{i=1}^k T_i^{\floor{a_{\ell i}(n)}}}f_\ell
\end{align}
(here, $\E\limits_{n\in[N]}:=\frac{1}{N}\sum\limits_{n=1}^N$ as customary in the field),
where $(X, \CX, \mu, T_1, \ldots, T_k)$ is a system of commuting invertible measure-preserving transformations on a standard probability space $(X, \CX, \mu)$ (henceforth referred to as \textit{system}); $f_1, \ldots, f_\ell$ are $L^\infty(\mu)$ functions; and $a_{ji}:\R\to\R$ belong to the broad class $\CH$ of \textit{Hardy functions of polynomial growth}, defined in Section \ref{SS: Hardy field functions}, that includes real polynomials (e.g. $\sqrt{2}t^2 + t$), fractional powers (e.g. $t^{3/2}$), and functions involving exponentials and logarithms (e.g. $t\log t$ or $\exp(\sqrt{\log t})$). As a special subclass of interest, we study averages 
\begin{align}\label{E: general average 2}
    \E_{n\in[N]} T_1^{\floor{a_1(n)}}f_1\cdots T_\ell^{\floor{a_\ell(n)}}f_\ell
\end{align}
exemplified by
\begin{align}
\label{E: average example 1.5}
            \E_{n\in[N]} T_1^{\sfloor{n^{\sqrt{2}}}}f_1\cdot T_2^{\sfloor{n^{\sqrt{2}}+n^{1/2}}}f_2;
\end{align}
for the latter average, no convergence or recurrence properties have been known so far except when $T_1 = T_2$. For those broad classes of averages, we obtain novel seminorm estimates and use them to derive new ergodic identities and combinatorial applications that significantly generalize recent advances on multiple ergodic averages of commuting transformations along polynomials and Hardy sequences \cite{DFKS22, DKS22, DKS23, Fr21, FrKu22a, KoTs23, Ts22}. 

{As an example of our results, we not only show that \eqref{E: average example 1.5} converges in $L^2(\mu)$, but we in fact find its limit to explicitly equal $\E(f_1|\CI(T_1))\cdot \E(f_2|\CI(T_2))$, where $\E(f_j|\CI(T_j))$ is the conditional expectation of $f_j$ onto the $\sigma$-algebra of $T_j$-invariant functions. When $T_1, T_2$ are ergodic, this simplifies to $\int f_1\; d\mu\cdot \int f_2\; d\mu$. We then use this limiting formula in conjunction with the Furstenberg correspondence principle to deduce that each subset of $\Z^2$ of positive upper density contains a configuration 
\begin{align*}
    (m_1, m_2),\; (m_1 + \sfloor{n^{\sqrt{2}}}, m_2),\; (m_1, m_2 + \sfloor{n^{\sqrt{2}} + n^{1/2}})
\end{align*}
for some $n\neq 0$ (in fact, we can ensure that $n$ is a prime number).
}

\subsection{Joint ergodicity}\label{SS: joint ergodicity results}
Our most exciting results concern \textit{joint ergodicity}, whose definition we recall now.
\begin{definition}[Joint ergodicity]
    Sequences $a_1, \ldots, a_\ell:\Z\to\R$ are \textit{jointly ergodic} for a system $(X, \CX, \mu, T_1, \ldots, T_\ell)$ if     
\begin{align}\label{E: joint ergodicity}
        \lim_{N\to\infty}\norm{\E_{n\in[N]}\prod_{j=1}^\ell T_j^{\floor{a_{j}(n)}}f_j - \prod_{j=1}^\ell \int f_j\, d\mu}_{L^2(\mu)} = 0
    \end{align}
    holds for all $f_1, \ldots, f_\ell\in L^\infty(\mu)$.
\end{definition}
 Introduced in the 1980s by Berend and Bergelson \cite{BB84}, joint ergodicity and its variants have been a subject of intense scrutiny in the last few years \cite{Ber87, BeHK09, BLS16, BMR20, BS22, BS23, BFM22, CFH11, DFKS22, DKS22, DKS23, DKS24, Fr10, Fr12, Fr22, Fr21, FrKr05,  FrKu22b, FrKu22c, FrKu22a, KK19, Kouts21,  KoSu23, KoTs23, Ts22}.  We derive a number of new joint ergodicity results on Hardy sequences, which vary in terms of how we balance the assumptions imposed on Hardy sequences and the system. First, we list various independence conditions on Hardy sequences, {each of which gives rise to new joint ergodicity results under suitable  assumptions on the system.}
\begin{definition}[Independence properties of Hardy sequences]\label{D: independence}
Let $a_1, \ldots, a_\ell\in\CH$. We call them:
\begin{enumerate}
\item \textit{pairwise independent} if for any $i\neq j$ and $c_i, c_j\in\R$ not both zero, we have
\begin{align*}
    \lim_{t\to\infty}\frac{|c_i a_i(t) + c_j a_j(t)|}{\log t}=\infty;
\end{align*}
\item \textit{strongly independent} if every nontrivial linear combination of $a_1, \ldots, a_\ell$ stays logarithmically away from rational polynomials, i.e. for all $c_1, \ldots, c_\ell\in\R$ not all zero and for all $p\in\Q[t]$, we have
\begin{equation}\label{E: lafrp}
    \lim_{t\to\infty}\frac{|c_1 a_1(t) + \cdots + c_\ell a_\ell(t) - p(t)|}{\log t} = \infty;
\end{equation}
    \item \textit{strongly irrationally independent} if \eqref{E: lafrp} holds for all $c_1, \ldots, c_\ell\in\R$ not all rational and for all $p\in\Q[t]$;
    \item  \textit{irrationally independent} if \eqref{E: lafrp} holds for all $c_1, \ldots, c_\ell\in \mathbb{R}\backslash\mathbb{Q}_{\ast}:=(\mathbb{R}\backslash\mathbb{Q})\cup\{0\}$ not all zero and for all $p\in\Q[t]$.
\end{enumerate}
\end{definition}

Clearly, strong independence implies strong irrational independence,  which implies irrational independence, which in turn implies pairwise independence. For instance,
 \begin{enumerate}
     \item the family $t^{\sqrt{2}},\; t^{\sqrt{2}}+t^{1/2}$ from \eqref{E: average example 1.5} is strongly independent; same with $t^{3/2},\; t^{3/2}+t^{1/2}$ or any other family of linearly independent \textit{fractional polynomials} $a_j(t) = \sum\limits_{i=1}^d \alpha_{ji}t^{b_i}$ for \emph{noninteger} positive real powers $b_1, \ldots, b_d$.
     \item $t^{\sqrt{2}},\; t^{\sqrt{2}}+t$ is a  strongly irrationally independent family that is not strongly independent; more generally, families of linearly independent {fractional polynomials} $a_j(t) = \sum\limits_{i=1}^d \alpha_{ji}t^{b_i}$ for (possibly integer) positive reals $b_1, \ldots, b_d$ and rational coefficients $\alpha_{ji}$ will be strongly irrationally independent, but not necessarily strongly independent;
     \item $t^3+at^2+a^2t,$ $t^2+at,$ for $a\in\mathbb{R}\backslash\mathbb{Q},$ is an irrationally independent family that is not strongly irrationally independent;
     \item linearly independent real polynomials are pairwise independent, but they need not even be strongly irrationally independent (consider e.g. $\sqrt{2}t^2,\; \sqrt{2}t^2 + \sqrt{3} t$).
 \end{enumerate}

{While technical, the properties from Definition \ref{D: independence} are natural for our purposes, as:
\begin{itemize}
    \item the ergodic averages along sequences satisfying condition (i) admit Host-Kra seminorm estimates (see Theorem \ref{T: HK control}); {that is because under condition (i), analytic maneuvers applied to obtain Theorem \ref{T: HK control} allow us to distinguish between $f_i$ and $f_j$ for $i\neq j$;}
    \item condition (ii) characterizes those sequences whose nontrivial linear combinations are equidistributed modulo 1 (see Theorem~\ref{T: Boshernitzan} below) and are therefore predestined to be jointly ergodic in the light of the joint ergodicity criteria (see Theorem \ref{T: joint ergodicity criteria});
    \item condition (iii) is a sufficient condition for joint ergodicity in totally ergodic systems (see Theorem \ref{T: joint ergodicity for strongly independent Hardy sequences}, which extends an earlier result of Frantzikinakis \cite[Theorem~1.9 (i)]{Fr21}); and  
    \item condition (iv) is a sufficient condition for (total) joint ergodicity in totally ergodic systems (see the second part of Theorem \ref{T: joint ergodicity for strongly independent Hardy sequences}, which extends \cite[Theorem~E]{KoSu23} of the second and fourth authors).
\end{itemize}
}


Our first joint ergodicity result deals with pairwise independent Hardy sequences and weakly mixing transformations. We emphasize that our definition of $\CH$ from Section \ref{SS: Hardy field functions} includes the assumption of polynomial growth.
    \begin{theorem}[Joint ergodicity for pairwise independent Hardy sequences and weakly mixing transformations]\label{T: Joint ergodicity for weakly mixing transformations}
     Let $\ell\in\N$ and $a_1, \ldots, a_\ell\in \CH$ be pairwise independent. Then $a_1, \ldots, a_\ell$ are jointly ergodic for all systems $(X, \CX, \mu, T_1, \ldots, T_\ell)$ in which $T_1, \ldots, T_\ell$ are all weakly mixing. 
    \end{theorem}
While known for integer polynomials \cite{Ber87, FrKu22a}, Theorem \ref{T: Joint ergodicity for weakly mixing transformations} is completely new even for pairwise independent fractional polynomials, let alone pairwise independent elements of $\CH$ in general.

The assumption that pairwise linear combinations of $a_1, \ldots, a_\ell$ ``stay away'' from $\log$ is necessary because functions growing slower than $\log$ tend to behave badly with regards to norm convergence.
For instance, the averages $$\E_{n\in [N]} T^{\floor{\log_2 n}}f \quad \text{or } \quad \E_{n\in [N]} T^{\floor{n+\log_2 n}}f_1\cdot T^n f_2$$ {fail to converge in general by considering eigenfunctions of $T$ and breaking the average into dyadic intervals}. {Without the pairwise independence assumption, one cannot hope for a useful structural description of the average without making rather restrictive assumptions on the system.}

For strongly independent families of Hardy sequences, we obtain the following joint ergodicity result. In contrast to Theorem \ref{T: Joint ergodicity for weakly mixing transformations}, it makes no mixing assumptions on the system.
    \begin{theorem}[Joint ergodicity for strongly independent Hardy sequences]\label{T: joint ergodicity for strongly independent Hardy sequences}
     Let $\ell\in\N$ and $a_1, \ldots, a_\ell\in \CH$ be strongly independent. Then for all systems $(X, \CX, \mu, T_1, \ldots, T_\ell)$ and functions $f_1, \ldots, f_\ell\in L^\infty(\mu)$, we have
	\begin{align*}
        \lim_{N\to\infty}\norm{\E_{n\in[N]}\prod_{j=1}^\ell T_j^{\floor{a_{j}(n)}}f_j - \prod_{j=1}^\ell \E(f_j|\CI(T_j))}_{L^2(\mu)} = 0,
	\end{align*}
	where $\E(f_j|\CI(T_j))$ is the orthogonal projection of $f_j$ onto the invariant factor of $T_j$. 
     
     In particular, $a_1, \ldots, a_\ell$ are jointly ergodic for all systems $(X, \CX, \mu, T_1, \ldots, T_\ell)$ with $T_1, \ldots, T_\ell$ being ergodic.
 	\end{theorem}

  	As a special case, Theorem~\ref{T: joint ergodicity for strongly independent Hardy sequences} applied to \eqref{E: average example 1.5} implies that 
   \begin{align*}
       \lim_{N\to\infty}\norm{\E_{n\in[N]} T_1^{\sfloor{n^{\sqrt{2}}}}f_1\cdot T_2^{\sfloor{n^{\sqrt{2}}+n^{1/2}}}f_2 - \int f_1\; d\mu\cdot \int f_2\; d\mu}_{L^2(\mu)} = 0
   \end{align*}
   if $T_1$ and $T_2$ are ergodic, a result not known before. More generally, it yields joint ergodicity for arbitrary linearly independent fractional polynomials with noninteger powers. 

   The relevance of strong independence comes from the equidistribution theorem of Boshernitzan, which implies that every nontrivial linear combination of strongly independent Hardy sequences is equidistributed.

\begin{theorem}[{\cite[Theorem 1.3]{Boshernitzan-equidistribution}}]\label{T: Boshernitzan}
    Let $a\in\mathcal{H}$. The sequence $(a(n))_{n\in \N}$ is equidistributed $\mod~\!1$ if and only if for every $p\in \Q[t]$ we have \[\lim_{t\to \infty}\frac{|a(t)-p(t)|}{\log t}=\infty.\]
\end{theorem}

Theorems \ref{T: Joint ergodicity for weakly mixing transformations} and \ref{T: joint ergodicity for strongly independent Hardy sequences} vastly generalize previously known joint ergodicity results that, broadly speaking, have dealt with
integer polynomials \cite{Ber87, CFH11, DFKS22, FrKr05, FrKu22a}, real polynomials \cite{KK19, Ko18}, and Hardy field functions \cite{BMR20, Fr10, Fr12, Fr21,  Ts22}, under various distinctness assumptions (in some cases for a single transformation, in other for multiple ones).
In particular, they imply norm convergence for a large class of ergodic averages for which no convergence results have been known before, an example being again the average \eqref{E: average example 1.5}.

 For strongly irrationally independent Hardy sequences, we obtain the following variant of Theorem \ref{T: joint ergodicity for strongly independent Hardy sequences}.

\begin{theorem}[Rational Kronecker control for strongly irrationally independent Hardy sequences]\label{T: Krat control}
    Let $\ell\in\N$ and $a_1, \ldots, a_\ell\in \CH$ be strongly irrationally independent. Then for all systems $(X, \CX, \mu, T_1, \ldots, T_\ell)$ and functions $f_1, \ldots, f_\ell\in L^\infty(\mu)$, we have
 \begin{align*}
        \lim_{N\to\infty}\norm{\E_{n\in[N]}\prod_{j=1}^\ell T_j^{\floor{a_{j}(n)}}f_j-\prod_{j=1}^\ell T_j^{\floor{a_{j}(n)}}\E(f_j|\Krat(T_j))}_{L^2(\mu)} = 0
	\end{align*}
  Here, $\E(f_j|\Krat(T_j))$ is the orthogonal projection of $f_j$ onto the rational Kronecker factor $\Krat(T_j)$.

      In particular, $a_1, \ldots, a_\ell$ are jointly ergodic for all systems $(X, \CX, \mu, T_1, \ldots, T_\ell)$ with $T_1, \ldots, T_\ell$ totally ergodic. The latter conclusion also holds if $a_1, \ldots, a_\ell$ are irrationally independent rather than strongly irrationally independent.\footnote{ This last conclusion extends \cite[Theorem~A]{KoSu23} to multiple commuting transformations.}
\end{theorem}

The appearance of the rational Kronecker factors comes from the fact that while the exponential averages of $c_1 a_1(n) + \cdots + c_\ell a_\ell(n)$ with at least one $c_j$ irrational will always vanish by Theorem \ref{T: Boshernitzan}, the same may not be the case when all $c_j$'s are rational.

Theorem~\ref{T: Krat control} is a multidimensional generalization of \cite[Theorem 1.9]{Fr21}. Since the family of strongly irrationally independent Hardy sequences includes linearly independent polynomials $a_1, \ldots, a_\ell\in\Z[t]$, Theorem~\ref{T: Krat control} contains as special cases all previously known results for linearly independent integer polynomials \cite{CFH11, FrKr05, FrKr06, FrKu22a}.

Before proceeding to state more results, we remark that the occurrences of floor functions in the previous statements can be individually replaced by other rounding functions, such as the ceiling function (denoted by $\lceil\cdot\rceil$) or the closest integer function (denoted by $[[\cdot]]$)  since $\lceil  t\rceil=-\floor{-t}$ and $[[t]]=\floor{t+1/2}$ for all $t\in\R.$ This is also the case for various other results in this paper, a fact that we will not explicitly mention again.
     
So far, we have stated results for broad, naturally occurring classes of Hardy sequences under either no assumption on the system or some fixed assumption (weak mixing or total ergodicity of individual transformations). Now, we switch gears and move towards establishing necessary and sufficient conditions on the system under which pairwise independent Hardy sequences are jointly ergodic. The definition below will frequently appear in the following discussion. 
    \begin{definition}[Ergodic sequences of transformations]
    A sequence of measure-preserving transformations $(T_n)_n$ on a probability space $(X, \CX, \mu)$ is \textit{ergodic} if 
    \[  \lim_{N\to\infty}\norm{\E_{n\in[N]} T_n f - \int f\, d\mu}_{L^2(\mu)} = 0    \]     
    for all $f\in L^\infty(\mu)$.
    \end{definition}
    
    There is a general heuristic that proposes a set of ``natural'' necessary and sufficient conditions for joint ergodicity. The problem below, stated before as \cite[Problem~1]{DKS23} (for the general case) or \cite[Conjecture 6.1]{BMR20} (for fractional polynomials), inquires about the extent to which this heuristic holds.

    \begin{problem}[Joint ergodicity classification problem]\label{Pr: joint ergodicity problem}
    Let $\ell\in\N$. Classify the sequences $a_1, \ldots, a_\ell:\N\to\R$ that are jointly ergodic for an (arbitrary) system  $(X, \CX, \mu, T_1, \ldots, T_\ell)$ if and only if the following two conditions hold:
    \begin{enumerate}
    \item \emph{(Difference ergodicity condition)} $(T_i^{\floor{a_i(n)}}T_j^{-\floor{a_j(n)}})_n$ is ergodic on $(X, \CX, \mu)$ for all distinct $i, j\in[\ell]$; and
    \item \emph{(Product ergodicity condition)} $(T_1^{\floor{a_1(n)}}\times \cdots \times T_\ell^{\floor{a_\ell(n)}})_n$ is ergodic on the product system $(X^\ell, \CX^{\otimes \ell}, \mu^\ell)$.
    \end{enumerate}
    \end{problem}
     The two conditions appearing in the problem are natural generalizations of the joint ergodicity conditions provided by  Berend and Bergelson in their pioneer work \cite{BB84}.
    
    Despite ongoing progress in the recent years \cite{BLS16, DFKS22, DKS22, DKS23, FrKu22b, FrKu22a}, the joint ergodicity classification problem remains challenging and open. 
    Our contribution to this question consists in the following.
    \begin{theorem}[{Resolution of Problem \ref{Pr: joint ergodicity problem} for pairwise independent Hardy sequences}]\label{T: joint ergodicity}
    
    
    {A family of pairwise independent Hardy sequences of polynomial growth is jointly ergodic for a system if and only if it satisfies the difference and product ergodicity conditions on this system.}
    \end{theorem}
    Theorem~\ref{T: joint ergodicity} contains both Theorem~\ref{T: Joint ergodicity for weakly mixing transformations} and the second parts of Theorem~\ref{T: joint ergodicity for strongly independent Hardy sequences} and \ref{T: Krat control} as special cases, and it vastly extends previously known joint ergodicity results.

\subsection{Applications to recurrence and combinatorics}\label{SS: recurrence}
As an immediate consequence, Theorem \ref{T: joint ergodicity for strongly independent Hardy sequences} gives a multiple recurrence result with an explicit lower bound which is optimal in that it cannot be improved for weakly mixing systems. 
\begin{corollary}[Multiple recurrence for strongly independent Hardy sequences]\label{C: lower bounds on recurrence}
	Let $\ell\in\N$ and $a_1, \ldots, a_\ell\in \CH$ be strongly independent. Then for all systems $(X, \CX, \mu, T_1, \ldots, T_\ell)$ and sets $E\in\CX$, we have
	\begin{align*}
	\lim_{N\to\infty}\E_{n\in[N]}\mu(E\cap T_1^{-\floor{a_1(n)}}E\cap \cdots \cap T_\ell^{-\floor{a_\ell(n)}}E) \geq (\mu(E))^{\ell+1}.
	\end{align*}
\end{corollary}


A standard application of the Furstenberg correspondence principle to Corollary \ref{C: lower bounds on recurrence} not only gives many strongly independent Hardy progressions in sets of positive upper density, but it implies the existence of so-called {\em popular common differences}, i.e. the differences $n$ for which the upper density of the intersection 
$$E\cap (E - \bv_1\floor{a_1(n)})\cap \cdots \cap (E - \bv_\ell\floor{a_\ell(n)})$$ is close to what it would be if the sets $E,\; E - \bv_1\floor{a_1(n)},\; \ldots,\; E - \bv_\ell\floor{a_\ell(n)}$ were independent.
\begin{corollary}[Hardy progressions in positive density sets]\label{C: popular differences}
	Let $\ell\in\N$ and $a_1, \ldots, a_\ell\in \CH$ be strongly independent. Then for all sets $E\subseteq\Z^k$ with positive upper density\footnote{The upper density of $E\subseteq\Z^k$ is defined by $\overline{d}(E) = \limsup\limits_{N\to\infty} \frac{|E\cap\{-N, \ldots, N\}^k|}{(2N+1)^k}$.} 
    and vectors $\bv_1, \ldots, \bv_\ell\in\Z^k$, we have
	\begin{align*}
		\lim_{N\to\infty}\E_{n\in[N]}\overline{d}(E\cap (E - \bv_1\floor{a_1(n)})\cap \cdots \cap (E - \bv_\ell\floor{a_\ell(n)})) \geq (\overline{d}(E))^{\ell+1}.
	\end{align*}
        In particular, $E$ contains 
	\[
	\bm,\; \bm + \bv_1 \floor{a_1(n)},\; \ldots,\; \bm +  \bv_\ell\floor{a_\ell(n)}
	\]
	for some $\bm\in \Z^{k}$ and $n\in\N$.
\end{corollary}

In particular, Corollaries~\ref{C: lower bounds on recurrence} and \ref{C: popular differences} give lower bounds on the number of patterns 
\begin{align*}
    (m_1, m_2),\; (m_1 + \sfloor{n^{\sqrt{2}}}, m_2),\; (m_1, m_2 + \sfloor{n^{\sqrt{2}} + n^{1/2}})
\end{align*}
in subsets of $\Z^2$ of positive upper density. The same conclusion follows for patterns involving strongly independent real polynomials such as
\begin{align*}
    (m_1, m_2),\; (m_1 + \sfloor{\sqrt{2}n^3+n^2}, m_2),\; (m_1, m_2 + \sfloor{\sqrt{3}n^3-n}).
\end{align*}

\subsection{Results along primes}\label{SS: primes}
Our new joint ergodicity and recurrence results can be combined with a transference principle from \cite{KoTs23} to deduce new ergodic and combinatorial results along primes. This transference principle itself is based on deep result on the Gowers uniformity of the modified von Mangoldt function on short intervals \cite{MSTT23}.
Studying multiple ergodic averages along primes has rich history, see e.g. \cite{BKS19, Bo88, Fr22, FHK07, FHK10, KK19, Ko18b, Kouts22, KoTs23, KMTT24, Na91, Na93, S15, Wierdl88, WZ12} for some earlier works. 
We denote the set of primes by $\P$ and the number of primes up to $N$ by $\pi(N)$. We will also use the averaging notation $$\E_{p\in \P\cap [N]} a_p=\frac{1}{\pi(N)} \sum_{p\in \P\colon p\leq N}a_p.$$

Our first theorem follows by combining Theorem \ref{T: joint ergodicity for strongly independent Hardy sequences} with \cite[Theorem 1.2]{KoTs23}. 

\begin{theorem}[Joint ergodicity for strongly independent Hardy sequences along primes]\label{T: joint ergodicity for strongly independent Hardy sequences along prime numbers}
Let $\ell\in\N$, $a_1, \ldots, a_\ell\in \CH$ be pairwise independent and $(X, \CX, \mu, T_1,$ $ \ldots, T_\ell)$ be a system. Suppose that one of the following conditions holds:
      \begin{enumerate}
          \item $a_1, \ldots, a_\ell$ are strongly independent; or
           \item $T_{1},\dots,T_{\ell}$ are weakly mixing.
      \end{enumerate}
      Then
	\begin{align*}
        \lim_{N\to\infty}\norm{\E_{p\in \P\cap [N]}\prod_{j=1}^\ell T_j^{\floor{a_{j}(p)}}f_j - \prod_{j=1}^\ell \E(f_j|\CI(T_j))}_{L^2(\mu)} = 0.
	\end{align*}
     In particular, for all systems $(X, \CX, \mu, T_1, \ldots, T_\ell)$ with $T_1, \ldots, T_\ell$ ergodic, we have \begin{equation*}
         \lim\limits_{N\to\infty}\norm{\E_{p\in \P\cap [N]}\prod_{j=1}^\ell T_j^{\floor{a_{j}(p)}}f_j-\prod_{j=1}^{\ell} \int f_j\ d\mu}_{L^2(\mu)} = 0.
     \end{equation*}
\end{theorem}
Applying the previous theorem to functions $f_1=\cdots=f_{\ell}={ 1}_E$ for a measurable set $E$, we deduce the following multiple recurrence result.

\begin{corollary}[Multiple recurrence for strongly independent Hardy sequences along primes]\label{C: multiple recuurence along prime numbers}
    Let $\ell\in\N$ and $a_1, \ldots, a_\ell\in \CH$ be strongly independent. Then for all systems $(X, \CX, \mu, T_1, \ldots, T_\ell)$ and sets $E\in\CX$, we have \begin{equation*}
        \lim\limits_{N\to\infty} \E_{p\in \P\cap [N]} \mu(E\cap T_1^{-\floor{a_1(p)}}E\cap\dots\cap T_{\ell}^{-\floor{a_{\ell}(p)}}E)\geq (\mu(E))^{\ell+1}.
    \end{equation*}
\end{corollary}

Once again, Furstenberg's correspondence principle translates the previous multiple recurrence result to the combinatorial setting.

\begin{corollary}[Prime configurations in positive density sets]\label{}
    	Let $\ell\in\N$ and $a_1, \ldots, a_\ell\in \CH$ be strongly independent. Then for all sets $E\subseteq\Z^k$ with positive upper density and vectors $\bv_1, \ldots, \bv_\ell\in\Z^k$, we have  
	\begin{align*}
		\lim\limits_{N\to\infty} \E_{p\in \P\cap[N]}\overline{d}(E\cap (E - \bv_1\floor{a_1(p)})\cap \cdots \cap (E - \bv_\ell\floor{a_\ell(p)})) \geq (\overline{d}(E))^{\ell+1}.
	\end{align*}
	 In particular, $E$ contains 
	\[
	\bm,\; \bm + \bv_1 \floor{a_1(p)},\; \ldots,\; \bm +  \bv_\ell\floor{a_\ell(p)}
	\]
	for some $\bm\in \Z^k$ and $p\in\P$.
\end{corollary}

\subsection{Seminorm estimates}

The main new input that allows us to prove joint ergodicity results 
and their corollaries 
are novel seminorm estimates on averages \eqref{E: general average} and \eqref{E: general average 2}, which we then combine with joint ergodicity criteria from \cite{FrKu22a} and equidistribution results for Hardy sequences 
to derive Theorems~\ref{T: Joint ergodicity for weakly mixing transformations}, \ref{T: joint ergodicity for strongly independent Hardy sequences}, and \ref{T: joint ergodicity}. Specifically, we develop a generalization of box seminorms robust enough to control the averages \eqref{E: general average} with polynomial bounds under some necessary nondegeneracy condition on the Hardy sequences. The gist of our main result is captured in the following rather informal statement: for its (somewhat complicated) precise version, we refer the reader to Theorem~\ref{T: box seminorm bound no duals}.

\begin{theorem}[Generalized box seminorm estimates]\label{T: box seminorm bound intro}
    Let $k,\ell\in\N$ and $\ba_j=(a_{j1},\dots,$ $a_{jk})\in \CH^k$ be $k$-tuples of functions in $\CH$ for $1\leq j\leq \ell$. 
    Suppose that for every $0\leq j\leq \ell$ different from $1$, the tuple $\ba_1 -\ba_j$ has a coordinate that grows faster than $\log$. 
    Then for every system $(X, \CX, \mu, T_1, \ldots, T_k)$ there exists a generalized box seminorm on $\nnorm{\cdot}$ on $L^\infty(\mu)$ of degree $O_{\ba_1, \ldots, \ba_\ell}(1)$ such that 
    \begin{align*}
            \limsup_{N\to\infty}\norm{\E_{n\in[N]}\Bigbrac{\prod_{i=1}^k T_i^{\floor{a_{1i}(n)}}}f_1\cdots \Bigbrac{\prod_{i=1}^k T_i^{\floor{a_{\ell i}(n)}}}f_\ell}_{L^2(\mu)}^{O_{\ba_1, \ldots, \ba_\ell}(1)}\ll_{\ba_1, \ldots, \ba_\ell} \nnorm{f_1}
    \end{align*}
    for all 1-bounded functions $f_1, \ldots, f_\ell\in L^\infty(\mu)$.
\end{theorem}
Generalized box seminorms and their properties are defined and discussed at length in Section \ref{S: seminorms}. The important point is that the seminorm appearing in Theorem~\ref{T: box seminorm bound intro} has ``bounded degree'' and depends only on the (appropriately defined) ``leading coefficients'' of the sequences $\ba_1, \ba_1 - \ba_2, \ldots, \ba_1 - \ba_\ell$. As such, Theorem \ref{T: box seminorm bound intro} vastly generalizes previous results of similar flavor \cite{DFKS22, DKS22, DKS23}. To hint at what these generalized box seminorms look like, we look at perhaps the simplest nontrivial special case of Theorem~\ref{T: box seminorm bound intro}. For a double ergodic average 
\begin{align*}
    \E_{n\in[N]} T_1^{\sfloor{\sqrt{2}n}}f_1\cdot T_2^{\sfloor{\sqrt{3}n}}f_2,
\end{align*}
Theorem \ref{T: box seminorm bound intro} gives the estimate
\begin{align*}
    \limsup_{N\to\infty}\norm{\E_{n\in[N]} T_1^{\sfloor{\sqrt{2}n}}f_1\cdot T_2^{\sfloor{\sqrt{3}n}}f_2}_{L^2(\mu)}^8 \ll \nnorm{f_2}
\end{align*}
for the seminorm
\begin{multline*}\nnorm{f} = 
\left(\lim_{M_2\to\infty}\;\E_{m_2, m_2'\in[M_2]}\;\lim_{M_1\to\infty}\;\E_{m_1, m_1'\in[M_1]}\right.\\
\left|\int T_1^{\sfloor{-\sqrt{2}m_1}} T_2^{\sfloor{\sqrt{3}m_1}+\sfloor{\sqrt{3}m_2}}f\cdot T_1^{\sfloor{-\sqrt{2}m_1}} T_2^{\sfloor{\sqrt{3}m_1}+\sfloor{\sqrt{3}m_2'}}\overline{f}\right.\\
\left.\left. T_1^{\sfloor{-\sqrt{2}m_1'}} T_2^{\sfloor{\sqrt{3}m_1'}+\sfloor{\sqrt{3}m_2}}\overline{f}\cdot T_1^{\sfloor{-\sqrt{2}m_1'}} T_2^{\sfloor{\sqrt{3}m_1'}+\sfloor{\sqrt{3}m_2'}}f\; d\mu\right|^2\right)^{1/8}.
\end{multline*}
One should think of this seminorm as a box-type seminorm along real subgroups $\langle (0, \sqrt{3})\rangle$ and $\langle (-\sqrt{2}, \sqrt{3})\rangle$, the former of which corresponds to the action $(T_2^{\sfloor{\sqrt{3}n}})_n$ and the latter to the action of $(T_1^{\sfloor{-\sqrt{2}n}} T_2^{\sfloor{\sqrt{3}n}})_n$.
The existence of the limits in the definition of the seminorm is derived from known results in Section \ref{S: seminorms}.

For averages \eqref{E: general average 2} along pairwise independent Hardy sequences such as \eqref{E: average example 1.5}, we can improve the conclusion of Theorem~\ref{T: box seminorm bound intro} to a proper Host-Kra seminorm control.

\begin{theorem}[Host-Kra seminorm estimates for pairwise independent Hardy sequences]\label{T: HK control}
    Let $\ell\in\N$ and $a_1, \ldots, a_\ell\in \CH$ be pairwise independent. Then there exists a positive integer $s=O_{a_1, \ldots, a_\ell}(1)$ such that for all systems $(X, \CX, \mu, T_1, \ldots, T_\ell)$ and 1-bounded functions $f_1, \ldots, f_\ell\in L^\infty(\mu)$, we have
    \begin{align*}
            \limsup_{N\to\infty}\norm{    \E_{n\in[N]} T_1^{\floor{a_1(n)}}f_1\cdots T_\ell^{\floor{a_\ell(n)}}f_\ell}_{L^2(\mu)}^{O_{a_1, \ldots, a_\ell}(1)}\ll_{a_1, \ldots, a_\ell} \min_{1\leq j\leq \ell}\nnorm{f_j}_{s, T_j}.
    \end{align*}
\end{theorem}

{In contrast to earlier results for polynomials, Theorem \ref{T: HK control} cannot be strengthened to averages along more general F{\o}lner sequences. See e.g. \cite{BMR20a} for an explanation of issues arising in this case.}

A reader may notice that unlike most seminorm estimates in ergodic theory, ours are not only quantitative, but they also come with good polynomial bounds. To the best of our knowledge, no such quantitative seminorm estimates exist in ergodic literature even for averages along integer polynomials.\footnote{For averages \eqref{E: general average 2} of integer polynomials along a single transformation $T_1 = \cdots = T_\ell$, a quantitative seminorm estimate could likely be proved by carefully quantifying the usual PET argument, and this would avoid the use of concatenation tools ubiquitous in this article. However, no avoidance of concatenation seems possible for quantitative estimates for averages involving several transformations.} 
We state Theorems \ref{T: box seminorm bound intro} and \ref{T: HK control} quantitatively not only because we can; the quantitative nature of Theorem \ref{T: box seminorm bound intro} gives considerable advantage in the proof of Theorem \ref{T: HK control}, even if one were only interested in the qualitative control implied by Theorem \ref{T: HK control}. Indeed, the first proof of Host-Kra seminorm control for averages \eqref{E: general average 2} along integer polynomials obtained in \cite{FrKu22a} necessitated a quantification of (qualitative) box seminorm control for averages \eqref{E: general average} along integer polynomials from \cite{DFKS22}. The quantification used in \cite{FrKu22a} was completely ineffective and relied on an abstract functional analytic argument that turned qualitative seminorm control of a multilinear operator into a soft quantitative one. By contrast, our proof of Theorem \ref{T: HK control} is completely effective and therefore recovers one of the main results from \cite{FrKu22a} using arguably more elementary tools. 

While proving seminorm estimates for ergodic averages along Hardy sequences (e.g. \cite{BMR20, Fr10, Fr12, Ts22}, one splits the average over $[N]$ into one over short intervals, and then Taylor expands the Hardy sequences on short intervals as polynomials (it is this step that requires Hardy sequences to have polynomial growth).  Therefore, an important tool in all previous seminorm estimates (e.g. \cite{BMR20, Fr10, Fr12, Ts22}) are simultaneous Taylor approximations for Hardy sequences that make it possible to control the degrees and leading coefficients of approximating polynomials. As it turns out, the existing approximations are far too weak to deal with commuting transformations. One of our key technical advances is Proposition \ref{P: Taylor expansion ultimate}, a new result on simultaneous Taylor approximations. The content of this result and its role in the overall strategy is outlined in Section \ref{SSS: approximations}

If one quantitatively strengthened the joint ergodicity criteria from \cite{FrKu22a}, one could combine them with our Theorem \ref{T: HK control} to give an alternative proof for the rational Kronecker control for linearly independent integer polynomials along commuting transformations from \cite{FrKu22a} that would avoid altogether the Host-Kra structure theorem. As our ambitions have been exhausted in different directions, we leave the details to the interested reader.

As an additional application, the quantitative seminorm estimates provided by Theorem \ref{T: HK control} can be used to recover Gowers norm estimates for multidimensional integer polynomial progressions proved recently in \cite{Kuc23}. Again, we do not provide the details of this derivation.

Last but not least, we emphasize that the proof of Theorem~\ref{T: HK control} only requires Walsh's result from \cite{W12} for commuting transformations with mutlivariable linear iterates (see also \cite{ZK16}), its corresponding strengthening to averages along real linear polynomials which follows from \cite{Ko18}, and the Fubini-type principle for ergodic averages \cite{BL15} (we need these ingredients to show that the limits in the definition of generalized box seminorms are well-defined and can be permuted). Hence it is highly possible that the only required norm convergence results are pre-Walsh norm convergence of ergodic averages along integer linear  polynomials \cite{A10b, H09, T08}.  Combined with the Host-Kra structure theorem \cite{HK05a} and Leibman's results on the convergence of polynomial averages on nilmanifolds \cite{L05a, L05b}, Theorem~\ref{T: HK control} gives an alternative proof of the norm convergence of ergodic averages of commuting transformations along pairwise independent integer polynomials.\footnote{By contrast, the qualitative Host-Kra seminorm control for averages along pairwise independent integer polynomials proved in \cite{FrKu22a}, as well as the general convergence result for real polynomial iterates from \cite{Ko18}, crucially use more general implications of Walsh's result.}

\subsection{Further applications of our result}
Since the release of this manuscript, Theorem \ref{T: box seminorm bound intro} and its corollaries have played a key part in two important developments \cite{DKKST25, FrKu25}. In both cases, Theorem \ref{T: box seminorm bound intro} supplied initial seminorm estimates for the averages under consideration.
These original estimates were then improved to optimal ones using techniques that depended on the particular problem.

In \cite{DKKST25}, we used Theorem \ref{T: box seminorm bound intro} to resolve Problem \ref{Pr: joint ergodicity problem} for all Hardy sequences of polynomial growth. In particular, we extended Theorem \ref{T: joint ergodicity} to all families of Hardy sequences that are ``nondegenerate'' in some suitable sense. We also constructed a counterexample of ``degenerate'' Hardy sequences, which is jointly ergodic for some system but fails to satisfy the difference ergodicity condition.

In \cite{FrKu25}, identities like
\begin{align*}
    \lim_{N\to\infty}\norm{\E_{n\in[N]}T_1^{\sfloor{n^{3/2}}}f_1 \cdot T_2^{\sfloor{n^{3/2}}}f_2 - \E_{n\in[N]}T_1^nf_1 \cdot T_2^n f_2}_{L^2(\mu)} = 0
\end{align*}
have been established by using the following special case of Theorem \ref{T: box seminorm bound intro} in which all the iterates are multiples of the same sequence (this is the setting orthogonal to that of Theorem \ref{T: HK control}).
\begin{corollary}[Seminorm estimates in the pairwise dependent case]\label{T: bound for pairwise dependent}
        Let $\ell\in\N$, and let $a\in\CH$ satisfy $a\succ \log$. Then there exists a positive integer $s=O_{a,\ell}(1)$ such that for all systems $(X, \CX, \mu, T_1, \ldots, T_\ell)$ and 1-bounded functions $f_1, \ldots, f_\ell\in L^\infty(\mu)$, we have
    \begin{align*}
            \limsup_{N\to\infty}\norm{\E_{n\in[N]} T_1^{\floor{a(n)}}f_1\cdots T_\ell^{\floor{a(n)}}f_\ell}_{L^2(\mu)}^{O_{a, \ell}(1)}\ll_{a,\ell} \nnorm{f_1}_{\be_1^{\times s}, (\be_1-\be_2)^{\times s}, \ldots, (\be_1-\be_\ell)^{\times s}}.
    \end{align*}
\end{corollary}
Corollary \ref{T: bound for pairwise dependent} has previously been known only when $a$ is an integer polynomial \cite{DFKS22} or a strongly nonpolynomial Hardy field function \cite{DKS23}. It is therefore new even for real polynomials like $\sqrt{2}t^2+t$. 

More generally, if we replace each $T_j^{\floor{a(n)}}$ in the average by $\prod_{i=1}^k T_i^{\floor{\alpha_{ji}a(n)}}$ for some $\balpha_j = (\alpha_{j1}, \ldots, \alpha_{jk})\in\R^k$, then Theorem \ref{T: HK control} gives control by a seminorm in which every $\be_j$ is replaced by $\balpha_j$.

\subsection{Outline of the paper and new ideas}

The paper is structured as follows. After presenting our notation and the basic properties of Hardy sequences in Section~\ref{S: background}, we move on to define generalized box seminorms and derive its various properties in Section~\ref{S: seminorms}. We then dedicate Section~\ref{S: concatenation} to a quantitative concatenation result for averages of generalized box seminorms that emulates similar results proved recently in the finitary setting \cite{GS24, KKL24a, Kuc23, Pel20, PP19}. Its culmination are Proposition~\ref{P: concatenation for general groups} and Corollary~\ref{C: concatenation for general groups II}, which can be seen as the quantitative analogues of the Bessel-type inequality from \cite[Corollary 1.22]{TZ16}, as well as Corollary~\ref{C: iterated concatenation for general groups}, tailor-made for subsequent applications. After some preliminary lemmas in Section~\ref{S: preliminaries}, the bulk of the paper (Sections~\ref{S: linear and sublinear}-\ref{S: general case}) is then devoted to proving Theorems~\ref{T: box seminorm bound no duals} and \ref{T: box seminorm bound}, precise versions of Theorem~\ref{T: box seminorm bound intro}, which assert that generalized box seminorms control averages \eqref{E: general average} with polynomial bounds. We then upgrade, in Section~\ref{S: HK control}, the generalized box seminorm estimates to the Host-Kra seminorm estimates for pairwise independent Hardy sequences, in effect proving Theorem \ref{T: HK control}. In Section~\ref{S: joint ergodicity proofs}, we use these estimates to derive joint ergodicity results. Lastly, Appendix \ref{A: Hardy properties} list various elementary properties of Hardy sequences while Appendix \ref{A: approximations} proves Proposition \ref{P: Taylor expansion ultimate}, a new result on simultaneous Taylor approximations of Hardy sequences, which is of independent interest.

\subsubsection{Quantitative concatenation}

Here, we outline the main arguments together with new ideas that come into the proofs, starting with quantitative concatenation. The basic idea behind concatenation  is that averages of generalized box seminorms of the form 
\begin{align*}
    \E_{i\in I}\nnorm{f}_{G_{1i}, \ldots, G_{si}}^{2^s},
\end{align*}
where $I$ is a finite indexing set and $\nnorm{\cdot}_{G_{1i}, \ldots, G_{si}}$ 
is a generalized box seminorm along finitely generated subgroups $G_{1i}, \ldots, G_{si}\subseteq\R^k$, are controlled by an average of generalized box seminorms along subgroups $G_{ji_1}+\cdots + G_{ji_\ell}$ for $1\leq j\leq s$ {and indices $i_1, \ldots, i_\ell\in I$ which typically will be pairwise distinct}. 

The utility of such control lies in the fact that the larger subgroups $G_{1i_1}+\cdots + G_{1i_\ell}$ often admit a more explicit description than the original ones. As a model example, the arguments from Section \ref{S: concatenation} show that the average
\begin{align}\label{E: 1-deg concat ex}
    \E_{h\in[\pm H]} \nnorm{f}_{\langle\sqrt{2} h\rangle}^2 := \E_{h\in[\pm H]}\; \lim_{M\to\infty}\;\E_{m, m'\in[\pm M]} \int T^{\floor{\sqrt{2}h m}}f \cdot T^{\floor{\sqrt{2}h m'}}\overline{f}\; d\mu
\end{align}
(where $[\pm H]:=\{-H, -H+1, \ldots, H-1, H\}$)
can be controlled with polynomial bounds by an average like
\begin{align*}
    \E_{h_1, h_2, h_3, h_4\in[\pm H]} \nnorm{f}_{\langle\sqrt{2} h_1\rangle+\cdots + \langle\sqrt{2} h_4\rangle}^{16}.
\end{align*}
One can then show that for almost all tuples $(h_1, h_2, h_3, h_4)\in\Z^4$, the group $$\langle\sqrt{2} h_1\rangle+\cdots + \langle\sqrt{2} h_4\rangle :=\langle \sqrt{2}(m_1 h_1 + m_2 h_2 + m_3 h_3 + m_4 h_4):\; m_1, m_2, m_3, m_4\in\Z\rangle$$ is close (in an appropriate sense) to the group $\langle \sqrt{2}\rangle$. Hence, once $H$ is large enough, the average \eqref{E: 1-deg concat ex} can be quantitatively bounded by a single seminorm along the group $\langle \sqrt{2}\rangle$. 

Section \ref{S: concatenation} provides the first quantitative concatenation results in the ergodic setting, which can be seen as quantitative improvements on the original concatenation results from \cite{TZ16}. While our proofs largely follow the finite-field argument from \cite{Kuc23} (and its integer adaptation from \cite{KKL24a}), the transition to the ergodic setting requires nontrivial modifications of certain technical bits. The output of Section \ref{S: concatenation} is then crucially used in Section \ref{S: polynomial concatenation} to deliver Proposition \ref{P: Box seminorm control for polynomials twisted by duals}, the version of Theorem \ref{T: box seminorm bound intro} for polynomials, in a way that will be explained shortly.

\subsubsection{Seminorm estimates: overview}

As already mentioned, the lion's share of our effort goes into deducing Theorem~\ref{T: box seminorm bound} which gives generalized box seminorm estimates on averages along Hardy sequences. Theorem~\ref{T: box seminorm bound} is then used to derive Theorem~\ref{T: box seminorm bound no duals}, a precise formulation of Theorem~\ref{T: box seminorm bound intro}, as well as Theorem \ref{T: HK control}, Host-Kra seminorm estimates for averages along pairwise independent Hardy sequences. The proof of Theorem \ref{T: box seminorm bound} is modelled on existing works that obtain box seminorm control for polynomial averages of commuting transformations \cite{DFKS22} and their finitary analogues \cite{KKL24a, Kuc23}, as well as Host-Kra seminorm control for averages of a single transformation along Hardy sequences \cite{Ts22} and averages of commuting transformations along the same Hardy iterate \cite{DKS23}. Yet dealing with Hardy sequences \textit{and} commuting transformations simultaneously raises the level of technicalities to the heights unseen in the aforementioned papers. To illustrate the point with one telling example, the arguments from  \cite{DFKS22, DKS23, KKL24a, Kuc23} invoke concatenation results once while \cite{Ts22} does not use them at all; we need them at two separate occasions. 

The proof of seminorm estimate for averages \eqref{E: general average} (Theorems \ref{T: box seminorm bound intro}, \ref{T: box seminorm bound no duals} and \ref{T: box seminorm bound}) is split into five different cases depending on the shape of Hardy sequences, where the simpler cases are invoked in the proofs of the more difficult cases. Rewriting \eqref{E: general average} as
\begin{align}\label{E: general average rephrased}
    \E_{n\in[N]} T^{\floor{\ba_1(n)}}f_1 \cdots T^{\floor{\ba_\ell(n)}}f_\ell 
\end{align}
for $T^{\floor{\ba_j(n)}} := T_1^{\sfloor{a_{j1}(n)}}\cdots T_k^{\sfloor{a_{jk}(n)}}$, we can list these five cases as follows:
\begin{enumerate}
    \item $\ba_j$'s are linear, i.e. $a_{ji}(t) = \alpha_{ji} t$ for some $\alpha_{ji}\in\R$ for all $j,i$; 
    \item $\ba_j$'s are sublinear, i.e. $a_{ji}(t)\prec t$ for all $j, i$,  e.g.
    \begin{align}\label{E: sublinear example}
        \ba_1(t) = (\sqrt{2}t^{1/2}, (\log t)^2),\quad \ba_2(t)= (\log t, \sqrt{3}t^{1/2});
    \end{align}
    \item $\ba_j\in\R^k[t]$;
    \item $\ba_j$'s are sums of a sublinear and a polynomial term, e.g. 
    \begin{align}\label{E: sublinear + polynomial example}
        \ba_1(t) = (\sqrt{2}t^2, (\log t)^2),\quad \ba_2(t)= (t^{2/3}, \sqrt{3}t^{2});
    \end{align}
    \item $\ba_j$'s are arbitrary; this includes the iterates
    \begin{align*}
        \ba_1(t) = (t^{\sqrt{2}}, 0), \quad \ba_2(t) = (0, t^{\sqrt{2}} + t^{1/2})
    \end{align*}
    from \eqref{E: average example 1.5}.
\end{enumerate}

\subsubsection{Seminorm estimates for linear and sublinear sequences}

When $\ba_j$'s are linear, the proof consists of $\ell$ applications of the van der Corput inequality, much like in the by now classical proof of Host-Kra seminorm control for Furstenberg averages 
\begin{align*}
    \E_{n\in[N]} T^n f_1\cdots T^{\ell n} f_\ell
\end{align*}
from \cite{HK05a}. They are followed by an extra application of the Cauchy-Schwarz inequality needed to remove the errors accrued while comparing $\floor{\alpha (n+m)}$ and $\floor{\alpha n}+\floor{\alpha m}$. The linear case, handled in Lemma \ref{L: linear averages} and Corollary \ref{C: linear averages}, already showcases the need for a more robust notion of seminorms than the regular box seminorms.

When $\ba_j$'s are sublinear, the proof of the generalized box seminorm estimates (isolated as Proposition \ref{P: sublinear} and Corollary \ref{C: sublinear}) is fairly classical and amounts to reducing the problem to the linear case. The starting point is to decompose each $\ba_j$ as a $\R^k$-linear combination of some functions $g_1, \ldots, g_m\in\CH$ satisfying the growth condition
\begin{align*}
1\prec g_1(t) \prec \cdots \prec g_m(t) \prec t;    
\end{align*}
 for \eqref{E: sublinear example}, we would take $$g_1(t) = \log t,\quad g_2(t) = (\log t)^2,\quad g_3(t) = t^{1/2}.$$ 

We then split 
\begin{align}\label{E: passing to short intervals}
    \E_{n\in[R]} \approx \E_{N\in[R]} \E_{n\in(N, N+L(N)]}
\end{align}
for some $1\prec L(t) \prec t$. The function $L$ is chosen in such a way as to make the contribution of $g_1, \ldots, g_{m-1}$ constant on each short interval; for \eqref{E: sublinear example}, any $L\in\CH$ satisfying $t^{1/2}\prec L(t)\prec t^{3/4}$ would do, giving roughly
\begin{multline*}
        \E_{n\in[R]}T_1^{\sfloor{\sqrt{2}n^{1/2}}}T_2^{\sfloor{(\log n)^2}}f_1 \cdot T_1^{\sfloor{\log n}} T_2^{\sfloor{\sqrt{3}n^{1/2}}} f_2\\ 
        \approx \E_{N\in[R]} \E_{n\in[L(N)]} T_1^{\sfloor{-\frac{\sqrt{2}}{2N^{1/2}} n}}\brac{T_1^{\sfloor{\sqrt{2}N^2}}T_2^{\sfloor{(\log N)^2}}f_1} \cdot T_2^{\sfloor{-\frac{\sqrt{3}}{2N^{1/2}} n}} \brac{T_1^{\sfloor{\log N}} T_2^{\sfloor{\sqrt{3}N^2}}f_2}
\end{multline*}
(for expository purposes, we ignore all the error terms that come from comparing an integer part of a sum with a sum of integer parts).
We then split each short interval into a union of arithmetic progressions and change variables so that on each progression, the coefficient in front of $n$ becomes independent of $N$ (for \eqref{E: sublinear example}, the appropriate step of arithmetic progressions to make this happen would be $\sim 2 N^{1/2}$). If the leading terms in $\ba_1$ and $\ba_j$ are different for all $j\neq 1$, then the claimed bound follows by invoking the linear case for the average over $n$. Otherwise we need to additionally analyze the average over $N$ to disentangle from $\ba_1$ those $\ba_j$ that have the same leading term, much like this was done in \cite[Proposition 5.3]{Ts22}.

\subsubsection{Averages with dual twists}
From the sublinear case onwards, we frequently need to deal with averages \eqref{E: general average rephrased} that have 1-bounded weights $c_n$ and are twisted by \textit{dual sequences}, i.e. sequences of the form $\CD(\floor{b(n)})$ for $b\in\CH$ and $\CD(n) = T^{n}\CD_{s, T}(f)$, where $\CD_{s, T}(f)$ is the level-$s$ dual function of $f$ with respect to $T$ defined in Section \ref{SS: dual functions}. In effect, we work with averages
\begin{align}\label{E: twisted general average}
    \E_{n\in[N]} c_n\cdot \prod_{j\in[\ell]} T^{\floor{\ba_j(n)}}f_j \cdot \prod_{j\in[J]}\CD_j(\floor{b_j(n)}),
\end{align}
where $c_n$ are bounded complex weights, $\CD_1, \ldots, \CD_J$ are dual sequences of bounded level and $b_1, \ldots, b_J$ are Hardy sequences of bounded growth. 
The motivation for studying these more general expressions comes from the proof of Theorem \ref{T: HK control}, in which we gradually replace arbitrary bounded functions $f_j$ by dual functions $\CD_{s,T}(g_j)$. (The reason why the replacement of arbitrary bounded functions to dual functions gives us an advantage is that much like nilsequences, dual sequences can be removed by finitely many applications of the van der Corput inequality.) A key tool to get rid of dual sequences from our expression is Lemma \ref{L: removing duals finitary}, a finitary version of \cite[Proposition~6.1]{Fr12}.

Whenever we deal with averages twisted by dual sequences, needed for the proof of Theorem \ref{T: HK control}, we typically have to make some maximum growth/degree assumption on the sequence $\ba_1$ acting on a function $f_1$ for which we want to obtain a seminorm estimate. However, in the absence of dual sequences, we can usually compose the integral with an iterate of maximum growth/degree, ensuring that the now modified sequence $\ba_1$ has itself maximum growth/degree. For instance, in the sublinear average twisted by dual sequences for which we prove seminorm estimates in Proposition \ref{P: sublinear}, the sequence $\ba_1$ must have maximum growth among $\ba_1, \ldots, \ba_\ell$. But in the sublinear average with no dual sequences considered in Corollary \ref{C: sublinear}, this assumption is no longer necessary. A similar disparity will later appear between Theorem \ref{T: box seminorm bound} (which covers averages \eqref{E: twisted general average} twisted by dual functions for general Hardy sequences except it requires $\ba_1$ to have positive fractional degree) and Theorem \ref{T: box seminorm bound no duals} (which covers averages \eqref{E: general average} with no dual sequences and therefore also works if $\ba_1$ is subfractional).

\subsubsection{Seminorm estimates for polynomial averages}
Seminorm estimates for polynomial averages are split over two sections, \ref{S: polynomial PET} and \ref{S: polynomial concatenation}, and involve a number of intermediate results, some of purely auxiliary character and others fine-tuned for later applications.
In contrast with the sublinear case, the entire argument for polynomials has to be carried out finitarily, i.e. at most stages we work with \eqref{E: twisted general average} for fixed $N$; the reason is that the proof of Theorem~\ref{T: box seminorm bound intro} actually requires seminorm estimates for finite polynomial averages. Section \ref{S: polynomial PET} contains the PET argument; its main outcomes are Proposition \ref{P: PET for polynomials twisted by duals}, which asserts that finite polynomial averages are controlled by finitary counterparts of generalized box seminorms, as well as Corollary \ref{C: PET for polynomials twisted by duals}, its infinitary version. In Section~\ref{S: polynomial concatenation}, we combine these PET estimates with a concatenation argument to derive Proposition~\ref{P: Box seminorm control for polynomials twisted by duals}, (generalized box seminorm estimates for polynomials). Previously, Proposition \ref{P: Box seminorm control for polynomials twisted by duals} was known only for integer polynomials and in a purely qualitative form \cite[Theorem~2.5]{DFKS22} due to its reliance on the Tao-Ziegler concatenation \cite{TZ16}. By contrast, Proposition~\ref{P: Box seminorm control for polynomials twisted by duals} holds for all real polynomials and is fully quantitative. 

The proofs in Sections \ref{S: polynomial PET} and \ref{S: polynomial concatenation} combine the latest technical advances from the ergodic and combinatorial settings with our novel notion of generalized box seminorms. Throughout the PET argument, we need to carefully track the coefficients of the polynomials emerging at the intermediate stages of the argument. This is done using a sophisticated coefficient tracking scheme from \cite{KKL24a} that itself refines an older scheme from \cite{DFKS22}. The polynomial concatenation argument from Section \ref{S: polynomial concatenation} emulates an analogous argument recently carried out in the combinatorial setting \cite{KKL24a}, and it combines the general concatenation argument from Section \ref{S: concatenation} with equidistribution results from \cite{KKL24a} for massive systems of multilinear forms that arise during the PET procedure. 

\subsubsection{Seminorm estimates in the polynomial + sublinear case}
The arguments for the sublinear and polynomial cases come together in Section \ref{S: polynomial + sublinear}, where we give seminorm estimates for sequences that are sums of sublinear and polynomial parts. Emulating the proof of \cite[Proposition 5.1]{Ts22}, we decompose a single average into a double average as in \eqref{E: passing to short intervals}, choosing the length $L$ of the short intervals in such a way as to render the sublinear terms constant on each short interval. For the family \eqref{E: sublinear + polynomial example}, any $1\prec L(t) \prec t^{1/3}$ would do, giving the approximation
\begin{multline*}
        \E_{n\in[R]}T_1^{\sfloor{\sqrt{2}n^2}}T_2^{\sfloor{(\log n)^2}}f_1 \cdot T_1^{\sfloor{n^{2/3}}} T_2^{\sfloor{\sqrt{3}n^2}} f_2\\ 
        \approx \E_{N\in[R]} \E_{n\in[L(N)]} T_1^{\sfloor{\sqrt{2}(N+n)^2}}\brac{T_2^{\sfloor{(\log N)^2}}f_1} \cdot T_2^{\sfloor{\sqrt{3}(N+n)^2}} \brac{T_1^{\sfloor{N^{2/3}}}f_2}.
\end{multline*}
The average over $n$ thus becomes a polynomial average, and we control it by an average of (finitary versions of) generalized box seminorms using the finitary results from Section~\ref{S: polynomial PET}. Then we move to the average over $N$, which contains the sublinear part of our Hardy sequences, and we handle it using Proposition \ref{P: sublinear}.  Once the sublinear part is dealt with, we concatenate the average of finitary generalized box seminorms obtained in the first step using results from Section \ref{S: polynomial concatenation}. This last step constitutes a marked departure from how the analogous argument in the single transformation case has been concluded in \cite{Ts22}. In that work, a simple PET argument sufficed to control the average of seminorms by a Host-Kra seminorm; in our case, we need the full power of the results from Sections \ref{S: polynomial PET}-\ref{S: polynomial concatenation} to finish the argument.

\subsubsection{Seminorm estimates in the general case}
All these proofs, already formidably technical, fade in comparison with the proofs from Section~\ref{S: general case} that complete the proof of Theorem~\ref{T: box seminorm bound no duals}, the precise form of Theorem~\ref{T: box seminorm bound intro}. The bulk of the section is devoted to the proof of Theorem~\ref{T: box seminorm bound}, seminorm estimates for general averages \eqref{E: twisted general average} twisted by dual sequences. Theorem~\ref{T: box seminorm bound} is the main component in the proof of Theorem~\ref{T: box seminorm bound no duals} with the caveat that the former requires $\ba_1$ (or more generally, at least one of $\ba_1, \ldots, \ba_\ell$) to have positive fractional degree. The missing case of Theorem~\ref{T: box seminorm bound no duals} in which all $\ba_1, \ldots, \ba_\ell$ are subfractional is supplied by Corollary~\ref{C: sublinear}.

The starting point of the proof of Theorem~\ref{T: box seminorm bound} is to split the average into short intervals as in \eqref{E: passing to short intervals} and Taylor approximate the sequences $\ba_j$ by polynomials on each short interval. Having passed to short intervals, we use the PET results for polynomials from Section \ref{S: polynomial PET} to control the finite averages over $n$ by a finite average (over integer tuples $\um, \um', \uh$) of finitary generalized box seminorms. It then takes a tremendous amount of effort to massage these resulting expressions until the directions of these finitary seminorms get realized as Hardy sequences in $N$. As it turns out, these new Hardy sequences are sums of polynomial and sublinear terms. This clever reduction, originally performed for a single transformation in \cite{Ts22}, allows us to use Proposition \ref{P: sublinear + polynomial} in order to bound the average over $N$. The result of these maneuvers is an average (over the tuples $\um, \um', \uh$ obtained at the previous step) of generalized box seminorms whose directions are polynomials with coefficients coming from the set of (appropriately defined) ``leading coefficients'' of $\ba_1, \ba_1 - \ba_2, \ldots, \ba_1 - \ba_\ell$. A final application of the concatenation results from Section \ref{S: polynomial concatenation} gives us a single generalized box seminorm along the leading coefficients of $\ba_1, \ba_1 - \ba_2, \ldots, \ba_1 - \ba_\ell$.

\subsubsection{Simultaneous Taylor approximations}\label{SSS: approximations}
Throughout Section \ref{S: general case}, even more so than in the previous sections, we go to considerable lengths to track the coefficients of polynomials and Hardy sequences that appear at intermediate stages. This is needed to ensure that the seminorm of $f_1$ that ultimately controls the original average \eqref{E: general average} has the ``correct'' directions, corresponding to the leading coefficients of $\ba_1, \ba_1 - \ba_2, \ldots, \ba_1-\ba_\ell$. As it transpires in the proof, we need a much finer understanding of the degrees and coefficients of simultaneous Taylor expansions than in previous works \cite{BMR20, Fr10, Fr12, KoTs23, R22, Ts22, Ts24}. These new results on simultaneous Taylor expansions, presented in Proposition \ref{P: Taylor expansion ultimate} and proved in Appendix \ref{A: approximations}, are of independent interest and will likely play a significant role in further works on Hardy sequences.


The rough idea behind Proposition \ref{P: Taylor expansion ultimate} is as follows. Throughout the proof, we split our (tuples of) Hardy sequences $\ba_1, \ldots, \ba_\ell$ into polynomial and nonpolynomial parts; we then take a basis $g_1, \ldots, g_m$ for the nonpolynomial part of our Hardy sequences subject to the growth condition $1\prec g_1(t) \prec \cdots \prec g_m(t)$. In order to track the coefficients of $\ba_1, \ldots, \ba_\ell$ and other intermediate sequences in terms of $g_1(t), \ldots, g_m(t), t, t^2, \ldots$, we need to find functions $H, L$ for which the Taylor approximants $g_{jN}$'s of $g_j$'s on the interval $(N, N + L(N)]$ satisfy the following properties:
\begin{enumerate}
    \item if $g_j$ has positive fractional degree, i.e. it grows faster than $t^\delta$ for some $\delta>0$, then the degree of $g_{jN}$ exceeds some a priori chosen threshold (in particular, $g_{jN}$ has higher degree than the polynomial parts of all our Hardy sequences, which is needed to separate the polynomial and nonpolynomial components of the Hardy sequences throughout the analysis);
    \item the Taylor approximants of any two $g_i, g_j$ have the same degree if and only if $g_i, g_j$ have the same fractional degree;
    \item after the change of variables $n\mapsto \floor{H(N)}n+r$, the leading coefficients $\chi_{g_j}(N)$ of $g_{jN}$ are increasing, sublinear, and have distinct growth rates (without changing variables this way, they would go to 0). Only then can we invoke the polynomial + sublinear case to deal with the average over $N$.
\end{enumerate}

The key difference between our Taylor approximations and those previously established for a single transformation $T$ (e.g. \cite{Fr09, Fr10, Ts22}) is that in those previous works, it was sufficient for the iterates appearing in various intermediate averages to be pairwise distinct. This condition alone was enough to control the averages by some seminorm of $T$. Here, however, each iterate is a sum of a polynomial in $\R^k[t]$ and a $\R^k$-linear combination of $g_1, \ldots, g_m$; hence each coefficient could in principle correspond to a different power of $T_1, \ldots, T_k$, giving rise to a different seminorm. To get seminorm estimates with correct directions, we thus need to be more delicate. It does not suffice for 
sublinear Hardy sequences $\chi_{g_1},\ldots, \chi_{g_m}$ to be pairwise distinct. Rather, 
they need to have \textit{distinct growth rates}. Only then can we keep track of the coefficients throughout the analysis in a way that will give us the desired seminorm estimates at the end. 

While we refrain  from discussing here the highly technical proof of Proposition \ref{P: Taylor expansion ultimate} contained in Appendix \ref{A: approximations}, we do point out that it considerably departs from proofs of earlier simultaneous Taylor approximation results, using ideas vaguely related to the theory of equidistribution and Diophantine approximations.

 

\subsubsection{Seminorm smoothing for pairwise independent polynomials}
Having proved generalized box seminorm estimates for general averages of the form \eqref{E: general average}, we focus our attention on the somewhat more specialized case of averages \eqref{E: general average 2} in which the Hardy sequences are pairwise independent. In Section \ref{S: HK control}, we deduce that the latter family of averages is controlled by Host-Kra seminorms, proving Theorem~\ref{T: HK control}. The proof goes via a variant of the seminorm smoothing argument developed in \cite{FrKu22a}. As it turns out, the argument does not really use any particular properties of Hardy sequences other than the estimates from Theorem~\ref{T: box seminorm bound} and Proposition~\ref{P: sublinear} (the latter supplies the case of all $\ba_1, \ldots, \ba_\ell$ being subfractional missing from Theorem~\ref{T: box seminorm bound}). We therefore write the proof of Theorem~\ref{T: HK control} down for an arbitrary family of sequences satisfying generalized seminorm estimates of the sort given by Theorem~\ref{T: box seminorm bound} and Proposition~\ref{P: sublinear}. This more abstract formulation makes the proof cleaner and renders more general estimates (Proposition~\ref{P: smoothing}) that may find applications beyond this paper. 

While the general strategy behind the proofs in Section~\ref{S: HK control} closely resembles the strategy from \cite{FrKu22a}, quite a few of the technical tools used abundantly in \cite{FrKu22a} need to be adjusted. Perhaps the most important issue is that the degree-1 generalized box seminorms do not admit as clean inverse theorems as degree-1 Host-Kra seminorms. To see where this issue arises, consider the average 
\begin{align}\label{E: average example 2}
    \E_{n\in[N]} T_1^{\floor{\sqrt{2}n^{3/2} + n\log n}}f_1\cdot T_2^{\floor{\sqrt{3}n^{3/2}}}f_2.
\end{align}
To bound it by Host-Kra seminorms, we need to invoke Host-Kra seminorm estimates for
\begin{align}\label{E: average example 2.1}
        \E_{n\in[N]} T_1^{\floor{\sqrt{2}n^{3/2} + n\log n}}f_1\cdot T_1^{\floor{\sqrt{2}n^{3/2}}}f_2.
\end{align}
If we copied blindly the argument from \cite{FrKu22a}, we would reduce from \eqref{E: average example 2} to \eqref{E: average example 2.1} by restricting to functions $f_2$ that are $T_1^{-\sqrt{2}}T_2^{\sqrt{3}}$-invariant. However, $T_1^{-\sqrt{2}}T_2^{\sqrt{3}}$ is not a well-defined transformation, and so to make this crucial reduction work, we need to find a meaningful way of talking about ``$T_1^{-\sqrt{2}}T_2^{\sqrt{3}}$-invariance'' in the context of our problem.

The following diagram illustrates the logical relationship between various seminorm estimates established in this paper. 

\medskip
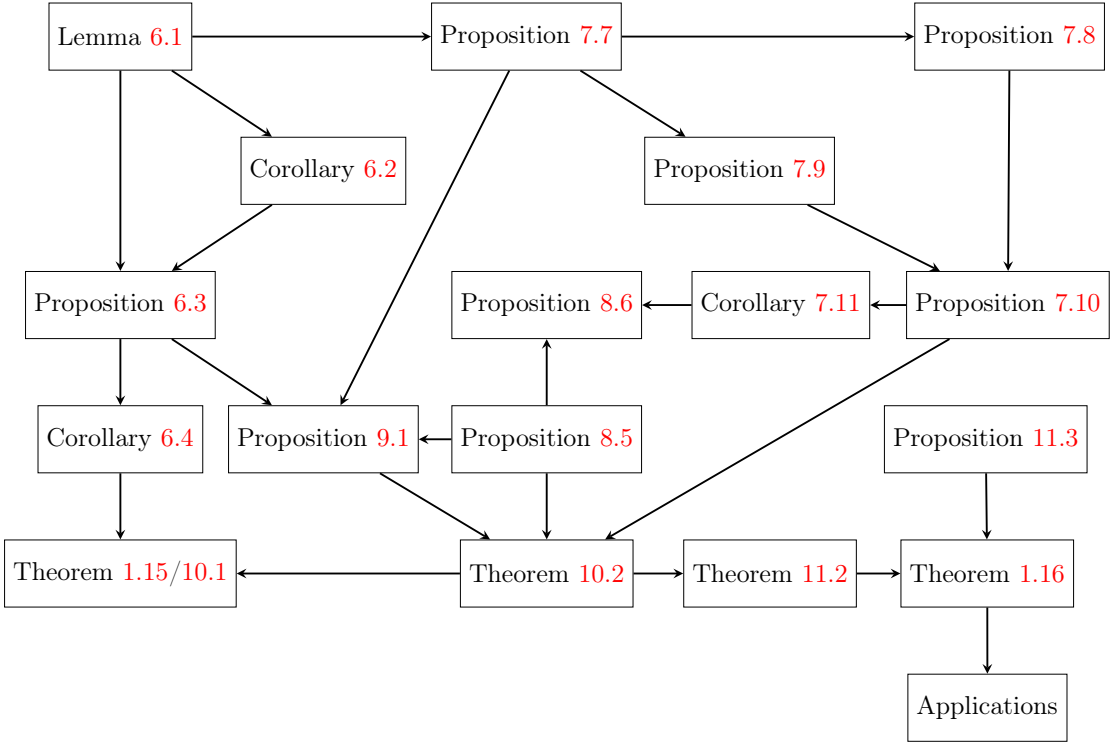
\begin{center}
\scalebox{0.89}{
\begin{tikzpicture}[
    box/.style = {draw, minimum size=1cm},
    arrow/.style = {->, >=stealth, thick}
]
\node[box] (L61) at (0*3,5*2) {Lemma \ref{L: linear averages}};
\node[box] (C62) at (1*3,4*2) {Corollary \ref{C: linear averages}};
\node[box] (P77) at (2*3,5*2) {Proposition \ref{P: PET I}};
\node[box] (P78) at (4*3.285,5*2) {Proposition \ref{P: PET for polynomials with shifts}};
\node[box] (P63) at (0*3,3*2) {Proposition \ref{P: sublinear}};
\node[box] (C64) at (0*3,2*2) {Corollary \ref{C: sublinear}};
\node[box] (P79) at (3*3.05,4*2) {Proposition \ref{P: PET for polynomials with shifts 2}};
\node[box] (P86) at (2.1*3,3*2) {Proposition \ref{P: Box seminorm control for polynomials twisted by duals}};
\node[box] (C711) at (3.2*3.05,3*2) {Corollary \ref{C: PET for polynomials twisted by duals}};
\node[box] (P710) at (4.3*3.05,3*2) {Proposition \ref{P: PET for polynomials twisted by duals}};
\node[box] (P85) at (2.1*3,2*2) {Proposition \ref{P: polynomial concatenation}};
\node[box] (P91) at (1*3,2*2) {Proposition \ref{P: sublinear + polynomial}};
\node[box] (T102) at (2.1*3,1*2) {Theorem \ref{T: box seminorm bound}};
\node[box] (T112) at (3.2*3,1*2) {Theorem \ref{T: box seminorm bound rephrased}};
\node[box] (T113) at (4.2*3.05,1*2) {Theorem \ref{T: HK control}};
\node[box] (P113) at (4.125*3.1,2*2) {Proposition \ref{P: smoothing}};
\node[box] (Apps) at (4.1*3.125,0*2) { Applications};
\node[box] (T1102) at (0*3,1*2) {Theorem \ref{T: box seminorm bound intro}/\ref{T: box seminorm bound no duals}};

\draw[arrow] (L61) -- (P77);
\draw[arrow] (L61) -- (P63);
\draw[arrow] (C62) -- (P63);
\draw[arrow] (P63) -- (C64);
\draw[arrow] (C64) -- (T1102);
\draw[arrow] (P77) -- (P79);
\draw[arrow] (P77) -- (P91);
\draw[arrow] (P78) -- (P710);
\draw[arrow] (P79) -- (P710);
\draw[arrow] (P85) -- (P86);
\draw[arrow] (C711) -- (P86);
\draw[arrow] (P710) -- (C711);
\draw[arrow] (P710) -- (T102);
\draw[arrow] (P91) -- (T102);
\draw[arrow] (P85) -- (T102);
\draw[arrow] (T102) -- (T112);
\draw[arrow] (T113) -- (Apps);
\draw[arrow] (P113) -- (T113);
\draw[arrow] (L61) -- (C62);
\draw[arrow] (P63) -- (P91);
\draw[arrow] (P85) -- (P91);
\draw[arrow] (T102) -- (T1102);
\draw[arrow] (T112) -- (T113);
\draw[arrow] (P77) -- (P78);

\end{tikzpicture}
}
\end{center}

 \subsection*{Acknowledgments}
 We would like to thank Nikos Frantzikinakis for plenty of useful discussions and thorough feedback on the first version of the paper. We also thank Noah Kravitz for comments. 

\section{Notation, background on Hardy sequences, basic definitions}\label{S: background}

\subsection{Basic notation}\label{SS:notation}
We start with presenting basic notation used throughout the paper.

The symbols $\C, \R, \Z, \N_0, \N$ stand for the sets of complex numbers, real numbers, integers, nonnegative integers, and positive integers {respectively}.    With $\T$, we denote the one dimensional torus, and we often identify it with $\R/\Z$ or  with $[0,1)$. For a ring $R$, we denote the collection of polynomials over $R$ in variables $n_1, \ldots, n_k$ via $R[n_1, \ldots, n_k]$. If $E\subseteq G$ is a subset of a group $G$, then $\langle E \rangle$ is the subgroup generated by $E$.

We frequently use the interval notation to mean either a real interval or its restriction to $\Z$:
\begin{enumerate}
    \item for $M<N$, we set $[M,N]$ to be either $\{\ceil{M},\ldots,\floor{N}\}$ or, occasionally, the real interval from $M$ to $N$;
    \item for $N>0$, we set $[N]$ to be either $\{1, \ldots, \floor{N}\}$ or, occasionally, the real interval from 0 to $N$;
    \item for $N>0$, we set $[\pm N]$ to be either $[-\floor{N}, \floor{N}]$ or, occasionally, the real interval from $-N$ to $N$.
\end{enumerate}
When we average over an interval $[M,N], [N], [\pm N]$ or their product, then we always think of the intervals as restricted to integers; when however we look at expressions like $G\cap[\pm N]^k$ for some real subgroup $G\subseteq \R^k$, then $[\pm N]$ always denotes a real interval (and similarly with $[N]$ or $[M,N]$). In other cases, the meaning should be clear from the context.

For a finite set $E,$ we define the average of a sequence $a:E\to \C$ as 
\[\E_{n\in E}a(n) :=\frac{1}{|E|}\sum_{n\in E}a(n).\] We will mainly use finite averages along subsets of integers of the form $[N]^s$ or $[\pm N]^s$, also known as \textit{Ces\`aro averages}.

We let $\CC z := \overline{z}$ be the complex conjugate of $z\in \C$.

For a set $E$, we define its indicator function by $1_E$. Whenever convenient, we also use the notation $1(E)$ for an event $E$.

Typically, we use underlined symbols $\uh$ to denote elements of $\R^s$ (which we think as tuples of parameters), and we employ the bold notation $\bx$ to denote elements of $\R^k$ (which we think as points in a space); sometimes we also use bold notation to denote $s$-tuples of points in $\R^k$, e.g., $\bm=(\bm_1, \ldots, \bm_s)\in(\R^s)^k$.

We often write $\ueps\in\{0,1\}^s$ for a vector of 0s and 1s of length $s$. For $\ueps\in\{0,1\}^s$ and $\uh, \uh'\in\R^s$, we set $\ueps\cdot \uh:=\eps_1 h_1+\cdots+ \eps_s h_s$, $\abs{\uh} := |h_1|+\cdots+|h_s|$ and
\begin{gather}\label{E: h^eps}
    h_j^{\eps_j} =\begin{cases}
        h_j,\; &\eps_j =0\\
        h_j',\; &\eps_j = 1
    \end{cases}
    \quad \textrm{for}\quad j\in[s].
\end{gather}

Given a system $(X, \CX, \mu, T_1, \ldots, T_k)$ and $\bm=(m_{1},\dots,m_{k})\in\R^k$, we set $\floor{\bm} :=  (\floor{m_1}, \ldots, \floor{m_k})$ and $$T^{\floor{\bm}}:=T_1^{\floor{m_1}}\cdots T_k^{\floor{m_k}}.$$
For $j\in[k]$, we set $\be_j$ to be the unit vector in $\R^k$ in the $j$-th direction, and we let $\be_0 = \mathbf{0}$, so that $T^{\be_j} = T_j$ for $j\in[k]$ and $T^{\be_0}$ is the identity transformation. 

For $M>0$, we say that a function $f\in L^\infty(\mu)$ is \textit{$M$-bounded} if $\norm{f}_{L^\infty(\mu)}\leq M$.

For a $\sigma$-algebra $\CA\subseteq\CX$ and $f\in L^1(\mu)$, we let $\E(f|\CA)$ be the conditional expectation of $f$ with respect to $\CA$. By abuse of notation, we can identify the $\sigma$-algebra $\CA$ with the algebra of functions $L^2(\CA)$. The algebras that we often work with include the \textit{invariant factor}
$$\CI(T) = \{f\in L^2(\mu):\; Tf = f\},$$
and the \textit{rational Kronecker factor} 
\begin{align*}
    \Krat(T) = \bigvee_{i=1}^\infty \CI(T^i) = \overline{\{f\in L^\infty(\mu):\; T^i f = f\;\; \textrm{for\; some}\;\; i\in\N\}}.
\end{align*}

Given two functions $a,b:(t_0,\infty)\to\mathbb{R}$ for some $t_0\in\R$ we write
\begin{enumerate}
    \item $a\ll b$, $a = O(b)$, $b\gg a$, or $b = \Omega(a)$ if there exists $C>0$ and $t_1\geq t_0$ such that $|a(t)| \leq C |b(t)|$ for all $t\geq t_1$; 
    \item $a \asymp b$ if $a\ll b$ and $a \gg b$;
    \item $a \prec b$, $b \succ a$ or $a = o(b)$ if $\lim\limits_{t\to\infty}\frac{a(t)}{b(t)} = 0$;
    \item $a \lll b$ if there exists $\delta > 0$ such that $|a(t)|t^\delta \prec |b(t)|$.
\end{enumerate}
If the absolute constant in (i) depends on some parameters $t_1, \ldots, t_n$, then we write $a\ll_{t_1, \ldots, t_n} b$, $a = O_{t_1, \ldots, t_n}(b)$, etc. If in (ii) we want to emphasize what parameter we are taking to infinity, we can also write $a = o_{t\to\infty; t_1, \ldots, t_n} (b)$ to emphasize that $t_1, \ldots, t_n$ are kept fixed while $t$ goes to infinity.

\subsection{Hardy field functions}\label{SS: Hardy field functions}

Let $B$ be the collection of equivalence classes of real valued functions defined on some halfline $(t_0,\infty)$ for $t_0\geq 0$, where two functions that eventually agree are identified. These equivalence classes are called \emph{germs} of functions.  A \emph{Hardy field} is a subfield of the ring $(B, +, \cdot)$ that is closed under differentiation.\footnote{By abuse of notation, we use the word \emph{function} while referring to elements of $B$ which in reality are germs of functions, understanding that all the operations defined and statements made for elements of $B$ are considered only for sufficiently large values of $t\in \mathbb{R}$.} 
 
 As in \cite{Ts22}, we will work with a Hardy field $\mathcal{H}_0$ which is closed under composition and compositional
inversion of functions (whenever the latter is defined). That is, if $a_1, a_2\in \mathcal{H}_0$ with $\lim\limits_{t\to\infty}a_2(t)=\infty,$ then $a_1\circ a_2,$ $a_2^{-1}\in \mathcal{H}_0$. We say that $a\in\CH$ has \emph{polynomial growth} if it satisfies $a(t)\ll t^d$ for some $d>0$, and we then let
\begin{align*}
    \CH = \{a\in \CH_0:\; a(t) \ll t^d\;\; \textrm{for\; some}\;\; d>0\}
\end{align*}
be the restriction of $\CH_0$ to the functions of polynomial growth. The set $\CH$ is still a Hardy field closed under compositions, but it is no longer closed under inverses since, for instance, $\log^{-1} = \exp$ is not in $\CH$.

Usually, one considers Hardy fields $\CH_0$ that contain the class $\mathcal{LE}$ of \textit{logarithmico-exponential Hardy field functions}, i.e. functions defined on $(t_0,\infty)$ for some $t_0\geq 0$ by a finite combination of symbols $+, -, \times, \div, \sqrt[n]{\cdot}, \exp, \log$ acting on the real variable $t$ and on real constants. While $\mathcal{LE}$ is not closed under compositional inverses, it is contained in the Hardy field of Pfaffian functions that is. See \cite{Fr10, Fr12, Hardy12} for more on Hardy fields in general and $\mathcal{LE}$ in particular. 

We commonly refer to sequences of the form $(a(n))_{n\in\N}$ as \textit{Hardy sequences}. Since ultimately we only care about their asymptotic behavior, we can modify their first finitely many values whenever convenient, or make them arbitrary if $a$ is only defined on $[t_0, \infty)$ for some $t_0>0$. This is consistent with our viewpoint on Hardy functions as germs rather than individual functions.

For a Hardy function $a\in\CH$, we define its \textit{fractional degree} to be
     $$\fracdeg a:= \lim_{t\to\infty}\frac{ta'(t)}{a(t)};$$
     the limit always exists since Hardy fields are closed under differentiation, multiplication and division, and it is a real number (see Lemma~\ref{L:Relation of growth rates}).
     If additionally $\CH$ contains $\log$ and is closed under compositions, we can also express $\fracdeg a$ slightly more intuitively as
     \[
     \fracdeg a =\lim\limits_{t\to\infty} \frac{\log a(t)}{\log t}
     \]
     (the equivalence of the two formulations follows from the L' H\^opital's rule applied to the latter expression).
     For instance, the functions $t^{3/2}, t^{3/2}\log t, t^{3/2}/\log t,$ {while of different growths,} all have fractional degree $3/2$. For tuples $\ba\in\CH^k$, we similarly define $$\fracdeg \ba := \max_{i\in[k]}\fracdeg a_i.$$

The definition of fractional degree is quite natural as it allows us to succinctly rephrase various properties of Hardy functions. With this definition, we for instance have
\begin{align*}
    a \lll b \quad \Longleftrightarrow \quad \fracdeg a < \fracdeg b,
\end{align*}
and the condition of $a$ having polynomial growth can be restated concisely as $\fracdeg a < \infty$. For a list of basic properties of fractional degrees, see Lemma~\ref{L: properties of fracdeg}. 

In order to facilitate future discussion, we name various naturally occuring properties of Hardy functions. Thus, we call $a\in\CH$:
\begin{enumerate}
    \item \textit{strongly nonpolynomial} if $t^{d}\prec a(t)\prec t^{d+1}$ for some $d\in\N_0$;
    \item  \textit{subfractional} if $\fracdeg a = 0$, or equivalently if $a(t)\prec t^\delta$ for every $\delta>0$;
    \item \textit{sublinear} if $a(t)\prec t$.
\end{enumerate}
Thus, $t^C$ is strongly nonpolynomial for any $C\in\R_+\setminus{\Z}$ and sublinear for any $C<1$; $t\log t$ is strongly nonpolynomial; and $(\log t)^C$ is subfractional (and hence also sublinear) for any $C>0$. Occasionally, we also call $\ba\in \CH^k$ subfractional (resp. sublinear) if these properties hold for each coordinate of $\ba$.

In our PET argument, we will need to treat polynomial and strongly nonpolynomial functions separately, essentially because polynomial functions vanish after differentiating finitely many times whereas strongly nonpolynomial functions do not. Using the lemma below, we can always decompose a collection of functions from $\CH^k$ as a linear combination of polynomials, strongly nonpolynomial parts, and error terms.

\begin{lemma}[{\cite[Lemma A.3]{R22}}]\label{decomposition}
Let $k,\ell\in\N$ and $\ba_1,...,\ba_\ell\in\mathcal{H}^k$. Then there exist $m\in\N$, functions $g_1,...,g_m\in \mathcal{H}$, vectors $\balpha_{j,i}\in\R^k$ and polynomials $\bp_1,...,\bp_\ell\in\R^k[t]$ such that \begin{enumerate}
    \item $g_1\prec g_2\prec...\prec g_m$,
    \item $t^{d_i}\prec g_i(t)\prec t^{d_i+1}$ for some $d_i\in \N_0$ (i.e. they are strongly nonpolynomial) and
    \item for all $j\in[\ell]$ we have \begin{equation*}
        \ba_j=\sum_{i=1}^{m} \balpha_{j,i} g_i+\bp_j+\br_j,
    \end{equation*}
    where $\br_j(t)\to\mathbf{0}$ as $t\to\infty$.
\end{enumerate}
\end{lemma} 

\section{Generalized box seminorms: basic theory}\label{S: seminorms}
An important new tool in this paper is a generalization of box seminorms which is useful for studying the limiting behavior in $L^2(\mu)$ of multiple ergodic averages of commuting transformations along Hardy sequences. Throughout the entire Section \ref{S: seminorms}, we fix a system $(X, \CX, \mu, T_1, \ldots, T_k)$. Let $G_1, \ldots, G_s\subseteq\R^k$ be finitely generated subgroups.
For functions $f_\ueps\in L^\infty(\mu)$ indexed by $\ueps\in\{0,1\}^s$, we define the \textit{Gowers-Cauchy-Schwarz inner product} along $G_1, \ldots, G_s$ to be
\begin{align}\label{E: GCS}
    \langle (f_\ueps)_\ueps\rangle_{\substack{G_1, \ldots, G_s}} 
    &= \E_{\bm_1, \bm_1'\in G_1}\cdots \E_{\bm_s, \bm_s'\in G_s}\int \prod_{\ueps\in\{0,1\}^s}\CC^{|\ueps|} T^{\floor{\bm_1^{\eps_1}}+\cdots + \floor{\bm_s^{\eps_s}}}f_\ueps\, d\mu,
\end{align}
where $$\E\limits_{\bm\in G} = \lim\limits_{M\to\infty}\E\limits_{\bm\in G\cap[\pm M]^k} \quad \textrm{and}\quad \E\limits_{\bm, \bm'\in G} = \lim\limits_{M\to\infty}\E\limits_{\bm, \bm'\in G\cap[\pm M]^k}$$ (here and in the majority of this section, $[\pm M]$ denotes the real interval from $-M$ to $M$), and also
$$\bm_i^{\eps_i} = \begin{cases}
\bm_i,\; &\eps_i = 0\\ 
\bm'_i,\; &\eps_i = 1
\end{cases}$$
much like in \eqref{E: h^eps}. For instance, if $s=2$ and $G_i = \langle \alpha_i \be_i\rangle$, then 
\begin{align*}
        \langle f_{00}, f_{01}, f_{10}, f_{11}\rangle_{\substack{G_1, G_2}} 
        = \E_{m_1,m_1'\in \Z}  \E_{m_2, m_2'\in\Z} \int &T_1^{\floor{\alpha_1 m_1}}T_2^{\floor{\alpha_2 m_2}} f_{00} \cdot T_1^{\floor{\alpha_1 m_1}}T_2^{\floor{\alpha_2 m_2'}}\overline{f_{01}}\\
        &T_1^{\floor{\alpha_1 m_1'}}T_2^{\floor{\alpha_2 m_2}}\overline{f_{10}}\cdot T_1^{\floor{\alpha_1 m_1'}}T_2^{\floor{\alpha_2 m_2'}}f_{11}\, d\mu.
\end{align*}

It is not a priori clear that the iterated limit in \eqref{E: GCS} exists; we show now that this is indeed the case by proving the $L^2(\mu)$ convergence of a much more general class of averages.

\begin{lemma}\label{L: convergence}
    Let $\ell, s\in\N$,  $G_1, \ldots, G_s\subseteq\R^k$ be finitely generated subgroups and $\gamma_{ji}\in\Z$ for $j\in[\ell]$ and $i\in[s]$. Then for all $f_1, \ldots, f_\ell\in L^\infty(\mu)$, the followings hold:
    \begin{enumerate}
        \item The iterated $L^2(\mu)$ limit 
        \begin{align}\label{E: iterated limit}
            \E_{\bm_1\in G_1} \cdots \E_{\bm_s\in G_s}\prod_{j\in[\ell]}T^{\gamma_{j1}\floor{\bm_1}+\cdots + \gamma_{js}\floor{\bm_s}}f_j
        \end{align}
        exists; in fact, we can replace $G_i\cap [\pm M]^k$ in the definition of $\E_{\bm_i\in G_i}$ by any F\o lner sequence on $G_i$;\footnote{ A \emph{F\o lner sequence} on an abelian group $(G,+)$ is a sequence $(I_M)_M$ of finite subsets of $G$ so that $\lim\limits_{M\to\infty}\frac{|I_M\Delta (I_M+g)|}{|I_M|}=0$ for all $g\in G.$} 
        \item The order of taking limits does not matter, i.e. for any permutation $\sigma$ on $[s]$, the limit \eqref{E: iterated limit} equals 
        \begin{align*}
                \E_{\bm_{\sigma(1)}\in G_{\sigma(1)}} \cdots \E_{\bm_{\sigma(s)}\in G_{\sigma(s)}}\prod_{j\in[\ell]}T^{\gamma_{j1}\floor{\bm_1}+\cdots + \gamma_{js}\floor{\bm_s}}f_j.
        \end{align*}
                \item The single $L^2(\mu)$ limit
        \begin{align}\label{E: single limit}
            \E_{(\bm_1, \ldots, \bm_s)\in G_1\times \cdots \times G_s}\prod_{j\in[\ell]}T^{\gamma_{j1}\floor{\bm_1}+\cdots + \gamma_{js}\floor{\bm_s}}f_j
        \end{align}
        exists along any F{\o}lner sequence on $G_1\times \cdots \times G_s$ and equals \eqref{E: iterated limit}.
    \end{enumerate}
\end{lemma}

\begin{proof}
First, we deduce the existence of the limit in \eqref{E: single limit} from a straightforward multivariable generalization of \cite[Theorem~3.2]{Ko18},\footnote{More specifically, \cite[Theorem~3.2]{Ko18} can immediately be adapted to multiple variables. Condition~(i) in \cite[Theorem 3.2]{Ko18} follows from a straightforward extension of \cite{W12} due to \cite{ZK16}, and Condition~(ii) in \cite[Theorem 3.2]{Ko18} can be derived in a similar way as in the proof of \cite[Theorem 1.4]{Ko18} using classical equidistribution results on tori. We use here the fact that an expression with a linear iterate  $T^{\floor{\alpha n+\beta}}$ can be seen as $S_\alpha^n \cdot S_\beta,$ where $(S_t)_{t\in\R}$ is a flow in a suitable extension of the system $(X, \CX, \mu, T)$ (see \cite{Ko18} for details). Hence, as we highlighted already in the introduction, we only need to use the linear multivariable variants of the results from \cite{W12, ZK16}.} and the goal of the forthcoming manipulations is to express \eqref{E: single limit} as an average along a F\o lner sequence on $\Z^K$ for some $K\in\N$ in order to apply the multivariable generalization of \cite[Theorem~3.2]{Ko18}.

Fix  a F\o lner sequence $(I_M)_{M\in\N}$ on 
\begin{align*}
    G = G_1\times \cdots \times G_s.
\end{align*}
Fix also $j\in[s]$.
Each group $G_j$ is finitely generated, {and so by the fundamental theorem of finitely generated abelian groups, it is group-isomorphic to $\Z^{K_j}$ for some $K_j\in\N$. We can therefore find $\b_{j1}, \ldots, \b_{jK_j}\in G_j$ and a group isomorphism
\begin{align*}
\phi_j: \Z^{K_j} &\to  G_j\\
(m_{j1}, \ldots, m_{jK_j}) & \mapsto \sum_{i\in[K_j]} \b_{ji} m_{ji}.
\end{align*}
Letting 
\begin{gather*}
G = G_1\times \cdots \times G_s\quad \textrm{and}\quad  K = K_1 + \cdots + K_s,
\end{gather*}
it is then easy to see that $(\phi\inv(I_M))_{M\in\N}$ is a F{\o}lner sequence on $\Z^K$.
Letting $$A(\bm) = \prod_{j\in[\ell]}T^{\gamma_{j1}\floor{\bm_1}+\cdots + \gamma_{js}\floor{\bm_s}}f_j,$$ the average that we investigate equals
\begin{align*}
	\E_{(\bm_1, \ldots, \bm_s)\in G\cap I_M} A(\bm)
=\E_{\um\in \phi^{-1}(I_M)}A\brac{\phi(\um)}.
\end{align*}   
}
    
Inducting on $s$, we now prove (i), showing at each stage that the iterated limit \eqref{E: iterated limit} exists and equals \eqref{E: single limit}. For $s=1$, the existence of the limit in  \eqref{E: single limit} and its independence of the choice of F\o lner sequence is a special case of (iii). Let $s>1$. By the induction hypothesis, we have 
    \begin{align*}
       &\E_{\bm_2\in G_2}\cdots \E_{\bm_s\in G_s} \prod_{j\in[\ell]}T^{\gamma_{j1}\floor{\bm_1}+\cdots + \gamma_{js}\floor{\bm_s}}f_j
       =  \E_{(\bm_2, \ldots, \bm_s)\in G_2 \times \cdots \times G_s}\prod_{j\in[\ell]}T^{\gamma_{j1}\floor{\bm_1}+\cdots + \gamma_{js}\floor{\bm_s}}f_j
    \end{align*} 
    for every $\bm_1\in G_1$,
    where the limits can be taken along any F\o lner sequences. This fact together with (iii) and \cite[Lemma 1.1]{BL15} imply that the limit in \eqref{E: iterated limit} exists and equals the single limit in (iii). Hence (i) follows.

    To see that (ii) holds, we note that the expression \eqref{E: single limit} does not depend on the ordering of $G_1, \ldots, G_s$, and hence the equality of the expressions \eqref{E: single limit} and \eqref{E: iterated limit} implies immediately that the indices $1, \ldots, s$ can be permuted.
\end{proof}

For $\bm, \bm'\in\R^s$ and $f\in L^\infty(\mu)$, we set $$\Delta_{({\bm}, {\bm'})}f = T^{\floor{\bm}}f\cdot T^{\floor{\bm'}}\overline{f}.$$ This is a symmetric version of the classical \textit{multiplicative derivative}. We define the \textit{generalized box seminorm} of $f$ along subgroups $G_1, \ldots, G_s\subseteq\R^k$ by the formula
\begin{align*}
    \nnorm{f}_{\substack{G_1, \ldots, G_s}} &:= \langle (f)_\ueps\rangle_{\substack{G_1, \ldots, G_s}}^{1/2^{s}}\\
    &= \brac{\E_{\bm_1, \bm_1'\in G_1}\cdots \E_{\bm_s, \bm_s'\in G_s}\int \Delta_{({\bm_1}, {\bm'_1}), \ldots, ({\bm_s}, {\bm'_s})}f\, d\mu}^{1/2^{s}}\\
    &= \brac{\E_{\bm_1, \bm_1'\in G_1}\cdots \E_{\bm_s, \bm_s'\in G_s}\int \prod_{\ueps\in\{0,1\}^s}\CC^{|\ueps|}T^{\floor{\bm_1^{\eps_1}}+\cdots + \floor{\bm_s^{\eps_s}}} f\, d\mu}^{1/2^{s}}.
\end{align*}
For instance, if $G_i = \langle \alpha_i \be_i\rangle$, then 
\begin{align*}
    \nnorm{f}_{G_1}^2 = \E_{m_1, m_1'\in \Z}\int T_1^{\floor{\alpha_1 m_1}}f \cdot T_1^{\floor{\alpha_1 m_1'}}\overline{f}\, d\mu
\end{align*}
and
\begin{align*}
    \nnorm{f}_{\substack{G_1, G_2}}^4= \E\limits_{m_{1},m'_{1}\in\Z}\E\limits_{m_{2},m'_{2}\in\Z} \int &T_1^{\floor{\alpha_1 m_1}}T_2^{\floor{\alpha_2 m_2}} f \cdot T_1^{\floor{\alpha_1 m_1}}T_2^{\floor{\alpha_2 m'_2}}\overline{f}\\
        &\cdot T_1^{\floor{\alpha_1 m'_1}}T_2^{\floor{\alpha_2 m_2}}\overline{f}\cdot T_1^{\floor{\alpha_1 m_1'}}T_2^{\floor{\alpha_2 m_2'}}f\, d\mu.
\end{align*}
We note that the formula for $ \nnorm{f}_{\substack{G_1, \ldots, G_s}}$ is well-defined since we can rewrite it as
\begin{align*}
    \brac{\E_{\bm_1, \bm_1'\in G_1}\cdots \E_{\bm_{s-1},\bm'_{s-1}\in G_{s-1}}\int \abs{\E_{\bm_s\in G_s}\Delta_{({\bm_1}, {\bm_1'}), \ldots, ({\bm_{s-1}}, {\bm'_{s-1}})}T^{\floor{\bm_s}}f}^2\, d\mu}^{1/2^{s}};
\end{align*}
hence the expression under the root is nonnegative, which allows us to take the root in the first place.\footnote{ From now on, because of the existence of limits and bounded convergence theorem, we will freely interchange summation with integration without mentioning it.}  From the definition, we immediately get the inductive formula
\begin{align}\label{E: inductive formula}
    \nnorm{f}_{\substack{G_1, \ldots, G_s}}^{2^s} = \E_{\bm_1, \bm_1'\in G_1}\cdots \E_{\bm_{s'},\bm'_{s'}\in G_{s'}}\nnorm{ \Delta_{({\bm_1}, {\bm_1'}), \ldots, ({\bm_{s'}}, {\bm'_{s'}})}f}_{G_{s'+1}, \ldots, G_s}^{2^{s-s'}}
\end{align}
for each $s'\in[s]$.

If $G_1, \ldots, G_s\subseteq \Z^k$, then $\nnorm{\cdot}_{G_1, \ldots, G_s}$ equals the usual box seminorm along $G_1, \ldots, G_s$ as defined e.g., in \cite{DFKS22, S21,TZ16}. 

We also define another family of seminorms which will show up more naturally in the PET procedure:
\begin{align*}
    \nnorm{f}^+_{\substack{G_1, \ldots, G_s}} &:=  \brac{\E_{\bm_1, \bm_1'\in G_1}\cdots \E_{\bm_s, \bm_s'\in G_s}\abs{\int \Delta_{({\bm_1}, {\bm_1'}), \ldots, ({\bm_s}, {\bm_s'})}f\, d\mu}^2}^{1/2^{s+1}}\\
      &= \brac{\E_{\bm_1, \bm_1'\in G_1}\cdots \E_{\bm_s, \bm_s'\in G_s}\int \Delta_{({\bm_1}, {\bm_1'}), \ldots, ({\bm_s}, {\bm_s'})}f\otimes\overline{f}\, d(\mu\times\mu)}^{1/2^{s+1}}\\
      &=\nnorm{f\otimes\overline{f}}_{G_1, \ldots, G_s}^{1/2}.
\end{align*}
Most of the groups that we end up dealing with will have a single generator, in which case the notation can be somewhat simplified. Whenever $G_i = \langle \bbeta_i\rangle$ for some $\bbeta_i\in\R^k$, we set
\begin{align*}
    \nnorm{f}_{\bbeta_1, \ldots, \bbeta_s}:=\nnorm{f}_{\substack{G_1, \ldots, G_s}}\quad \textrm{and}\quad \nnorm{f}^+_{\bbeta_1, \ldots, \bbeta_s}:=\nnorm{f}^+_{\substack{G_1, \ldots, G_s}};
\end{align*}
and if a vector $\bbeta$ appears $s$ times, then we abbreviate its $s$ occurrences as $\bbeta^s$, e.g.
\begin{align*}
    \nnorm{f}_{\bbeta_1, \bbeta_1, \bbeta_2} := \nnorm{f}_{\bbeta_1^2, \bbeta_2}\quad \textrm{and}\quad \nnorm{f}_{\bbeta_1, \bbeta_1, \bbeta_2}^+ := \nnorm{f}_{\bbeta_1^2, \bbeta_2}^+.
\end{align*}
We also use further conventions for iterated multiplicative derivatives. Wherever convenient, we write
\begin{align}\label{E: iterated multiplicative derivative}
    \Delta_{\bbeta_1, \ldots, \bbeta_s; (\um, \um')} = \Delta_{(\bbeta_1 m_1, \bbeta_1 m_1'), \ldots, (\bbeta_s m_s, \bbeta_s m_s')}.
\end{align}
for $\um=(m_{1},\dots,m_{s}), \um'=(m'_{1},\dots,m'_{s})\in\Z^s$.

Because of Lemma \ref{L: convergence}, we can replaced the iterated limits in the definition of generalized box seminorms by single limits, and we can permute the order of taking subgroups, obtaining the identities
\begin{align}\label{E: permutation invariance}
    \nnorm{f}_{\bbeta_1, \ldots, \bbeta_s} = \nnorm{f}_{\bbeta_{\sigma(1)}, \ldots, \bbeta_{\sigma(s)}}\quad \textrm{and}\quad \nnorm{f}^+_{\bbeta_1, \ldots, \bbeta_s} = \nnorm{f}^+_{\bbeta_{\sigma(1)}, \ldots, \bbeta_{\sigma(s)}}
\end{align}
for any permutation $\sigma:[s]\to[s]$.

Like Host-Kra seminorms or box seminorms, generalized box seminorms satisfy the Gowers-Cauchy-Schwarz inequality, which we will derive from the standard finitary Gowers-Cauchy-Schwarz inequality (it can be proved e.g., by adapting \cite[Lemma B.2]{GT10b}). 
\begin{lemma}[Finitary Gowers-Cauchy-Schwarz inequality]\label{L: GCS finitary}
    Let $s\in\N$, $E_1, \ldots, E_s$ be finite nonempty sets, $(X, \CX, \mu)$ be a probability space, and $A_\ueps(m_1, \ldots, m_s)\in L^\infty(\mu)$ for every $(m_1, \ldots, m_s)\in E_1\times \cdots \times E_s$ and $\ueps\in\{0,1\}^s$. Then
    \begin{multline*}
        \abs{\E_{m_1, m_1'\in E_1}\cdots \E_{m_s, m_s'\in E_s}\int \prod_{\ueps\in\{0,1\}^s}\CC^{|\ueps|} A_\ueps(m_1^{\eps_1}, \ldots, m_s^{\eps_s})\; d\mu }^{2^s}\\
        \leq \prod_{\underline{\omega}\in\{0,1\}^s}\E_{m_1, m_1'\in E_1}\cdots \E_{m_s, m_s'\in E_s}\int \prod_{\ueps\in\{0,1\}^s}\CC^{|\ueps|} A_{\underline{\omega}}(m_1^{\eps_1}, \ldots, m_s^{\eps_s})\; d\mu.
    \end{multline*}
\end{lemma}

\begin{lemma}[Gowers-Cauchy-Schwarz inequality]\label{L: GCS}
    Let $s\in\N$ and  $G_1, \ldots, G_s\subseteq\R^k$ be finitely generated subgroups. For all functions $f_\ueps\in L^\infty(\mu)$ indexed by $\ueps\in\{0,1\}^s$, we have
    \begin{align*}
        \abs{\langle (f_\ueps)_\ueps\rangle_{\substack{G_1, \ldots, G_s}}}\leq\prod_{\ueps\in\{0,1\}^s}\nnorm{f_\ueps}_{G_1, \ldots, G_s}. 
    \end{align*}
\end{lemma}
\begin{proof}
    The proof follows easily from Lemma \ref{L: GCS finitary}; our ability to change the order of limits in a way that allows us to deduce the infinitary statement from the finitary one is a consequence of Lemma \ref{L: convergence}.
\end{proof}

Much like it is the case for Host-Kra and box seminorms, the Gowers-Cauchy-Schwarz inequality is a crucial fact ensuring that the generalized box seminorms are indeed seminorms.

\begin{lemma}[Triangle inequality]\label{L: triangle}
    Let $s\in\N$ and $G_1, \ldots, G_s\subseteq\R^k$ be finitely generated subgroups. For all $f, g\in L^\infty(\mu)$, we have
    \begin{align}\label{E: triangle 1}
        &\nnorm{f+g}_{\substack{G_1, \ldots, G_s}}\leq \nnorm{f}_{\substack{G_1, \ldots, G_s}} + \nnorm{g}_{\substack{G_1, \ldots, G_s}}\\
        \label{E: triangle 2}
        \quad \textrm{and}\quad &\nnorm{f+g}_{\substack{G_1, \ldots, G_s}}^+\leq \nnorm{f}_{\substack{G_1, \ldots, G_s}}^+ + \nnorm{g}_{\substack{G_1, \ldots, G_s}}^+.
    \end{align}
\end{lemma}
\begin{proof}
    The lemma for \eqref{E: triangle 1} follows in the standard way by rewriting $\nnorm{f+g}_{\substack{G_1, \ldots, G_s}}^{2^s}$ as a Gowers-Cauchy-Schwarz inner product along $G_1, \ldots, G_s$, then splitting it using multilinearity into terms involving $f$ and $g$, and bounding each of them using Lemma \ref{L: GCS}. For \eqref{E: triangle 2}, we express $(\nnorm{f+g}^+_{G_1, \ldots, G_s})^2=\nnorm{(f+g)\otimes(f+g)}_{G_1, \ldots, G_s}$ and split
    \begin{align*}
        (f+g)\otimes \overline{(f+g)} = f\otimes \overline{f}+g\otimes\overline{f} + f\otimes\overline{g}+g\otimes\overline{g},
    \end{align*}
    so that
    \begin{align*}
        (\nnorm{f+g}^+_{G_1, \ldots, G_s})^2 \leq \nnorm{f\otimes \overline{f}}_{G_1, \ldots, G_s}+\nnorm{g\otimes \overline{f}}_{G_1, \ldots, G_s}+\nnorm{f\otimes \overline{g}}_{G_1, \ldots, G_s}+\nnorm{g\otimes \overline{g}}_{G_1, \ldots, G_s}.
    \end{align*}
    From the definition of seminorms and Lemma \ref{L: convergence}, we have
    \begin{multline*}
        \nnorm{f\otimes \overline{g}}_{G_1, \ldots, G_s}^{2^s} = \E_{(\bm_1, \bm_1',\ldots,\bm_s, \bm_s')\in G_1\times G_1\times\cdots\times G_s\times G_s}\\
        \int \prod_{\ueps\in\{0,1\}^s}\CC^{|\ueps|}T^{\floor{\bm_1^{\eps_1}}+\cdots + \floor{\bm_s^{\eps_s}}} f\, d\mu \cdot \overline{\int \prod_{\ueps\in\{0,1\}^s}\CC^{|\ueps|}T^{\floor{\bm_1^{\eps_1}}+\cdots + \floor{\bm_s^{\eps_s}}} g\, d\mu},
    \end{multline*}
    and so by the Cauchy-Schwarz inequality, we get
    \begin{align*}
        \nnorm{f\otimes\overline{g}}_{G_1, \ldots, G_s}\leq \nnorm{f}_{G_1, \ldots, G_s}^+\cdot \nnorm{g}_{G_1, \ldots, G_s}^+.
    \end{align*}
    Hence
    \begin{align*}
        (\nnorm{f+g}^+_{G_1, \ldots, G_s})^2 \leq (\nnorm{f}_{G_1, \ldots, G_s}^+)^2 + 2\nnorm{f}_{G_1, \ldots, G_s}^+ \nnorm{g}_{G_1, \ldots, G_s}^+ +(\nnorm{g}_{G_1, \ldots, G_s}^+)^2,
    \end{align*}
    and the result follows by taking square roots on both sides.
\end{proof}

Similarly to Host-Kra and box seminorms, generalized box seminorms satisfy an important monotonicity property.
\begin{lemma}[Monotonicity property]\label{L: monotonicity property}
    Let $s\in\N$ and $G_1, \ldots, G_{s+1}\subseteq\R^k$ be finitely generated subgroups. For all $f\in L^\infty(\mu)$, we have
    \begin{align*}
        &\nnorm{f}_{G_1, \ldots, G_s}\leq \nnorm{f}_{G_1, \ldots, G_{s+1}}\\
        \quad \textrm{and}\quad &\nnorm{f}_{G_1, \ldots, G_s}^+\leq \nnorm{f}_{G_1, \ldots, G_{s+1}}^+.
    \end{align*}
\end{lemma}
\begin{proof}
    The second inequality follows from the first and the fact that $\nnorm{f}^+_{G_1, \ldots, G_s} = \nnorm{f\otimes\overline{f}}_{G_1, \ldots, G_s}^{1/2}$ while the first one is a standard consequence of the Gowers-Cauchy-Schwarz inequality.
\end{proof}

The following lemma gives sufficient conditions for two generalized box seminorms to agree.
\begin{lemma}[Invariance property]
    Let $s\in\N$ and $G_1, G_1'\ldots, G_s, G_s'\subseteq\R^k$ be finitely generated subgroups with the property that
    \begin{align}\label{E: ergodicity property}
        \E_{\bm\in G_i} T^{\floor{\bm}}f = \E_{\bm\in G_i'} T^{\floor{\bm}}f
    \end{align}
    for every $f\in L^\infty(\mu)$ and $i\in[s]$.
    Then 
    \begin{align*}
        \langle (f_\ueps)_\ueps\rangle_{G_1, \ldots, G_s} = \langle (f_\ueps)_\ueps\rangle_{G'_1, \ldots, G'_s}
    \end{align*}
    for all $f_\ueps\in L^\infty(\mu)$ indexed by $\ueps\in\{0,1\}^s$.
\end{lemma}
\begin{proof}
    We write $\ueps = (\Tilde{\ueps},\eps_s)$, so that $\langle (f_\ueps)_\ueps\rangle_{G_1, \ldots, G_s}$ equals
\begin{align*}
    \E_{\bm_1, \bm_1'\in G_1}\cdots \E_{\bm_{s-1}, \bm_{s-1}'\in G_{s-1}} 
    \int &\E_{\bm_s\in G_s} \prod_{\tilde{\ueps}\in\{0,1\}^{s-1}}\CC^{|\tilde{\ueps}|} T^{\floor{\bm_1^{\eps_1}}+\cdots + \floor{\bm_{s-1}^{\eps_{s-1}}} + \floor{\bm_s}}f_{\tilde{\ueps}0}\\
    &\E_{\bm'_s\in G_s} \prod_{\tilde{\ueps}\in\{0,1\}^{s-1}}\CC^{|\tilde{\ueps}|} T^{\floor{\bm_1^{\eps_1}}+\cdots + \floor{\bm_{s-1}^{\eps_{s-1}}} + \floor{\bm_s'}}\overline{f_{\tilde{\ueps}1}}\, d\mu.
\end{align*}
Using \eqref{E: ergodicity property} for $G_s$, we rewrite the two averages inside the integral as averages over $G_s'$ instead. It follows that
\begin{align*}
    \langle (f_\ueps)_\ueps\rangle_{G_1, \ldots, G_{s-1}, G_s} = \langle (f_\ueps)_\ueps\rangle_{G_1, \ldots, G_{s-1}, G'_s}.
\end{align*}
Using the fact that we can swap the order of summation in the formula for Gowers-Cauchy-Schwarz inner product thanks to Lemma \ref{L: convergence}, we iterate the procedure over the indices $i=1, \ldots, s-1$, obtaining the claim.
\end{proof}

Much like for box seminorms, we can compare the generalized box seminorm along $G_1, \ldots, G_s$ to a seminorm along its subgroups.
\begin{lemma}\label{L: seminorms of subgroups}
   Let $G_1, G_1', \ldots, G_s, G_s'\subseteq\R^k$ be finitely generated subgroups such that $G'_i\subseteq G_i$ for each $i\in[s]$. For every 1-bounded function $f\in L^\infty(\mu)$, we have
\begin{align*}
    \nnorm{f}_{G_1, \ldots, G_s}^+\ll_{k,s} \nnorm{f}_{G'_1, \ldots, G'_s}^+.
\end{align*}

Conversely, if the index of $G'_{i}$ in $G_{i}$ is at most $D$ for all $i\in[s]$, then
\begin{align*}
\nnorm{f}_{G'_1, \ldots, G'_s}^+\ll_{D,s} \nnorm{f}_{G_1, \ldots, G_s}^+.
\end{align*}
\end{lemma}
\begin{proof}
    We prove both inequalities for $s=1$ (in which case we simply write $G'=G'_1$, and $G=G_1$) since the cases $s>1$ follow by induction using the inductive formula \eqref{E: inductive formula} and the fact that we can permute the order of groups in the definition of box norms (Lemma~\ref{L: convergence} (ii)).
    
    For the first inequality, we first note that if $G'\subseteq G$, then
    \begin{align*}
        \E_{\bm\in G}A(\bm) = \E_{\bm\in G}A(\bm + \bk)
    \end{align*}
    for every $\bk\in G'$ (as long as the first limit exists), and hence
    \begin{align}\label{E: shifting by a subgroup}
        \E_{\bm\in G}A(\bm) = \E_{\bk\in G}\E_{\bm\in G}A(\bm + \bk).
    \end{align}

    Now, by definition of the degree 1 seminorm, we have
    \begin{align}\label{E: degree 1 + identity}
        \brac{\nnorm{f}_{G}^+}^4 &= \int \E_{\bm\in G}(T\times T)^{\floor{\bm}}f\otimes \overline{f}\cdot \E_{\bm'\in G}(T\times T)^{\floor{\bm'}}\overline{f}\otimes f\; d(\mu\times\mu).
    \end{align}
    The identity \eqref{E: shifting by a subgroup} and Lemma \ref{L: convergence} give
    \begin{align*}
        \E_{\bm\in G}(T\times T)^{\floor{\bm}}\overline{f}\otimes f = \E_{\bm\in G}\E_{\bk\in G'}(T\times T)^{\floor{\bm+\bk}}\overline{f}\otimes f.
    \end{align*}
    Inserting this identity into \eqref{E: degree 1 + identity} and applying the Cauchy-Schwarz inequality, we get
    \begin{align*}
        \brac{\nnorm{f}_{G}^+}^4 &\leq \E_{\bm\in G}\int \abs{\E_{\bk\in G'}(T\times T)^{\floor{\bm+\bk}}\overline{f}\otimes f}^2\; d(\mu\times\mu).
    \end{align*}
    This looks almost like what we aim for except for the presence of $\bm$. We can remove it by composing the integral with $(T\times T)^{-\floor{\bm}}$, swapping the subtraction inside the fractional part and making the error term $\floor{\bm+\bk}-\floor{\bm}$ independent of $\bm$ and $\bk$ using the pigeonhole principle (this causes a loss in the multiplicative factor of $O_{k}(1)$ since there are that many choices for the error term). This gives us a bound of the form
    \begin{align*}
        \brac{\nnorm{f}_{G}^+}^4 &\ll_k \int \E_{\bk\in G'}(T\times T)^{\floor{\bk}+\br}\overline{f}\otimes f\cdot \E_{\bk'\in G'}(T\times T)^{\floor{\bk'}+\br'}\overline{f}\otimes f\; d(\mu\times\mu)
    \end{align*}    
    for some $\br, \br'\in \Z^k$. These error terms can be removed using an application of the Gowers-Cauchy-Schwarz inequality, after which the claimed inequality follows.
    
    For the converse direction, assume that
    the index of $G'$ in $G$ is at most $D$. If the averages of a function $A: G\to [0, \infty)$ along $G, G'$ converge, then 
    \begin{align*}
        \E_{\bm\in G'}A(\bm) &= \lim_{M\to\infty}\frac{1}{|G'\cap[\pm M]^k|}\sum_{\bm\in G'\cap[\pm M]^k}A(\bm)\\
        &\leq \lim_{M\to\infty}\frac{D}{|G\cap[\pm M]^k|}\sum_{\bm\in G\cap[\pm M]^k}A(\bm) = D\cdot \E_{\bm\in G}A(\bm);
    \end{align*}
    the point is that nonnegativity of $A$ allows us to extend the summation from $G'$ to $G$. Applying this argument twice to averages over $\bm$ and $\bm'$, we get that
    \begin{align*}
        \brac{\nnorm{f}_{G'}^+}^4 = \E_{\bm,\bm'\in G'}\abs{\int \Delta_{(\bm, \bm')}f\, d\mu}^2 \ll_D \E_{\bm,\bm'\in G}\abs{\int \Delta_{(\bm, \bm')}f\, d\mu}^2 = \brac{\nnorm{f}_{G}^+}^4,
    \end{align*} and the claim follows.
\end{proof}

Generalized box seminorms are also essentially invariant under rescalings of the groups, as indicated by the result below.
\begin{lemma}[Rescaling property]\label{L: dilating seminorms}
    Let $k,K, s\in\N$, $\alpha\geq 1$ be a real number, $G_1,\ldots, G_{s}\subseteq\R^k$ be finitely generated subgroups such that for every $i\in[s]$, we have $G_i = \langle \bg_{i1}, \ldots, \bg_{iK_i}\rangle$ with $K_i\leq K$. For $i\in[s], j\in[K_i]$, let $\alpha_{ij}\in\R\setminus\{0\}$ satisfy $1/\alpha\leq \alpha_{ij}\leq \alpha$, and let 
    \begin{align*}
        G'_i = \langle \alpha_{i1}\bg_{i1}, \ldots, \alpha_{iK_i} \bg_{iK_i}\rangle .
    \end{align*}
    Then for any $f\in L^\infty(\mu)$, we have
    \begin{align*}
        \nnorm{f}^+_{G'_1, \ldots, G'_s}\asymp_{\alpha, k, K, s} \nnorm{f}^+_{G_1, \ldots, G_s}.
    \end{align*}
    {Whenever $s\geq 2$, we also have
    \begin{align*}
        \nnorm{f}_{G'_1, \ldots, G'_s}\asymp_{\alpha, k, K, s} \nnorm{f}_{G_1, \ldots, G_s}.
    \end{align*}}
\end{lemma}
\begin{proof}
We only prove the first statement, as the second follows similarly. 

We prove the forward inequality $\nnorm{f}^+_{G'_1, \ldots, G'_s}\ll_{\alpha, K, s} \nnorm{f}^+_{G_1, \ldots, G_s}$ for $s=1$, and the higher degree case will follow from the inductive formula \eqref{E: inductive formula} and the fact that we can permute the subgroups thanks to Lemma \ref{L: convergence}. The converse inequality then follows by inverting the roles of subgroups $G_i, G_i'$. We let all the constants depend on $\alpha, K, s$.
   
    Recall that 
    \begin{align*}
        \brac{\nnorm{f}_{G'}^+}^4 &= \E_{\bm, \bm'\in \langle \alpha_1 \bg_1, \ldots, \alpha_K \bg_K\rangle}\abs{\int \Delta_{(\bm, \bm')}f\, d\mu}^4\\
        &= \E_{\substack{\um, \um'\in\Z^K}} \abs{\int \Delta_{(\alpha_1 \bg_1 m_1 + \cdots + \alpha_K \bg_K m_K, \alpha_1 \bg_1 m'_1 + \cdots + \alpha_K \bg_K m'_K)}f\, d\mu}^4\\
        &= \int \abs{\E_{\um\in\Z^K}(T\times T)^{\floor{\alpha_1 \bg_1 m_1 + \cdots + \alpha_K \bg_K m_K}}f\otimes \overline{f}}^2\; d(\mu\times\mu),
    \end{align*}
where we are dropping the indices $i$ for clarity (here, $[\pm M]$ denotes the appropriate integer interval). 
    We observe first that 
    \begin{align}\label{E: integer part approximation}
        \abs{\floor{\alpha_1 \bg_1 m_1 + \cdots + \alpha_K \bg_K m_K} - \floor{\bg_1 \floor{\alpha_1 m_1} + \cdots + \bg_K \floor{\alpha_K m_K}}}\ll 1,
    \end{align}
    and so by the pigeonhole principle, there exist $\br\in\Z^k$ and a set $E\subset\Z^K$ such that
    \begin{multline*}
        \brac{\nnorm{f}_{G'}^+}^4 \ll \int \abs{\E_{\um\in\Z^K}1_E(\um) \cdot (T\times T)^{\floor{\bg_1 \floor{\alpha_1 m_1} + \cdots + \bg_K \floor{\alpha_K m_K}}+\br}f\otimes \overline{f}}^2\; d(\mu\times\mu)
    \end{multline*}
    ($E$ consists of those tuples for which the difference between two tuples in \eqref{E: integer part approximation} is $\br$).
    By composing with $(T\times T)^{-\br}$, we can assume that $\br = \mathbf{0}$.
   Each integer takes the form $\floor{\alpha_i m_i}$ for  at most $O(1)$ many $m_i\in\Z$, and hence 
        \begin{align*}
        \brac{\nnorm{f}_{G'}^+}^4 \ll \int \abs{\E_{\um\in\Z^K}c_{\um} \cdot (T\times T)^{\floor{\bg_1 m_1 + \cdots + \bg_K m_K}}f\otimes \overline{f}}^2\; d(\mu\times\mu)
    \end{align*}
    for some nonnegative $O(1)$-bounded weight $c_{\um}$. The right-hand side can then be expanded as 
    \begin{align*}
         \E_{\substack{\um, \um'\in\Z^K}} c_{\substack{\um, \um'}} \cdot \abs{\int \Delta_{(\bg_1 m_1 + \cdots + \bg_K m_K, \bg_1 m'_1 + \cdots + \bg_K m'_K)}f}^2\; d\mu
    \end{align*}
    for some nonnegative, $O(1)$-bounded weight $c_{\um, \um'}$. The weight can be removed by non-negativity at the cost of losing a factor $O(1)$, and the result follows.
     \end{proof}

Last but not least, the two families of seminorms that we have defined, those with and without plus, are related to each other as follows.
\begin{lemma}\label{L: plus vs normal seminorm}
    Let $s\in\N$ and $G_1,\ldots, G_{s+1}\subseteq\R^k$ be finitely generated subgroups. For every $f\in L^\infty(\mu)$, we have
    \begin{align*}
        \nnorm{f}_{G_1, \ldots, G_s}\leq\nnorm{f}_{G_1, \ldots, G_s}^+\leq\nnorm{f}_{G_1, \ldots, G_{s+1}}.
    \end{align*}
\end{lemma}
\begin{proof}
    The first inequality follows from a straightforward application of the Cauchy-Schwarz inequality. For the second inequality, we prove the case $s=1$, and the case $s>1$ will follow from the inductive formula \eqref{E: inductive formula}. By definition, we have 
    \begin{align*}
        (\nnorm{f}_{G_1}^+)^4 = \E_{\bm_1, \bm_1'\in G_1}\abs{\int \Delta_{(\bm_1, \bm_1')}f\,d\mu}^2.
    \end{align*}
    For every subgroup $G_2$, we can introduce an extra averaging over $G_2$ so that
    \begin{align*}
        \int \Delta_{(\bm_1, \bm_1')}f\,d\mu = \E_{\bm_2\in G_2}\int T^{\floor{\bm_2}}\Delta_{(\bm_1, \bm_1')}f\,d\mu, 
    \end{align*}
    and hence
    \begin{align*}
        (\nnorm{f}_{G_1}^+)^4 = \E_{\bm_1, \bm_1'\in G_1}\abs{\E_{\bm_2\in G_2}\int T^{\floor{\bm_2}}\Delta_{(\bm_1, \bm_1')}f\,d\mu}^2.
    \end{align*}
    By the Cauchy-Schwarz inequality, we have
    \begin{align*}
        (\nnorm{f}_{G_1}^+)^4 \leq \E_{\bm_1, \bm_1'\in G_1}\int\abs{\E_{\bm_2\in G_2} T^{\floor{\bm_2}}\Delta_{(\bm_1, \bm_1')}f\,d\mu}^2,
    \end{align*}
    and since $\int\abs{\E_{\bm_2\in G_2} T^{\floor{\bm_2}}\Delta_{(\bm_1, \bm_1')}f\,d\mu}^2 = \nnorm{\Delta_{(\bm_1, \bm_1')}f}_{G_2}^2$,
    we get the claimed inequality from the inductive formula \eqref{E: inductive formula}.
\end{proof}
 
 \subsection{Dual functions and sequences}\label{SS: dual functions}
 To generalized box seminorms we can assign a notion of dual functions that naturally extends dual functions for Host-Kra seminorms. Let $s\in\N$ and $\{0,1\}^s_* = \{0,1\}^s\setminus\{\underline{0}\}$. For a function $f\in L^\infty(\mu)$ and a finitely generated subgroup $G\subseteq\R^k$ we define the \textit{level-$s$ dual function of $f$ with respect to $G$} to be
\begin{align*}
    \CD_{s, G}(f) := \E_{\bm_1, \bm_1', \ldots, \bm_s, \bm_{s'}\in G_1}\prod_{\ueps\in\{0,1\}^s_*}\CC^{|\ueps|}T^{\floor{\bm_1^{\eps_1}}-\floor{\bm_1}+\cdots + \floor{\bm_s^{\eps_s}}-\floor{\bm_s}} f;
\end{align*}
by Lemma \ref{L: convergence}, the iterated limit exists and equals a single limit. Thus, for instance, 
\begin{align*}
    \CD_{1, G}(f) &= \E_{\bm_1, \bm_1'\in G}T^{\floor{\bm'_1}-\floor{\bm_1}}\overline{f}\\
    \CD_{2, G}(f) &= \E_{\bm_1, \bm_1',\bm_2,\bm_2'\in G}T^{\floor{\bm'_1}-\floor{\bm_1}}\overline{f}\cdot T^{\floor{\bm'_2}-\floor{\bm_2}}\overline{f}\cdot T^{\floor{\bm'_1}-\floor{\bm_1}+\floor{\bm'_2}-\floor{\bm_2}}f.
\end{align*}
Dual functions satisfy the identity
\begin{align}\label{dual identity}
    \nnorm{f}_{s, G}^{2^s} = \int f \cdot \CD_{s,G}(f)\, d\mu,
\end{align}
which justifies their name.

If $G = \langle \b \rangle$ for some $\b\in\R^k$, we simply write $\CD_{s, \b}(f) = \CD_{s, G}(f)$; similarly, if $\b = \be_j$, then we may write $\CD_{s, T_j}(f) = \CD_{s,\b}(f)$. In the latter case, the dual functions take the better known form
\begin{align*}
    \CD_{s, T_j}(f) := \lim_{M\to\infty}\E_{\um\in [\pm M]^s}\prod_{\ueps\in\{0,1\}^s_*}\CC^{|\ueps|}T_j^{\ueps\cdot\um}f;
\end{align*}
the average over $[\pm M]^s$ can also be replaced by the one-sided average over $[M]^s$.

For $d\in\N$, we denote
 \begin{equation}\label{E:Ds}
 \FD_d(T_{1},\dots,T_{k}):= \{(T_j^n \CD_{s, T_j}(f))_{n\in\Z}\colon\; \norm{f}_{L^\infty(\mu)}\leq 1,\ j\in[k],\ s\in[d]\}
 \end{equation}
to be the set of sequences of 1-bounded functions coming from dual functions of degree up to $d$ for the transformations $T_1,\ldots, T_k$. We call each $(T_j^n \CD_{s, T_j}(f))_{n\in\Z}$ a \textit{level-$s$ dual sequence}. When there is no confusion regarding the system $(X,\mathcal{X},\mu,T_{1},\dots,T_{k})$, we write $\FD_d:=\FD_d(T_{1},\dots,T_{k})$ for short.

\section{Quantitative concatenation of generalized box seminorms}\label{S: concatenation}
Our goal now is to present concatenation results for averages of generalized box seminorms. Contrary to the original concatenation work of Tao-Ziegler \cite{TZ16} which introduced the notion of concatenation of seminorms for the first time, our result is fully quantitative. Its proof closely follows (and significantly borrows from) recent quantitative concatenation arguments in the discrete setting \cite{KKL24a, Kuc23, Pel20, PP19}. Once again, we fix a system $(X, \CX, \mu, T_1, \ldots, T_k)$ throughout this section.

One of the key tools in our argument is the following lemma which allows us to concatenate averages along sumsets. 
\begin{lemma}[Sumset lemma]\label{L: sumset lemma}
    Let $G,G'\subseteq\R^k$ be finitely generated groups. Then 
    \begin{align*}
        \E_{\bm'\in G'}\E_{\bm\in G}T^{\floor{\bm-\bm'}}f = \E_{\bm\in G+G'}T^{\floor{\bm}}f
    \end{align*}
    for every $f\in L^\infty(\mu)$.
\end{lemma}
\begin{proof}
    Let $G_0 = G\cap G'$; {we note that $G_0$ is finitely generated itself since it is a subgroup of a finitely generated abelian group $G$}. Our first goal is to decompose
    \begin{align}\label{E: direct sum decomposition}
        G = G_0 \oplus K \oplus E \quad \textrm{and}\quad G' = G_0 \oplus K' \oplus E'
    \end{align}
    for some subgroups $K\subseteq G,\; K'\subseteq G'$ and finite sets $E\subseteq G,\; E'\subseteq G'$. If $G/G_0$ is finite, then there exists a finite set $E\subseteq G$ for which $G = G_0 \oplus E$. Otherwise $G/G_0$ is infinite, and since it is finitely generated,\footnote{A quotient $G/G_0$ of a finitely generated group $G$ is finitely generated since if $\bg_1, \ldots, \bg_n$ generate $G$, then their cosets generate $G/G_0$.} there must exist $\bg_1\in G$ such that $\langle\bg_1\rangle$ intersects $G_0$ trivially. Next, either $G/(G_0 \oplus \langle\bg_1\rangle)$ is finite, and so $G = G_0 \oplus \langle\bg_1\rangle \oplus E$ for some finite set $E$, or there exists $\bg_2\in G$ such that $\langle\bg_2\rangle$ intersects $G_0 \oplus \langle\bg_1\rangle$ trivially. Iterating this way, we find a subgroup $K\subseteq G$ and a set $E\subseteq G$ satisfying the claim; similarly for $G'$.

    As a consequence of \eqref{E: direct sum decomposition}, we can decompose each $\bm\in G+G'$ uniquely as
    \begin{align}\label{E: sumset decomposition}
        \bm = \bg + \bk + \be - \bk' - \be'
    \end{align}
    for {$\bg\in G_0,$} $\bk\in K$, $\be\in E$, $\bk'\in K'$, $\be'\in E'$. 

    Given any $\bg'\in G_0$, we have 
    \begin{align*}
        \E_{\bm\in G+G'}T^{\floor{\bm}}f = \E_{\bm\in G+G'}T^{\floor{\bm-\bg'}}f = \E_{\bg'\in G_0} \E_{\bm\in G+G'}T^{\floor{\bm-\bg'}}f,
    \end{align*}
    which by the Fubini-type principle \cite[Lemma 1.1]{BL15} equals
    \begin{align}\label{E: Fubini}
        \E_{(\bm, \bg')\in (G+G')\times G_0}T^{\floor{\bm-\bg'}}f.
    \end{align}
    Given the uniqueness of the decomposition \eqref{E: sumset decomposition}, for every F{\o}lner sequence $(I_M)_M$ on $(G+G')\times G_0$, we define
    \begin{align*}
        \tilde{I}_M = \{(\bg + \bk + \be, \bg'+\bk'+\be')\in G\times G':\; (\bg + \bk + \be - \bk' - \be', \bg')\in I_M\};
    \end{align*}
    it is an easy exercise to see that $(\tilde{I}_M)_M$ is a F{\o}lner sequence on $G\times G'$. Hence \eqref{E: Fubini} can be rephrased 
    \begin{align*}
        \E_{(\bm,\bm')\in G\times G'}T^{\floor{\bm-\bm'}}f = \E_{\bm'\in G'}\E_{\bm\in G}T^{\floor{\bm-\bm'}}f,
    \end{align*}
    the latter equality being once again a consequence of \cite[Lemma 1.1]{BL15}.
\end{proof}

In order to illustrate better how concatenation works, we start with a simple concatenation result for degree 1 seminorms.
\begin{lemma}[Concatenation of degree 1 norms]\label{L: concatenation degree 1}
Let $k\in\N$, $I$ be a finite indexing set and $G_i\subseteq\R^k$ be a finitely generated subgroup for each $i\in I$. Then for every 1-bounded function $f\in L^\infty(\mu)$, we have
\begin{align*}
    \brac{\E_{i\in I}\nnorm{f}^+_{G_i}}^{16}\ll_k \E_{i, i'\in I}(\nnorm{f}^+_{G_i+G_{i'}})^4.
\end{align*}
\end{lemma}
In this and subsequent results in this section, the implicit constants and $O(1)$ terms do not depend on the system. Also, we frequently use Lemma \ref{L: convergence} (without mentioning it explicitly) to swap the order of limits, replace an iterated limit by a single limit, or do the opposite. 

\begin{proof}
Let $\veps = \E_{i\in I}\nnorm{f}^+_{G_i}$. Composing the emerging integral with $T^{\floor{\bm'}}$ and expanding the definition of the seminorm, we get 
\begin{align*}
    \E_{i\in I}\E_{\bm, \bm'\in G_i}\abs{\int f\cdot T^{\floor{\bm'} - \floor{\bm}}\overline{f}\, d\mu}^2 =\veps^4.
\end{align*}
Passing to the product system allows us to get rid of the absolute value, giving
\begin{align*}
    \int f\otimes\overline{f}\cdot \E_{i\in I} \E_{\bm,\bm'\in G_i} (T\times T)^{\floor{\bm'} - \floor{\bm} }\overline{f}\otimes f\, d(\mu\times\mu) = \veps^4.
\end{align*}
An application of the Cauchy-Schwarz inequality duplicates the group $G_i$, giving
\begin{align*}
    \E_{i,i'\in I}\E_{\bm,\bm'\in G_i}\E_{\bm'', \bm'''\in G_{i'}} \int (T\times T)^{\floor{\bm'} - \floor{\bm}}f\otimes\overline{f}\cdot  (T\times T)^{\floor{\bm'''}-\floor{\bm''}}\overline{f}\otimes f\, d(\mu\times\mu) \geq \veps^8.
\end{align*}
We now compose the integral with $(T\times T)^{-(\floor{\bm'} - \floor{\bm})}$ and  change the order of summations, so that
\begin{align*}
    \E_{i,i'\in I}\int\E_{\bm,\bm'\in G_i}\E_{\bm'', \bm'''\in G_{i'}}  f\otimes\overline{f}\cdot  (T\times T)^{\floor{\bm'''}-\floor{\bm''}-\floor{\bm'}+\floor{\bm}}\overline{f}\otimes f\, d(\mu\times\mu) \geq \veps^8.
\end{align*}
By going back from the product system to the original system, we can easily see that each integral above is real and nonnegative, and so we employ
the pigeonhole principle to choose $\br\in\Z^k$ with $|\br|\leq k$ independently of $i, i',\bm,\bm',\bm'',\bm'''$ so that
\begin{align*}
    \E_{i,i'\in I}\int f\otimes \overline{f}\cdot \E_{\bm,\bm'\in G_i}\E_{\bm'',\bm'''\in G_{i'}}(T\times T)^{\floor{(\bm-\bm'-\bm''+\bm''')}+\br}\overline{f}\otimes f\, d(\mu\times \mu) \gg_k \veps^8.
\end{align*}
Splitting the double average $\E_{\bm,\bm'\in G_i}$ as $\E_{\bm'\in G_i}\E_{\bm\in G_i}$ (and similarly for the average over $G_{i'}$) and changing variables $\bm_0 = \bm - \bm'$ as well as $\bm_1 = \bm''-\bm'''$ (for each fixed $\bm'\in G_i, \bm'''\in G_i$), we get
\begin{align*}
    \E_{i,i'\in I}\int f\otimes \overline{f}\cdot \E_{\bm_0\in G_i}\E_{\bm_1\in G_{i'}}(T\times T)^{\floor{\bm_0 - \bm_1}+\br}\overline{f}\otimes f\, d(\mu\times \mu) \gg_k \veps^8.
\end{align*}
By Lemma \ref{L: sumset lemma}, we can replace the iterated limit by a single limit over the sumset $G_i+G_{i'}$, getting
\begin{align*}
    \E_{i,i'\in I}\int f\otimes \overline{f}\cdot \E_{\bm\in G_i+G_{i'}}(T\times T)^{\floor{\bm}+\br}\overline{f}\otimes f\, d(\mu\times \mu) \gg_k \veps^8.
\end{align*}
The claim follows from one more application of the Cauchy-Schwarz inequality followed by composing the resulting integral with $(T\times T)^{-\br}$.
\end{proof}

\begin{example}
Let $G_{h_1, h_2} = \langle \alpha \be_1 h_1+\beta \be_2 h_2\rangle$ for all $(h_1, h_2)\in\Z^2$. Then Lemma \ref{L: concatenation degree 1} gives
\begin{align*}
    \brac{\limsup_{H\to\infty}\E_{h_1, h_2\in[\pm H]}{\nnorm{f}^+_{G_{h_1, h_2}}}}^{16}\ll_k \limsup_{H\to\infty}\E_{h_1, h_1', h_2, h'_2\in[\pm H]}{\nnorm{f}^+_{G_{h_1, h_2}+G_{h_1', h_2'}}}.
\end{align*}

For all $(h_1, h_1', h_2, h_2')\in\Z^4$, the group $G_{h_1, h_2}+G_{h_1', h_2'}$ contains the subgroup
\begin{align*}
    G'_{h_1, h_1', h_2, h_2'} = \langle \alpha \be_1 (h_1 h_2'-h_1' h_2),\beta\be_2 (h_1 h_2'-h_1' h_2)\rangle,
\end{align*}
and hence 
\begin{align*}
    \limsup_{H\to\infty}\brac{\E_{h_1, h_2\in[\pm H]}\nnorm{f}^+_{G_{h_1, h_2}}}^{16}\ll_k \limsup_{H\to\infty}\E_{h_1, h_1', h_2, h'_2\in[\pm H]}{\nnorm{f}^+_{G'_{h_1, h_1', h_2, h_2'}}}
\end{align*}
by Lemma \ref{L: seminorms of subgroups}. Importantly, the subgroups $G'_{h_1, h_1', h_2, h_2'}$ are finite index subgroups of $\langle \alpha \be_1, \beta\be_2\rangle$ for all $(h_1, h_1', h_2, h_2')\in\Z^4$ except a zero density set of tuples satisfying $h_1 h_2'-h_1' h_2 = 0$, and so by Lemma \ref{L: seminorms of subgroups}, we have
\begin{align*}
    \lim_{H\to\infty}\E_{h_1, h_2\in[\pm H]}\nnorm{f}^+_{G_{h_1, h_2}} = 0    \quad \textrm{whenever}\quad \nnorm{f}^+_{\langle \alpha \be_1, \beta\be_2\rangle} = 0.
\end{align*} 
 Later on, we will quantify this reasoning, showing that in fact
\begin{align*}
    \brac{\limsup_{H\to\infty}\E_{h_1, h_2\in[\pm H]}\nnorm{f}^+_{G_{h_1, h_2}}}^{O(1)}\ll_k \nnorm{f}^+_{\langle \alpha \be_1, \beta\be_2\rangle}.
\end{align*} 
\end{example}

The argument becomes more complicated when we deal with averages of generalized box seminorms of degree $s>1$. In handling this general case, we will iteratively use the lemma below, which is an ergodic version of \cite[Lemma 2.2]{Kuc23} and \cite[Lemma 6.1]{KKL24a}.
Its proof relies on rather elementary maneuvers that involve the Gowers-Cauchy-Schwarz inequality, multiple applications of the inductive formula \eqref{E: inductive formula} and a simple change of variables.

\begin{lemma}\label{L: concatenation lemma}
    Let $s\in\N$, $I$ be a finite indexing set and $G_i, K_{1i}, \ldots, K_{si}$ be finitely generated subgroups of $\R^k$ for each $i\in I$. For each 1-bounded function $f\in L^\infty(\mu)$, we have
    \begin{align*}
        \brac{{\E_{i\in I}\nnorm{f}_{G_i, K_{1i}, \ldots, K_{si}}^{+}}}^{2^{3s+4}} \ll_{k,s} \E_{i, i'\in I}\brac{\nnorm{f}_{K_{1i}, \ldots, K_{si}, K_{1 i'}, \ldots, K_{s i'}, G_i+G_{i'}}^{+}}^{2^{2s+2}}.
    \end{align*}
\end{lemma}

\begin{proof} 
    Let $\veps = {\E_{i\in I}\nnorm{f}_{G_i, K_{1i}, \ldots, K_{si}}^{+}}$ and $K_i = K_{1i}\times \cdots\times K_{si}$. 
    Using the H\"older inequality, expanding the definition of the seminorm and passing to the product system, we obtain
    \begin{align*}
        \E_{i\in I} \E_{\substack{\bk, \bk'\in K_i}}\E_{\bm, \bm'\in G_i}\int \Delta_{(\bk_1, \bk'_1), \ldots, (\bk_s, \bk'_s), (\bm, \bm')}\tilde{f}\; d\tilde{\mu} \geq
        \veps^{2^{s+2}}, 
    \end{align*}
    where $\tilde{T}=T\times T,\; \tilde{f}=f\otimes \overline{f}$ and $\tilde{\mu}=\mu\times\mu$. Composing the integral with $\tilde{T}^{-\floor{\bm}}$ and changing the order of summations, we obtain
    \begin{multline*}
        \int \tilde{f}\cdot \E_{i\in I} \E_{\substack{\bk, \bk'\in K_i}}\E_{\bm, \bm'\in G_i}\tilde{T}^{\floor{\bm'} - \floor{\bm}}\overline{\tilde{f}}\\ \prod_{\ueps\in\{0,1\}^{s}_{\ast}}\CC^{|\ueps|} T^{\floor{\bk_1^{\eps_1}}+\cdots + \floor{\bk_s^{\eps_s}}}(\tilde{f}\cdot \tilde{T}^{\floor{\bm'} - \floor{\bm}}\overline{\tilde{f}}) \; d\tilde{\mu} \geq
        \veps^{2^{s+2}},
    \end{multline*}
    where we recall that $\{0,1\}^s_* = \{0,1\}^s\setminus\{\underline{0}\}$.
We then apply the Cauchy-Schwarz inequality to remove $\Tilde{f}$ and change the order of summation back, so that
\begin{multline*}
    \E_{i, i'\in I} \E_{\substack{\bm, \bm'\in G_i,\\ \bm'',\bm'''\in G_{i'}}}\E_{\substack{\bk,\bk' \in K_i,\\ \bk'', \bk''' \in K_{i'}}}\int \tilde{T}^{\floor{\bm'} - \floor{\bm}}\overline{\tilde{f}} \cdot \tilde{T}^{\floor{\bm'''}-\floor{\bm''}}{\tilde{f}}\\
    \prod_{\ueps\in\{0,1\}^{s}_{\ast}}\CC^{|\ueps|} T^{\floor{\bk_1^{\eps_1}}+\cdots + \floor{\bk_s^{\eps_s}}}(\tilde{f}\cdot \tilde{T}^{\floor{\bm'} - \floor{\bm}}\overline{\tilde{f}})\\
    \prod_{\ueps'\in\{0,1\}^{s}_{\ast}}\CC^{|\ueps'|} T^{\floor{{\bk''_1}^{\eps'_1}}+\cdots + \floor{{\bk''_s}^{\eps'_s}}}(\overline{\tilde{f}}\cdot \tilde{T}^{\floor{\bm'''}-\floor{\bm''}}\tilde{f})\; d\tilde{\mu} \geq
        \veps^{2^{s+3}},
\end{multline*}
where
\begin{align*}
    \bk_i^{\eps_i} = \begin{cases}
\bk_i,\; &\eps_i = 0\\ 
\bk'_i,\; &\eps_i = 1
\end{cases},\quad \textrm{and}\quad 
    {\bk''_i}^{\eps'_i} = \begin{cases}
\bk''_i,\; &\eps'_i = 0\\ 
\bk'''_i,\; &\eps'_i = 1
\end{cases}.
\end{align*}

The crucial observation is that for each fixed $i,i'\in I$, $\bm,\bm'\in G_i,$ and $\bm'',\bm'''\in G_{i'}$, the average
\begin{multline*}
\E_{\substack{\bk,\bk' \in K_i,\\ \bk'', \bk''' \in K_{i'}}}\int \tilde{T}^{\floor{\bm'} - \floor{\bm}}\overline{\tilde{f}} \cdot \tilde{T}^{\floor{\bm'''}-\floor{\bm''}}{\tilde{f}}\cdot
    \prod_{\ueps\in\{0,1\}^{s}_{\ast}}\CC^{|\ueps|} T^{\floor{\bk_1^{\eps_1}}+\cdots + \floor{\bk_s^{\eps_s}}}(\tilde{f}\cdot \tilde{T}^{\floor{\bm'} - \floor{\bm}}\overline{\tilde{f}})\\
    \prod_{\ueps'\in\{0,1\}^{s}_{\ast}}\CC^{|\ueps'|} T^{\floor{{\bk''_1}^{\eps'_1}}+\cdots + \floor{{\bk''_s}^{\eps'_s}}}(\overline{\tilde{f}}\cdot \tilde{T}^{\floor{\bm'''}-\floor{\bm''}}\tilde{f})\; d\tilde{\mu}
\end{multline*}
is a Gowers-Cauchy-Schwarz inner product along $K_{1i}, \ldots, K_{si}, K_{1 i'}, \ldots, K_{s i'}$; more specifically, it
    can be written as 
       \begin{equation*}
    \begin{split}
     \E_{\substack{\bk,\bk' \in K_i,\\ \bk'', \bk''' \in K_{i'}}}\int \prod_{(\ueps,\ueps')\in\{0,1\}^{2s}}\CC^{|{\ueps}|+ |{\ueps'}|} \tilde{T}^{\floor{\bk_1^{\eps_1}}+\cdots + \floor{\bk_s^{\eps_s}} + \floor{{\bk''_1}^{\eps'_1}}+\cdots + \floor{{\bk''_s}^{\eps'_s}}}\tilde{g}_{\ueps, \ueps',\bm,\bm'}\, d\tilde{\mu}
       \end{split}
    \end{equation*}
    for some functions $\tilde{g}_{\ueps, \ueps',\bm,\bm'}=g_{\ueps, \ueps',\bm,\bm'}\otimes\overline{g_{\ueps, \ueps',\bm,\bm'}}$ with
    \begin{align}\label{E: 0 coordinate function}
        \tilde{g}_{\underline{{0}},\underline{{0}},\bm,\bm'} &=\tilde{T}^{\floor{\bm'} - \floor{\bm}}\overline{\tilde{f}} \cdot \tilde{T}^{\floor{\bm'''}-\floor{\bm''}}{\tilde{f}}\\
        \nonumber
        &= ({T}^{\floor{\bm'} - \floor{\bm}}\overline{{f}}\cdot \tilde{T}^{\floor{\bm'''}-\floor{\bm''}}{{f}})\otimes ({T}^{\floor{\bm'} - \floor{\bm}}{{f}}\cdot \tilde{T}^{\floor{\bm'''}-\floor{\bm''}}\overline{{f}}).
    \end{align}
    By the Gowers-Cauchy-Schwarz inequality (Lemma \ref{L: GCS}) and a change of the order of summations, we thus have
    \begin{align*}
        \E_{i, i'\in I} \E_{\substack{\bm, \bm'\in G_i,\\ \bm'',\bm'''\in G_{i'}}}\int\E_{\substack{\bk,\bk' \in K_i,\\ \bk'', \bk''' \in K_{i'}}} \Delta_{(\bk, \bk'), (\bk'', \bk''')}(\tilde{T}^{\floor{\bm'} - \floor{\bm}}\overline{\tilde{f}} \cdot \tilde{T}^{\floor{\bm'''}-\floor{\bm''}}{\tilde{f}})\; d\tilde{\mu} \geq \veps^{2^{3s+3}}. 
    \end{align*}
    The product structure of the function $\tilde{f}$ implies that each integral above is real and nonnegative (since it is a square of an integral over the original system). Composing the integral with $\tilde{T}^{-(\floor{\bm'} - \floor{\bm})}$ and moving all the summation and subtraction inside the integer part, we argue as in the proof of Lemma \ref{L: concatenation degree 1} to find $\br\in\Z^k$ such that 
    \begin{align*}
        \E_{i, i'\in I} \E_{\substack{\bm_0\in G_i,\\ \bm_1\in G_{i'}}}\E_{\substack{\bk,\bk' \in K_i,\\ \bk'', \bk''' \in K_{i'}}}\int \Delta_{(\bk, \bk'), (\bk'', \bk''')}(\overline{\tilde{f}} \cdot \tilde{T}^{\floor{\bm_1-\bm_0}+\br}{\tilde{f}})\; d\tilde{\mu} \gg_{k,s} \veps^{2^{3s+3}}.
    \end{align*}
    Lemma \ref{L: sumset lemma} and a rearrangement gives
    \begin{align*}
        \E_{i, i'\in I} \E_{\substack{\bk,\bk' \in K_i,\\ \bk'', \bk''' \in K_{i'}}}\int \Delta_{(\bk, \bk'), (\bk'', \bk''')}\overline{\tilde{f}} \cdot \E_{\substack{\bm\in G_i+G_{i'}}}\Delta_{(\bk, \bk'), (\bk'', \bk''')}\tilde{T}^{\floor{\bm}+\br}{\tilde{f}}\; d\tilde{\mu} \gg_{k,s} \veps^{2^{3s+3}}.
    \end{align*}
    A final application of the Cauchy-Schwarz inequality gives
    \begin{align*}
        \E_{i, i'\in I} \E_{\substack{\bk,\bk' \in K_i,\\ \bk'', \bk''' \in K_{i'}}}\E_{\substack{\bm,\bm'\in G_i+G_{i'}}}\int \Delta_{(\bk, \bk'), (\bk'', \bk'''), (\bm,\bm')}{\tilde{f}}\; d\tilde{\mu} \gg_{k,s} \veps^{2^{3s+4}},
    \end{align*}
    and the result follows upon moving back to the original system.
\end{proof}

By repeatedly applying Lemma \ref{L: concatenation lemma}, we derive the following concatenation result, which is an ergodic version of \cite[Proposition 2.3]{Kuc23} and \cite[Proposition 6.2]{KKL24a}.
\begin{proposition}[Concatenation of box seminorms, version I]\label{P: concatenation for general groups}
    Let $s\in\N$, $I$ be a finite indexing set and $G_{1i}, \ldots, G_{si}\subseteq \R^k$ be finitely generated subgroups for each $i\in I$. For all 1-bounded functions $f\in L^\infty(\mu)$, we have
    \begin{align*}
        \brac{\E_{i\in I}\nnorm{f}_{G_{1i}, \ldots, G_{si}}^+}^{O_{s}(1)} \ll_{k,s} \E_{\substack{i_\ueps\in I:\\ \ueps\in\{0,1\}^s}}\nnorm{f}^+_{\substack{\{G_{j i_\ueps}+G_{j i_{\ueps'}}:\ j\in[s],\ \ueps,\ueps'\in\{0,1\}^s\\ \mathrm{with}\ (\epsilon_1, \ldots,  \epsilon_{s-j}) = (\epsilon'_1, \ldots,  \epsilon'_{s-j}),\ \epsilon_{s+1-j} =1,\ \epsilon'_{s+1-j}=0\}}}.
    \end{align*}
\end{proposition}

\begin{proof}
    The proof of Proposition \ref{P: concatenation for general groups} relies on a gradual concatenation of the ``unconcatenated'' subgroups $G_{1i}, \ldots, G_{si}$ in the average using Lemma \ref{L: concatenation lemma}. We repeatedly use the inductive formula \eqref{E: inductive formula} in order to reinterpret the average in such a way  that successive applications of Lemma \ref{L: concatenation lemma} concatenate the subgroups $G_{1i}, \ldots, G_{si}$ one by one. 
    
    In order to simplify the notation, we shall write in this proof $(\bm,\bk)$ for what we would typically write as $(\bm,\bm')$; we also assume that $\bm_j^{\ueps}, \bk_j^{\ueps}$ are always elements of $G_{j i_\ueps}$ and write $\E_{\bm_j^{\ueps}, \bk_j^\ueps} = \E_{(\bm_j^{\ueps}, \bk_j^\ueps)\in G_{j i_\ueps}\times G_{j i_\ueps}}$.
    We let all the constants depend on $k,s$, noting however that the powers of $\veps$ will not depend on $k$.

       \smallskip
    \textbf{The case $s=2$:}
    \smallskip
    
    For illustrative purposes, we first prove  the statement of Proposition \ref{P: concatenation for general groups} for degree 2 box seminorms.  Let $\veps = \E_{i\in I}\nnorm{f}^+_{G_{1i}, G_{2 i}}$. We first want to concatenate the group $G_{2 i}$. By Lemma \ref{L: concatenation lemma}, we have
\begin{align}\label{E: concatenation s=2 1}
    \E_{i_0, i_1\in I}{(\nnorm{f}^+_{G_{1 i_0}, G_{1 i_1}, G_{2 i_0}+G_{2 i_1}})}^{2^4}\gg\veps^{O(1)},
\end{align}
and so the group $G_{2 i}$ indeed got concatenated. The price we paid for this is that while applying Lemma \ref{L: concatenation lemma}, the unconcatenated group $G_{1i}$ doubled into $G_{1 i_0}$ and $G_{1 i_1}$. The next goal is therefore to concatenate these groups one by one. Using the inductive formula for box seminorms and the 1-boundedness of $f$, we deduce from \eqref{E: concatenation s=2 1} that\footnote{The only reason why we keep the seminorm $\nnorm{\Delta_{(\bm^1_1, \bk^1_1)}f}^+_{G_{1 i_0}, G_{2 i_0} + G_{2 i_1}}$ raised to the power $2^s$ is so that the averages $\E_{\bm^1_1, \bk^1_1}\E_{i_0\in I}\brac{\nnorm{\Delta_{(\bm^1_1, \bk^1_1)}f}^+_{G_{1 i_0}, G_{2 i_0} + G_{2 i_1}}}^{2^3}$ are guaranteed to converge by Lemma~\ref{L: convergence}. Without the power of $2^3$ being there, the convergence would not be guaranteed. }
\begin{align*}
    \E_{i_1\in I}\E_{\bm^1_1, \bk^1_1}\E_{i_0\in I}\brac{\nnorm{\Delta_{(\bm^1_1, \bk^1_1)}f}^+_{G_{1 i_0}, G_{2 i_0} + G_{2 i_1}}}^{2^3}\gg\veps^{O(1)}. 
\end{align*}
Applying Lemma \ref{L: concatenation lemma} to each\footnote{Technically, to use Lemma \ref{L: concatenation lemma} as stated, we need first to use 1-boundedness to bound $\E_{i_0\in I}\brac{\nnorm{\Delta_{(\bm^1_1, \bk^1_1)}f}^+_{G_{1 i_0}, G_{2 i_0} + G_{2 i_1}}}^{2^3}\leq \E_{i_0\in I}{\nnorm{\Delta_{(\bm^1_1, \bk^1_1)}f}^+_{G_{1 i_0}, G_{2 i_0} + G_{2 i_1}}}$.}
$\E_{i_0\in I}\brac{\nnorm{\Delta_{(\bm^1_1, \bk^1_1)}f}^+_{G_{1 i_0}, G_{2 i_0} + G_{2 i_1}}}^{2^3}$, we deduce that
\begin{align*}
    \E_{i_{00}, i_{01}, i_1\in I}\E_{\bm^1_1, \bk^1_1}\brac{\nnorm{\Delta_{(\bm^1_1, \bk^1_1)}f}^+_{G_{1 i_{00}}+G_{1 i_{01}},  G_{2 i_{00}} + G_{2 i_1}, G_{2 i_{01}} + G_{2 i_1}}}^{2^4}\gg\veps^{O(1)}.
\end{align*}
The inductive formula for box seminorms and Lemma \ref{L: sumset lemma} once again give us 
\begin{align*}
    \E_{i_{00}, i_{01}\in I}\E_{\substack{\bm^{00}_1, \bk^{00}_1,\\ \bm^{01}_1, \bk^{01}_1}} \E_{i_1\in I}\brac{\nnorm{\Delta_{(\bm^{00}_1+\bm^{01}_1, \bk^{00}_1+\bk^{01}_1)}f}^+_{G_{1 i_1},  G_{2 i_{00}} + G_{2 i_1}, G_{2 i_{01}} + G_{2 i_1}}}^{2^4}\gg\veps^{O(1)}.
\end{align*}
We then apply Lemma \ref{L: concatenation lemma} for the last time, this time to each expression $$\E_{i_1\in I}\brac{\nnorm{\Delta_{(\bm^{00}_1+\bm^{01}_1, \bk^{00}_1+\bk^{01}_1)}f}^+_{G_{1 i_1},  G_{2 i_{00}} + G_{2 i_1}, G_{2 i_{01}} + G_{2 i_1}}}^{2^4},$$ obtaining
\begin{align*}
    \E_{i_{00}, i_{01}, i_{10}, i_{11}\in I}\E_{\substack{\bm^{00}_1, \bk^{00}_1,\\ \bm^{01}_1, \bk^{01}_1}} \brac{\nnorm{\Delta_{(\bm^{00}_1+\bm^{01}_1, \bk^{00}_1+\bk^{01}_1)}f}^+_{\substack{G_{1 i_{10}}+G_{1 i_{11}},  G_{2 i_{00}} + G_{2 i_{10}}, G_{2 i_{00}} + G_{2 i_{11}},\\  G_{2 i_{01}} + G_{2 i_{10}}, G_{2 i_{01}} + G_{2 i_{11}}}}}^{2^{6}}\gg\veps^{O(1)}.    
\end{align*}
The inductive formula for box seminorms, 1-boundedness, and Lemma \ref{L: sumset lemma} then imply that
\begin{align*}
    \E_{i_{00}, i_{01}, i_{10}, i_{11}\in I}\nnorm{f}^+_{\substack{G_{1 i_{00}}+ G_{1 i_{01}}, G_{1 i_{10}}+G_{1 i_{11}},  G_{2 i_{00}} + G_{2 i_{10}},\\ G_{2 i_{00}} + G_{2 i_{11}},  G_{2 i_{01}} + G_{2 i_{10}}, G_{2 i_{01}} + G_{2 i_{11}}}}\gg\veps^{O(1)},
\end{align*}
finishing the $s=2$ case.

 We note that the proof of Proposition~\ref{P: concatenation for general groups} for $s=2$ relies on 3 applications of Lemma~\ref{L: concatenation lemma}. More generally, the proof for an arbitrary $s\geq 2$ will require $1+2+\cdots + 2^{s-1} = 2^s-1$ applications of Lemma~\ref{L: concatenation lemma}.

    \smallskip
    \textbf{The general case:}
    \smallskip

We move on to prove the general case. Starting with $\veps = \E_{i\in I}\nnorm{f}_{G_{1i}, \ldots, G_{si}}^+$, we apply Lemma~\ref{L: concatenation lemma} to bound
\begin{align}\label{E: concatenation s>2 1}
    \E_{i_0, i_1\in I}\brac{\nnorm{f}^+_{G_{1 i_0}, \ldots, G_{(s-1) i_0}, G_{1 i_1}, \ldots, G_{(s-1) i_1}, G_{s i_0}+G_{s i_1}}}^{2^{2s}}\gg\veps^{O(1)}.
\end{align}
We note that we passed from having $s$ unconcatenated groups indexed by $i$ to $s-1$ unconcatenated groups indexed by $i_0$ and another $s-1$ unconcatenated groups indexed by $i_1$ (in addition to the concatenated group $G_{s i_0}+G_{s i_1}$). Thus, the total number of groups almost doubled, but what matters is that for each index $i_0, i_1$, the number of unconcatenated groups with this index went down by 1. At the next stage of the argument, we will apply Lemma~\ref{L: concatenation lemma} twice to concatenate $G_{(s-1) i_0}$ first and then $G_{(s-1) i_1}$. As a consequence, the two indices $i_0, i_1$ will be replaced by four indices $i_{00}, i_{01}, i_{10}, i_{11}$, and for each of them we will have exactly $s-2$ unconcatenated groups. 
We will continue in this manner: at each stage, the number of indices $i_\ueps$ will double, but the number of unconcatenated groups with each index $i_\ueps$ will decrease by 1. Eventually, on the $s$-th step, we will be left with $2^{s-1}$ unconcatenated groups $G_{1 i_\ueps}$, and $2^{s-1}$ applications of Lemma~\ref{L: concatenation lemma} will allow us to concatenate them all without producing any new unconcatenated groups. This will finish the argument.

This is the general strategy; let us see in detail what happens at the second stage, i.e., after obtaining the bound \eqref{E: concatenation s>2 1}. Using the inductive formula for box seminorms, we can rephrase \eqref{E: concatenation s>2 1} as
\begin{align*}
    \E_{i_1\in I} \E_{\substack{\bm^1_1, \ldots, \bm^1_{s-1},\\ \bk^1_1, \ldots, \bk^1_{s-1}}}\E_{i_0\in I}\brac{\nnorm{\Delta_{(\bm^1_1, \bk^1_1), \ldots, (\bm^1_{s-1}, \bk^1_{s-1})} f}^+_{G_{1 i_0}, \ldots, G_{(s-1) i_0}, G_{s i_0}+G_{s i_1}}}^{2^{s+1}}\gg\veps^{O(1)}.
\end{align*}
 For each fixed $i_1, \bm^1_1, \ldots, \bm^1_{s-1}, \bk^1_1, \ldots, \bk^1_{s-1}$, we apply Lemma~\ref{L: concatenation lemma} separately to each average over $i_0$, obtaining 
\begin{multline*}
    \E_{i_1\in I} \E_{\substack{\bm^1_1, \ldots, \bm^1_{s-1},\\ \bk^1_1, \ldots, \bk^1_{s-1}}} \E_{i_{00}, i_{01}\in I}\\
    \brac{\nnorm{\Delta_{(\bm^1_1, \bk^1_1), \ldots, (\bm^1_{s-1}, \bk^1_{s-1})} f}^+_{\substack{G_{1 i_{00}}, \ldots, G_{(s-2) i_{00}}, G_{1 i_{01}}, \ldots, G_{(s-2) i_{01}},\\ G_{(s-1) i_{00}}+G_{(s-1) i_{01}}, G_{s i_{00}}+G_{s i_1}, G_{s i_{01}}+G_{s i_1}}}}^{2^{2s}}\gg\veps^{O(1)}.    
\end{multline*}
We rearrange the inequality above using the inductive formula for box seminorms as 
\begin{multline*}
    \E_{i_{00}, i_{01}\in I} 
    \E_{\substack{\bm^{00}_1, \ldots, \bm^{00}_{s-1},\\ \bm^{01}_1, \ldots, \bm^{01}_{s-1}}} 
    \E_{\substack{\bk^{00}_1, \ldots, \bk^{00}_{s-1},\\ \bk^{01}_1, \ldots, \bk^{01}_{s-1}}} \E_{i_1\in I}\\
    \brac{\bignnorm{\Delta_{\substack{(\bm^{00}_1, \bk^{00}_1), \ldots, (\bm^{00}_{s-2}, \bk^{00}_{s-2}),\\ (\bm^{01}_1, \bk^{01}_1), \ldots, (\bm^{01}_{s-2}, \bk^{01}_{s-2}),\\ (\bm^{00}_{s-1}+\bm^{01}_{s-1}, \bk^{00}_{s-1}+\bk^{01}_{s-1})}} f}^+_{\substack{G_{1 i_1}, \ldots, G_{(s-1) i_1},\\ G_{s i_{00}}+G_{s i_1}, G_{s i_{01}}+G_{s i_1}}}}^{2^{s+2}}\gg\veps^{O(1)}.      
\end{multline*}
in order to concatenate $G_{(s-1) i_1}$. By Lemma \ref{L: concatenation lemma} applied separately to each average over $i_1\in I$, we have
\begin{multline*}
    \E_{\substack{i_{00}, i_{01},\\ i_{10}, i_{11}\in I}}
    \E_{\substack{\bm^{00}_1, \ldots, \bm^{00}_{s-1},\\ \bm^{01}_1, \ldots, \bm^{01}_{s-1}}} 
    \E_{\substack{\bk^{00}_1, \ldots, \bk^{00}_{s-1},\\ \bk^{01}_1, \ldots, \bk^{01}_{s-1}}} \\
    \brac{\bignnorm{\Delta_{\substack{(\bm^{00}_1, \bk^{00}_1), \ldots, (\bm^{00}_{s-2}, \bk^{00}_{s-2}),\\ (\bm^{01}_1, \bk^{01}_1), \ldots, (\bm^{01}_{s-2}, \bk^{01}_{s-2}),\\ (\bm^{00}_{s-1}+\bm^{01}_{s-1}, \bk^{00}_{s-1}+\bk^{01}_{s-1})}} f}^+_{\substack{G_{1 i_{10}}, \ldots, G_{(s-2) i_{10}}, G_{1 i_{11}}, \ldots, G_{(s-2) i_{11}},\\ G_{(s-1) i_{10}}+G_{(s-1) i_{11}}, G_{s i_{00}}+G_{s i_{10}}, G_{s i_{00}}+G_{s i_{11}},\\ G_{s i_{01}}+G_{s i_{10}}, G_{s i_{01}}+G_{s i_{11}}}}}^{2^{2s+2}}\gg\veps^{O(1)}. 
\end{multline*}
An application of the inductive formula for box seminorms and 1-boundedness then gives
\begin{align*}
    \E_{\substack{i_{00}, i_{01},\\ i_{10}, i_{11}\in I}} {\nnorm{f}^+_{\substack{G_{1 i_{00}}, \ldots, G_{(s-2) i_{00}}, G_{1 i_{01}}, \ldots, G_{(s-2) i_{01}},\\ G_{1 i_{10}}, \ldots, G_{(s-2) i_{10}}, G_{1 i_{11}}, \ldots, G_{(s-2) i_{11}},\\ G_{(s-1) i_{00}}+G_{(s-1) i_{01}}, G_{(s-1) i_{10}}+G_{(s-1) i_{11}},\\ G_{s i_{00}}+G_{s i_{10}}, G_{s i_{00}}+G_{s i_{11}},\\ G_{s i_{01}}+G_{s i_{10}}, G_{s i_{01}}+G_{s i_{11}}}}}\gg\veps^{O(1)},
\end{align*}
which can be written more compactly as
\begin{align*}
   \E_{\substack{i_{00}, i_{01},\\ i_{10}, i_{11}\in I}} {\nnorm{f}^+_{\substack{\{G_{j i_\ueps}:\ \ueps\in\{0,1\}^2,\ j=1, \ldots, s-2\},\\ \{G_{(s-1) i_\ueps}+G_{(s-1) i_{\ueps'}}:\ \ueps,\ueps'\in\{0,1\}^2\ \mathrm{with}\ \epsilon_1 = \epsilon'_1,\ \epsilon_2=1, \epsilon'_2=0\},\\  
    \{G_{s i_\ueps}+G_{s i_{\ueps'}}:\ \ueps,\ueps'\in\{0,1\}^2\ \mathrm{with}\ \epsilon_1=1, \epsilon'_1=0\}}}}\gg\veps^{O(1)}.
\end{align*}

We have thus successfully concatenated all the groups $G_{(s-1) i_\ueps}$ and $G_{s i_\ueps}$. 

At the next stage, we concatenate the groups $G_{(s-2) i_\ueps}$. Applying Lemma~\ref{L: concatenation lemma} and the inductive formula for box seminorms four times like before, each time to an average over $i_{00}, i_{01}, i_{10}, i_{11}$ respectively, we arrive at the inequality
\begin{align*}
        \E_{\substack{i_{\ueps}\in I,\\ \ueps\in\{0,1\}^3}} \nnorm{f}^+_{\substack{\{G_{j i_\ueps}:\ \ueps\in\{0,1\}^3,\ j\in[s-3]\},\\ \{G_{j i_\ueps}+G_{j i_{\ueps'}}:\ j= s-2, s-1, s,\ \ueps,\ueps'\in\{0,1\}^3\\ \mathrm{with}\ (\epsilon_1, \ldots, \epsilon_{s-j}) = (\epsilon'_1, \ldots, \epsilon'_{s-j}),\ \epsilon_{s+1-j}=1,\; \epsilon'_{s+1-j}=0\}}}\gg\veps^{O(1)}.
\end{align*}
This time, we have successfully concatenated the groups $G_{(s-2) i_\ueps}$. At the next step, 8 applications of Lemma~\ref{L: concatenation lemma} and the inductive formula for box seminorms allow us to concatenate groups $G_{(s-3) i_\ueps}$. Continuing the argument like this, we arrive, after a total of 
\begin{align*}
    1 + 2 + 2^2 + \cdots + 2^{s-1} = 2^s - 1
\end{align*}
applications of Lemma~\ref{L: concatenation lemma} and the inductive formula for box seminorms, at the claimed inequality. 
\end{proof}

For applications, the following weaker but notationally lighter corollary of Proposition~\ref{P: concatenation for general groups} is sufficient. It can be seen as an ergodic version of \cite[Corollary~2.4]{Kuc23} or \cite[Corollary~6.3]{KKL24a}.

\begin{corollary}[Concatenation of box seminorms, version II]\label{C: concatenation for general groups II}
Let $s\in\N$, $t=2^s$, $I$ be a finite indexing set and $G_{1i}, \ldots, G_{si}\subseteq\R^k$ be finitely generated subgroups for each $i\in I$. For all 1-bounded functions $f\in L^\infty(\mu)$, we have
    \begin{align}\label{E: concatenation simplified}
        \brac{\E_{i\in I}{\nnorm{f}^+_{G_{1i}, \ldots, G_{si}}}}^{O_s(1)} \ll_{k,s} \E_{\substack{k_1, \ldots, k_t\in I}}\nnorm{f}^+_{\{G_{j k_{i_1}}+G_{j k_{i_2}}:\ j\in[s],\ 1\leq i_1 < i_2 \leq t\}}.
    \end{align}
\end{corollary}

 Corollary \ref{C: concatenation for general groups II} plays the role of the Bessel-type inequality from \cite[Corollary 1.22]{TZ16}. However, its advantage over \cite[Corollary 1.22]{TZ16} is that we are quantitatively comparing averages of generalized box seminorms rather than the $L^2(\mu)$ norms of projections onto appropriate Host-Kra factors. Since usually we have no quantitative information on the latter, the results from \cite{TZ16} are in effect of purely qualitative nature, and they were used as such in previous ergodic works on polynomial concatenation \cite{DFKS22, DKS22}. The most important difference is of course that we quantitatively compare two averages of box seminorms while an analogous comparison in \cite{TZ16} is fully qualitative. By contrast, the quantitative nature of Corollary \ref{C: concatenation for general groups II} is a key improvement that allows us to obtain quantitative statements in Theorems \ref{T: box seminorm bound intro} and \ref{T: HK control}.

 A further difference between Corollary \ref{C: concatenation for general groups II} and \cite[Corollary 1.22]{TZ16} concerns the nature of seminorms and groups that appear in the right-hand side of the inequality \eqref{E: concatenation simplified}. In our result, we only sum subgroups with the same index $j$ whereas results from \cite{TZ16} involve subgroups of the form $G_{jk}+G_{j'k'}$ for all indices $j,j'$. The price that we pay is that our argument necessitates the introduction of $t=2^s$ indices $k_1, \ldots, k_t$ on the right hand side of \eqref{E: concatenation simplified} while the arguments in \cite{TZ16} allow to average on the right hand side over only two indices $k, k'$, and also the degree of the generalized box seminorms in our output is much larger than the one in \cite{TZ16}.
 At the end of the day, however, these inefficiencies are of little importance; what matters is that the number $t$ of indices and the degree of the seminorms are both $O_s(1)$.
 
\begin{proof}
    Let $\veps = \E_{i\in I}{\nnorm{f}^+_{G_{1i}, \ldots, G_{si}}}$. We let all the constants depend on $k,s$, noting however that powers of $\veps$ will not depend on $k$.
    Proposition~\ref{P: concatenation for general groups} immediately implies that
    \begin{align*}
        \E_{\substack{i_\ueps\in I:\\ \ueps\in\{0,1\}^s}}\nnorm{f}^+_{\substack{\{G_{j i_\ueps}+G_{j i_{\ueps'}}:\ j\in[s],\ \ueps,\ueps'\in\{0,1\}^s\\ \mathrm{with}\ (\epsilon_1, \ldots,  \epsilon_{s-j}) = (\epsilon'_1, \ldots,  \epsilon'_{s-j}),\ \epsilon_{s+1-j} =1,\ \epsilon'_{s+1-j}=0\}}} \gg \veps^{O(1)}.
    \end{align*}
    Ordering $\{0,1\}^s$ lexicographically and identifying it with $[2^s]=[t]$ in an order-preserving way, we obtain a subset $\CK_0$ of $\CK := $$\{(i_1,i_2) \in [t]\times [t]: i_1<i_2\}$ such that
    \begin{align*}
        \E_{\substack{k_1, \ldots, k_t\in I}}\nnorm{f}^+_{\substack{\{G_{j k_{i_1}}+G_{j k_{i_2}}:\; j\in[s],\ ({i_1}, {i_2})\in\CK_0\}}} \gg \veps^{O(1)}.
    \end{align*}    
    By Lemma~\ref{L: monotonicity property}, we can replace $\CK_0$ by the full set $\CK$, giving the claimed result. 
\end{proof}

We shall use the following consequence of Corollary~\ref{C: concatenation for general groups II}, which can be compared with \cite[Corollary 2.5]{Kuc23} and \cite[Corollary 6.4]{KKL24a}, and which is obtained from an iterated application of Corollary~\ref{C: concatenation for general groups II}. Its advantage is that it allows us to take the directions in the concatenated box seminorms to be arbitrarily long sums of the original directions rather than just double sums, as is the case in Corollary~\ref{C: concatenation for general groups II}.
\begin{corollary}[Iterated concatenation of box seminorms]\label{C: iterated concatenation for general groups}
    Let $s\in\N$, $\ell$ be a nonnegative integer power of 2, $I$ be a finite indexing set and $G_{1i}, \ldots, G_{si}\subseteq \R^k$ be finitely generated subgroups for each $i\in I$. There exists a positive integer $t=O_{\ell, s}(1)$ such that for all 1-bounded functions $f\in L^\infty(\mu)$, we have
    \begin{align}\label{E: iterated concatenation}
         \brac{\E_{i\in I}\nnorm{f}^+_{G_{1i}, \ldots, G_{si}}}^{O_{\ell, s}(1)} \ll_{k, \ell, s} \E_{\substack{k_1, \ldots, k_t\in I}}\nnorm{f}^+_{\substack{G_{jk_{i_1}}+\cdots + G_{jk_{i_\ell}}:\; j\in[s],\; 1\leq i_1 < \cdots < i_\ell\leq t}}.
    \end{align}
\end{corollary}

It will be important for our applications later on that in the sum $G_{jk_{i_1}}+\cdots + G_{jk_{i_\ell}}$ inside \eqref{E: iterated concatenation},
the indices $i_1, \ldots, i_\ell$ are all distinct since otherwise the group might not be larger than its summands.

\begin{proof}
    The proof proceeds by induction on $\ell$, which we recall is a power of 2. For $\ell=1$, the claim is trivial (with $t=1$), and for $\ell=2$, it follows from Corollary~\ref{C: concatenation for general groups II} (with $t=2^s$). We will just prove the cases $\ell = 4$ since this will make the strategy for the general case clear enough. 

    Let $\veps = \E_{i\in I}{\nnorm{f}^+_{G_{1i}, \ldots, G_{si}}}$. Then by Corollary \ref{C: concatenation for general groups II}, we have
    \begin{align}\label{E: iterated concatenation induction}
        \E_{\substack{k_1, \ldots, k_{t_1}\in I}}\nnorm{f}^+_{\{G_{j k_{i_1}}+G_{j k_{i_2}}:\ j\in[s],\ 1\leq i_1 < i_2 \leq t_1\}} \gg \veps^{O(1)}
    \end{align}
    with $t_1 = 2^s$
    (once again, the constants are allowed to depend on $k,\ell, s$ except that the powers of $\veps$ do not depend on $k$).
    To prove the result for $\ell=4$, we apply Corollary~\ref{C: concatenation for general groups II} to \eqref{E: iterated concatenation induction}, taking $(k_1, \ldots, k_{t_1})\in I^{t_1}$ in place of $i\in I$. This gives
    \begin{align}\label{E: iterated concatenation induction 2}
        \E_{\substack{k_{i l}\in I:\\ i\in[t_1],\; l\in[t_1']}}\nnorm{f}^+_{\substack{\{G_{j k_{i_1l_1}}+G_{j k_{i_2l_1}} + G_{j k_{i_1l_2}}+G_{j k_{i_2l_2}}:\\ j\in[s],\; 1\leq i_1< i_2\leq t_1,\; 1\leq l_1<l_2\leq t_1'\}}}\gg\veps^{O(1)}
    \end{align}
    for some positive integer $t_1' = O(1)$. Reindexing $(i, l)\in[t_1]\times[t_1']$ as $i\in[t_2]$ and using the monotonicity property (Lemma~\ref{L: monotonicity property}), we argue as in the proof of Corollary~\ref{C: concatenation for general groups II} that
    \begin{align*}
                \E_{\substack{k_1, \ldots, k_{t_2}\in I}}\nnorm{f}^+_{\{G_{j k_{i_1}}+\cdots + G_{j k_{i_4}}:\ j\in[s],\ 1\leq i_1< \cdots < i_4\leq t_2\}}\gg\veps^{O(1)}.
    \end{align*}
    This proves the case $\ell = 4$. We continue in the same manner for larger $\ell$.
\end{proof}

\section{Standard techniques for averages along Hardy sequences}\label{S: preliminaries}

We have so far proved basic properties of generalized box seminorms and developed the theory of their quantitative concatenation. From now on, we will show how these seminorms can be used to control multiple ergodic averages along Hardy sequences. Before we embark on this, we list several standard technical lemmas that shall be profusely applied in the process. 
We start with two variants of the classical van der Corput lemma. The derivation of van der Corput-type inequalities is standard, and we refer the reader to \cite{BM16} for more background. 
\begin{lemma}\label{L: finitary van der Corput}
Let $N\in\mathbb{N}$ and $v_1, \ldots , v_N$ be 1-bounded vectors in an inner product space. For every natural number $H\leq N$, we have
\begin{equation*}
\norm{\E_{n\in [N]}v_n}^2\leq 6\E_{h\in[\pm H]}\abs{\E_{n\in[N]}\langle v_n, v_{n+h}\rangle} + \frac{6H}{N}
\end{equation*} 
and
\begin{equation*}
\norm{\E_{n\in [N]}v_n}^2\leq 3\E_{h,h'\in[\pm H]}\E_{n\in[N]}\langle v_{n+h}, v_{n+h'}\rangle + \frac{3H}{N}.
\end{equation*} 
\end{lemma}

We then move on to prove the finitary version of \cite[Proposition 6.1]{Fr12} which allows us to remove from our average dual functions evaluated at Hardy sequences. We state two versions (based on two different versions of the van der Corput trick from Lemma \ref{L: finitary van der Corput}), each of which will be useful at different stages of the PET argument. We recall for the convenience of the reader that the collection $\FD_d$ of dual sequences of degree at most $d$ is defined in Section \ref{SS: dual functions}.

\begin{lemma}\label{L: removing duals finitary}
			Let $d, J, k\in\N$, and suppose that $b_1, \ldots, b_J\in \mathcal{H}$ satisfy the growth condition $b_j(t)\ll t^d$. Then there exists a positive integer $s=O_{d, J}(1)$ such that for all systems $(X,\mathcal{X},\mu, T_1, \ldots, T_k)$, sequences of 1-bounded functions $A$ and $\CD_1, \ldots, \CD_J\in \FD_d$, and numbers $n_0\in\N_0$ and $K, N\in\N$ with $K\leq N$, we have
			\begin{multline*}
				\norm{\E_{n\in[N]} A(n_0+n)\cdot\prod_{j\in[J]} \CD_{j}(\floor{b_j(n_0+n)})}_{L^2(\mu)}^{2^s}\\
				\ll_{d,J, k} \E\limits_{\uk\in[\pm K]^s}    \sup_{|c_n|\leq 1}    \norm{\E_{n\in[N]} c_n \cdot \prod_{\ueps \in\{0,1\}^s} \mathcal{C}^{|\ueps|}A(n_0 + n+\ueps \cdot\uk)}_{L^2(\mu)}\\
    +\frac{K}{N}+ o_{N\to\infty; d, J, k, b_1, \ldots, b_J}(1).
			\end{multline*}
   Similarly, we have
   \begin{multline*}
				\norm{\E_{n\in[N]} A(n_0+n)\cdot\prod_{j\in[J]} \CD_{j}(\floor{b_j(n_0+n)})}_{L^2(\mu)}^{2^s}\\
				\ll_{d,J, k} \E\limits_{\uk, \uk'\in[\pm K]^s}    \sup_{|c_n|\leq 1}    \norm{\E_{n\in[N]} c_n \cdot \prod_{\ueps \in\{0,1\}^s} \mathcal{C}^{|\ueps|}A(n_0 + n + ({\underline{1}}-\ueps) \cdot\uk + \ueps\cdot\uk'
    )}_{L^2(\mu)}\\
    +\frac{K}{N}+ o_{N\to\infty; d, J, k, b_1, \ldots, b_J}(1).
			\end{multline*}

   If $b_1, \ldots, b_J$ are polynomials, then the term $o_{N\to\infty; d, J, k, b_1, \ldots, b_J}(1)$ can be taken to be 0.
		\end{lemma}

\begin{proof}
The proof proceeds by a bounded number of applications of the van der Corput inequality in a PET-like manner (much like in \cite[Proposition 6.1]{Fr12}), where at each step we reduce to an average of smaller complexity.
To obtain the first claim, we will repeatedly use the first part of Lemma~\ref{L: finitary van der Corput} in order to remove all the dual terms. The second claim will follow analogously from the second part of Lemma~\ref{L: finitary van der Corput}. 

Our first task is to rephrase the average that we study in a way that facilitates the induction. By the definition of the dual functions, we can find {positive integers $l,\ell=O_{d, J}(1)$}, tuples of Hardy field functions $\bb_1,\ldots, \bb_\ell\in\{0,b_1,\ldots,b_J\}^d$, and 1-bounded sequences $d_1,\ldots, d_\ell:{\Z}^l\to L^{\infty}(\mu)$ for which 
\begin{align}\label{E: dual representation}
    \prod_{j\in [J]} \CD_{j}(\floor{b_j(n_0+n)})=\lim_{M\to\infty}\E_{{\bm}\in [\pm M]^l}\prod_{j\in [\ell]} d_{j}({\bm}+\floor{\bb_j(n_0+n)}).
\end{align}
For instance, we have\footnote{We note that the following representation is not unique: we could equally well set $\bb_3(n) = (0, b_1(n), 0)$ if we wanted to ``assign'' $b_1(n)$ to $m_2$ rather than $m_1$.} 
\begin{align*}
    T_1^{\floor{b_1(n_0+n)}}\CD_{2, T_1}(f_1)\cdot T_2^{\floor{b_2(n_0+n)}}\CD_{1, T_2}(f_2) = \lim_{M\to\infty}\E_{{\bm}\in [\pm M]^3}\prod_{j\in[4]}d_j(\bm + \floor{\bb_j(n_0+n)})
\end{align*}
for \begin{gather*}
    d_j(m_1, m_2, m_3) = \begin{cases}
        T_1^{m_1}\overline{f_1},\; &j=1\\
        T_1^{m_2}\overline{f_1},\; &j=2\\
        T_1^{m_1+m_2}{f_1},\; &j=3\\
        T_2^{m_3}\overline{f_2},\; &j=4
    \end{cases}, \quad \textrm{and}\quad 
    \bb_j(n) = \begin{cases}
        (b_1(n), 0, 0),\; &j=1\\
        (0, b_1(n), 0),\; &j=2\\
        (b_1(n), 0, 0),\; &j=3\\
        (0, 0, b_2(n)),\; &j=4
    \end{cases}.
\end{gather*} 

We remark that in \eqref{E: dual representation}, we have used Lemma \ref{L: convergence} to combine limits coming from various dual functions.

We then set
\begin{align*}
    \veps :=
    \sup_{E\subseteq \N}\limsup\limits_{M\to\infty}\norm{\E_{n\in[N]} A(n_0+n)\cdot\E_{{\bm}\in [\pm M]^l}\prod_{j\in [\ell]} d_{j}({\bm}+\floor{\bb_j(n_0+n)})\cdot 1_E(n)}_{L^2(\mu)}.
\end{align*}

Let $W$ be the type of the family of sequences $\bb_1, \ldots, \bb_\ell$ defined as in \cite[Subsection~5.1]{Fr12}.\footnote{Technically, \cite[Subsection~5.1]{Fr12} defines the type for $l$-tuples $(\CB_1,\ldots,$ $\CB_J),$ where $\CB_i=(b_{1i}, \ldots, b_{\ell i})$, but it is easy to move between the two formalisms.} We notice that since $d, l, \ell$ are fixed, we have, by \cite[Lemma~5.3]{Fr12}, only $O_{d,\ell, J}(1)$ many $W$'s. Assuming that $b_{11}$ is the function of the largest growth rate, we proceed by induction on $W.$
\medskip

{\bf Base case:} Assume that $b_{ji}$ converge in $\R$ for all $j,i$. 
Consequently, the sequences $\floor{\bb_j(n_0+n)}$ are constant for $n$ large enough. Hence the average over $\bm$ becomes independent of $n$ and can be removed using the H\"older inequality, giving 
$$\sup_{E\subseteq \N}\norm{\E_{n\in [N]}A(n_0+n)\cdot 1_{E}(n)}_{L^2(\mu)} + o_{N\to\infty; \bb_1, \ldots, \bb_\ell}(1) \geq \veps,$$
with the error term corresponding to the small values $n$. The result follows immediately since $1_E(n)$ is a 1-bounded weight. 

{\bf Inductive step:} Suppose now that $b_{11}(t)\to\infty$, and assume that for all families of $2\ell$ sequences with type $W'<W$, the statement holds with $s = s_0(W', l, 2\ell)$. We then set 
$$ s_0(W, l, \ell):= \max_{W'<W}(s_0(W', l, 2\ell))+1.$$

Swapping the order of summation of $n$ and $\bm$, using the Cauchy-Schwarz inequality, and then the van der Corput inequality (the first part of Lemma~\ref{L: finitary van der Corput}), we get

\begin{multline*} \E_{k_1\in [\pm K]}\limsup\limits_{M\to\infty}\Bigg\|\E_{n\in [N]}A(n_0+n+k_1)\cdot \overline{A}(n_0+n)\cdot \E_{{\bm}\in [\pm M]^l} \prod_{j\in [\ell]} d_{j}({\bm}+\floor{\bb_j(n_0+n+k_1)}) \\
\overline{d}_{j}({\bm}+\floor{\bb_j(n_0+n)})\cdot 1_{E}(n+k_1)\cdot 1_{E}(n)\Bigg\|_{L^2(\mu)}+\frac{K}{N} \gg \veps^2.
\end{multline*}
We now choose the vector $\bb$ that reduces the type of $W$ and make the change of variables ${\bm}\to {\bm}-\floor{\bb(n_0+n)}$, so that
\begin{multline*} \E_{k_1\in [\pm K]}\limsup\limits_{M\to\infty}\Bigg\|\E_{n\in [N]}A(n_0+n+k_1)\cdot \overline{A}(n_0+n)\\
\E_{{\bm}\in [\pm M]^l} \prod_{j\in [\ell]} d_{j}({\bm}+\floor{\bb_j(n_0+n+k_1)-\bb(n_0+n)}
+\br_{j,k_1}(n))\\ \overline{d}_{j}({\bm}+\floor{\bb_j(n_0+n)-\bb(n_0+n)}+\br'_{j}(n))\cdot 1_{E_{k_1}}(n)\Bigg\|_{L^2(\mu)} +\frac{K}{N} \gg \veps^2,
\end{multline*}
where the error terms $\br, \br'$ take values in $\{0,1\}^l,$ and $E_{k_1}=E\cap (E-k_1).$ 

Partitioning $E_{k_1}$ into $O_{l, \ell}(1)$ sets in which the $\br, \br'$ are constant, we get
\begin{multline*}
\sup_{E\subseteq [N]}\E_{k_1\in [\pm K]}\limsup\limits_{M\to\infty}\Bigg\|\E_{n\in [N]}A(n_0+n+k_1)\cdot \overline{A}(n_0+n)\\ \E_{{\bm}\in [\pm M]^l} \prod_{j\in [\ell]} d_{j}({\bm}+\floor{\bb_j(n_0+n+k_1)-\bb(n_0+n)}) \\
\overline{d}_{j}({\bm}+\floor{\bb_j(n_0+n)-\bb(n_0+n)})\cdot 1_{E}(n)\Bigg\|_{L^2(\mu)} +\frac{K}{N} \gg_{l,\ell} \veps^2.
\end{multline*}
Since the new averages have smaller type, the claim follows by the induction hypothesis.
\end{proof}  

We also need a lemma for removing error terms that take finitely many values. This lemma, which is a variation on \cite[Lemma 3.2]{Ts22}, will be used to remove error terms that come from swapping the order of summing and taking integer parts. { We omit the proof.}
\begin{lemma}\label{L: errors}
    Let $k, \ell,  N\in\N$ and assume that for every $j\in[\ell]$, the functions $\br_{j}:\Z\to\Z^k$ take values in a finite set $S$. Then for any family of sequences $\ba_{j}:\Z\to\Z^k$ there exist $\br'_1, \ldots,\br'_\ell\in\Z^k$ (where we can take $\br'_1 = \mathbf{0}$) such that for all systems $(X, \CX, \mu, T_1, \ldots, T_k)$ and 1-bounded functions $f_1, \ldots, f_\ell\in L^\infty(\mu)$, we have
    \begin{multline*}
      \sup_{|c_n|\leq 1}\norm{  \E_{n\in[N]}  c_{n}\cdot \prod_{j\in[\ell]}T^{\ba_{j}(n)+\br_{j}(n)}f_j}_{L^2(\mu)}
      \ll_{k,\ell, S}
      \sup_{|c_{n}|\leq 1}\norm{\E_{n\in[N]}  c_{n}\cdot \prod_{j\in[\ell]}T^{\ba_{j}(n)+\br'_j}f_j}_{L^2(\mu)}.
    \end{multline*} 
\end{lemma}

Lastly, we shall use a version of \cite[Lemma 3.3]{Ts22} that allows us to pass to short intervals.

\begin{lemma}\label{L: double averaging}
Let $d, \ell\in\N$ and $L\in\CH$ with $1\prec L(t)\prec t$. Let also $(X, \CX, \mu)$ be a probability space and for each $n\in[N]$ and $N\in\N$, let $A_{nN}: L^\infty(\mu)^\ell\to L^\infty(\mu)$ be a function such that $|A_{nN}(f_1, \ldots, f_\ell)|\leq 1$ for all 1-bounded functions $f_1, \ldots, f_\ell\in L^\infty(\mu)$. Then 
\begin{multline*}
    \brac{\limsup_{N\to\infty}\sup_{\norm{f_1}_\infty, \ldots, \norm{f_\ell}_\infty\leq 1}\sup_{|c_n|\leq 1}\norm{\E_{n\in[N]}c_n\cdot A_{nN}(f_1, \ldots, f_\ell)}_{L^2(\mu)}}^d\\ 
    \leq \limsup_{R\to\infty}\sup_{\norm{f_1}_\infty, \ldots, \norm{f_\ell}_\infty\leq 1}\sup_{|c_n|\leq 1} \E\limits_{N\in[R]} \norm{\E_{n\in(N, N+L(N)]}c_n\cdot A_{nR}(f_1, \ldots, f_\ell)}_{L^2(\mu)}^d.
\end{multline*}
\end{lemma}
\begin{proof}
    Let 
    \begin{align*}
        \veps = \limsup_{N\to\infty}\sup_{\norm{f_1}_\infty, \ldots, \norm{f_\ell}_\infty\leq 1}\sup_{|c_n|\leq 1}\norm{\E_{n\in[N]}c_n\cdot A_{nN}(f_1, \ldots, f_\ell)}_{L^2(\mu)}.
    \end{align*}
    We can find 1-bounded functions $f_{jN}$ and constants $c_{nN}$ approximating the suprema well enough so that
\begin{align*}
     \veps = \limsup_{N\to\infty}\norm{\E_{n\in[N]}c_{nN}\cdot A_{nN}(f_{1N}, \ldots, f_{\ell N})}_{L^2(\mu)}.
\end{align*}    
Setting $B_{nN} = c_{nN}\cdot A_{nN}(f_{1N}, \ldots, f_{\ell N})$ and using the same argument as in the proof of \cite[Lemma 3.3]{Ts22}, we observe that
\begin{align*}
    \lim_{R\to\infty}\norm{\E_{n\in[R]}B_{nR} - \E\limits_{N\in[R]}\E_{n\in(N, N+L(N)]}B_{nR}}_{L^2(\mu)} = 0,
\end{align*}
and so
\begin{align*}
    \limsup_{R\to\infty}\E\limits_{N\in[R]}\norm{\E_{n\in(N, N+L(N)]}c_{nR}\cdot A_{nR}(f_{1R}, \ldots, f_{\ell R})}_{L^2(\mu)}\geq \veps.
\end{align*}
The bound then follows by bounding 
\begin{multline*}
    \norm{\E_{n\in(N, N+L(N)]}c_{nR}\cdot A_{nR}(f_{1R}, \ldots, f_{\ell R})}_{L^2(\mu)}\\ \leq \sup_{\norm{f_1}_\infty, \ldots, \norm{f_\ell}_\infty\leq 1}\sup_{|c_n|\leq 1}\E\limits_{N\in[R]} \norm{\E_{n\in(N, N+L(N)]}c_n\cdot A_{nR}(f_1, \ldots, f_\ell)}_{L^2(\mu)},
\end{multline*}
taking the limit $R\to\infty$ 
and applying the H\"older inequality.
\end{proof}

\section{Seminorm estimates for linear and sublinear sequences}\label{S: linear and sublinear}
Having listed all the prerequisites, we are ready to give the first applications of our generalized box seminorms. In this section, we will prove that these seminorms quantitatively control multiple ergodic averages along linear and sublinear Hardy sequences. We start with a standard lemma that allows us to bound finitary multiple ergodic averages along linear functions.
\begin{lemma}\label{L: linear averages}
    Let $(X, \CX, \mu, T_1,...,T_k)$ be a system and $\balpha_{1}, \ldots, \balpha_\ell\in\R^k$ be nonzero. For any $M, N\in\N$ with $M\leq N$
    and 1-bounded $f_1, \ldots, f_\ell\in L^\infty(\mu)$, we have
    \begin{multline*}
        \sup_{|c_n|\leq 1}\norm{\E_{n\in[N]}c_n \cdot \prod_{j\in[\ell]}T^{\floor{\balpha_j n}}f_j}_{L^2(\mu)}^{2^{\ell+1}}\\
        \ll_{k, \ell} \E_{\um, \um'\in[\pm M]^\ell}\abs{\int \Delta_{\balpha_1, \balpha_1 - \balpha_\ell, \ldots, \balpha_1 - \balpha_{2}; (\um, \um')} f_1\, d\mu}^2 + \frac{M}{N}.
    \end{multline*}
\end{lemma}
We recall here the convention \eqref{E: iterated multiplicative derivative} for the iterated multiplicative derivative.

The proof of Lemma~\ref{L: linear averages} is standard and follows from several applications of Lemma~\ref{L: finitary van der Corput}. Taking the supremum over 1-bounded $f_2, \ldots, f_\ell$ followed by the limits $N\to\infty$ and $M\to\infty$ (in this order), we obtain the following infinitary corollary.
\begin{corollary}\label{C: linear averages}
    Let $(X, \CX, \mu, T_1,...,T_k)$ be a system and $\balpha_{1}, \ldots, \balpha_\ell\in\R^k$ be nonzero. For any 1-bounded $f_1\in L^\infty(\mu)$, we have
    \begin{multline*}
        \limsup_{N\to\infty}\sup_{\substack{\norm{f_2}_\infty, \ldots, \norm{f_\ell}_\infty\leq 1}}\sup_{|c_n|\leq 1}\norm{\E_{n\in[N]}c_n \cdot \prod_{j\in[\ell]}T^{\floor{\balpha_j n}}f_j}_{L^2(\mu)}
        \ll_{k, \ell} \nnorm{f_1}_{\balpha_1, \balpha_1 - \balpha_\ell, \ldots, \balpha_1 - \balpha_{2}}^+.
    \end{multline*}    
\end{corollary}

The seminorm control for linear sequences can then be used to prove a similar result for sublinear sequences, whose proof requires considerably more effort, though. We first prove a result for averages twisted by dual sequences and then use it to derive a stronger corollary for sublinear averages without dual twists.
\begin{proposition}\label{P: sublinear}
    Let $d, k, \ell, m\in\N$, $J\in\N_0$, $0\leq m_1 \leq m$ and $g_1,...,g_m, b_1, \ldots, b_J\in \mathcal{H}$ be Hardy functions satisfying the following growth conditions:
    \begin{enumerate}
        \item $1 \prec g_1(t) \prec \cdots \prec g_m(t) \prec t$;
        \item $g_i(t)\ll \log t$ iff $i\leq m_1$;
        \item $b_j(t)\ll t^d$ for every $j\in[J]$.
    \end{enumerate}
    For $j\in[\ell]$, let 
    \begin{align*}
        \ba_j = \sum_{i\in[m]}\balpha_{ji} g_i + \br_j\in\CH^k \quad\textrm{for\; some}\quad \balpha_{ji}\in\R^k\quad \textrm{and}\quad \lim_{t\to\infty}\br_j(t)=\mathbf{0},
    \end{align*}
    and for distinct $j,j'\in[0, \ell]$, let
    \begin{align*}
        d'_{jj'} = \max\{i\in[m]:\; \balpha_{ji}\neq \balpha_{j'i}\}
    \end{align*}
    (where $\ba_0 = \mathbf{0}$ and $\balpha_{0i} = \mathbf{0}$). 
    Suppose that the following two conditions are satisfied:
    \begin{enumerate}
        \item for every $j\in[0,\ell]\setminus\{1\}$, we have $d'_{1j}>m_1$, i.e., at least one coordinate of $\ba_1 - \ba_{j}$ grows faster than log; 
        \item $\balpha_{1m}\neq \textbf{0}$, i.e., $\ba_1$ has maximum growth.
    \end{enumerate}
    Then there exists a positive integer $s=O_{d, J, \ell, m}(1)$ such that for any system $(X,\CX, \mu,T_1,$ $...,T_k)$ and 1-bounded $f_1\in L^\infty(\mu)$,
    we have 
    \begin{multline*}
         \limsup_{N\to\infty}\sup_{\substack{\norm{f_2}_\infty, \ldots, \norm{f_\ell}_\infty\leq 1,\\ \CD_1, \ldots, \CD_J \in \FD_d}} \sup_{|c_n|\leq 1}\norm{\E_{n\in [N]} c_{n}\cdot \prod_{j\in [\ell]}  T^{\floor{\ba_j(n)}} f_j \cdot \prod_{j\in[J]}\CD_j(\floor{b_j(n)})}_{L^2(\mu)}^{O_{d, J, \ell, m}(1)}\\ 
         \ll_{d, J, k, \ell, m} \nnorm{f_1}_{\balpha_{1d'_{10}}^s, (\balpha_{1d'_{12}}-\balpha_{2d'_{12}})^s, \ldots,  (\balpha_{1d'_{1\ell}}-\balpha_{\ell d'_{1\ell}})^s}^+.
    \end{multline*}
\end{proposition}

\begin{proof}[Proof]
We assume that $m>m_1$ (i.e. $g_m\succ \log$) since in the case $m=m_1$ the proposition is vacuously true. The idea of the proof is to first apply Lemma \ref{L: removing duals finitary} to remove the dual sequences, then pass to short intervals appropriately chosen so that on Taylor expanding $g_1, \ldots, g_m$, all but the contribution of $g_m$ to $\ba_1, \ldots, \ba_\ell$ becomes constant. Subsequently, we pass to arithmetic progressions on which our sequences become linear, and for those we can obtain the claimed seminorm bound using Corollary \ref{C: linear averages}.

We let all the quantities in this proof depend on $d, k, m, J, \ell$ without mentioning the dependence explicitly, noting however that $s$ (to be chosen later) and powers of $\veps$ do not depend on $k$.

\smallskip
 \textbf{Step 1: Passing to short intervals, removing the dual functions and changing  variables.}
\smallskip

Let
\begin{align*}
    \veps = \limsup_{N\to\infty}\sup_{\substack{\norm{f_2}_\infty, \ldots, \norm{f_\ell}_\infty\leq 1,\\ \CD_1, \ldots, \CD_J \in \FD_d}} \sup_{|c_n|\leq 1}\norm{\underset{n\in [N]}{\E} c_{n}\cdot \prod_{j\in[\ell]}  T^{\floor{\ba_j(n)}} f_j\cdot \prod_{j\in[J]}\CD_j(\floor{b_j(n)})}_{L^2(\mu)}.
\end{align*}
 Using Lemma \ref{L: double averaging}, we pass to short intervals of some length $1\prec L(t)\prec t$ to be chosen later, obtaining
\begin{multline*}
    \limsup_{R\to\infty}\sup_{\substack{\norm{f_2}_\infty, \ldots, \norm{f_\ell}_\infty\leq 1,\\ \CD_1, \ldots, \CD_J \in \FD_d}} \sup_{|c_n|\leq 1}\E\limits_{N\in[R]}\\
    \norm{\E_{n\in [L(N)]} c_{N+n}\cdot \prod_{j\in[\ell]}  T^{\floor{\ba_j(N+n)}} f_j\cdot \prod_{j\in[J]}\CD_j(\floor{b_j(N+n)})}_{L^2(\mu)}\geq \veps.
\end{multline*}
Subsequently, we apply Lemma \ref{L: removing duals finitary} in order to remove the functions $\CD_1, \ldots, \CD_J$, so that
\begin{multline*}
    \limsup_{R\to\infty}\;\sup_{\substack{\norm{f_2}_\infty, \ldots, \norm{f_\ell}_\infty\leq 1}} \;\E\limits_{N\in[R]}\;\E\limits_{\uk, \uk'\in[\pm K_N]^{s_1}}\;\sup_{|c_n|\leq 1}\\
    \norm{\E_{n\in [L(N)]} c_n \cdot \prod_{j\in[\ell]} \prod_{\ueps\in\{0,1\}^{s_1}}\CC^{|\ueps|}  T^{\floor{\ba_j(N+n+({\underline{1}}-\ueps) \cdot\uk + \ueps\cdot\uk')}} f_j}_{L^2(\mu)}+
    \frac{K_N}{L(N)}\gg \veps^{O(1)}
\end{multline*}
for some integer $s_1 = O(1)$ (independent of $k$) and $K_N>0$ that we choose shortly.

Recall that $\ba_j = \sum_{i\in[m]}\balpha_{ji} g_i$ and $g_1\prec \cdots \prec g_m$. We want to choose the length $L(t)$ in such a way that the functions $g_1, \ldots, g_{m-1}$ are approximately constant on each interval $(N, N+L(N)]$ while $g_m$ can be approximated by a linear polynomial. This can be accomplished by choosing any $L$ in the range 
$$g_m'(t)^{-1}\prec L(t)\prec \min(g_m''(t)^{-1/2}, g_{m-1}'(t)^{-1}).$$
That this range is nonempty can be seen as follows: the relation $g_m'(t)^{-1}\prec g_m''(t)^{-1/2}$ is a consequence of Lemma~\ref{L: degree inequality} and the sublinearity of $g_m$ while $g_m'(t)^{-1}\prec  g_{m-1}'(t)^{-1}$ follows from the L'H\^opital rule. Using Lemma~\ref{L: degree inequality} once more, we get $1\prec g_m'(t)^{-1}$ and $\min(g_m''(t)^{-1/2}, g_{m-1}'(t)^{-1})\prec t,$ from which we have $1\prec L(t)\prec t,$ and so the application of Lemma~\ref{L: double averaging} above is justified. 

We additionally let $K_N = K \floor{|g'_m(N)^{-1}|}$ for some $K>0$; with this choice, the quantity $K_N$ satisfies the upper bound $\lim\limits_{N\to\infty}\frac{K_N}{K L(N)} = 0$ for each fixed $K$. If $n\in[L(N)]$ and $\uk, \uk'\in[\pm K_N]^{s_1}$, then $$n+({\underline{1}}-\ueps) \cdot\uk + \ueps\cdot\uk' \in [-s_1 K_N, L(N)+s_1 K_N]\subseteq[-o(K L(N)), L(N)(1+o(K))],$$ and so by Taylor expanding the functions $g_i$ around $N$, we can approximate
\begin{multline*}
    g_i(N+n+({\underline{1}}-\ueps) \cdot\uk + \ueps\cdot\uk')\\
    = \begin{cases} g_m(N) + g'_m(N) (n+({\underline{1}}-\ueps) \cdot\uk + \ueps\cdot\uk') + r_{mNn\ueps\uk\uk'},\; &i = m\\
    g_i(N)+r_{iNn\ueps\uk\uk'},\;  &i\in[m-1]
    \end{cases}
\end{multline*}
for error terms $r_{iNn\ueps\uk\uk'}$ satisfying\footnote{The bound on the error term follows since for every $i<m$, the error term $r_{iNn\ueps\uk\uk'}$ coming from the Taylor expansion of $g_i(N+n)$ at $N$ equals $g'_i(N) n_0$ for some $n_0\in [-s_1 K_N, L(N)+s_1 K_N]$ and can be bounded by $r_{iNn\ueps\uk\uk'}\prec g_i'(N) L(N)(1+o(K)) \prec 1$ for every fixed $K$. A similar argument works for $i=m$, this time however we use the second derivative of $g_m$.}
\begin{align*}
    \lim_{N\to\infty}\;\max_{n\in[L(N)]}\;\max_{\uk,\uk'\in[\pm K_N]^{s_1}}|r_{iNn\ueps\uk\uk'}| = {0}\quad\textrm{for\; all}\quad (i,\ueps)\in[m]\times\{0,1\}^{s_1}.
\end{align*}

Then
\begin{align*}
    &\ba_j(N+n+({\underline{1}}-\ueps) \cdot\uk + \ueps\cdot\uk')\\
    &\qquad\qquad\qquad= \balpha_{jm} g'_m(N) (n+({\underline{1}}-\ueps) \cdot\uk + \ueps\cdot\uk')+\sum_{i\in[m]}\balpha_{ji}g_i(N) + \br_{jNn\ueps\uk\uk'}\\
    &\qquad\qquad\qquad= \balpha_{jm} g'_m(N) (n+({\underline{1}}-\ueps) \cdot\uk + \ueps\cdot\uk') + \ba_j(N)+ \br_{jNn\ueps\uk\uk'}
\end{align*}
for some error terms $\br_{jNn\ueps\uk\uk'}\in\R^k$ satisfying
\begin{align*}
    \lim_{N\to\infty}\max_{n\in[L(N)]}\max_{\uk,\uk'\in[\pm K_N]^{s_1}}|\br_{jNn\ueps\uk\uk'}| = \mathbf{0}\quad\textrm{for\; all}\quad (j,\ueps)\in[\ell]\times\{0,1\}^{s_1}.
\end{align*}
By Lemma \ref{L: errors} and the pigeonhole principle, we can find $\br_{j\ueps}\in\Z^k$ for which
\begin{multline*}
    \limsup_{R\to\infty}\;\sup_{\substack{\norm{f_2}_\infty, \ldots, \norm{f_\ell}_\infty\leq 1}}\; \E\limits_{N\in[R]}\; \E\limits_{\uk, \uk'\in[\pm K_N]^{s_1}}\; \sup_{|c_n|\leq 1}\\
    \norm{\E_{n\in [L(N)]} c_n \cdot \prod_{j\in[\ell]}\prod_{\ueps\in\{0,1\}^{s_1}}\CC^{|\ueps|}  T^{\floor{\balpha_{jm} g'_m(N) (n+({\underline{1}}-\ueps) \cdot\uk + \ueps\cdot\uk')}+\floor{\ba_j(N)}+\br_{j\ueps}} f_j}_{L^2(\mu)}\gg \veps^{O(1)}.
\end{multline*}

The next step is to remove the coefficient $g'_m(N)$ so that all the sequences become linear. This is accomplished by splitting the range $[L(N)]$ of $n$ into arithmetic progressions of difference $\floor{|g'_m(N)|^{-1}}$, and similarly with the ranges of $\uk,\uk'$. Passing to arithmetic progressions, we in effect
substitute $n\floor{|g'_m(N)|^{-1}} + r$ for $n$, $k_i\floor{|g'_m(N)|^{-1}} + r_i$ for $k_i$ and $k_i'\floor{|g'_m(N)|^{-1}} + r_i'$ for $k_i'$. Letting $$\Tilde{L}(N) =\frac{L(N)}{\floor{|g'_m(N)|^{-1}}}\succ 1$$ 
be the length of the new interval in $n$, we deduce that 
\begin{multline*}
    \limsup_{R\to\infty}\; \sup_{\substack{\norm{f_{2\ueps}}_\infty, \ldots, \norm{f_{\ell\ueps}}_\infty\leq 1}}\;
    \E\limits_{N\in[R]}\; \E\limits_{\uk,\uk'\in[\pm K]^{s_1}}\; \E_{\substack{r, r_1, \ldots, r_{s_1},\\ r_1', \ldots, r_{s_1}'\in [g'_m(N)^{-1}]}}\;
     \sup_{|c_n|\leq 1}\\ \norm{\underset{n\in [\tilde{L}(N)]}{\E} c_n \cdot \prod_{j\in[\ell]} \prod_{\ueps\in\{0,1\}^{s_1}}\CC^{|\ueps|} 
     T^{\floor{\balpha_{jm}  (n + ({\underline{1}}-\ueps) \cdot\uk + \ueps\cdot\uk') + \balpha_{jm}  g_m'(N)(r+({\underline{1}}-\ueps) \cdot\ur + \ueps\cdot\ur'))}+\floor{\ba_j(N)}+\br_{j\ueps}}
     f_{j}}_{L^2(\mu)}\\
     \gg \veps^{O(1)}.
\end{multline*}
We then split the large integer part and handle the error terms using Lemma \ref{L: errors} (possibly modifying $\br_{j\ueps}$ if needed); since $g_m'(N)(r+({\underline{1}}-\ueps) \cdot\ur + \ueps\cdot\ur') \ll 1$, this term can be ignored altogether at the cost of $O(1)$ factor in the lower bound. Thus 
\begin{multline*}
    \limsup_{R\to\infty}\; \sup_{\substack{\norm{f_{2\ueps}}_\infty, \ldots, \norm{f_{\ell\ueps}}_\infty\leq 1}}\;
    \E\limits_{N\in[R]}\; \E\limits_{\uk,\uk'\in[\pm K]^{s_1}}\;
     \sup_{|c_n|\leq 1}\\ \norm{\underset{n\in [\tilde{L}(N)]}{\E} c_n \cdot \prod_{j\in[\ell]} \prod_{\ueps\in\{0,1\}^{s_1}}\CC^{|\ueps|} T^{\floor{\balpha_{jm}n}+ \floor{\balpha_{jm}k_1^{\eps_1}}+ \cdots + \floor{\balpha_{jm}k_{s_1}^{\eps_{s_1}}}+\floor{\ba_j(N)}+\br_{j\ueps}} f_{j}}_{L^2(\mu)} \gg \veps^{O(1)}.
\end{multline*}

We choose functions $f_{jR}$ for $j\geq 2$ close enough to the supremum, while also setting $f_{1R} = f_{1}$, so that upon applying the property $$\limsup\limits_{R\to\infty}(A_R + B_R)\leq \limsup_{R\to\infty}A_R + \limsup_{R\to\infty} B_R$$ to exchange $\limsup\limits_{R\to\infty}$ with $\E\limits_{\uk, \uk'\in[\pm K]^{s_1}}$, we get
\begin{multline*}
    \E\limits_{\uk,\uk'\in[\pm K]^{s_1}}\limsup_{R\to\infty} \E\limits_{N\in[R]} \sup_{|c_n|\leq 1}\\
    \norm{\underset{n\in [\tilde{L}(N)]}{\E} c_n \cdot \prod_{j\in[\ell]}\prod_{\ueps\in\{0,1\}^{s_1}} \CC^{|\ueps|} T^{\floor{\balpha_{jm}n}+ \floor{\balpha_{jm}k_1^{\eps_1}}+ \cdots + \floor{\balpha_{jm}k_{s_1}^{\eps_{s_1}}}+\floor{\ba_j(N)}+\br_{j\ueps}} f_{jR}}_{L^2(\mu)}  \gg \veps^{O(1)}.
\end{multline*}

At this point, we split into the cases $\ell =1$ and $\ell>1$.

\smallskip
\textbf{Step 2: The case $\ell = 1$.} 
\smallskip

 If $\ell = 1$, then we can compose the integral defining the $L^2(\mu)$ norm above with $T^{-\floor{\ba_1(N)}}$, so that the expression above reduces to
\begin{multline*}
    \E_{\uk,\uk'\in[\pm K]^{s_1}}\limsup_{R\to\infty}\E\limits_{N\in[R]}
    \sup_{|c_n|\leq 1}\\
    \norm{\underset{n\in [\tilde{L}(N)]}{\E} c_n \cdot T^{\floor{\balpha_{1m}n}}\brac{\prod_{\ueps\in\{0,1\}^{s_1}}\CC^{|\ueps|} T^{\floor{\balpha_{1m}k_1^{\eps_1}}+ \cdots + \floor{\balpha_{1m}k_{s_1}^{\eps_{s_1}}}+\br_{1\ueps}} f_{1}}}_{L^2(\mu)} \gg \veps^{O(1)}.
\end{multline*}
 Corollary \ref{C: linear averages} then gives
\begin{align*}
    \E_{\uk,\uk'\in[\pm K]^{s_1}}\bignnorm{\prod_{\ueps\in\{0,1\}^{s_1}}\CC^{|\ueps|} T^{\floor{\balpha_{1m}k_1^{\eps_1}}+ \cdots + \floor{\balpha_{1m}k_{s_1}^{\eps_{s_1}}}+\br_{1\ueps}} f_{1}}_{\balpha_{1m}}^+\gg\veps^{O(1)}.
\end{align*}
Applying the H\"older inequality, taking $K\to\infty$, expanding the expression above (with the seminorm raised to an appropriate powers) as a Gowers-Cauchy-Schwarz inner product, and using the translation invariance of seminorms, we obtain $\nnorm{f_1}_{\balpha_{1m}^{s_1+1}}^+\gg \veps^{O(1)}$ from the Gowers-Cauchy-Schwarz inequality (Lemma \ref{L: GCS}).

\smallskip
\textbf{Step 3: The case $\ell>1$.}
\smallskip

We move on to the case $\ell>1$. This case is much more notationally complicated than the previous one, yet the underlying goal is simple: we want to reduce to an average of length less than $\ell$ and apply the induction hypothesis.  Reordering if necessary, we can find $\ell_0\in[\ell]$ such that $\balpha_{1m} = \balpha_{j m}$ iff $j\in[\ell_0]$. Upon setting
\begin{align*}
    F_{RN\uk\uk'} := \prod_{j\in[\ell_0]}T^{\floor{\ba_j(N)}}\brac{\prod_{\ueps\in\{0,1\}^{s_1}}\CC^{|\ueps|} T^{\floor{\balpha_{1m}k_1^{\eps_1}}+ \cdots + \floor{\balpha_{1m}k_{s_1}^{\eps_{s_1}}}+\br_{j\ueps}}f_{j R}}
\end{align*}
and
\begin{align*}
    F_{jRN\uk\uk'} := T^{\floor{\ba_j(N)}}\prod_{\ueps\in\{0,1\}^{s_1}}\CC^{|\ueps|} T^{\floor{\balpha_{jm}k_1^{\eps_1}}+ \cdots + \floor{\balpha_{jm}k_{s_1}^{\eps_{s_1}}}+\br_{j\ueps}}f_{j R}\quad \textrm{for}\quad j\in[\ell_0+1,\ell],    
\end{align*}
we have
\begin{multline*}
    \E_{\uk,\uk'\in[\pm K]^{s_1}}\limsup_{R\to\infty} \E\limits_{N\in[R]} \sup_{|c_n|\leq 1}\\
    \norm{\underset{n\in [\tilde{L}(N)]}{\E} c_{n}\cdot T^{\floor{\balpha_{1m}n}} F_{RN\uk\uk'} \cdot\prod_{j=\ell_0+1}^{\ell}T^{\floor{\balpha_{jm}n}}F_{jRN\uk\uk'}}_{L^2(\mu)}  \gg \veps^{O(1)}.
\end{multline*}
By Lemma \ref{L: linear averages} (here we use the assumption that $\balpha_{1m}\neq \mathbf{0}$), for any $H>0$ we have 
\begin{multline*}
    \E_{\uk,\uk'\in[\pm K]^{s_1}} \E_{\uh, \uh'\in[\pm H]^{\ell-\ell_0+1}} \limsup_{R\to\infty}\E\limits_{N\in[R]}\\ 
    \abs{\int \Delta_{\balpha_{1m}, \balpha_{1m} - \balpha_{\ell m}, \ldots, \balpha_{1 m} - \balpha_{(\ell_0+1)m}; (\uh, \uh')} F_{RN\uk\uk'}\, d\mu}^2 \gg \veps^{O(1)}.
\end{multline*}
The definition of $F_{RN\uk\uk'}$ gives us
\begin{align*}
    \E_{\uk,\uk'\in[\pm K]^{s_1}} \E_{\uh, \uh'\in[\pm H]^{\ell-\ell_0+1}}\limsup_{R\to\infty}\E\limits_{N\in[R]}
    \abs{\int \prod_{j\in[\ell_0]} T^{\floor{\ba_j(N)}}f_{jR\uk\uk'\uh\uh'}\, d\mu}^2  \gg \veps^{O(1)},
\end{align*}
where 
\begin{multline*}
    f_{jR\uk\uk'\uh\uh'} := \Delta_{\balpha_{1m}, \balpha_{1m} - \balpha_{\ell m}, \ldots, \balpha_{1 m} - \balpha_{(\ell_0+1)m}; (\uh, \uh')}\\
    \prod_{\ueps\in\{0,1\}^{s_1}}\CC^{|\ueps|} T^{\floor{\balpha_{1m}k_1^{\eps_1}}+ \cdots + \floor{\balpha_{1m}k_{s_1}^{\eps_{s_1}}}+\br_{j\ueps}}f_{j R}
\end{multline*}
for each $j\in[\ell_0]$, $\uk,\uk'\in[\pm K]^{s_1}$ and $\uh, \uh'\in[\pm H]^{\ell-\ell_0+1}$. Composing with $T^{-\floor{\ba_{\ell_0}(N)}}$, swapping subtraction inside integer parts, and using Lemma \ref{L: errors} to handle error terms, we find $\br_j\in\Z^k$ 
such that
\begin{align*}
    \E_{\uk,\uk'\in[\pm K]^{s_1}} \E_{\uh, \uh'\in[\pm H]^{\ell-\ell_0+1}}\limsup_{R\to\infty}\E\limits_{N\in[R]}
    \abs{\int \prod_{j\in[\ell_0]} T^{\floor{\ba_j(N)-\ba_{\ell_0}(N)}+\br_j}f_{jR\uk\uk'\uh\uh'}\, d\mu}^2  \gg \veps^{O(1)}.
\end{align*}
Passing to the product system, taking $\E\limits_{N\in[R]}$ inside the average and applying the Cauchy-Schwarz inequality, we deduce that
\begin{multline*}
    \E_{\uk,\uk'\in[\pm K]^{s_1}} \E_{\uh, \uh'\in[\pm H]^{\ell-\ell_0+1}} \limsup_{R\to\infty}\\
    \norm{\E\limits_{N\in[R]}\prod_{j\in[\ell_0-1]} (T\times T)^{\floor{\ba_j(N)-\ba_{\ell_0}(N)}+\br_j}(f_{jR\uk\uk'\uh\uh'} \otimes \overline{f_{jR\uk\uk'\uh\uh'}})}_{L^2(\mu\times \mu)} \gg \veps^{O(1)}.
\end{multline*}
Each average indexed by $\uk\uk'\uh\uh'$ has length $\ell_0-1\leq \ell - 1$, and so applying the induction hypothesis and recalling that $f_{1\uk\uk'\uh\uh'} := f_{1R\uk\uk'\uh\uh'}$ does not depend on $R$, we deduce that there exists a positive integer $s_2=O(1)$ (independent of $k$) such that
\begin{align*}
    \E_{\uk,\uk'\in[\pm K]^{s_1}} \E_{\uh, \uh'\in[\pm H]^{\ell-\ell_0+1}}\nnorm{f_{1\uk\uk'\uh\uh'}}^+_{\balpha_{1m}^{s_2}, (\balpha_{1d'_{12}} - \balpha_{2d'_{12}})^{s_2}, \ldots, (\balpha_{1d'_{12}} - \balpha_{\ell_0 d'_{1\ell_0}})^{s_2}}  \gg \veps^{O(1)}.
\end{align*}
Note that $d'_{1j}=m$ for $j\in[\ell_{0}+1,\ell]$. The result follows on applying the H\"older inequality, taking limits (first $H\to\infty$ and then $K\to\infty$), and using the definition of $f_{1\uk\uk'\uh\uh'}$ together with the inductive formula \eqref{E: inductive formula} and Lemma \ref{L: GCS} (for the average over $\uk$) to remove the shifts $\br_{j\ueps}$.

While applying the induction hypothesis, we implicitly use the assumption that some coordinate of the sequences $\ba_1 - \ba_{\ell_0}$ and $(\ba_1 - \ba_{\ell_0}) - (\ba_j - \ba_{\ell_0}) = \ba_1 - \ba_j$ grows faster than $\log$, which follows from the assumption that $d'_{1 j }>m_1$ for all $j\neq 1$.
\end{proof}

If the average has no dual sequences, we can remove the assumption of $\ba_1$ having maximum growth.
\begin{corollary}\label{C: sublinear}
    Let $d, k, \ell, m\in\N$, $0\leq m_1 \leq m$ and $g_1,...,g_m\in \mathcal{H}$ be Hardy functions satisfying the following growth conditions:
    \begin{enumerate}
        \item $1 \prec g_1(t) \prec \cdots \prec g_m(t) \prec t$;
        \item $g_i(t)\ll \log t$ iff $i\leq m_1$.
    \end{enumerate}
    For $j\in[\ell]$, let 
    \begin{align*}
        \ba_j = \sum_{i\in[m]}\balpha_{ji} g_i + \br_j\in\CH^k \quad\textrm{for\; some}\quad \balpha_{ji}\in\R^k\quad \textrm{and}\quad \lim_{t\to\infty}\br_j(t)=\mathbf{0},
    \end{align*}
    and for distinct $j,j'\in[0, \ell]$, let
    \begin{align*}
        d'_{jj'} = \max\{i\in[m]:\; \balpha_{ji}\neq \balpha_{j'i}\}
    \end{align*}
    (where $\ba_0 = \mathbf{0}$ and $\balpha_{0i} = \mathbf{0}$). 
    Suppose that for all $j\in[0,\ell]\setminus\{1\}$, we have $d'_{1j}>m_1$, i.e., at least one coordinate of $\ba_1 - \ba_{j}$ grows faster than log. 
    Then there exists a positive integer $s=O_{d, \ell, m}(1)$ such that for any system $(X,\CX, \mu,T_1,...,T_k)$ and 1-bounded $f_1\in L^\infty(\mu)$, we have 
    \begin{multline*}
         \limsup_{N\to\infty}\sup_{\substack{\norm{f_2}_\infty, \ldots, \norm{f_\ell}_\infty\leq 1}} \sup_{|c_n|\leq 1}\norm{\E_{n\in [N]} c_{n}\cdot \prod_{j\in [\ell]}  T^{\floor{\ba_j(n)}} f_j}_{L^2(\mu)}^{O_{d, \ell, m}(1)}\\ 
         \ll_{d, k, \ell, m} \nnorm{f_1}_{\balpha_{1d'_{10}}^s, (\balpha_{1d'_{12}}-\balpha_{2d'_{12}})^s, \ldots,  (\balpha_{1d'_{1\ell}}-\balpha_{\ell d'_{1\ell}})^s}^+.
    \end{multline*}
\end{corollary}
\begin{proof}
Once again, we assume that $m>m_1$ since otherwise the result follows vacuously.  
If $\balpha_{1m}\neq \mathbf{0}$, then the estimate follows immediately from Proposition \ref{P: sublinear}. If $\balpha_{1m} = \mathbf{0}$, then we perform the standard trick of composing the average with $T^{-\floor{\ba_{i}(n)}}$ for some index $i\in[\ell]$ for which $\balpha_{im}\neq\mathbf{0}$ (we can assume without loss of generality that such an index exists). This gives
    \begin{align*}
        \norm{\E_{n\in [N]} c_{n}\cdot \prod_{j\in [\ell]}  T^{\floor{\ba_j(n)}} f_j}_{L^2(\mu)}
        = \int f_i \cdot \E_{n\in[N]}c_n\cdot T^{-\floor{\ba_i(n)}}F_N\cdot \prod_{\substack{j\in[\ell],\\ j\neq i}} T^{\floor{\ba_j(n)}-\floor{\ba_i(n)}}f_j\, d\mu
    \end{align*}
     for some 1-bounded function $F_N$. By applying the Cauchy-Schwarz inequality, swapping the order of subtraction and taking integer parts, and using Lemma \ref{L: errors} to handle the ensuing error terms, we can bound the expression above by $O_{d, k,\ell, m}(1)$ times 
     \begin{align*}
         \sup_{\substack{\norm{f_2}_\infty, \ldots, \norm{f_\ell}_\infty\leq 1}} \sup_{|c_n|\leq 1}\norm{\E_{n\in [N]} c_n\cdot \prod_{\substack{j\in[\ell]}} T^{\floor{\tilde{\ba}_j(n)}}f_j}_{L^2(\mu)},
     \end{align*}
     where
    \begin{align*}
        \tilde{\ba}_j = \begin{cases} \ba_j - \ba_i,\; &j\neq i\\
        -\ba_i,\; & j = i
        \end{cases}.
    \end{align*}
    Crucially, the sequence $\tilde{\ba}_1 =\ba_1 -\ba_i$ has maximum growth, and so the limsup as $N\to\infty$ of the average above can be controlled using Proposition \ref{P: sublinear}. Moreover, the families
    \begin{align*}
        \{\ba_1, \ba_1-\ba_j: j\in[2,\ell]\} \quad\textrm{and}\quad \{\tilde{\ba}_1, \tilde{\ba}_1-\tilde{\ba}_j: j\in[2,\ell]\}
    \end{align*}
    are identical, and hence so are their leading coefficients. The claimed estimate follows from an application of Proposition \ref{P: sublinear}.
\end{proof}
We remark that the trick of composing the average with $T^{-\floor{\ba_{i}(n)}}$ would not work when applied to an average twisted by dual functions like the one considered in Proposition~\ref{P: sublinear}; the issue is that the resulting terms $T^{-\floor{\ba_{i}(n)}}\CD(\floor{b(n)})$ cannot be removed using Lemma~\ref{L: removing duals finitary}. Therefore here and in the results to come,  averages twisted by dual functions will usually have some growth assumptions on the sequence $\ba_1$ with respect to which we obtain seminorm estimates.

\section{PET: seminorm estimates for polynomials}\label{S: polynomial PET}
The purpose of this section is to take first steps towards proving generalized box seminorm bounds for another naturally occurring family of sequences: polynomials over $\R^k$. Even though many seminorm estimates for multiple ergodic averages along polynomials have been obtained in the past, no work so far has done it in this generality, not least because the appropriate family of seminorms has not been previously identified. The purpose of this section is to show in Proposition \ref{P: PET for polynomials twisted by duals} that finite averages along polynomials twisted by dual sequences are controlled by averages of (finite analogues of) generalized box seminorms. The argument whereby we achieve this is known in the literature as PET, and it consists of many applications of the van der Corput inequality combined with keeping careful track of coefficients of the polynomials that appear in the intermediate steps of the induction. From Proposition \ref{P: PET for polynomials twisted by duals} we then deduce in Corollary \ref{C: PET for polynomials twisted by duals} that infinite averages are controlled quantitatively by averages of generalized box seminorms. In the next section, we will combine these auxiliary results with the quantitative concatenation machinery to identify a single generalized box seminorm that controls a given polynomial average with polynomial bounds. 

\subsection{Setting up the induction scheme}

In order to carry out the PET argument in general, we first need to set up some notation that allows us to talk about general multidimensional polynomial families. The formalism that we employ is borrowed directly from \cite{KKL24a}, which itself draws heavily on previous works performing the PET argument (e.g. \cite{Ber87, BL96, CFH11, DFKS22, DKS22, Pel20}).

Let $d, k, \ell, s_1,s_2\in\N$. {In the rest of the section,} $\uh$ is always a tuple in $\Z^{s_1}$  unless explicitly stated otherwise. All the polynomial families that we shall consider will take the following form.
\begin{definition}[Essentially distinct and normal polynomial families]
    Let $\CQ = (\q_1, \ldots,$ $\q_{s_2})$ be a family of polynomials in $\R^k[n, \uh]$ of degree at most $d$. We call $\CQ$ \textit{essentially distinct} if the polynomials $\q_1, \ldots, \q_{s_2}$ and their pairwise differences are nonconstant as polynomials in $n$. In a similar spirit, we say that $\q_1$ is \emph{essentially distinct} from $\q_2, \ldots, \q_s$ if $\q_1, \q_1 - \q_2, \ldots, \q_1 - \q_s$ are nonconstant as polynomials in $n$.
    Lastly, we  call $\CQ$ \textit{normal} if it is essentially distinct, $\q_1$ has the highest degree with respect to the variable $n$, and $\q_1(0, \uh), \ldots, \q_{s_2}(0,\uh)$ are the zero polynomial for all $\uh\in\Z^{s_1}$.
\end{definition}

Let $\CP = (\p_1, \ldots, \p_\ell)$ be a normal family of polynomials in $\R^k[n]$ of degree at most $d$. Without loss of generality, we can assume that $\deg\p_1 = d$. 
We denote the coefficients of the polynomials $\p_j$ by $\bp_j(n) = \sum_{i=1}^d \bbeta_{ji}n^i$, and we  also set $\bp_0 := \mathbf{0}$ and $\bbeta_{0i} := \mathbf{0}$ (applying the same convention for families of polynomials in $\R^k[n,\uh]$ whenever convenient). During the PET argument, we will subject $\CP$ to an iterated application of the following operation.
      \begin{definition}[van der Corput operation]
          Let $\CQ = (\q_1, \ldots, \q_{s_2})$ be a normal family of polynomials in $\R^k[n,\uh]$. Given $m\in[s]$, we define the new polynomial family
      \begin{align*}
            \partial_m\CQ := &(\sigma_{h_{s_1+1}}{\q}_1-{\q}_m, \ldots, \sigma_{h_{s_1+1}}{\q}_{s_2} - {\q}_m, {\q}_1-{\q}_m, \ldots,  {\q}_{s_2}-{\q}_m)^*,
      \end{align*}
      where $\sigma_{h_{s_1+1}}\q(n, \uh)$ is the polynomial in $\mathbb{R}^{k}[n,\uh,h_{s_{1}+1}]$   defined by
      \begin{align*}
          \sigma_{h_{s_1+1}}\q(n, \uh) := \q(n+h_{s_1+1}, \uh) - \q(h_{s_1+1}, \uh),
      \end{align*}
      and the $^*$ operation removes all monomials independent of $n$ and all but the first copy of each polynomial that appears multiple times.
            \end{definition}
            We refer to the operation $\CQ\mapsto \partial_m \CQ$ as a \textit{van der Corput operation} on $\CQ$ since the new family $\partial_m \CQ$ will arise from $\CQ$ via an application of the van der Corput inequality. It is through this operation that we obtain a new polynomial family from $\CQ$ at each step of the PET induction scheme.

      Our next goal is to describe the properties of $\CP$ that are preserved by the van der Corput operation. Let $\CQ = (\q_1, \ldots, \q_{s_2})$ be a normal family of polynomials in $\R^k[n, \uh]$, and write 
      \begin{align}\label{E: q_j}
        \bq_j(n,\uh) = \sum_{i=1}^d \bgamma_{ji}(\uh)n^i      
      \end{align}
     for some $\bgamma_{ji}\in\R[\uh]$. In other words, the $\bgamma_{ji}$'s are the coefficients of $\bq_j$ as a polynomial in $n$, and they are themselves polynomials in $\uh$. For $j,j'\in[0,s_2]$, we also define
      $$d_{jj'} := \max\{i\in[d]:\; \bgamma_{ji}(\uh)\neq \bgamma_{j'i}(\uh)\}$$
      to be the degree of $\q_j - \q_{j'}$ as a polynomial in $n$,
      where we set $\bq_0 := \mathbf{0}$ and $\bgamma_{0i}:=\mathbf{0}$ for all $i\in\N_0$.
      We want to show that if $\CQ$ is a family obtained from $\CP$ at some stage of the PET induction scheme, then the polynomials $\bgamma_{ji}$ preserve several useful properties. The following definition is taken from \cite{KKL24a}.
      \begin{definition}[Descendent family of polynomials]\label{D: descendence}
      A {normal} family $\CQ = (\q_1, \ldots, \q_{s_2})$ of polynomials in $\R^k[n, \uh]$ \textit{descends from} $\CP$ if for each $j\in[s_2]$, the polynomials $\bgamma_{ji}$ in \eqref{E: q_j} can be expressed as\footnote{The multinomial coefficient
      \begin{align*}
          {{\binom{u_1+\cdots + u_{s_1}+i}{u_1, \ldots, u_{s_1}}}} = \frac{(u_1+\cdots+u_{s_1}+i)!}{u_1!\cdots u_{s_1}!\cdot i!}
      \end{align*}
      is the coefficient of $\uh^{\uu}n^i = h_1^{u_1}\cdots h_{s_1}^{u_{s_1}} n^i$ in $(n+h_1+\cdots + h_{s_1})^{i+u_1+\cdots+u_{s_1}}$.}
          \begin{align}\label{E: gamma_ji}
              \bgamma_{ji}(\uh) =  \sum_{\substack{\uu\in \N_0^{s_{1}},\\ |\uu|\leq d-i}} {{\binom{u_1+\cdots + u_{s_1}+i}{u_1, \ldots, u_{s_1}}}}(\bbeta_{w_{j\uu}(|\uu|+i)}-\bbeta_{w_{\uu}(|\uu|+i)})\uh^\uu
          \end{align}
        for some indices $w_{\uu}, w_{j\uu}\in[0,\ell]$ satisfying the following properties: 
      \begin{enumerate}
          \item\label{i: w_1} (Coefficients of $\q_1$) $w_{1\uu} = 1$ for all $\uu$;
          \item\label{i: multilinear} (Multilinearity) for each $j$, the polynomial $\bgamma_{1d_{1j}}-\bgamma_{jd_{1j}}$ (the leading coefficient of $\bq_1 - \bq_{j}$ as a polynomial in $n$) is multilinear in $\uh$;
        \item\label{i: leading coeffs} (Leading coefficients) for each $j$ and $\uu$, the vector $\bbeta_{1(|\uu|+d_{1j})}-\bbeta_{w_{j\uu}(|\uu|+d_{1j})}$ is nonzero only if it is the leading coefficient of  $\bp_1-\bp_{w_{j\uu}}$.
      \end{enumerate}
      \end{definition}

The following lemma ascertains that the van der Corput operation preserves the abovementioned properties inherited by $\CQ$ from $\CP$ by virtue of descending from $\CP$. While the following two results were proved for integer polynomials, the exact same arguments work for real polynomials.
    
    \begin{proposition}[{\cite[Proposition 4.6]{KKL24a}}]\label{P: vdC preserves descendancy}
        Let $d, k, \ell, s_1, s_2\in\N$, and let $\CP = (\p_1, \ldots, \p_\ell)$ be a normal family of polynomials in $\R^k[n]$ with coefficients given by $\bp_j(n) = \sum_{i=1}^d \bbeta_{ji}n^i$.
        Suppose that the family $\CQ = (\q_1, \ldots, \q_{s_2})$ of polynomials in $\R^k[n,\uh]$ descends from $\CP$. 
        Then the family $\partial_m \CQ$ also descends from $\CP$ for each $m\in[s_2]$ for which $\sigma_{h_{s_1+1}}{\q}_1-{\q}_m$
        has the highest degree as a polynomial in $n$ among the elements of $\partial_m \CQ$. 
    \end{proposition}

    The next result ensures that we reach a family of linear polynomials after a bounded number of van der Corput operations; this fact is the foundation of all PET induction schemes. 
        \begin{proposition}[{\cite[Proposition 4.7]{KKL24a}}]\label{P: vdC terminates}
        Let $d, k, \ell\in\N$ and $\CP$ be a normal family of  polynomials in $\R^k[n]$ of degree at most $d$. There exist positive integers $s=O_{d,\ell}(1)$ and $m_1, \ldots, m_s$ such that for every $i\in[s]$ the family $\CQ_i = \partial_{m_i}\cdots\partial_{m_1}\CP$ descends from $\CP$, and moreover the elements of $\CQ_s$ are linear in $n$. 
    \end{proposition}
    We remark that $s$ is independent of the dimension $k$ of the underlying space.

    The following lemma shows how we pass from the family $\CQ$ to $\partial_m\CQ$ in applications. In the language of \cite{DFKS22, DKS22}, the lemma below says that the new family $\partial_m\CQ$ is \textit{1-standard}. 
    \begin{lemma}\label{L: bounding averages in vdC}
        Let $d, k, \ell, s_1, s_2, H, N\in\N$ with $H\leq N$ and $\CP$ be a normal family of polynomials. Let $\CQ = \{\q_1, \ldots, \q_{s_2}\}$ be a family of polynomials in $\R^k[n,\uh]$ with $\deg \q_1 >1$, and suppose that $\CQ$ and $\partial_m\CQ$ both descend from $\CP$ for some $m\in[s_2]$. For every 1-bounded $f_1\in L^\infty(\mu)$ and $n_0\in\Z$, we have
        \begin{multline*}
            \sup_{\norm{f_2}_\infty, \ldots, \norm{f_s}_\infty\leq 1}\sup_{|c_n|\leq 1}\norm{\E_{n\in[N]}c_n\cdot T^{\floor{\q_1(n_0 + n,\uh)}}f_1\cdots T^{\floor{\q_{s_2}(n_0 + n,\uh)}}f_s}_{L^2(\mu)}^2\\
            \ll_{d,k, \ell, s_1,s_2}\E_{h_{s_1+1}\in[\pm H]} \sup_{\substack{\norm{f_\bq}_\infty\leq 1,\\ \bq\in\partial_m\CQ\setminus\{\sigma_{h_{s_1+1}}\q_1 - \q_m\}}}\sup_{|c_n|\leq 1}\\
            \norm{\E_{n\in[N]}c_n\cdot \prod_{\bq\in\partial_m\CQ} T^{\floor{\q(n_0 + n,\uh, h_{s_1+1})}}f_{\bq}}_{L^2(\mu)}+\frac{H}{N}
        \end{multline*}
        for $f_{\sigma_{h_{s_1+1}}\q_1 - \q_m} = f_1$.
    \end{lemma}
    \begin{proof}
        First, we observe crucially that all the polynomials
        \begin{align}\label{E: polys after vdC}
            \sigma_{h_{s_1+1}}{\q}_2-{\q}_m, \ldots, \sigma_{h_{s_1+1}}{\q}_{s_2} - {\q}_m, {\q}_1-{\q}_m, \ldots,  {\q}_{s_2}-{\q}_m
        \end{align}
        are essentially distinct from $\sigma_{h_{s_1+1}}{\q}_1-{\q}_m$ since all the polynomials in $\CQ$ are essentially distinct from $\bq_1$. In making this deduction, we crucially use the assumption that $\q_1$ is nonlinear in $n$ for otherwise ${\q}_1-{\q}_m$ and $\sigma_{h_{s_1+1}}{\q}_1-{\q}_m$ would differ by a polynomial in $(\uh, h_{s_1+1})$. Consequently, after applying the $*$ operation, no other polynomial from \eqref{E: polys after vdC} gets identified with $\sigma_{h_{s_1+1}}{\q}_1-{\q}_m$. We then apply the van der Corput inequality (Lemma~\ref{L: finitary van der Corput}), compose with $T^{-\floor{\bq_m(n)}}$, and use Lemma \ref{L: errors} in order to make error terms independent of $n$. While applying the van der Corput inequality, we remove the terms that do not depend on $n$, and we also group together functions that correspond to classes of iterates that are not essentially distinct. The essential distinctness of $\sigma_{h_{s_1+1}}{\q}_1-{\q}_m$ from other polynomials ensures that $f_{\sigma_{h_{s_1+1}}\q_1 - \q_m}$ is not grouped together with anything else. Finally, we compose the integral with some expression $T^{-\ba(\uh, h_{s_1+1})}$ in such a way as to ensure that $f_{\sigma_{h_{s_1+1}}\q_1 - \q_m} = f_1$.
    \end{proof}

    \subsection{Preliminary estimates for finite polynomial averages}
    Having set up the formalism for studying polynomial averages and described how they change after applying the van der Corput inequality, we move on to provide first seminorm estimates for finite polynomial averages. The starting point for all the results in this section is the following result that essentially follows from iterated applications of Lemma \ref{L: bounding averages in vdC}.
		\begin{proposition}[PET bound for polynomials]\label{P: PET I}
			Let $d, k, \ell\in\N$. There exist positive integers $s_1, s_2=O_{d, \ell}(1)$  with the following property: for all $H,M, N\in\N$ with $H,M\leq N$, $n_0\in\Z$, systems $(X, \CX, \mu, T_1, \ldots, T_k)$, 1-bounded functions $f_1\in L^\infty(\mu)$, and polynomials $\p_1, \ldots, \p_\ell\in\R^k[n]$ with degrees at most $d$,
             form $\p_j(n) = \sum_{i=0}^d \bbeta_{ji} n^i$, and such that $\bp_1$ is essentially distinct from $\bp_2, \ldots, \bp_\ell$, we have the bound
			\begin{multline}\label{E: PET bound inequality}
				\sup_{\norm{f_2}_\infty, \ldots, \norm{f_\ell}_\infty\leq 1}\sup_{|c_n|\leq 1}\norm{\E_{n\in[N]}c_n\cdot T^{\floor{\p_1(n_0 + n)}}f_1\cdots T^{\floor{\p_\ell(n_0 + n)}}f_\ell}_{L^2(\mu)}^{O_{d,\ell}(1)}\\
        \ll_{d,k, \ell} \E\limits_{\uh\in[\pm H]^{s_1}} \E_{\um, \um'\in[\pm M]^{s_2}} \abs{\int \Delta_{{\bc_1(\uh)}, \ldots, {\bc_{s_2}(\uh)}; (\um,\um')} f_1\, d\mu}
        + \frac{H}{N}
        +\frac{M}{N},
			\end{multline}
			for nonzero polynomials $\bc_1, \ldots, \bc_{s_2}\in\R^{k}[\uh]$ that depend only on $\bp_1, \ldots, \bp_\ell$, are multilinear and take the form
			\begin{align}\label{E: polynomials c_j}
				\bc_{j}(\uh) = \sum_{\substack{\uu\in \{0,1\}^{s_1},\\ |\uu|\leq d-1}} 
    (|\uu|+1)! \cdot (\bbeta_{1(|\uu|+1)}-\bbeta_{w_{j\uu}(|\uu|+1)})\uh^\uu
			\end{align}
                for some indices $w_{j\uu}\in[0,\ell]$ (with $\bbeta_{0(|\uu|+1)}:=\mathbf{0}$). Moreover, each nonzero vector $\bbeta_{1(|\uu|+1)}-\bbeta_{w_{j\uu}(|\uu|+1)}$ is the leading coefficient of $\p_1 - \p_{w_{j\uu}}$.
		\end{proposition}
  Note that in the proposition above, the polynomials and functions are chosen for fixed $N$, meaning that they can vary with $N$. The absolute constants depend only on the degree and number of polynomials as well as the number of transformations; they can be chosen uniformly so that they do not depend on the coefficients of the polynomials.  

  \begin{proof}
      We allow all the quantities to depend on $d,k, \ell$, noting however that $s_1, s_2$ will be independent of $k$.
      
      We assume that $n_0 = 0$ and that $\p_1$ has maximum degree among $\p_1, \ldots, \p_\ell$, and in the end we will explain the necessary modifications in the general case. Under this assumption, the family $\CP := \{\bp_1 - \bp_1(0), \ldots, \bp_\ell-\bp_\ell(0)\}$ is normal, and so we can use Proposition  \ref{P: vdC terminates} in order to find $s_1=O(1)$ (independent of $k$) and $m_1, \ldots, m_{s_1}\in\N$ such that for $i\in[s_1]$ the families $\CQ_i := \partial_{m_i}\cdots\partial_{m_1}\CP$ are normal and descend from $\CP$, and the family $\CQ_{s_1}$ is linear in $n$ and has length $s_2$. The van der Corput operation at most doubles the length of the polynomial family, giving us a crude estimate $s_2 = |\CQ_{s_1}|\leq 2^{s_1}|\CP|\ll 1$. We note in particular that the upper bound on $s_2$ is independent of $k$.
      
      Let $$\b_{11}(\uh)n, \ldots, \b_{s_21}(\uh)n$$ be the elements of $\CQ_{s_1}$. Applying Lemma \ref{L: bounding averages in vdC} inductively to each $\CQ_0, \ldots, \CQ_{s_1-1}$ (with $\CQ_0 = \CP$), we deduce that 
    \begin{multline*}
        \sup_{\norm{f_2}_\infty, \ldots, \norm{f_\ell}_\infty\leq 1}\sup_{|c_n|\leq 1}\norm{\E_{n\in[N]}c_n\cdot T^{\floor{\p_1(n)}}f_1\cdots T^{\floor{\p_\ell(n)}}f_\ell}_{L^2(\mu)}^{2^{s_1}}\\
        \ll\E\limits_{\uh\in[\pm H]^{s_1}} \sup_{\norm{f_{2}}_\infty, \ldots, \norm{f_{\ell}}_\infty\leq 1} \sup_{|c_n|\leq 1}\norm{\E_{n\in[N]} c_n\cdot \prod_{j\in[s_2]} T^{\floor{\b_{j1}(\uh)n}} f_{j}}_{L^2(\mu)}
        +\frac{H}{N}.
    \end{multline*}
    Since the family $\CQ_{s_1}$ descends from $\CP$, the polynomials $\b_{11}, \ldots, \b_{s_21}$ are all distinct. The claimed inequality then follows from Lemma \ref{L: linear averages} upon letting $\bc_1, \ldots, \bc_{s_2}$ be the polynomials $\b_{1 1}, \b_{11}-\b_{21}, \ldots, \b_{11} - \b_{s_2 1}$. That the polynomials $\bc_1, \ldots, \bc_{s_2}$ are nonzero is a consequence of the distinctness of $\b_{11}, \ldots, \b_{s_21}$. The formula \eqref{E: polynomials c_j} follows from patching together various properties that make up Definition \ref{D: descendence}.
    
    When $n_0\neq 0$, then we carry out the argument as above but with $n_0+n$ in place of $n$ at every step. Arriving at the linear polynomials $$\b_{11}(\uh)(n_0+n), \ldots, \b_{s_21}(\uh)(n_0+n),$$ we simply compose with $T^{-\floor{\b_{11}(\uh)n_0}}$ and dispose of the error terms using Lemma \ref{L: errors}, obtaining the exact same result as in the case $n_0 = 0$.

     When $\p_1$ does not have maximum degree, then we argue as in the proof of Corollary~\ref{C: sublinear}.
  \end{proof}

The following two results extend Proposition \ref{P: PET I} to families of polynomials with cubic shifts that we will need to consider later. The statement and proof differ somewhat depending on whether $\p_1$ is linear or not, and so we consider these two cases separately.

\begin{proposition}[Finitary PET bound for polynomials with cubic shifts, nonlinear case]\label{P: PET for polynomials with shifts}
    Let $s\in\N_0$ and $d, k, \ell\in \N$ with $d\geq 2$. There exist positive integers $s_1, s_2=O_{d, \ell,s}(1)$ with the following property: for all $H, K, M, N\in\N$ with $H,M\leq N$, $n_0\in\Z$, systems $(X, \CX, \mu, T_1, \ldots, T_k)$,  1-bounded functions $f_{1\underline{1}}\in L^\infty(\mu)$, and polynomials $\p_1, \ldots, \p_\ell\in\R^k[n]$ with degrees at most $d$,
             form $\p_j(n) = \sum_{i=0}^d \bbeta_{ji} n^i$, and such that $\bp_1$ is nonlinear and essentially distinct from $\bp_2, \ldots, \bp_\ell$, we have
    \begin{multline*}
        \sup_{\substack{\norm{f_{j\ueps}}_\infty\leq 1,\\ (j,\ueps)\neq(1,\underline{1})}}\sup_{|c_n|\leq 1}\norm{\E_{n\in[N]}c_n\cdot\prod_{j\in[\ell]}\prod_{\ueps\in\{0,1\}^s} T^{\floor{\bp_j(n_0+n+\ueps\cdot\uk)}}f_{j\ueps}}_{L^2(\mu)}^{O_{d,\ell,s}(1)}\\         \ll_{d, k, \ell, s} 
        \E\limits_{\uh\in[\pm H]^{s_1}} \E_{\um, \um'\in[\pm M]^{s_2}} \abs{\int \Delta_{\bc_1(\uh, \uk), \ldots, \bc_{s_2}(\uh, \uk); (\um,\um')} f_{1\underline{1}}\, d\mu} 
         + \frac{H}{N}+\frac{M}{N}
    \end{multline*}
    for all but at most $O_{d,k, \ell,s}(K^{s-1})$ values $\uk\in [\pm K]^{s}$ and nonzero polynomials $\bc_1, \ldots, \bc_{s_2}\in\R^k[\uh, \uk]$ that depend only on $\bp_1, \ldots, \bp_\ell$ and $s$, are multilinear and take the form
    \begin{align}\label{E: c_j with h, k}  
        \bc_{j}(\uh, \uk) = \sum_{\substack{(\uu,\uv)\in \N_0^{s+s_1},\\ |\uu|+|\uv|\leq d-1}} (|\uu|+|\uv|+1)!\cdot(\bbeta_{1(|\uu|+|\uv|+1)}-\bbeta_{w_{j\uu\uv}(|\uu|+|\uv|+1)})\uh^\uu \uk^\uv,
    \end{align}
                for some indices $w_{j\uu\uv}\in[0,\ell]$ (with $\bbeta_{0(|\uu|+|\uv|+1)}:=\mathbf{0}$). Moreover, each nonzero vector $\bbeta_{1(|\uu|+|\uv|+1)}-\bbeta_{w_{j\uu\uv}(|\uu|+|\uv|+1)}$ is the leading coefficient of $\p_1 - \p_{w_{j\uu\uv}}$.
\end{proposition}

    \begin{proof}
As in the proof of Proposition \ref{P: PET I}, we can assume without loss of generality that $n_0 = 0$.
Fix $\uk\in[\pm K]^{s}$ and let 
\begin{align*}
    \veps =
    \sup_{\substack{\norm{f_{j\ueps}}_\infty\leq 1,\\ (j,\ueps)\neq(1,\underline{1})}}\sup_{|c_n|\leq 1}\norm{\E_{n\in[N]}c_n\cdot\prod_{j\in[\ell]}\prod_{\ueps\in\{0,1\}^s} T^{\floor{\bp_j(n+\ueps\cdot\uk)}}f_{j\ueps}}_{L^2(\mu)}.
\end{align*}
We allow all the quantities in the proof to depend on $d,k,\ell,s$ but not on $H, K, M,N, n_0$ (noting however that $s_1, s_2$ and the powers of $\veps$ will be independent of $k$). Regrouping the polynomials if necessary, we pick $\ell_0\in[\ell]$ such that $\p_1, \ldots, \p_{\ell_0}$ are nonlinear while $\p_{\ell_0+1}, \ldots, \p_{\ell}$ are linear. By assumption, $\ell_0\geq 1$, i.e. $\p_1$ and possibly several other polynomials are nonlinear. Splitting $\floor{\p_j(n+\ueps\cdot\uk)}$ into $\floor{\bbeta_{j1}n} + \floor{\bbeta_{j1} \ueps\cdot\uk}$ and the error terms for $j>\ell_0$ and subsequently removing the error terms using Lemma \ref{L: errors}, we obtain
\begin{multline*}
    \sup_{\substack{\norm{f_{j\ueps}}_\infty\leq 1,\\ (j,\ueps)\neq(1,\underline{1})}}\; \sup_{\norm{f_{\ell_0+1}}_\infty, \ldots, \norm{f_{\ell}}_\infty\leq 1}\;
    \sup_{|c_n|\leq 1}\\
    \norm{\E_{n\in[N]}c_n\cdot\prod_{j\in[\ell_0]}\prod_{\ueps\in\{0,1\}^s} T^{\floor{\bp_j(n+\ueps\cdot\uk)}}f_{j\ueps}\cdot \prod_{j=\ell_0+1}^\ell T^{\floor{\bp_j(n)}}f_{j}}_{L^2(\mu)}\gg\veps
\end{multline*}
(we note that $f_{j\ueps}$ for $j>\ell_0$ will depend on $\uk$).
After this regrouping, all the polynomials in the average above are essentially distinct for all but $O(K^{s-1})$ values of $\uk\in[\pm K]^s$, which will make it amenable to an application of Proposition \ref{P: PET I} later on.

We can rewrite
 \begin{align*}
     \bp_j(n+\ueps\cdot\uk) = \sum_{i=0}^d \bbeta_{ji}(n+\ueps\cdot\uk)^i = \sum_{i=0}^d\bbeta_{ji}\sum_{l=0}^i{{\binom{i}{l}}}n^l(\ueps\cdot\uk)^{i-l} = \sum_{l=0}^d\bbeta_{j\ueps\uk l}n^l
 \end{align*}
 for $\bbeta_{j\ueps\uk l} = \sum_{i=l}^d \bbeta_{ji}{{\binom{i}{l}}}(\ueps\cdot\uk)^{i-l}$.
 By Proposition \ref{P: PET I} applied with $f_{1\underline{1}}$ in place of $f_1$ to all the averages over $n$ except those corresponding to the $O(K^{s-1})$ exceptional values of $\uk$, there exist positive integers $s_1, s_2=O(1)$ (independent of $k$) such that
 \begin{align*}
        \E\limits_{\uh\in[\pm H]^{s_1}} \E_{\um, \um'\in[\pm M]^{s_2}} \abs{\int \Delta_{{\bc_1(\uh, \uk)}, \ldots, {\bc_{s_2}(\uh, \uk)}; (\um,\um')} f_{1\underline{1}}\, d\mu}
        +\frac{H}{N}+\frac{M}{N} \gg \veps^{O(1)}
			\end{align*}
			for nonzero polynomials $\bc_1, \ldots, \bc_{s_2}:\R^k[\uh, \uk]$. Moreover, the polynomials $\bc_1, \ldots, \bc_{s_2}$ are multilinear, depend only on $\bp_1, \ldots, \bp_\ell$ and $\uk$, and take the form
			\begin{align}\label{E: c_j preliminary}
	        \bc_{j}(\uh, \uk) = \sum_{\substack{\uu\in \{0,1\}^{s_1},\\ |\uu|\leq d-1}} (|\uu|+1)!\cdot (\bbeta_{1\underline{1}\uk(|\uu|+1)}-\bbeta_{\tilde{w}_{j\uu}\uk(|\uu|+1)})\uh^\uu,			
			\end{align}
		for some $\tilde{w}_{j\uu}= (w_{j\uu}, \ueps)\in[0,\ell]\times\{0,1\}^{s}$, where $\bbeta_{1\underline{1}\uk(|\uu|+1)}=\bbeta_{1(|\uu|+1)}$, and  $\bbeta_{1\underline{1}\uk(|\uu|+1)}-\bbeta_{\tilde{w}_{j\uu}\uk(|\uu|+1)}$ is the leading coefficient of $\p_1(n+\underline{1}\cdot\uk) - \p_{{w}_{j\uu}}(n+\ueps\cdot\uk)$. 
  
  Our task is to massage the formula \eqref{E: c_j preliminary} for $\bc_j(\uh, \uk)$ in order to reach the formula \eqref{E: c_j with h, k} from the statement of the proposition. Specifically, we want to show that the coefficient of $\uh^\uu \uk^\uv$ in \eqref{E: c_j preliminary} takes the form given by \eqref{E: c_j with h, k}, and the way to proceed is to scrutinize the expressions $\bbeta_{1\underline{1}\uk(|\uu|+1)}-\bbeta_{\tilde{w}_{j\uu}\uk(|\uu|+1)}$ for fixed $j$, $\uu$ and $\ueps$. 
  If $\ueps = \underline{1}$, or if the polynomials $\p_1, \p_{w_{j\uu}}$ have distinct leading coefficients, then the leading coefficient of $$\p_1(n+\underline{1}\cdot\uk) - \p_{{w}_{j\uu}}(n+\ueps\cdot\uk)$$ is the same as the leading coefficient of $\p_1(n) - \p_{{w}_{j\uu}}(n)$ and equals $\bbeta_{1(|\uu|+1)}-\bbeta_{w_{j\uu}(|\uu|+1)}$ for $|\uu| = \deg(\p_1-\p_{w_{j\uu}}) -1$. Then, for each $\uv$, the coefficient of $\uh^\uu \uk^\uv$ in \eqref{E: c_j preliminary} matches the formula from \eqref{E: c_j with h, k} with $w_{j\uu\uv} = w_{j\uu}$ if $|\uv|=0$ and $w_{j\uu\uv} = 1$ otherwise.
  
  The remaining case is when $\ueps\neq \underline{1}$ and $\p_1, \p_{w_{j\uu}}$ have the same leading coefficients. Assuming without loss of generality that $\deg \p_1 = \deg \p_{w_{j\uu}} = d$ (so that $\bbeta_{1d}=\bbeta_{w_{j\uu}d}$), the leading coefficient of $\p_1(n+\underline{1}\cdot\uk) - \p_{{w}_{j\uu}}(n+\ueps\cdot\uk)$ becomes 
  \begin{align*}
      \bbeta_{1\underline{1}\uk(|\uu|+1)}-\bbeta_{\tilde{w}_{j\uu}\uk(|\uu|+1)} = d\bbeta_{1d}(\underline{1}-\ueps)\cdot\uk  + \bbeta_{1(d-1)}-\bbeta_{w_{j\uu}(d-1)}
  \end{align*}
  for $|\uu| = d-2$, and so
  \begin{multline*}
      (|\uu|+1)!\cdot (\bbeta_{1\underline{1}\uk(|\uu|+1)}-\bbeta_{\tilde{w}_{j\uu}\uk(|\uu|+1)})\uh^\uu\\ 
      = \sum_{\substack{\uv\in\{0,1\}^s,\\ |\uv|\leq d-(|\uu|+1)}}(|\uu|+|\uv|+1)! \cdot (\bbeta_{1(|\uu|+|\uv|+1)}-\bbeta_{{w}_{j\uu\uv}(|\uu|+|\uv|+1)})\uh^\uu\uk^{\uv}
  \end{multline*}
  with
  \begin{align*}
      w_{j\uu\uv} = \begin{cases} w_{j\uu},\; &|\uv| = 0\\
      0,\; &|\uv| = 1\\
      1,\; & |\uv|>1
      \end{cases}.
  \end{align*}
  We implicitly use here the observations that $\bbeta_{1(|\uu|+|\uv|+1)}-\bbeta_{{w}_{j\uu\uv}(|\uu|+|\uv|+1)}$ vanishes when $|\uv|\geq 2$ and that  $(|\uu|+|\uv|+1)! = (|\uu|+1)! \cdot d $ for $|\uv|=1$ since in this case we have $|\uu|+|\uv|+1 = |\uu|+2 = d$. Hence for such $\uu$, the coefficient of $\uh^\uu \uk^\uv$ in \eqref{E: c_j preliminary} once again takes the form prescribed by \eqref{E: c_j with h, k}.
\end{proof}

We now prove a version of Proposition \ref{P: PET for polynomials with shifts} in which $\p_1$ is linear. We note that the average that we aim to bound slightly differs from the one in the previous proposition; this has to do with the way in which Proposition \ref{P: PET for polynomials with shifts 2} is meant to be applied. 

\begin{proposition}[Finitary PET bound for polynomials with cubic shifts, linear case]\label{P: PET for polynomials with shifts 2}
    Let $d, k, \ell\in \N$  and $s\in\N_0$. There exist positive integers $s_1, s_2=O_{d, \ell,s}(1)$ with the following property: for all $H, K, M, N\in\N$ with $H,M\leq N$, $n_0\in\Z$, systems $(X, \CX, \mu, T_1, \ldots,$ $ T_k)$,  1-bounded functions $f_{1{\ueps}}\in L^\infty(\mu)$, and polynomials $\p_1, \ldots, \p_\ell\in\R^k[n]$ with degrees at most $d$,
             form $\p_j(n) = \sum_{i=0}^d \bbeta_{ji} n^i$, and such that $\bp_1$ is linear and essentially distinct from $\bp_2, \ldots, \bp_\ell$, we have
    \begin{multline*}
        \sup_{\substack{\norm{f_{2\ueps}}_\infty, \ldots, \norm{f_{\ell\ueps}}_\infty\leq 1}}\sup_{|c_n|\leq 1}\norm{\E_{n\in[N]}c_n\cdot\prod_{j\in[\ell]}\prod_{\ueps\in\{0,1\}^s} T^{\floor{\bp_j(n_0+n+ ({\underline{1}}-\ueps) \cdot\uk + \ueps\cdot\uk')
        }}f_{j\ueps}}_{L^2(\mu)}^{O_{d,\ell,s}(1)}\\         \ll_{d, k,\ell, s}  
        \E\limits_{\uh\in[\pm H]^{s_1}}\E_{\um, \um'\in[\pm M]^{s_2}}\\
        \abs{\int \prod_{\ueps\in\{0,1\}^{s_1}}
        T^{ \floor{\bbeta_{11}k_1^{\eps_1}}+\cdots + \floor{\bbeta_{11}k_s^{\eps_s}}
        +\br_{1\ueps}}\Delta_{\bc_1(\uh), \ldots, \bc_{s_2}(\uh); (\um,\um')} f_{1\ueps}\, d\mu}
        +  \frac{H}{N}+\frac{M}{N}
    \end{multline*}
    for all but at most $O_{d,k,\ell,s}(K^{2s-1})$ values $(\uk,\uk')\in [\pm K]^{2s}$,
    some shifts $\br_{1\ueps}\in\Z^k$ and nonzero polynomials $\bc_1, \ldots, \bc_{s_2}:\R^k[\uh]$ that depend only on $\bp_1, \ldots, \bp_\ell, s$ and satisfy the same properties as those listed in Proposition \ref{P: PET I}.
\end{proposition}

\begin{proof}
Let 
\begin{align*}
    \veps = 
    \sup_{\substack{\norm{f_{2\ueps}}_\infty, \ldots, \norm{f_{\ell\ueps}}_\infty\leq 1}}\sup_{|c_n|\leq 1}\norm{\E_{n\in[N]}\prod_{j\in[\ell]}\prod_{\ueps\in\{0,1\}^s} T^{\floor{\bp_j(n_0+n+(\underline{1}-\ueps)\cdot\uk+\ueps\cdot\uk')}}f_{j\ueps}}_{L^2(\mu)}.
\end{align*}
We allow all the quantities to depend on $d, J, k, \ell$, noting however that $s_1, s_2$ and powers of $\veps$ will be independent of $k$.
    We regroup the polynomials so that $\p_j$ is linear if and only if $j\in[\ell_0]$ for some $\ell_0\in[\ell]$. By the pigeonhole principle, there exist $\br_{j\ueps\uk\uk'}\in\Z^k$ for which 
    \begin{multline*}
   \sup_{\substack{\norm{f_{2\ueps}}_\infty, \ldots, \norm{f_{\ell\ueps}}_\infty\leq 1}}\sup_{|c_n|\leq 1}
    \left|\!\left|\E_{n\in[N]}\prod_{j\in[\ell_0]}T^{\floor{\bbeta_{j1}n}}\prod_{\ueps\in\{0,1\}^s} T^{\floor{\bbeta_{j1}k_1^{\eps_1}}+\cdots + \floor{\bbeta_{j1}k_s^{\eps_s}}+\br_{j\ueps\uk\uk'}}f_{j\ueps}\right.\right.\\ \left.\left. \prod_{j=\ell_0+1}^\ell \prod_{\ueps\in\{0,1\}^{s_1}} T^{\floor{\bp_j(n+(\underline{1}-\ueps)\cdot\uk+\ueps\cdot\uk')}} f_{j\ueps}\right|\!\right|_{L^2(\mu)}\gg\veps.
\end{multline*}
For all $(\uk,\uk')\in[\pm K]^{2s}$ save an exceptional set of size $O(K^{2s-1})$, the polynomials in the expression above are essentially distinct as polynomials in $n$, and so we can apply Proposition \ref{P: PET I} to these ``good'' values $(\uk,\uk')$ to find positive integers $s_1, s_2 = O(1)$ (independent of $k$) such that
\begin{multline*}
    \E\limits_{\uh\in[\pm H]^{s_1}}\E_{\um, \um'\in[\pm M]^{s_2}}
    \left|\int \prod_{\ueps\in\{0,1\}^{s_1}} T^{\floor{\bbeta_{11}k_1^{\eps_1}}+\cdots + \floor{\bbeta_{11}k_s^{\eps_s}}+\br_{1\ueps\uk\uk'}}\right.\\
    \left.   \vphantom{\prod_{\ueps\in\{0,1\}^{s_1}}}
    \brac{\Delta_{\bc_1(\uh,\uk,\uk'), \ldots, \bc_{s_2}(\uh,\uk,\uk'); (\um,\um')} f_{1\ueps}}\, d\mu\right| +  \frac{H}{N}+\frac{M}{N}\gg\veps^{O(1)}
\end{multline*}
			for nonzero polynomials $\bc_1, \ldots, \bc_{s_2}\in\R^k[\uh,\uk,\uk']$. Moreover, the polynomials $\bc_1, \ldots, \bc_{s_2}$ are multilinear, depend only on $\bp_1, \ldots, \bp_\ell$, $\uk,\uk'$ and $s$, and take the form
			\begin{align*}
	        \bc_{j}(\uh,\uk,\uk') = \sum_{\substack{\uu\in \{0,1\}^{s_1},\\ |\uu|\leq d-1}} (|\uu|+1)!\cdot (\bbeta_{1\underline{1}\uk\uk'(|\uu|+1)}-\bbeta_{\tilde{w}_{j\uu}\uk\uk'(|\uu|+1)})\uh^\uu,			
			\end{align*}
		for some $$\tilde{w}_{j\uu}= (w_{j\uu}, \ueps) \in [0,\ell]\times\{0,1\}^{s},$$ where $\bbeta_{1\underline{1}\uk\uk'(|\uu|+1)}=\bbeta_{1(|\uu|+1)}$ and  $\bbeta_{1\underline{1}\uk\uk'(|\uu|+1)}-\bbeta_{\tilde{w}_{j\uu}\uk\uk'(|\uu|+1)}$ is the leading coefficient of
  \begin{align}\label{E: leading coeff 7.8}
    \bbeta_{11}n - \p_{{w}_{j\uu}}(n+(\underline{1}-\ueps)\cdot\uk+\ueps\cdot\uk').    
  \end{align}
   The leading coefficient of \eqref{E: leading coeff 7.8} is either $\bbeta_{11}$ (if $w_{j\uu}=0$), or $\bbeta_{11}-\bbeta_{j1}$ (if $w_{j\uu}\in[\ell_0]$), or it is minus times the leading coefficient of $\p_{{w}_{j\uu}}$ whenever $j\in[\ell_0+1,\ell]$ since then $\deg \p_{{w}_{j\uu}} \geq 2$. It follows from the observation that the collection of possible leading coefficients of \eqref{E: leading coeff 7.8} is the same as the collection of possible leading coefficients of $\p_{1}, \p_{1}-\p_2, \ldots, \p_1-\p_\ell$. In particular, the leading coefficients are independent of $\uk$, and so are the polynomials $\bc_1, \ldots, \bc_{s_2}$.
\end{proof}

\subsection{Estimates for polynomial averages with dual twists}
In later applications, we need to consider polynomial averages twisted by dual sequences. The next two results allow us to bound such averages by an average of generalized box seminorms (and their finitary versions), building on Propositions \ref{P: PET for polynomials with shifts} and \ref{P: PET for polynomials with shifts 2}.

\begin{proposition}[Finitary PET bound for polynomials with dual twists]\label{P: PET for polynomials twisted by duals}
    Let $d, k, \ell\in \N$ and $J\in\N_0$. There exist positive integers $s_1, s_2=O_{d,J,\ell}(1)$ with the following property: for all $H,M,N\in\N$ with $H,M\leq N$, $n_0\in\N_0$,  systems $(X, \CX, \mu, T_1, \ldots, T_k)$,  1-bounded functions $f_{1}\in L^\infty(\mu)$, Hardy sequences $b_1, \ldots, b_J\in\CH$ of growth $b_j(t)\ll t^d$, and polynomials $\p_1, \ldots, \p_\ell\in\R^k[n]$ with degrees at most $d$,
             form $\p_j(n) = \sum_{i=0}^d \bbeta_{ji} n^i$, and such that $\bp_1$ is essentially distinct from $\bp_2, \ldots, \bp_\ell$, we have
    \begin{multline*}
        \sup_{n_0\in\N_0}\sup_{\substack{\norm{f_2}_\infty, \ldots, \norm{f_\ell}_\infty\leq 1,\\ \CD_1, \ldots, \CD_J\in\FD_d}}\sup_{|c_n|\leq 1}\norm{\E_{n\in[N]}c_n\cdot \prod_{j\in[\ell]} T^{\floor{\bp_j(n_0+n)}}f_j\cdot \prod_{j\in[J]}\CD_j(\floor{b_j(n_0+n)})}_{L^2(\mu)}^{O_{d,J,\ell}(1)}\\
        \ll_{d, J, k, \ell}  \E\limits_{\uh\in[\pm H]^{s_1}} \E_{\um, \um'\in[\pm M]^{s_2}} \abs{\int \Delta_{\bc_1(\uh), \ldots, \bc_{s_2}(\uh); (\um,\um')} f_{1}\, d\mu}\\
        + \frac{1}{H} + \frac{1}{M} + \frac{H}{N}+\frac{M}{N} + o_{N\to\infty; d, J, k, \ell, b_1, \ldots, b_J}(1)
    \end{multline*}
    for nonzero polynomials $\bc_1, \ldots, \bc_{s_2}\in\R^{k}[\uh]$ that depend only on $\bp_1, \ldots, \bp_\ell$, $d$ and $L$ and satisfy the same properties as those listed in Proposition \ref{P: PET I}.
\end{proposition}
\begin{proof}
 Let 
\begin{align*}
    \veps =  \sup_{n_0\in\N_0}\sup_{\substack{\norm{f_2}_\infty, \ldots, \norm{f_\ell}_\infty\leq 1,\\ \CD_1, \ldots, \CD_J\in\FD_d}}\sup_{|c_n|\leq 1}\norm{\E_{n\in[N]}c_n\cdot\prod_{j\in[\ell]} T^{\floor{\bp_j(n_0+n)}}f_j\cdot \prod_{j\in[J]}\CD_j(\floor{b_j(n_0+n)})}_{L^2(\mu)}.
\end{align*}
We allow all the quantities to depend on $d, J, k, \ell$, noting however that $s_0, s_1, s_2$ and powers of $\veps$ will not depend on $k$. We split into two cases based on the degree of $\p_1$.

\smallskip
\textbf{Case 1: $\deg \p_1 >1$.}
\smallskip

Suppose first that $\deg \p_1 >1$.
By Lemma \ref{L: removing duals finitary} applied with $K=H$, there exists a positive integer $s_0=O(1)$ (independent of $k$) such that
\begin{multline*}
    \E_{\uk\in[\pm H]^{s_0}}\sup_{n_0\in\N_0}\sup_{\substack{\norm{f_2}_\infty, \ldots, \norm{f_\ell}_\infty\leq 1,}}\sup_{|c_n|\leq 1}\norm{\E_{n\in[N]}c_n\cdot\prod_{j\in[\ell]} T^{\floor{\bp_j(n_0+n+\ueps\cdot\uk)}}f_j}_{L^2(\mu)}\\
    +\frac{H}{N} + o_{N\to\infty}(1)
    \gg \veps^{O(1)}
\end{multline*}
(the $o_{N\to\infty}(1)$ term is allowed to depend on $b_1, \ldots, b_J$ in addition to $d,k,\ell, L$).
 Proposition \ref{P: PET for polynomials with shifts} gives us positive integers $s_1, s_2=O(1)$ independent of $k$ and polynomials $\bc_1, \ldots, \bc_{s_2}\in\R^k[\uh, \uk]$ satisfying the properties listed in Proposition \ref{P: PET I} for which
    \begin{multline*}
        \E_{(\uh, \uk)\in[\pm H]^{s_0+s_1}} \E_{\um, \um'\in[\pm M]^{s_2}} \abs{\int \Delta_{\bc_1(\uh, \uk), \ldots, \bc_{s_2}(\uh, \uk); (\um,\um')} f_{1}\, d\mu}\\ 
        + \frac{1}{H}
         + \frac{H}{N}+\frac{M}{N}+ o_{N\to\infty}(1)\gg \veps^{O(1)}
    \end{multline*}
    (the error $1/H$ comes from the exceptional set of $\uk$'s mentioned in Proposition \ref{P: PET for polynomials with shifts}).
    The result follows upon renaming $(\uh, \uk)$ as $\uh$ and $s_0+s_1$ as $s_1$.

\smallskip
\textbf{Case 2: $\deg \p_1 = 1$.}
\smallskip

If $\deg \bp = 1$, then Lemma \ref{L: removing duals finitary} applied with $K=M$ gives a positive integer $s_0=O(1)$ (independent of $k$) such that
\begin{multline*}
    \E_{\uk, \uk'\in[\pm M]^{s_0}}\sup_{n_0\in\N_0}\sup_{\substack{\norm{f_2}_\infty, \ldots, \norm{f_\ell}_\infty\leq 1}}\sup_{|c_n|\leq 1}\\
    \norm{\E_{n\in[N]}c_n\cdot\prod_{j\in[\ell]} T^{\floor{\bp_j(n_0+n+({\underline{1}}-\ueps) \cdot\uk + \ueps\cdot\uk'
    )}}f_j}_{L^2(\mu)}+\frac{M}{N} + o_{N\to\infty}(1)
    \gg \veps^{O(1)}
\end{multline*}
(again, the $o_{N\to\infty}(1)$ term is allowed to depend on $b_1, \ldots, b_J$ in addition to $d,k,\ell, L$).
Proposition~\ref{P: PET for polynomials with shifts 2} then gives positive integers $s_1, s_2 = O(1)$ (also independent of $k$), polynomials $\bc_1, \ldots, \bc_{s_2}\in\R^k[\uh]$ satisfying the properties listed in Proposition \ref{P: PET I}, and shifts $\br_{1\ueps}\in\Z^k$ such that
\begin{multline*}
    \E_{\uk, \uk'\in[\pm M]^{s_0}}\E\limits_{\uh\in[\pm H]^{s_1}} \E_{\um, \um'\in[\pm M]^{s_2}}\\ \abs{\int \prod_{\ueps\in\{0,1\}^{s_0}} T^{\floor{\bbeta_{11}k_1^{\eps_1}}+\cdots + \floor{\bbeta_{11}k_{s_0}^{\eps_{s_0}}}
    +\br_{1\ueps}}\Delta_{\bc_1(\uh), \ldots, \bc_{s_2}(\uh); (\um,\um')} f_1\, d\mu}\\
    + \frac{1}{M}
         + \frac{H}{N}+\frac{M}{N} + o_{N\to\infty}(1)
    \gg\veps^{O(1)}.
\end{multline*}
This is almost what we want except for the shifts $\br_{1\ueps}$. After squaring the absolute value using the Cauchy-Schwarz inequality and passing to the product system, the shifts $\br_{1\ueps}$ can however be removed using Lemma \ref{L: GCS finitary} (the finitary Gowers-Cauchy-Schwarz inequality) applied to the average over $\uk, \uk'$ as well as translation invariance. This gives
 \begin{multline*}
    \E_{\uk, \uk'\in[\pm M]^{s_0}}\E\limits_{\uh\in[\pm H]^{s_1}}\E_{\um, \um'\in[\pm M]^{s_2}}\\
    \abs{\int \Delta_{\bbeta_{11}^{s_0}; (\uk, \uk')}\Delta_{\bc_1(\uh), \ldots, \bc_{s_2}(\uh); (\um,\um')} f_1\, d\mu}^2
    + \frac{1}{M}
         + \frac{H}{N}+\frac{M}{N} + o_{N\to\infty}(1)
    \gg\veps^{O(1)}.
\end{multline*}
The result follows upon including $s_0$ copies of $\bbeta_{11}$ into the family $\bc_1, \ldots, \bc_{s_2}$ (this new family of polynomials still satisfies the conditions set in Proposition \ref{P: PET I}) and renaming $s_0+s_2$ as $s_2$, $(\uk,\um)$ as $\um$ and $(\uk',\um')$ as $\um'$.
\end{proof}

The result below is a straightforward infinitary corollary of Proposition \ref{P: PET for polynomials twisted by duals}.
\begin{corollary}[Infinitary PET bound for polynomials with dual twists]\label{C: PET for polynomials twisted by duals}
    Let $d, k, \ell\in \N$ and $J\in\N_0$. There exist positive integers $s_1, s_2=O_{d,J,\ell}(1)$ with the following property: for all $H\in\N$, systems $(X, \CX, \mu, T_1, \ldots, T_k)$,  1-bounded functions $f_{1}\in L^\infty(\mu)$, Hardy sequences $b_1, \ldots, b_J\in\CH$ of growth $b_j(t)\ll t^d$, and polynomials $\p_1, \ldots, \p_\ell\in\R^k[n]$ with degrees at most $d$,
             form $\p_j(n) = \sum_{i=0}^d \bbeta_{ji} n^i$, and such that $\bp_1$ is essentially distinct from $\bp_2, \ldots, \bp_\ell$, we have
    \begin{multline*}
        \limsup_{N\to\infty}\sup_{n_0\in\N_0}\sup_{\substack{\norm{f_2}_\infty, \ldots, \norm{f_\ell}_\infty\leq 1,\\ \CD_1, \ldots, \CD_J\in\FD_d}}\sup_{|c_n|\leq 1}\\
        \norm{\E_{n\in[N]}c_{n}\cdot\prod_{j\in[\ell]} T^{\floor{\bp_j(n_0+n)}}f_j\cdot \prod_{j\in[J]}\CD_j(\floor{b_j(n_0+n)})}_{L^2(\mu)}^{O_{d,J,\ell}(1)}\\
        \ll_{d, J, k, \ell}  \E\limits_{\uh\in[\pm H]^{s_1}} \nnorm{f_1}_{\bc_1(\uh), \ldots, \bc_{s_2}(\uh)}^+ + \frac{1}{H}
    \end{multline*}
    for nonzero polynomials $\bc_1, \ldots, \bc_{s_2}\in\R^{k}[\uh]$ that depend only on $\bp_1, \ldots, \bp_\ell$, $d$ and $L$ and satisfy the same properties as those listed in Proposition \ref{P: PET I}. 
\end{corollary}

\section{Quantitative concatenation for seminorms   with polynomial parameters}\label{S: polynomial concatenation}

The purpose of this section is to show that the averages 
\begin{align}\label{E: average of polynomial box seminorms}
    \limsup_{H\to\infty}\E\limits_{\uh\in[\pm H]^{s_1}}\nnorm{ f}^+_{\bc_1(\uh), \ldots, \bc_{s_2}(\uh)}
\end{align}
of box seminorms along multilinear polynomials $\bc_1, \ldots, \bc_{s_2}\in\R[\uh]$ that have appeared in the previous section can be quantitatively controlled by a single box seminorm in an appropriate choice of directions. The way in which we ``concatenate'' the box seminorms appearing in \eqref{E: average of polynomial box seminorms} is closely modeled on the finitary argument recently developed in \cite{KKL24a}, with both the general strategy and the main inputs being directly borrowed from that paper.

The quantitative concatenation of \eqref{E: average of polynomial box seminorms} into a single box seminorm relies on three main ingredients. The first one is a consequence of Corollary \ref{C: iterated concatenation for general groups} that allows us to replace the polynomials $\bc_1, \ldots, \bc_{s_2}$ by ones in more variables that are consequently better distributed in their ranges than the original polynomials. The second ingredient is an equidistribution result for a system of multilinear forms that appear while concatenating \eqref{E: average of polynomial box seminorms}. The last component is a technical lemma that allows us to eliminate all the ``nongeneric'' tuples $\uh$ that might potentially bring trouble. The last two components are taken verbatim from \cite{KKL24a}, while the first one is an ergodic adaptation of the arguments from that paper.

Below we state and prove these results; we illustrate the main ideas with the following example.
\begin{example}\label{Ex: concatenation example}
Suppose that $s=1$ and consider
\begin{align}\label{E: c 1}
    \bc(h_1, h_2) = \bbeta_{11} h_1 h_2 + \bbeta_{10} h_1.    
\end{align}
By passing to the product system, we can deal with $\nnorm{\cdot}_{\bc(\uh)}$ rather than $\nnorm{\cdot}_{\bc(\uh)}^+$.
For the sake of simplicity, suppose that $\bbeta_{11}, \bbeta_{10}\in\Z^k$ so that we do not have to deal with integer parts.
Let
\begin{align*}
   \veps = \E\limits_{\uh\in[\pm H]^{2}}{\nnorm{f}^2_{\bc(\uh)}} &= \E_{h_1, h_2\in[\pm H]}\lim_{M\to\infty}\E_{m,m'\in[\pm M]}{\int T^{\bc(h_1, h_2)m'}f\cdot T^{\bc(h_1, h_2)m}\overline{f}\, d\mu}\\
    &= \lim_{M\to\infty}{\int f\cdot \E_{h_1, h_2\in[\pm H]}\E_{m\in[\pm M]} T^{\bc(h_1, h_2)m}\overline{f}\, d\mu};
\end{align*}
the second line follows by composing with $T^{-\bc(h_1, h_2)m'}$ and recalling from Lemma~\ref{L: sumset lemma} that the limits $\lim\limits_{M\to\infty}\E_{m,m'\in[\pm M]}S^{m-m'} f$ and $\lim\limits_{M\to\infty}\E_{m\in[\pm M]}S^{m} f$ are the same for any transformation $S$. Our point is to apply the Cauchy-Schwarz inequality several times to increase the number of variables so that the expression $\bc(h_1, h_2)m$ is replaced by one better distributed inside $\langle \bbeta_{11}, \bbeta_{10}\rangle$. Applying the Cauchy-Schwarz inequality and invoking the 1-boundedness of $f$, we have
\begin{align*}
    \veps^2 &\leq \lim_{M\to\infty}\E_{\substack{h_{11}, h_{12},\\ h_{21}, h_{22}\in[\pm H]}}\E_{m_1, m_2\in[\pm M]}\int T^{\bc(h_{11}, h_{21})m_1}f\cdot  T^{\bc(h_{12}, h_{22})m_2}\overline{f}\, d\mu\\
    &= \lim_{M\to\infty}\int f\cdot \E_{\substack{h_{11}, h_{12},\\ h_{21}, h_{22}\in[\pm H]}}\E_{m_1, m_2\in[\pm M]} T^{\bc(h_{12}, h_{22})m_2 - \bc(h_{11}, h_{21})m_1}\overline{f}\, d\mu.
\end{align*}
We note that after this and the next applications of the Cauchy-Schwarz inequality, the limit as $M\to\infty$ exists because the expression that we are arriving at is a finite average of box seminorms.

The polynomial $\bc(h_{12}, h_{22})m_2 - \bc(h_{11}, h_{21})m_1$ is much better equidistributed inside its range than $\bc(h_1, h_2)m$ was. However, it is still not quite as well distributed as necessary. We therefore apply the Cauchy-Schwarz inequality one more time, getting 
\begin{align*}
   \veps^4 &\leq \lim_{M\to\infty}\E_{\substack{h_{1k}, h_{2k}\in[\pm H]:\\ k\in[4]}}\; \E_{\substack{m_k\in[\pm M]:\\ k\in[4]}}\\
   &\qquad\qquad\qquad\qquad\int T^{\bc(h_{12}, h_{22})m_2 - \bc(h_{11}, h_{21})m_1} f\cdot
   T^{\bc(h_{13}, h_{23})m_3 - \bc(h_{14}, h_{24})m_4} \overline{f}\, d\mu \\ 
    &= \lim_{M\to\infty}\E_{\substack{h_{1k}, h_{2k}\in[\pm H]:\\ k\in[4]}}\; \E_{\substack{m_k\in[\pm M]:\\ k\in[4]}}\\
    &\qquad\qquad\qquad\qquad\int f\cdot T^{\bc(h_{11}, h_{21})m_1 - \bc(h_{12}, h_{22})m_2 + \bc(h_{13}, h_{23})m_3 - \bc(h_{14}, h_{24})m_4} \overline{f}\, d\mu.
\end{align*} 
By definition, we have
\begin{multline*}
    \bc(h_{11}, h_{21})m_1 - \bc(h_{12}, h_{22})m_2 + \bc(h_{13}, h_{23})m_3 - \bc(h_{14}, h_{24})m_4\\ = \bbeta_{11}(h_{11} h_{21} m_1 - h_{12} h_{22}m_2 + h_{13} h_{23}m_3 - h_{14} h_{24}m_4)\\ + \bbeta_{10}(h_{11}  m_1 - h_{12} m_2 + h_{13} m_3 - h_{14} m_4)
\end{multline*}
As $h_{1k}, m_k$ range over $[\pm H], [\pm M]$ respectively, the bilinear form $h_{11}  m_1 - h_{12} m_2 + h_{13} m_3 - h_{14} m_4$ will range fairly uniformly over $[\pm 4 HM]$; however, this is not the case for the trilinear form $h_{11} h_{21} m_1 - h_{12} h_{22}m_2 + h_{13} h_{23}m_3 - h_{14} h_{24}m_4$. To uniformize the coefficient of $\bbeta_{11}$, we apply the Cauchy-Schwarz inequality two more times, each time only doubling the variables $h_{2k}, m_{k}$. This gives us
\begin{align*}
    \veps^{16} &\leq \lim_{M\to\infty}\E_{\substack{h_{1k_1}, h_{2k_1 k_2}\in[\pm H]:\\ k_1, k_2\in[4]}}\; \E_{\substack{m_{k_1 k_2}\in[\pm M]:\\ k_1, k_2\in[4]}}\int f\cdot T^{\sum_{k_1, k_2\in[4]}(-1)^{k_1+k_2}\bc(h_{1k_1}, h_{2k_1k_2})m_{k_1k_2}} \overline{f}\, d\mu.
\end{align*}
Since the range of all the variables $h_{1k_1}, h_{2k_1k_2}, m_{k_1k_2}$ is symmetric around the origin, we can make suitable choices of variables to remove the $(-1)^{k_1+k_2}$ factor, obtaining
\begin{align*}
    \veps^{16} &\leq \lim_{M\to\infty}\E_{\substack{h_{1k_1}, h_{2k_1 k_2}\in[\pm H]:\\ k_1, k_2\in[4]}}\; \E_{\substack{m_{k_1 k_2}\in[\pm M]:\\ k_1, k_2\in[4]}}\int f\cdot T^{\sum_{k_1, k_2\in[4]}\bc(h_{1k_1}, h_{2k_1k_2})m_{k_1k_2}} \overline{f}\, d\mu.
\end{align*}
The polynomial appearing in the average above takes the form
\begin{multline}\label{E: c 2}
    \sum_{k_1, k_2\in[4]}\bc(h_{1k_1}, h_{2k_1k_2})m_{k_1k_2}\\
    = \bbeta_{11}(h_{11}(h_{211}m_{11} + \cdots +h_{214}m_{14}) + \cdots + h_{14}(h_{241}m_{41} + \cdots + h_{244}m_{44}))\\
    + \bbeta_{10}(h_{11}(m_{11} + \cdots +m_{14}) + \cdots + h_{14}(m_{41} + \cdots + m_{44})).
\end{multline}
Fix $k_1\in[4]$. As $h_{2k_1 k_2}, m_{k_1 k_2}$ range over $[\pm H], [\pm M]$ respectively, it turns out that the expression
\begin{align}\label{E: vector of multilinear forms}
    \begin{pmatrix}
        h_{2k_11}m_{k_11} + h_{2k_1 2}m_{k_1 2} + h_{2k_1 3}m_{k_1 3} + h_{2k_1 4}m_{k_1 4}\\
        m_{k_11} + m_{k_1 2} + m_{k_1 3} + m_{k_1 4}
    \end{pmatrix}    
\end{align}
is fairly uniformly distributed inside $[\pm 4HM]\times [\pm 4M]$ in that any element of this set takes the form \eqref{E: vector of multilinear forms} for no more than $O(H^3 M^2)$ values $(h_{2k_1 k_2}, m_{k_1 k_2})_{k_2\in[4]}$. Morally, this allows us to replace \eqref{E: vector of multilinear forms} by $\begin{pmatrix}
    t_{k_1 1}\\ t_{k_1 0}
\end{pmatrix}$
ranging in $[\pm 4HM]\times [\pm 4M]$ at the cost of a loss of a factor $O(1)$ in the bound, yielding
\begin{align*}
    \veps^{16} &\ll \limsup_{M\to\infty}\E_{\substack{h_{1k_1}\in[\pm H]:\\ k_1\in[4]}}\; \E_{\substack{t_{k_1 1}\in[\pm 4HM]:\\ k_1 \in[4]}}\; \E_{\substack{t_{k_1 0}\in[\pm 4M]:\\ k_1 \in[4]}}\abs{\int f\cdot T^{\sum_{k_1\in[4]}\tilde{\bc}(h_{1k_1}, t_{k_11}, t_{k_1 0})} \overline{f}\, d\mu}
\end{align*}
for 
\begin{align*}
    \tilde{\bc}(h_{1k_1}, t_{k_11}, t_{k_1 0}) = \bbeta_{11} h_{1k_1} t_{k_1 1} + \bbeta_{10} h_{1k_1} t_{k_1 0}.
\end{align*}
We have thus reduced the trilinear polynomial $\bc(h_{1k_1}, h_{2k_1k_2})m_{k_1k_2}$ to the bilinear polynomial $\tilde{\bc}(h_{1k_1}, t_{k_11}, t_{k_1 0})$. The iterate then takes the form
\begin{align*}
    \sum_{k_1\in[4]}\tilde{\bc}(h_{1k_1}, t_{k_11}, t_{k_1 0}) = \bbeta_{11}(h_{11}t_{11}+\cdots + h_{14} t_{41}) + \bbeta_{10}(h_{11}t_{10} + \cdots + h_{14}t_{40}).
\end{align*}
A similar argument as before shows that as $h_{1k_1}, t_{k_1 1}, t_{k_1 0}$ range over $[\pm H], [\pm 4HM], [\pm 4M]$ respectively, the vectors
\begin{align}\label{E: vector of multilinear forms 2}
    \begin{pmatrix}
        h_{11}t_{11}+h_{12}t_{21}+h_{13}t_{31} + h_{14} t_{41}\\
        h_{11}t_{10}+h_{12}t_{20}+h_{13}t_{30} + h_{14} t_{40}        
    \end{pmatrix}
\end{align}
are uniformly distributed in their range, which now takes the form $[\pm 16H^2 M]\times[\pm 16HM]$ in that no element of this box equals \eqref{E: vector of multilinear forms 2} more than $O(1)$ times the expected number. This allows us to replace each vector \eqref{E: vector of multilinear forms 2} by a vector $\begin{pmatrix}
    t_{1}\\ t_{0}
\end{pmatrix}$ ranging in $[\pm 16H^2 M]\times[\pm 16HM]$, giving
\begin{align*}
\veps^{16} &\ll \limsup_{M\to\infty}\E_{\substack{t_{ 1}\in[\pm 16H^2M]}}\;\E_{\substack{t_{0}\in[\pm 16HM]}}\abs{\int f\cdot T^{\bbeta_{11}t_1 + \bbeta_{10}t_0} \overline{f}\, d\mu}.
\end{align*}
Squaring both sides and passing to the product system, we get 
\begin{align*}
    \veps^{32}&\ll \lim_{M\to\infty}\E_{\substack{t_{ 1}\in[\pm 16H^2M]}}\E_{\substack{t_{0}\in[\pm 16HM]}}\int f\otimes \overline{f}\cdot (T\times T)^{\bbeta_{11}t_1 + \bbeta_{10}t_0} \overline{f}\otimes f\, d(\mu\times\mu)\\
    &=\lim_{M\to\infty}\E_{\substack{t_0, t_1\in[\pm M]}}\int f\otimes \overline{f}\cdot (T\times T)^{\bbeta_{11}t_1 + \bbeta_{10}t_0} \overline{f}\otimes f\, d(\mu\times\mu);
\end{align*}
the fact that the limit along $[\pm 16H^2M]\times [\pm 16HM]$ exists and can be replaced by a limit along $[\pm M]^\times[\pm M]$ follows from Lemma~\ref{L: convergence}. It follows that
$$\nnorm{f}^+_{G} = \nnorm{f\otimes \overline{f}}_G^{1/2}\gg \veps^{8},$$
where $G = \langle \bbeta_{11}, \bbeta_{10}\rangle$, completing the concatenation step.

For longer averages involving higher degree polynomials, there will naturally be additional complications. Replacing the polynomials $\bc_1, \ldots, \bc_s$ by better distributed ones will be the subject of Proposition \ref{P: concatenation of polynomials I}. The relevant equidistribution results also become more complicated and are covered in Proposition~\ref{P: systems of multilinear equations}. Lastly, we will need to restrict the tuples of $\uh$ to those which are not too small and whose elements have pairwise small greatest common divisor. That such tuples are generic is guaranteed by Lemma~\ref{L: H_l}. However, the underlying idea behind dealing with the general case of \eqref{E: average of polynomial box seminorms} is similar: we first dramatically increase the number of variables  to make the polynomials better distributed, and then we use this well-distribution to replace all multilinear polynomials by single variables. 

\end{example}

\begin{proposition}\label{P: concatenation of polynomials I}
    Let $d, k, s_1, s_2\in\N$ and $\ell$ be a nonnegative integer power of 2. There exists a positive integer $t=O_{\ell, s_1, s_2}(1)$ such that for all polynomials $\bc_1, \ldots, \bc_{s_2}\in\R^k[\uh]$ of degree at most $d$ and coefficients $O(1)$ (where $\uh$ are vectors in $\Z^{s_{1}}$), numbers $H\in\N$, and 1-bounded functions $f\in L^\infty(\mu)$, we have
    \begin{multline}\label{E: more variables}
        \brac{\E\limits_{\uh\in[\pm H]^{s_1}}\nnorm{f}^+_{\bc_1(\uh), \ldots, \bc_{s_2}(\uh)}}^{O_{d, \ell, s_1, s_2}(1)}\\
        \ll_{d, k, \ell, s_1, s_2} \E_{\substack{h_{lk_1\cdots k_l}\in[\pm H]:\\ l\in[s_1],\; k_1, \ldots, k_{s_1}\in[t]}}
        \nnorm{f}^+_{\substack{\langle \bc_j(h_{1k_{1i_1}}, \ldots, h_{s_1 k_{1i_1}\cdots k_{s_1i_{s_1}}}):\; i_1, \ldots, i_{s_1}\in[\ell]\rangle:\\ j\in[s_2],\; 1\leq k_{l1} < \cdots < k_{l \ell}\leq t,\; l\in[s_1]}}.
    \end{multline}
\end{proposition}

Proposition \ref{P: concatenation of polynomials I} generalizes the first step of Example \ref{Ex: concatenation example} in which we increase the number of variables. In Example \ref{Ex: concatenation example}, the role of $\bc_1$ in the left-hand side of \eqref{E: more variables} is played  by \eqref{E: c 1}, and the subgroup in the right-hand side of \eqref{E: more variables} consists of expressions \eqref{E: c 2} with $m$'s ranging over $\Z$;
the bound \eqref{E: more variables} then holds with $s_1 = 2$ (since $\bc_1$ involves two parameters $h_1, h_2$), $s_2 = 1$ (corresponding to the degrees of the input box seminorms), $t = 4$ (corresponding to the range of the parameters $k_1, k_2$) and $\ell = 4$ (corresponding to adding four different $k_2$'s in \eqref{E: vector of multilinear forms} and four different $k_1$'s in \eqref{E: vector of multilinear forms 2}). 

\begin{proof}
Let 
\begin{align*}
    \veps = \E\limits_{\uh\in[\pm H]^{s_1}}\nnorm{f}^+_{\bc_1(\uh), \ldots, \bc_{s_2}(\uh)}.
\end{align*}
    The proof consists of multiple applications of Corollary \ref{C: iterated concatenation for general groups}. First, we apply it with the indexing set $\uh\in[\pm H]^{s_1}$ to find a positive integer $t_1 = O_{\ell, s_2}(1)$ such that
    \begin{align*}
        \E\limits_{\uh_1, \ldots, \uh_{t_1}\in[\pm H]^{s_1}}\nnorm{f}^+_{\substack{\langle\bc_j(\uh_{k_{11}}),\; \ldots,\; \bc_j(\uh_{k_{1 \ell}}) \rangle:\\ j\in[s_2],\; 1\leq k_{11} < \cdots < k_{1 \ell}\leq t_1}}\gg_{k, \ell, s_2} \veps^{O_{\ell, s_2}(1)}.
    \end{align*}

    In the second application of Corollary \ref{C: iterated concatenation for general groups}, we rewrite $\uh_k = (h_{1k}, \Tilde{\uh}_{1k})$, and we take the indexing set to be $(\tilde{\uh}_{11}, \ldots, \tilde{\uh}_{1 t_1})\in[\pm H]^{(s_1-1)t_1}$. Then Corollary \ref{C: iterated concatenation for general groups}, applied separately to each $(h_{11}, \ldots, h_{1t_1})$, gives a positive integer $t_2 = O_{\ell, s_2}(1)$ for which
    \begin{multline*}
        \E_{\substack{h_{1k_1}\in[\pm H]:\\ k_1\in[t_1]}}\;
        \E_{\substack{\Tilde{\uh}_{1k_1 k_2}\in[\pm H]^{s_1-1}:\\ k_1\in [t_1],\; k_2\in[t_2]}}
        \nnorm{f}^+_{\substack{\langle \bc_j(h_{1k_{11}}, \Tilde{\uh}_{1k_{11}k_{21}}),\; \ldots,\; \bc_j(h_{1k_{11}}, \Tilde{\uh}_{1k_{11}k_{2\ell}}),\\
        \ldots, \bc_j(h_{1k_{1\ell}}, \Tilde{\uh}_{1k_{1\ell}k_{21}}),\; \ldots,\; \bc_j(h_{1k_{1\ell}}, \Tilde{\uh}_{1k_{1\ell}k_{2\ell}})\rangle:\\ j\in[s],\; 1\leq k_{11} < \cdots < k_{1 \ell}\leq t_1,\; 1\leq k_{21} < \cdots < k_{2 \ell}\leq t_2 }} \gg_{k, \ell, s_2} \veps^{O_{\ell, s_2}(1)}.
    \end{multline*}

    At the subsequent step, we similarly write $\Tilde{\uh}_{1k_1k_2}=(h_{2k_1k_2}, \Tilde{h}_{2k_1k_2}).$ Taking $$(\tilde{h}_{2k_1k_2})_{k_1\in[t_1], k_2\in[t_2]}\in[\pm H]^{(s_1-2)t_1 t_2}$$ as the indexing set, we apply Corollary \ref{C: iterated concatenation for general groups} separately to each $(h_{1k_1}, h_{2k_1k_2})_{k_1, k_2}$, obtaining a positive integer $t_3 = O_{\ell, s_2}(1)$ for which
    \begin{multline*}
        \E_{\substack{h_{1k_1}\in[\pm H]:\\ k_1\in[t_1]}}\; \E_{\substack{h_{1k_1 k_2}\in[\pm H]:\\ k_1\in [t_1],\; k_2\in[t_2]}}\;
        \E_{\substack{\Tilde{\uh}_{2k_1 k_2 k_3}\in[\pm H]^{s_1-2}:\\ k_1\in [t_1],\; k_2\in[t_2],\; k_3\in[t_3]}}\\
        \nnorm{f}^+_{\substack{\langle\bc_j(h_{1k_{1i_1}}, h_{2k_{1i_1}k_{2i_2}}, \Tilde{\uh}_{2k_{1i_1}k_{2i_2}k_{3i_3}}):\; i_1, i_2, i_3\in[\ell]\rangle:\\ j\in[s_2],\; 1\leq k_{l1} < \cdots < k_{l \ell}\leq t_l,\; l\in[3]}} \gg_{k, \ell, s_2} \veps^{O_{\ell, s_2}(1)}.
    \end{multline*}
    Continuing in this fashion $s_1-3$ more times, we find positive integers $t_4, \ldots, t_{s_1}=O_{\ell, s_1, s_2}(1)$ such that 
        \begin{align*}
        \E_{\substack{h_{lk_1\cdots k_l}\in[\pm H]:\\ l\in[s_1],\; k_1\in[t_1],\; \ldots,\; k_{s_1}\in[t_{s_1}]}}
        \nnorm{f}^+_{\substack{\langle \bc_j(h_{1k_{1i_1}}, \ldots, h_{s_1k_{1i_1}\cdots k_{s_1i_{s_1}}}):\; i_1, \ldots, i_{s_1}\in[\ell]\rangle:\\ j\in[s_2],\; 1\leq k_{l1} < \cdots < k_{l \ell}\leq t_l,\; l\in[s_1]}}\gg_{k, \ell, s_1, s_2} \veps^{O_{\ell, s_1, s_2}(1)}.
    \end{align*}
    We could just end here, however, it is notationally convenient if we can deal with only one $t_l$. Setting $t=\max(t_1, \ldots, t_{s_1})$, we deduce the result from the monotonicity property of the seminorms.
\end{proof} 

Each group $\langle \bc_j(h_{1k_{1i_1}}, \ldots, h_{s_1 k_{1i_1}\cdots k_{s_1i_{s_1}}})\colon\; i_1, \ldots, i_{s_1}\in[\ell]\rangle$ in the left-hand side of \eqref{E: more variables} consists of multilinear forms like \eqref{E: c 2}. Importantly, as long as $\ell$ is sufficiently large ($\ell\geq 3$ will do), the  system of multilinear forms that comes from taking all the groups in the right hand side of \eqref{E: more variables} is equidistributed in its range. This statement is made precise by the following proposition.

\begin{proposition}[Equidistribution of systems of multilinear forms, {\cite[Proposition~7.5]{KKL24a}}]\label{P: systems of multilinear equations}
    Let $\veps>0$, $\ell, s_1, s_2, t\in\N$ with $3\leq\ell\leq t$, and let $H, M\in\N$ with $$\veps^{-O_{\ell, s_1, s_2, t}(1)}\ll_{\ell, s_1, s_2, t} H\leq M.$$ Let 
    \begin{align}\label{E: K}
        \CK = \{(k_{li})_{\substack{(l,i)\in[s_1]\times[\ell]}}\in[t]^{s_1\ell}:\; 1\leq k_{li}<\cdots < k_{l\ell}\leq t\;\; \textrm{for\; all}\;\; l\in[s_1]\},
    \end{align}
    and for each $l\in[s_1]$, let also
    \begin{multline}\label{E: H_l}
        \CH_{l,\veps} = \{(h_{lk_1\cdots k_l})_{k_1, \ldots, k_l\in[t]}\in[\pm H]^{t^l}:\; |h_{lk_1\cdots k_l}-h_{lk''_1\cdots k''_l}|\geq \veps H,\\ \gcd(h_{lk_1\cdots k_l}-h_{lk''_1\cdots k''_l}, h_{lk'_1\cdots k'_l}-h_{lk''_1\cdots k''_l})\leq \veps\inv\\
    \textrm{for\; distinct}\; (k_1, \ldots, k_l),\; (k'_1, \ldots, k'_l),\; (k''_1, \ldots, k''_l)\in[t]^l\}.
    \end{multline}
    Then 
        \begin{multline*}
\max_{\substack{n_{j\uk\uu}\in\Z:\; j\in[s],\\ \uk\in\CK,\; \uu\in\{0,1\}^{s_1}}}\;\E_{\substack{m_{j\uk i_1\cdots i_{s_1}}\in[\pm M]:\\ j\in[s_2],\; \uk\in\CK,\; i_1, \ldots,  i_{s_1}\in[\ell]}}\;
                \E_{\substack{h_{lk_{1}\cdots k_{l}}\in[\pm H]:\\  k_1, \ldots, k_{s_1}\in[t],\; l\in[s_1]}}\; \brac{\prod_{l\in[s_1]}1_{\CH_{l, \veps}}((h_{lk_{1}\cdots k_{l}})_{k_1, \ldots, k_l})}\\
        \prod_{j\in[s_2]}\; \prod_{\substack{\uk\in\CK}}\prod_{\substack{\uu\in\{0,1\}^{s_1}
        }} 1\Bigbrac{\sum\limits_{i_1, \ldots, i_{s_1}\in[\ell]} h_{1k_{1i_1}}^{u_1} \cdots h_{s_1k_{1i_1}\cdots k_{s_1i_{s_1}}}^{u_{s_1}}m_{j\uk i_1\cdots i_{s_1}} = n_{j\uk\uu}}\\
        \ll_{\ell, s_1, s_2, t} \veps^{-O_{\ell, s_1, s_2, t}(1)}  M^{-2^{s_1}s_2|\CK|}H^{-s_1 2^{s_1-1}s_2|\CK|},
    \end{multline*}
    where we recall our convention for the indicator function from Section \ref{S: background}.

\end{proposition}
On the first read, we advise the reader to assume that $t=\ell$, so that the indexing set $\CK$ consists of only one element, $k_{li_l}$ can be replaced by $k_l$, and the product $\prod_{\substack{\uk\in\CK}}$ disappears.

The point of restricting the range of $h$'s to \eqref{E: H_l} is that multilinear forms like $(h-h'')(h'-h'')$ that appear in the proof of Proposition \ref{P: systems of multilinear equations} are only well distributed in its range if $h-h'', h'-h''$ are not too small and do not have large gcd. The next lemma shows that the set of such well-behaved $h$'s is generic.
\begin{lemma}[{\cite[Lemma 7.8]{KKL24a}}]\label{L: H_l}
    Let $\veps>0$, $H, l, t\in\N$, and $\CH_{l,\veps}$ be as in \eqref{E: H_l}. Then all but a $O_{l,t}(\veps)$ proportion of elements of $[\pm H]^{t^l}$ lie in $\CH_{l, \veps}$.
\end{lemma}

In what follows, we let $A_{d,s_1} := \left\{\uu\in\N_0^{s_1}:\ |\uu|\leq d \right\}$. The next result establishes that averages of generalized box seminorms in which boxes are generated by multilinear polynomials $\bc_1, \ldots, \bc_{s_2}$ can be controlled with polynomial losses by a single generalized box seminorm of bounded degree along subgroups spanned by the coefficients of the polynomials $\bc_1, \ldots, \bc_{s_2}$. 
    \begin{proposition}[Quantitative concatenation of box seminorms along polynomials]\label{P: polynomial concatenation}
        Let $d, k, s_1, s_2\in\N$. There exists a positive integer $s_3=O_{d,s_1, s_2}(1)$ with the following property: for all systems $(X, \CX, \mu, T_1, \ldots, T_k)$, 1-bounded functions $f\in L^\infty(\mu)$, and polynomials $\bc_1, \ldots, \bc_{s_2}\in\R^k[\uh]$ 
        (where $\uh$ are vectors in $\Z^{s_{1}}$)
        of the form
        \begin{align*}
            \bc_j(\uh) =  \sum_{\substack{\uu\in A_{d, s_1}}}\bv_{j\uu}\uh^{\uu}
        \end{align*}
        for some $\bv_{j\uu}\in\R^{k}$,
        we have
        \begin{align*}
            \limsup_{H\to\infty}\brac{\E\limits_{\uh\in[\pm H]^{s_1}}\nnorm{f}^+_{{\bc_1(\uh)}, \ldots, {\bc_{s_2}(\uh)}}}^{O_{s_1,s_2}(1)} \ll_{k, s_1, s_2} \nnorm{f}^+_{G_1^{s_3}, \ldots, G_{s_2}^{s_3}},
        \end{align*}
        where $G_j = \langle \bv_{j\uu}:\ \uu\in A_{d,s_1}\rangle$ for $j\in[s_2]$.
    \end{proposition}
    \begin{proof}
       Set
        \begin{align*}
            \veps = \E\limits_{\uh\in[\pm H]^{s_1}}\nnorm{f}_{{\bc_1(\uh)}, \ldots, {\bc_{s_2}(\uh)}}^+.
        \end{align*}
               We let all the constants depend on $k, s_1, s_2$, noting however that powers of $\veps$ do not depend on $k$.
    By Proposition \ref{P: concatenation of polynomials I} applied with $\ell = 4$,\footnote{The reason for taking $\ell = 4$ is as follows: 4 is the smallest power of 2 greater than or equal to 3, and we have a lower bound $\ell\geq 3$ in Proposition \ref{P: systems of multilinear equations} that we will apply shortly.}  there exists a positive integer $t=O(1)$ (independent of $k$) for which
    \begin{align}\label{E: unexpanded}
        \E_{\substack{h_{lk_1\cdots k_l}\in[\pm H]:\\ l\in[s_1],\; k_1, \ldots, k_{s_1}\in[t]}}
        \nnorm{f}_{\substack{\langle \bc_j(h_{1k_{1i_1}}, \ldots, h_{s_1k_{1i_1}\cdots k_{s_1i_{s_1}}}):\; i_1, \ldots, i_{s_1}\in[4]\rangle:\\ j\in[s_2],\;        \uk\in\CK}} \gg \veps^{O(1)};
    \end{align}
    we have also applied here Lemma \ref{L: plus vs normal seminorm} to replace $\nnorm{\cdot}^+$ with $\nnorm{\cdot}$, which we can by increasing $t$ by 1, say.
    For each choice of $j\in[s_2], \uk\in\CK$ and $h$'s, the group $$\langle \bc_j(h_{1k_{1i_1}}, \ldots, h_{s_1k_{1i_1}\cdots k_{s_1i_{s_1}}}):\; i_1, \ldots, i_{s_1}\in[4]\rangle$$ consists of elements
    \begin{align}\label{E: expandeddd}
        \sum_{\uu\in\{0,1\}^{s_1}} \bv_{j\uu} \sum\limits_{i_1, \ldots, i_{s_1}\in[4]} h_{1k_{1i_1}}^{u_1}\cdots h_{s_1 k_{1i_1}\cdots k_{s_1i_{s_1}}}^{u_{s_1}} m_{j\uk i_1\cdots i_{s_1}}
    \end{align}
    with $m$'s ranging over $\Z$, and so applying the H\"older inequality to \eqref{E: unexpanded} and expanding the definition of the seminorm, we get
    \begin{multline*}
        \lim_{M\to\infty}\E_{\substack{h_{lk_1\cdots k_l}\in[\pm H]:\\ l\in[s_1],\; k_1, \ldots, k_{s_1}\in[t]}}\;\E_{\substack{m_{j\uk i_1\cdots i_{s_1}}, m_{j\uk i_1\cdots i_{s_1}}'\in[\pm M]:\\ j\in[s_2],\; \uk\in\CK,\; i_1, \ldots,  i_{s_1}\in[4]}}\\ 
        \int \Delta_{\substack{\big(\sum\limits_{i_1, \ldots, i_{s_1}\in[4]}\bc_j(h_{1k_{1i_1}}, \ldots, h_{s_1 k_{1i_1}\cdots k_{s_1 i_{s_1}}})m_{j \uk i_1\cdots i_{s_1}},\\
        \sum\limits_{i_1, \ldots, i_{s_1}\in[4]}\bc_j(h_{1k_{1i_1}}, \ldots, h_{s_1 k_{1i_1}\cdots k_{s_1 i_{s_1}}})m_{j \uk i_1\cdots i_{s_1}}'\big):\; j\in[s_2],\; \uk\in\CK}}\; f\; d\mu \gg \veps^{O(1)}.
    \end{multline*}    
    To simplify the expression above, we compose the integral with
    \begin{align*}
        T^{-\sum_{j\in[s_2]}\sum_{\uk\in\CK}\floor{\sum\limits_{i_1, \ldots, i_{s_1}\in[4]}\bc_j(h_{1k_{1i_1}}, \ldots, h_{s_1 k_{1i_1}\cdots k_{s_1 i_{s_1}}})m_{j \uk i_1\cdots i_{s_1}}'}}.
    \end{align*}
    Swapping the order of subtraction and integer parts, replacing $m_{j \uk i_1\cdots i_{s_1}}- m_{j \uk i_1\cdots i_{s_1}}'$ by $m_{j \uk i_1\cdots i_{s_1}}$, dealing with the error terms using Lemma \ref{L: errors}, and using the triangle inequality to get rid of resulting weights, we get
      \begin{multline}\label{E: expanded}
        \lim_{M\to\infty}\E_{\substack{h_{lk_1\cdots k_l}\in[\pm H]:\\ l\in[s_1],\; k_1, \ldots, k_{s_1}\in[t]}}\;\E_{\substack{m_{j\uk i_1\cdots i_{s_1}}\in[\pm 2M]:\\ j\in[s_2],\; \uk\in\CK,\; i_1, \ldots,  i_{s_1}\in[4]}}\\
        \abs{\int \prod_{\ueps\in\{0,1\}^{s_2|\CK|}}\CC^{|\ueps|}T^{\eps_{j\uk}\cdot\floor{\sum\limits_{i_1, \ldots, i_{s_1}\in[4]}\bc_j(h_{1k_{1i_1}}, \ldots, h_{s_1 k_{1i_1}\cdots k_{s_1 i_{s_1}}})m_{j \uk i_1\cdots i_{s_1}}}+\br_{\ueps}}f\; d\mu}\gg \veps^{O(1)}
    \end{multline}      
    for some $\br_\ueps\in\Z^k$.

    Our goal is to use Proposition \ref{P: systems of multilinear equations} to show that the gigantic system of multilinear forms
    \begin{align}\label{E: multilinear form}
        \sum\limits_{i_1, \ldots, i_{s_1}\in[4]} h_{1k_{1i_1}}^{u_1}\cdots h_{s_1 k_{1i_1}\cdots k_{s_1i_{s_1}}}^{u_{s_1}} m_{j\uk i_1\cdots i_{s_1}}
    \end{align}
    indexed by $j\in[s_2],\; \uk\in\CK,\; \uu\in\{0,1\}^{s_1}$ that appears in \eqref{E: expanded} after substituting $\bc_j(h_{1k_{1i_1}}, \ldots, h_{s_1 k_{1i_1}\cdots k_{s_1 i_{s_1}}})$ by an element of the form (\ref{E: expandeddd}) 
    is equidistributed in its range, in that we can replace each \eqref{E: multilinear form}  by a separate variable $n_{j\uk\uu}$ ranging over $[\pm C_{\uu}H^{|\uu|}M]$ for some $C_\uu>0$. The first step to achieve this is to throw out those tuples of $h$'s whose coordinates are either close to each other or have a large greatest common divisor.
    For some small $\veps'>0$ to be chosen later, let $\CH_l = \CH_{l,\veps'}$ be the set defined in \eqref{E: H_l}. By Lemma~\ref{L: H_l}, the set $\CH_{l,\veps'}$ contains all but a $O(\veps')$ proportion of elements of $[\pm H]^{t^l}$; that means that generically, the differences of the numbers $h_{lk_1 \cdots k_l}$ are far from 0 and have small greatest common divisor. Taking $\veps' = c\veps^{1/c}$ for a sufficiently small $c>0$, we can ignore in \eqref{E: expanded} all the tuples $(h_{l k_1\cdots k_l})_{k_1, \ldots, k_l}$ not in $\CH_l$. Thus, we obtain 
    \begin{multline*}
        \limsup_{M\to\infty}\E_{\substack{h_{lk_1\cdots k_l}\in[\pm H]:\\ l\in[s_1],\; k_1, \ldots, k_{s_1}\in[t]}}\; {{ \prod_{l\in[s_1]}1_{\CH_l}((h_{l k_1\cdots k_l})_{k_1, \ldots, k_l})}}\cdot\E_{\substack{m_{j\uk i_1\cdots i_{s_1}}\in[\pm M]:\\ j\in[s_2],\; \uk\in\CK,\; i_1, \ldots,  i_{s_1}\in[4]}}\\ 
        \abs{\int \prod_{\ueps\in\{0,1\}^{s_2|\CK|}}\CC^{|\ueps|}T^{\eps_{j\uk}\cdot\floor{\sum\limits_{i_1, \ldots, i_{s_1}\in[4]}\bc_j(h_{1k_{1i_1}}, \ldots, h_{s_1 k_{1i_1}\cdots k_{s_1 i_{s_1}}})m_{j \uk i_1\cdots i_{s_1}}}+\br_{\ueps}}f\; d\mu}\gg \veps^{O(1)}.
    \end{multline*}
    The introduction of extra summations over $n_{j\uk\uu}$ for $j\in[s_2], \uk\in\CK, \uu\in\{0,1\}^{s_1}$ allows us to bound
    \begin{multline*}
        \limsup_{M\to\infty} \sum_{\substack{n_{j\uk\uu}\ll H^{|\uu|}M:\\ j\in[s_2],\; \uk\in\CK,\; \uu\in\{0,1\}^{s_1}}}\CN(\un)\\
        \abs{\int \prod_{\ueps\in\{0,1\}^{s_2|\CK|}}\CC^{|\ueps|}T^{\eps_{j\uk}\cdot\floor{\sum_{\substack{\uu\in\{0,1\}^{s_1}}}\bv_{j\uu} n_{j\uk\uu}}+\br_{\ueps}}f\; d\mu}\gg \veps^{O(1)},
    \end{multline*}
    for 
    \begin{multline*}
        \CN(\un) = \E_{\substack{m_{j\uk i_1\cdots i_{s_1}}\in[\pm M]:\\ j\in[s_2],\; \uk\in\CK,\; i_1, \ldots,  i_{s_1}\in[\ell]}} \; 
        \E_{\substack{h_{lk_{1}\cdots k_{l}}\in[\pm H]:\\  k_1, \ldots, k_l\in[t],\; l\in[s_1]}}\; {\prod_{l\in[s_1]}1_{\CH_l}((h_{l k_1\cdots k_l})_{k_1, \ldots, k_l})} \\
        \prod_{j\in[s]}\; \prod_{\substack{\uk\in\CK}}\prod_{\substack{\uu\in\{0,1\}^{s_1}
        }} 1\Bigbrac{\sum\limits_{i_1, \ldots, i_{s_1}\in[\ell]} h_{1k_{1i_1}}^{u_1} \cdots h_{s_1 k_{1i_1}\cdots k_{s_1 i_{s_1}}}^{u_{s_1}}m_{j\uk i_1\cdots i_{s_1}} = n_{j\uk\uu}}.
    \end{multline*}
    The kernel $\CN$ is supported on the tuples $(n_{j\uk\uu})_{\substack{j\in[s_2],\; \uk\in\CK,\\ \uu\in\{0,1\}^{s_1}}}$ satisfying $n_{j\uk\uu}\ll H^{|\uu|} M$, of which there are 
    \begin{align*}
        \ll \prod_{j\in[s_2]}\prod_{\uk\in\CK} \prod_{\uu\in\{0,1\}^{s_1}} H^{|\uu|} M \ll  M^{2^{s_1} s_2 |\CK|} H^{s_1 2^{s_1-1} s_2 |\CK|} 
    \end{align*}
    due to the identity
        \begin{align*}
        \sum_{u_{1} = 0}^1\cdots \sum_{u_{s_1} = 0}^1 (u_{1}+\cdots + u_{s_1}) = \sum_{i=0}^{s_1} i {\binom{s_1}{i}} = s_12^{s_1-1}
    \end{align*}
    that can easily be proved by induction.
    
    Crucially, Proposition \ref{P: systems of multilinear equations} allows us to bound
    \begin{align*}
        \CN(\un)\ll \veps^{-O(1)}  M^{- 2^{s_1} s_2 |\CK|} H^{-s_1 2^{s_1-1} s_2 |\CK|}
    \end{align*}
    as long as $M\geq H\gg\veps^{-O(1)}$,
    and so we get the estimate
    \begin{align*}
        \limsup_{M\to\infty}\E_{\substack{n_{j\uk\uu}\ll H^{|\uu|}M:\\ j\in[s_2],\; \uk\in\CK,\; \uu\in\{0,1\}^{s_1}
        }}\abs{\int \Delta_{\sum_{\substack{\uu\in\{0,1\}^{s_1}
        }}\bv_{j\uu} n_{j\uk\uu}:\; j\in[s_2],\; \uk\in\CK}\; f\, d\mu}\gg \veps^{O(1)}.
    \end{align*}
    The claimed inequality promptly follows upon taking $M\to\infty$, followed by taking $H\to\infty$.\footnote{We note here that taking $H\to\infty$ is not strictly necessary; a close analysis of the argument shows that if
    \begin{align*}
        \E\limits_{\uh\in[\pm H]^{s_1}}\nnorm{f}^+_{{\bc_1(\uh)}, \ldots, {\bc_{s_2}(\uh)}}\geq \veps,
    \end{align*}
    then
    \begin{align*}
        \brac{\E\limits_{\uh\in[\pm H]^{s_1}}\nnorm{f}^+_{{\bc_1(\uh)}, \ldots, {\bc_{s_2}(\uh)}}}^{O_{s_1,s_2}(1)} \ll_{k, s_1, s_2} \nnorm{f}^+_{G_1^{s_3}, \ldots, G_{s_2}^{s_3}},
    \end{align*} as long as $H\gg_{k, s_1, s_2} \veps^{-O_{s_1, s_2}(1)}.$} 
    \end{proof}

Propositions \ref{P: PET I} and  \ref{P: polynomial concatenation} together give the quantitative control of polynomial averages by a bounded-degree generalized box seminorm that depends only on the leading coefficients of the polynomials and their differences.

\begin{proposition}[Box seminorm control for polynomial with dual twists]\label{P: Box seminorm control for polynomials twisted by duals}
    Let $d, k, \ell\in \N$, $J\in\N_0$. There exists a positive integer $s=O_{d,J,\ell}(1)$ with the following property: for all $n_0\in\N_0$, systems $(X, \CX, \mu, T_1, \ldots, T_k)$, 1-bounded functions $f_1\in L^\infty(\mu)$, Hardy sequences $b_1, \ldots, b_J\in\CH$ of growth $b_j(t)\ll t^d$, and polynomials $\bp_1, \ldots, \bp_\ell\in\R^k[n]$  of degrees at most $d$, form $\bp_j(n) = \sum_{i=0}^d \bbeta_{ji} n^i$, and such that $\bp_1$ is essentially distinct from $\bp_2, \ldots, \bp_\ell$, we have
    \begin{multline*}
        \limsup_{N\to\infty}\sup_{n_0\in\N_0}\sup_{\substack{\norm{f_2}_\infty, \ldots, \norm{f_\ell}_\infty\leq 1,\\ \CD_1, \ldots, \CD_J\in\FD_d}}
        \norm{\E_{n\in[N]}\prod_{j\in[\ell]} T^{\floor{\bp_j(n_0+n)}}f_j\cdot \prod_{j\in[J]}\CD_j(\floor{b_j(n_0+n)})}_{L^2(\mu)}^{O_{d, \ell, J}(1)}\\
        \ll_{d, J, k, \ell}  \nnorm{f_1}^+_{\bbeta_{1d_{10}}^s, (\bbeta_{1d_{12}} - \bbeta_{2d_{12}})^s, \ldots, (\bbeta_{1d_{1\ell}} - \bbeta_{\ell d_{1\ell}})^s},
    \end{multline*}
    where $d_{1j} = \deg (\p_1 -\p_j)$ for $j\in[0, \ell]$ and $\p_0 = {\bf 0}$.
\end{proposition}

\begin{proof}
Let 
\begin{multline*}
    \veps = \limsup_{N\to\infty}\sup_{n_0\in\N_0}\sup_{\substack{\norm{f_2}_\infty, \ldots, \norm{f_\ell}_\infty\leq 1,\\ \CD_1, \ldots, \CD_J\in\FD_d}}\sup_{|c_n|\leq 1}\\
        \norm{\E_{n\in[N]}\prod_{j\in[\ell]} T^{\floor{\bp_j(n_0+n)}}f_j\cdot \prod_{j\in[J]}\CD_j(\floor{b_j(n_0+n)})}_{L^2(\mu)}.
\end{multline*}
Corollary \ref{C: PET for polynomials twisted by duals} implies that 
    \begin{align*}
        \limsup_{H\to\infty}\E\limits_{\uh\in[\pm H]^{s_1}} \nnorm{f_1}_{\bc_1(\uh), \ldots, \bc_{s_2}(\uh)}^+ \gg_{d, J, k, \ell} \veps^{O_{d, \ell, J}(1)}
    \end{align*}
    for some $s_1,s_2 = O_{d, \ell, J}(1)$ (independent of $k$) and nonzero polynomials $\bc_1, \ldots, \bc_{s_2}\in\R^k[\uh]$ that depend only on $\bp_1, \ldots, \bp_\ell$, $d,J,\ell$ and take the form
    \begin{align*}
	\bc_{j}(\uh) = \sum_{\substack{\uu\in \{0,1\}^{s_1},\\ |\uu|\leq d-1}} (|\uu|+1)! \cdot (\bbeta_{1(|\uu|+1)}-\bbeta_{w_{j\uu}(|\uu|+1)})\uh^\uu
    \end{align*}
    for some indices $w_{j\uu}\in[0,\ell]$, where all nonzero vectors $\bbeta_{1(|\uu|+1)}-\bbeta_{w_{j\uu}(|\uu|+1)}$ are the leading coefficients of the polynomials $\p_1 - \p_{w_{j\uu}}$, and so they belong to the family
    \begin{align*}
        \bbeta_{1d_{10}},\; \bbeta_{1d_{12}} - \bbeta_{2d_{12}},\; \ldots,\; \bbeta_{1d_{1\ell}} - \bbeta_{\ell d_{1\ell}}.
    \end{align*}
    The claim follows from Proposition \ref{P: polynomial concatenation} as well as Lemma \ref{L: seminorms of subgroups}; the latter allows us to ignore the scaling factors $(|\uu|+1)!$.
\end{proof}

\section{Seminorms estimates in the polynomial + sublinear case}\label{S: polynomial + sublinear}
    The next case that we need to deal with is when our Hardy sequences are sums of sublinear and polynomial functions. 
  \begin{proposition}\label{P: sublinear + polynomial}
			Let $d, k, \ell, m\in\N$, $0\leq m_1\leq m$ be an integer,  $g_1,...,g_m\in \mathcal{H}$ be Hardy functions satisfying the growth conditions
   \begin{enumerate}
       \item $1 \prec g_1(t) \prec \cdots  \prec g_m(t) \prec t$;
       \item $g_i(t)\ll \log t$ iff $i\in[m_1]$. 
   \end{enumerate}
    For $j\in[\ell]$, let
     \begin{gather*}
     \ba_j(n) = \bu_j(n) + \bp_j(n) +\br_j(n)\in\CH^k\\
\textrm{for}\quad \bu_j(n) = \sum_{i\in[m]}\balpha_{ji} g_i,\quad\p_j(n) = \sum_{i=0}^d \bbeta_{ji} n^i\in\R^k[n]\quad \textrm{and}\quad \lim_{n\to\infty}\br_j(n)\to\mathbf{0}
     \end{gather*}
     for some $\balpha_{ji},\bbeta_{ji}\in\R^{k}$.
   Lastly, let
   \begin{align*}
        d_{jj'} &= \max\{i\in[d]:\; \bbeta_{ji}\neq \bbeta_{j'i}\}\\ 
        d'_{jj'} &= \max\{i\in[m]:\; \balpha_{ji}\neq \balpha_{j'i}\},
   \end{align*}
   and assume that $d_{1j}>0$ or $d'_{1j}>m_1$ for all $j\in[0,\ell]\setminus\{1\}$, where $\ba_0 = \bu_0=\bp_0 =\mathbf{0}$ (in other words, some coordinate of $\ba_1 -\ba_{j}$ grows faster than log).
   Then there exist a positive integer $s=O_{d, \ell, m}(1)$ and vectors
   \begin{align}\label{E: set of coeffs poly + sub}
       \bgamma_1, \ldots, \bgamma_s\in &\{\bbeta_{1d_{1j}}-\bbeta_{jd_{1j}}:\; j\in[0,\ell]\; \textrm{with}\; d_{1j}>0\}\\
       \nonumber\cup &\{\balpha_{1d'_{1j}}-\balpha_{jd'_{1j}}:\; j\in[0,\ell]\; \textrm{with}\; d_{1j} = 0\; \textrm{and}\; d'_{1j}>m_1\} 
   \end{align}
   such that for all systems $(X, \CX, \mu, T_1, \ldots, T_k)$ and 1-bounded functions $f_1\in L^\infty(\mu)$, we have
			\begin{multline*}
				\limsup_{N\to\infty}\sup_{\substack{\norm{f_2}_\infty, \ldots, \norm{f_\ell}_\infty\leq 1}}\sup_{|c_n|\leq 1}\norm{\E_{n\in[N]}c_n\cdot \prod_{j\in[\ell]} T^{\floor{\ba_j(n)}}f_j}_{L^2(\mu)}^{O_{d, \ell, m}(1)}
    \ll_{d, k, \ell, m} \nnorm{f_1}_{\bgamma_1, \ldots, \bgamma_s}^+.
			\end{multline*}
		\end{proposition}

The vectors in \eqref{E: set of coeffs poly + sub} are the leading coefficients of $\ba_1, \ba_1 - \ba_2, \ldots, \ba_1 - \ba_\ell$ with regards to the ordering of the generators of $\ba_1, \ldots, \ba_\ell$ as
\begin{align*}
         g_1(t),\; \ldots,\; g_{m}(t),\; t,\; t^2,\; \ldots,\; t^d
\end{align*}
which matches the growth rate of these generators.

The proof of Proposition \ref{P: sublinear + polynomial} follows a similar strategy as the proof of the analogous result in the single transformation case \cite[Proposition 5.1]{Ts22}. That is, we pass to short intervals in such a way that the sublinear terms are constant on each short interval. We apply the arguments for polynomials from Section \ref{S: polynomial PET} to handle the polynomial average over $n$ running over short intervals of length $L(N)$; then we use the results for sublinear functions from Section~\ref{S: linear and sublinear} to handle the average over $N$ that contains the sublinear parts of the original sequences. At the end, we apply the concatenation results from Section~\ref{S: polynomial concatenation} to concatenate the average of generalized box seminorms that arose at the earlier steps of the argument. This last step differs from the conclusion of the proof of \cite[Proposition~5.1]{Ts22} in which no concatenation was required.

    \begin{proof}
     Let 
    \begin{align*}
        \veps = \limsup_{N\to\infty}\sup_{\substack{\norm{f_2}_\infty, \ldots, \norm{f_\ell}_\infty\leq 1}}\sup_{|c_n|\leq 1}\norm{\E_{n\in[N]}c_n\cdot \prod_{j\in[\ell]} T^{\floor{\ba_j(n)}}f_j}_{L^2(\mu)}.
    \end{align*}
    We allow all the quantities below to depend on $d, k, \ell, m$ (however, powers of $\veps$ and the integers $s_i$ will stay independent of $k$). 

    If all the polynomials $\bp_1, \ldots, \bp_\ell$ are constant, then the result follows from Proposition~\ref{P: sublinear}. Suppose therefore that at least one of the polynomials is nonconstant. We assume that $\bp_1$ is nonconstant; if this is not the case and, say, $\p_\ell$ is nonconstant, then we can reduce to the previous case by composing the integral with $T^{-\floor{\p_\ell(n)}}$ and arguing as at the end of Proposition~\ref{P: PET I}.

    Fix functions $f_{jN}$ (with $f_{1N} = f_1$) that realize the limsup above. We then use Lemma~\ref{L: double averaging} to subdivide $\N$ into short intervals, obtaining
    \begin{align*}
        \limsup_{R\to\infty}
        \sup_{|c_n|\leq 1} \E\limits_{N\in[R]} 
        \norm{\E_{n\in[L(N)]}c_n\cdot \prod_{j\in[\ell]}T^{\floor{\ba_j(N+n)}}f_{jR}}_{L^2(\mu)}\geq \veps.
    \end{align*}    
    We choose the length $L$ of the short intervals in such a way that the sublinear parts $\bu_1, \ldots, \bu_\ell$ are approximately constant on each short interval. To achieve this, we pick $L(N)\prec {g'_m(N)}^{-1}$; with that choice, we have
    \begin{align*}
       \lim_{N\to\infty}\max_{n\in[L(N)]}|\ba_j(N+n) - (\bu_j(N) + \bp_j(N+n))| =0
    \end{align*}
    for every $j\in[\ell].$
    By Lemma \ref{L: errors} and the pigeonhole principle, there exist $\br_{j}\in\Z^k$ such that
    \begin{align*}
        \limsup_{R\to\infty}
        \E\limits_{N\in[R]}\sup_{|c_n|\leq 1}
        \norm{\E_{n\in[L(N)]}c_n\cdot \prod_{j\in[\ell]}T^{\floor{\bp_j(N+n)}}\brac{T^{\floor{\bu_j(N)}+\br_{j}}f_{jR}}}_{L^2(\mu)}
        \gg \veps.
    \end{align*}

   Rearranging if necessary, we can assume that $\bp_1, \ldots, \bp_{\ell_0}$ are identical and distinct from $\bp_{\ell_0+1}, \ldots, \bp_{\ell}$ for some $\ell_{0}\in[\ell]$. 

 Applying Proposition \ref{P: PET I} to the average over $n$, we obtain positive integers $s_1, s_2=O(1)$ (independent of $k$) such that 
\begin{align*}
    \limsup_{R\to\infty}\E\limits_{N\in[R]}\E\limits_{\uh\in[\pm H]^{s_1}}\E_{\um, \um'\in[\pm M]^{s_2}} \abs{\int \Delta_{\bc_1(\uh), \ldots, \bc_{s_2}(\uh); (\um,\um')} F_{RN}\, d\mu} 
    \gg\veps^{O(1)}
\end{align*}
for  some polynomials $\bc_1, \ldots, \bc_{s_2}\in\R^k[\uh]$ satisfying the conditions of Proposition \ref{P: PET I} and
    \begin{align*}
        F_{RN} = \prod_{j\in[\ell_0]}T^{\floor{\bu_j(N)}} f_{j R}.
    \end{align*}
    We note also that $\bc_1, \ldots, \bc_{s_2}$ do not depend on~$N$.
    Letting 
    \begin{align*}
        G_{jR\uh\um\um'} = \Delta_{\bc_1(\uh), \ldots, \bc_{s_2}(\uh); (\um,\um')} f_{jR}
    \end{align*}
    for every $j\in[\ell_0]$,
    we can rewrite the inequality above as
    \begin{align*}
   \limsup_{R\to\infty}\E\limits_{N\in[R]}\E\limits_{\uh\in[\pm H]^{s_1}}\E_{\um, \um'\in[\pm M]^{s_2}} \abs{\int \prod_{j\in[\ell_0]} T^{\floor{\bu_j(N)}} G_{jR\uh\uk\uk'\um\um'}\, d\mu} \gg\veps^{O(1)}     .
    \end{align*}
    Composing the integral with $T^{-\floor{\bu_{\ell_0}(N)}}$, swapping the minus sign inside the integer part, handling the error terms using Lemma \ref{L: errors}, passing to the product system, and using the Cauchy-Schwarz inequality and properties of limsups, we get
    \begin{multline*}
        \E\limits_{\uh\in[\pm H]^{s_1}}\E_{\um, \um'\in[\pm M]^{s_2}}\limsup_{R\to\infty}\\
        \norm{\E\limits_{N\in[R]}\prod_{j\in[\ell_0-1]} (T\times T)^{\floor{\bu_j(N)-\bu_{\ell_0}(N)}+\br_j} G_{jR\uh\um\um'}\otimes \overline{G_{jR\uh\um\um'}}}_{L^2(\mu\times\mu)} \gg \veps^{O(1)}
    \end{multline*}      
    for some $\br_j\in\Z^k$.
    By assumption, for any  $j\in[0,\ell_0]\setminus\{1\}$, we have $d'_{1j}>m_1$, i.e., some coordinate of $\bu_1- \bu_{j}$ grows faster than log.
    Since $\bu_1, \ldots, \bu_{\ell_0}$ are sublinear, we can apply Proposition~\ref{P: sublinear} to find a positive integer $s_3 = O(1)$ (independent of $k$) for which
    \begin{align*}
        \E\limits_{\uh\in[\pm H]^{s_1}}\E_{\um, \um'\in[\pm M]^{s_2}}
        \nnorm{G_{1R\uh\um\um'}}_{\substack{
        (\balpha_{1d'_{12}}-\balpha_{2d'_{12}})^{s_3}, \ldots,  (\balpha_{1d'_{1\ell_0}}-\balpha_{\ell_0 d'_{1\ell_0}})^{s_3}}}^+
        \gg \veps^{O(1)}.
    \end{align*}    
    Recalling the definition of $G_{1R\uh\um\um'}$, taking $M\to\infty$, and using the inductive formula, we get that 
    \begin{align*}
        \E\limits_{\uh\in[\pm H]^{s_1}}\nnorm{f_1}_{\substack{ 
        \bc_1(\uh), \ldots, \bc_{s_2}(\uh), (\balpha_{1d'_{12}}-\balpha_{2d'_{12}})^{s_3}, \ldots,  (\balpha_{1d'_{1\ell_0}}-\balpha_{\ell_0 d'_{1\ell_0}})^{s_3}}}^+ \gg \veps^{O(1)}.
    \end{align*}
    
    We then take the limsup as $H\to\infty$ and conclude by applying Proposition~\ref{P: polynomial concatenation} together with the previously observed fact that the coefficients of $\bc_1, \ldots, \bc_{s_2}$ satisfy the conditions of Proposition \ref{P: PET I}.
    \end{proof}

\section{Seminorm estimates for arbitrary Hardy sequences}\label{S: general case}
Finally, we are ready to state and prove the general case of Theorem \ref{T: box seminorm bound intro}.

  \begin{theorem}[Box seminorm bound]\label{T: box seminorm bound no duals}
			Let $d, k, \ell, m\in\N$, $0\leq m_1\leq m_2\leq m$ be integers and  $\CQ = \{g_1,...,g_m\}\subseteq \mathcal{H}$ be a collection of Hardy functions satisfying the following growth conditions:
   \begin{enumerate}
       \item $1 \prec g_1(t)\prec \cdots \prec g_m(t)\ll t^d$ are strongly nonpolynomial;
       \item $g_i(t)\ll\log t$ iff $i\in[m_1]$;
       \item $g_i(t)\prec t^\delta$ for every $\delta>0$ (i.e., $\fracdeg g_i = 0$) iff $i\in[m_2]$.
   \end{enumerate}
   For $j\in[\ell]$, let
     \begin{gather*}
     \ba_j = \bu_j + \bp_j +\br_j\in\CH^k\\
         \textrm{for}\quad \bu_j = \sum_{i\in[m]}\balpha_{ji} g_i,\quad\p_j(t) = \sum_{i=0}^d \bbeta_{ji} t^i\in\R^k[t]\quad \textrm{and}\quad \lim_{t\to\infty}\br_j(t)\to\mathbf{0}
     \end{gather*}
     for some $\balpha_{ji},\bbeta_{ji}\in\R^{k}$.
   Lastly, let
   \begin{align*}
        d_{jj'} &= \max\{i\in[d]:\; \bbeta_{ji}\neq \bbeta_{j'i}\}\\ 
        d'_{jj'} &= \max\{i\in[m]:\; \balpha_{ji}\neq \balpha_{j'i}\},
   \end{align*}
   and assume that $d_{1j}>0$ or  $d'_{1j}>m_1$ for all $j\in[0,\ell]\setminus\{1\}$, where $\ba_0 = \bu_0 = \bp_0 = \mathbf{0}$ (thus, $\ba_1- \ba_{j}$ has a coordinate that grows faster than log). 
   Then there exist a positive integer $s=O_{d, \ell, \CQ}(1)$ and vectors
\begin{align}
    \nonumber\bgamma_1, \ldots, \bgamma_s\in &\{\balpha_{1d'_{1j}} - \balpha_{jd'_{1j}}:\; j\in[0,\ell]\setminus\{1\},\; d'_{1j}> m_2\}\\
    \label{E: set of coefficients} \cup &\{\bbeta_{1d_{1j}} - \bbeta_{jd_{1j}}:\; j\in[0,\ell]\setminus\{1\},\; d'_{1j}\leq m_2,\; d_{1j}>0\}\\
    \nonumber  \cup &\{\balpha_{1d'_{1j}} - \balpha_{jd'_{1j}}:\; j\in[0,\ell]\setminus\{1\},\; m_1 < d'_{1j}\leq m_2,\; d_{1j} = 0\}
\end{align}
   such that for all systems $(X, \CX, \mu, T_1, \ldots, T_k)$ and 1-bounded functions $f_1\in L^\infty(\mu)$, we have
			\begin{multline*}
				\brac{\limsup_{N\to\infty}\sup_{\substack{\norm{f_2}_\infty, \ldots, \norm{f_\ell}_\infty\leq 1}}\sup_{|c_n|\leq 1}\norm{\E_{n\in[N]}c_n\cdot \prod_{j\in[\ell]} T^{\floor{\ba_j(n)}}f_j}_{L^2(\mu)}}^{O_{d, \ell, \CQ}(1)}\\
    \ll_{d, k, \ell, \CQ} \nnorm{f_1}_{\bgamma_1, \ldots, \bgamma_s}^+.
			\end{multline*}
		\end{theorem}

Before we prove Theorem \ref{T: box seminorm bound no duals}, we make several remarks regarding its content.
  First, it is a priori far from obvious where the set of directions \eqref{E: set of coefficients} comes from. We can think of $\ba_1, \ldots, \ba_\ell$ as $\R^k$-linear combinations of $g_1(t), \ldots, g_m(t), t, t^2, \ldots, t^d$, and we want to think of these generators as being ordered in an ascending way as
  \begin{align}\label{E: unnatural ordering}
     g_1(t),\; \ldots,\; g_{m_2}(t),\; t,\; t^2,\; \ldots,\; t^d, \; g_{m_2 + 1}(t),\; \ldots,\; g_m(t).
  \end{align}
  In other words, all strongly nonpolynomial functions of positive fractional degree take precedence over polynomials. This ordering naturally comes up in our argument, and it has to do with how we Taylor expand $g_1, \ldots, g_m$ on short intervals. Basically, the functions $g_1, \ldots, g_{m_2}$ will be constant on short intervals; by contrast, the functions $g_{m_2 + 1}(t),\; \ldots,\; g_m(t)$ will be approximated by polynomials whose degree in $n$ far surpasses $d$. 
  The collection \eqref{E: set of coefficients} can then be interpreted as the set of leading coefficients of $\ba_1, \ba_1-\ba_2, \ldots, \ba_1 - \ba_\ell$ with respect to \eqref{E: unnatural ordering}.
  
  Interestingly,  Theorem~\ref{T: box seminorm bound no duals} does not generalize Proposition~\ref{P: sublinear + polynomial}; rather, it invokes Proposition~\ref{P: sublinear + polynomial} in its proof. The reason why Proposition~\ref{P: sublinear + polynomial} is not merely a subcase of Theorem~\ref{T: box seminorm bound no duals} is that the set of directions appearing in the conclusion of the former is different from \eqref{E: set of coefficients}, as it corresponds to the ordering
  \begin{align}\label{E: natural ordering}
      g_1(t),\; \ldots,\; g_{m}(t),\; t,\; t^2,\; \ldots,\; t^d,
  \end{align}
  which is arguably more natural than \eqref{E: unnatural ordering} in that $g_1(t) \prec \cdots \prec g_m(t)\prec t \prec \cdots \prec t^d$ under the assumption of Proposition \ref{P: sublinear + polynomial}. It is an interesting open question whether a version of Theorem \ref{T: box seminorm bound no duals} might hold with a set of direction vectors corresponding to the ordering \eqref{E: natural ordering} for arbitrary strongly nonpolynomial $g_1, \ldots, g_m$.

Theorem \ref{T: box seminorm bound no duals} will be derived from the following more general result for averages twisted by dual sequences. We note that in Theorem \ref{T: box seminorm bound} below, $\ba_1$ cannot be subfractional. The missing case of Theorem \ref{T: box seminorm bound no duals} in which $\ba_1, \ldots, \ba_\ell$ are all subfractional is covered by Corollary \ref{C: sublinear}.
  \begin{theorem}[Box seminorm bound with dual twists]\label{T: box seminorm bound}
			Let $d, k, \ell, m\in\N$, $J\in\N_0$, $0\leq m_1\leq m_2\leq m$ be integers and  $\CQ = \{g_1,...,g_m, b_1, \ldots, b_J\}\subseteq \mathcal{H}$ be a collection of Hardy functions satisfying the following growth conditions:
   \begin{enumerate}
       \item $1 \prec g_1(t)\prec \cdots \prec g_m(t)\ll t^d$ are strongly nonpolynomial;
       \item $g_i(t)\ll\log t$ iff $i\in[m_1]$;
       \item $g_i(t)\prec t^\delta$ for every $\delta>0$ (i.e. $\fracdeg g_i = 0$) iff $i\in[m_2]$;
       \item $b_j(t)\ll t^d$ for every $j\in[J]$.
   \end{enumerate}
   For $j,j'\in[\ell]$, let $\ba_j, \bu_j, \bp_j, \br_j, d_{jj'}$ and $d'_{jj'}$ be as in Theorem \ref{T: box seminorm bound no duals}. Suppose that the following two conditions hold:
   \begin{enumerate}
       \item for all $j\in[0,\ell]\setminus\{1\}$, we have $d_{1j}>0$ or  $d'_{1j}>m_1$, i.e., $\ba_1- \ba_{j}$ has a coordinate that grows faster than log;
       \item $d_{10} > 0$ or $d'_{10} > m_2$, i.e., $\ba_1$ has positive fractional degree.
   \end{enumerate}
   Then there exist a positive integer $s=O_{d, \ell, \CQ}(1)$ and vectors $\bgamma_1, \ldots, \bgamma_s$ in \eqref{E: set of coefficients} such that for all systems $(X, \CX, \mu, T_1, \ldots, T_k)$ and 1-bounded functions $f_1\in L^\infty(\mu)$, we have
			\begin{multline*}
				\brac{\limsup_{N\to\infty}\sup_{\substack{\norm{f_2}_\infty, \ldots, \norm{f_\ell}_\infty\leq 1,\\ \CD_1, \ldots, \CD_J\in\FD_d}}\sup_{|c_n|\leq 1}\norm{\E_{n\in[N]}c_n\cdot \prod_{j\in[\ell]} T^{\floor{\ba_j(n)}}f_j \cdot \prod_{j\in[J]}\CD_j(\floor{b_j(n)})}_{L^2(\mu)}}^{O_{d, \ell, \CQ}(1)}\\
    \ll_{d, k, \ell, \CQ} \nnorm{f_1}_{\bgamma_1, \ldots, \bgamma_s}^+.
			\end{multline*}
		\end{theorem}

  We remark that in Theorems \ref{T: box seminorm bound no duals} and \ref{T: box seminorm bound}, we have recorded the dependence of several parameters on $\CQ$, the collection of generators $g_1, \ldots, g_m$ and the sequences $b_1, \ldots, b_J$. Chasing through the argument, one can however establish that the parameters dependent on $\CQ$ really depend on $J, m$ (the numbers of $g_i$'s and $b_i$'s respectively), the maximum growth of $b_1, \ldots, b_J$ (which is controlled by a single parameter $d$), and the collection $\fracdeg g_1, \ldots, \fracdeg g_m$ of the fractional degrees of the generators. Importantly, the parameters will depend on \textit{all} fractional degrees, not just the maximum one. 

  \begin{proof}[Proof of Theorem \ref{T: box seminorm bound no duals} assuming Theorem \ref{T: box seminorm bound} and Corollary \ref{C: sublinear}]
      If $\ba_1$ has positive fractional degree, then the desired estimate in Theorem \ref{T: box seminorm bound no duals} follows immediately from Theorem \ref{T: box seminorm bound}. If $\ba_1$ is subfractional but some $\ba_i$ with $i\in[2,\ell]$ has positive fractional degree, then the result can be deduced from  Theorem \ref{T: box seminorm bound} by composing the integral with $T^{-\floor{\ba_i(n)}}$ in the same way in which Corollary \ref{C: sublinear} is deduced from Proposition~\ref{P: sublinear}. Lastly, if all of $\ba_1, \ldots, \ba_\ell$ are subfractional, then the result follows from Corollary~\ref{C: sublinear} (we note that the second and third case have a nonempty overlap).
  \end{proof}

\subsection{Proof of Theorem \ref{T: box seminorm bound}}

 We are now ready to start the proof of Theorem \ref{T: box seminorm bound}.
  Let
\begin{multline}\label{E: Hardy average}
\veps = \limsup_{N\to\infty}\sup_{\substack{\norm{f_2}_\infty, \ldots, \norm{f_\ell}_\infty\leq 1,\\ \CD_1, \ldots, \CD_J\in\FD_d}}\sup_{|c_n|\leq 1}\norm{\E_{n\in[N]}c_n\cdot \prod_{j\in[\ell]} T^{\floor{\ba_j(n)}}f_j \cdot \prod_{j\in[J]}\CD_j(\floor{b_j(n)})}_{L^2(\mu)}.
\end{multline}
We let all constants depend on $d, k, \ell, \CQ$ (and hence also on $J, m$), noting however that the positive integers $s_i\in\N$ and powers of $\veps$ obtained throughout the argument will not depend on $k$.

 The argument through which we control the average \eqref{E: Hardy average} by a box seminorm consists of five main steps: 
\begin{enumerate}
    \item approximating the sequences $\ba_j$ by polynomials on short intervals;
    \item applying finitary PET bounds for polynomials;
    \item massaging the resulting expression until it becomes an average along Hardy sequences that are sums of polynomial and sublinear terms;
    \item invoking seminorm estimates for such averages (Proposition~\ref{P: sublinear + polynomial}); 
    \item applying concatenation to an average of seminorms with polynomial parameters.
\end{enumerate}

Because of the need to change the order of summations on a number of occasions, much of the argument has to be done with respect to finite averages. We also need much more delicate understanding of the Taylor expansions of nonpolynomial Hardy functions of distinct growth than in previous works on the subject.

\smallskip
 \textbf{Step 1: Passing to short intervals and Taylor expanding $g_i$'s. }
\smallskip 

Our first step is to split the intervals $[N]$ into short intervals $[N+1, N+L(N)]$ using Lemma~\ref{L: double averaging} and approximate the functions $g_i$ on short intervals by polynomials. The function $L\in\CH$ will be specified later; for now, its only important property is that $\fracdeg L\in (0,1)$.
Taylor expanding the functions $g_i$, we define 
\begin{align}\label{E: Taylor approximation}
g_{iN}(n) = \sum_{l=0}^{d_{g_i}} \frac{g_i^{(l)}(N)}{l!} n^l
\end{align}
for all $n\in [L(N)]$. The degree $d_{g_i}$ of the expansion satisfies 
\begin{align*}
    \lim\limits_{N\to\infty}|g_i^{(d_{g_i})}(N)|\cdot N^{d_{g_i}} = \infty\quad \textrm{and}\quad\lim_{N\to\infty}\max\limits_{n\in[L(N)]}|g_i(N+n) - g_{iN}(n)|= 0,
\end{align*}
and this happens whenever
\begin{align}\label{E: property of L}
    \bigabs{ g_i^{(d_{g_i})}(N)}^{-\frac{1}{d_{g_i}}}\prec L(N)\prec \bigabs{ g_i^{(d_{g_i} +1)}(N)}^{-\frac{1}{d_{g_i} +1}}.
\end{align}
In particular,  for $i\in[m_2]$, we have {$d_{g_{i}}=0$,} $g_{iN}(n) = g_i(N)$ and
\begin{align*}
\max_{n\in[L(N)]}|g_i(N+n) - g_{i}(N)|\leq L(N)\cdot|g'_i(N)|\to 0\quad \textrm{as}\quad N\to\infty.
\end{align*}
{Indeed, as $i\in [m_2],$ we have $g_i(N)\prec N^\delta$ for all $\delta>0,$ so by Lemma~\ref{L: Frantzikinakis growth inequalities}, we get $$g'_i(N)\ll N^{-1}g_i(N)\prec N^{\delta-1}$$ for all $\delta>0$; the claimed convergence to 0 then follows from the assumption that $L(N)\ll N^{1-\delta}$ for some $\delta>0$.}
Informally, this asserts that the functions $g_1, \ldots, g_{m_2}$ are essentially constant on the interval $(N, N+L(N)]$.

Thus, we have
\begin{align}\label{E: approximation by polynomial}
 \lim_{N\to\infty}\max\limits_{n\in[L(N)]}|\ba_j(N+n) - \ba_{jN}(n)|=0   
\end{align}
for
\begin{align*}
\ba_{jN}(n) = \sum_{i\in[m]} \balpha_{ji} g_{iN}(n) + \bp_j(N+n)\quad \textrm{and}\quad n\in[L(N)].
\end{align*}

The following result on simultaneous Taylor approximations of $g_1, \ldots, g_m$ is of utmost importance for us, as it tells us how to choose the function $L$, and it gives useful properties of the degrees $d_{g_i}$ of $g_{iN}$.
 \begin{proposition}[Common Taylor expansion]\label{P: Taylor expansion ultimate}
     Let $q\in\N$. Then there exist positive functions $H, L\in\CH$ with fractional degrees in $(0,1)$ such that the degrees $d_{g_i}$ of the Taylor approximants \eqref{E: Taylor approximation} satisfy the following conditions:
     \begin{enumerate}
         \item\label{i: arbitrarily large} if $i>m_2$ (i.e. $\fracdeg g_i>0$), then $d_{g_i}\geq q$, otherwise $d_{g_i} = 0$;
         \item\label{i: equal fractional degrees} for all $i,j\in[m]$, we have $d_{g_i} = d_{g_j}$ iff $\fracdeg g_i = \fracdeg g_j$;
         \item\label{i: property implying common Taylor expansion} for all $i, j \in [m_2+1,m]$,  we have\footnote{We recall that $a \lll b$ iff $\fracdeg a< \fracdeg b$ iff there exists $\delta>0$ such that $a(N)N^\delta\ll b(N)$.}
\begin{equation}\label{E: growth conditions 1}
    \bigabs{g_i^{(d_{g_i}) }(N) }^{-\frac{1}{d_{g_i}}}\lll H(N) \lll L(N) \lll   \bigabs{g_j^{(d_{g_j}+1) }(N)}^{-\frac{1}{d_{g_j}+1}};
   \end{equation}
         \item\label{i: comparing frac degs} for all $i\in[m_2+1,m]$, we have
         \begin{align*}
             \fracdeg \bigabs{g_i^{(d_{g_i}) }(N) }^{-\frac{1}{d_{g_i}}}\leq \fracdeg \bigabs{g_m^{(d_{g_m}) }(N) }^{-\frac{1}{d_{g_m}}},
         \end{align*}
         with equality if and only if $\fracdeg g_i = \fracdeg g_m$.
\item\label{i: property implying different chi} for all distinct $i,j\in [m]$, we have 
\begin{equation*}
    g_i^{(d_{g_i})}(N) H(N)^{d_{g_i}} \not \asymp g_j^{(d_{g_j})}(N) H(N)^{d_{g_j}};
\end{equation*}
in other words, these two functions have distinct growth rates. 
     \end{enumerate}

 Moreover, the function $H$ takes the form
 \begin{align*}
    H(N) = \begin{cases}
    |g_m^{(d_{g_m})}(N)|^{-1/d_{g_m}}N^\eta, \quad &\textrm{if}\quad d_{g_m}>0\\
    N^\eta, \quad &\textrm{if}\quad d_{g_m}=0
    \end{cases},
 \end{align*}
  where $\eta$ can be any number from an interval $(0, \eta_0)$ for some $\eta_0>0$ depending only on the numbers $\fracdeg g_i, d_{g_i}$, and $L$ is any function satisfying $L\succ H$ and \eqref{E: growth conditions 1}.
 \end{proposition}
 When applying Proposition \ref{P: Taylor expansion ultimate}, we will choose
 \begin{align}\label{E: degree of approximation}
     q = 2d\geq 2\max(\{\deg\bp_j:\; j\in[\ell]\}\cup\{\fracdeg g_i:\; i\in[m]\}),
 \end{align}
 remarking that any higher $q$ would do as well.
	
The first condition ensures that after passing to short intervals, those $g_i$ that are not sublinear can be expanded as polynomials of arbitrarily high degree (in particular, the degrees of their Taylor expansions are higher than the degrees of the polynomials $\bp_1, \ldots \bp_\ell$). The second  condition guarantees that strongly nonpolynomial functions of distinct fractional degrees can be expanded as polynomials of distinct degrees. The inequalities \eqref{E: growth conditions 1} provide the common Taylor expansion (as they imply \eqref{E: property of L}). They are also instrumental in showing that after the change of variables, the lengths of the intervals will go to $\infty$ (as $H\prec L$), and that the new coefficients obtained after changes of variables will grow faster than log. The property \eqref{i: comparing frac degs} will be used to show later on that these new coefficients are sublinear, and in fact, their fractional degrees are bounded away from 1. 
The last condition will ensure that these new coefficients have distinct growth rates.
We will use this proposition as a black box and postpone its proof until Appendix \ref{A: approximations}.

Splitting $\N$ into short intervals using Lemma \ref{L: double averaging}, approximating $\ba_j(N+n)$ by $\ba_{jN}$ for each $j\in[\ell]$ as in \eqref{E: approximation by polynomial} and handling the error terms using Lemma \ref{L: errors}, we obtain
\begin{multline*}
\limsup_{R\to\infty}\;    \sup_{\substack{\norm{f_2}_\infty, \ldots, \norm{f_\ell}_\infty\leq 1,\\ \CD_1, \ldots, \CD_J\in\FD_d}}\; \E\limits_{N\in[R]}\; \sup_{|c_n|\leq 1}\\ 
\norm{\E_{n\in[L(N)]}c_n \cdot \prod_{j\in[\ell]} T^{\floor{\ba_{jN}(n)}}f_j \cdot  \prod_{j\in[J]} \CD_j(\floor{b_j(N+n)})}_{L^2(\mu)}\gg \veps.
\end{multline*}

It is at this point that we use the assumption that $d_{10}>0$ or $d'_{10}>m_2$, i.e., $\ba_1$ has positive fractional degree: this assumption is required for the polynomial $\ba_{1N}$ to be nonconstant (which is necessary for otherwise $f_1$ would vanish from the average above). If $d_{10}>0$, then the polynomial part of $\ba_{1N}$ is nonconstant whereas if $d'_{10}>m_2$, then the nonpolynomial part of $\ba_{1N}$ is noncostant by Proposition \ref{P: Taylor expansion ultimate}\eqref{i: arbitrarily large}. 

Rearranging $\ba_1, \ldots, \ba_\ell$ if necessary, we may assume that there exists $\ell_{0}\in[\ell]$ such that $j\in[\ell_{0}]$ if and only if $\ba_1-\ba_j$ is subfractional, so that
\begin{align*}
    \ba_{j}= \ba_1 + \sum_{i\in[m_2]} (\balpha_{ji}-\balpha_{1i})g_i\quad \textrm{for}\quad j\in[\ell_0],
\end{align*}
and similarly for $\ba_{jN}$. Then
\begin{multline*}
\limsup_{R\to\infty}    \sup_{\substack{\norm{f_{\ell_0+1}}_\infty, \ldots, \norm{f_\ell}_\infty\leq 1,\\ \CD_1, \ldots, \CD_J\in\FD_d}}\E\limits_{N\in[R]}\sup_{|c_n|\leq 1}\\ 
\norm{\E_{n\in[L(N)]}c_n \cdot T^{\floor{\ba_{1N}(n)}} F_{RN}\cdot \prod_{j = \ell_0+1}^\ell T^{\floor{\ba_{jN}(n)}}f_j \cdot  \prod_{j\in[J]} \CD_j(\floor{b_j(N+n)})}_{L^2(\mu)}\gg \veps
\end{multline*}
for
\begin{align}\label{E: F_RN}
F_{RN} =  \prod_{j\in[\ell_0]} T^{\floor{\sum_{i\in[m_2]} (\balpha_{ji}-\balpha_{1i})g_i(N)}} f_{jR},
\end{align}
where $f_{1R} = f_1$ and $f_{jR}$ for $j\geq 2$ are some 1-bounded functions that bring the $L^2(\mu)$ norms above close to the supremum, {and the error terms are absorbed by the sup over~$c_{n}$}.

We set
\begin{gather*}
 d''_{jj'} := \deg\brac{\sum_{i\in[m]} (\balpha_{ji}-\balpha_{j'i}) g_{iN}}
\end{gather*}
to be the degree of the strongly nonpolynomial part of $\ba_{jN} - \ba_{j'N}$. We note from Proposition \ref{P: Taylor expansion ultimate}\eqref{i: arbitrarily large} that
\begin{align}\label{E: relationship between degrees}
    d''_{jj'} = 0 \iff d'_{jj'} \leq m_2,
\end{align}
and we also observe from the definition of $d''_{jj'}$ above that $d''_{jj'} = d_{g_i}$ for $i = d'_{jj'}$. We also subdivide $[m] = \CB_1 \cup \cdots \cup \CB_{m_3}$ into classes of indices
\begin{align*}
\CB_j = \{i\in[m]:\; d_{g_i} = j\}
\end{align*}
corresponding to the sublinear functions whose Taylor expansions have the same degree.
By Proposition~\ref{P: Taylor expansion ultimate}\eqref{i: equal fractional degrees}, all $g_i$ belonging to a fixed class $\CB_j$ have the same fractional degree. Moreover, if $i,i'\in\CB_j$ for the same $j$ and $i<i'$, then $g_i^{(d_{g_i})}\prec g_{i'}^{(d_{g_{i'}})}$; consequently, the functions $\{g_i^{(d_{g_i})}:\; i\in\CB_j\}$ are linearly independent. Because of this and the assumption~\eqref{E: degree of approximation},  the leading coefficient of $\ba_{jN}- \ba_{j' N}$ is
 \begin{align}\label{E: leading coeff nonpolynomial}
	 \sum_{i\in\CB_{d''_{jj'}}} (\balpha_{ji}-\balpha_{j'i}) \frac{g_i^{(d_{g_i})}(N)}{d_{g_i}!} = \sum_{i\in\CB_{d''_{jj'}}} (\balpha_{ji}-\balpha_{j'i}) \frac{g_i^{(d''_{jj'})}(N)}{d''_{jj'}!}
 \end{align}
 if $d''_{jj'} >0$ (or equivalently if $d'_{jj'}>m_2$) and 
 \begin{align}\label{E: leading coeff polynomial}
 	\bbeta_{jd_{jj'}} - \bbeta_{j' d_{jj'}}
 \end{align}
otherwise.

Before we apply our PET results, we want to change variables in order to ensure that the leading coefficients of the polynomials $g_{iN}$ grow faster than 1. Substituting $n\mapsto \floor{H(N)}n + r$  and setting $\tilde{L}(N) = L(N)/H(N)$ (which satisfies $\tilde{L}\succ 1$ by Proposition~\ref{P: Taylor expansion ultimate}), we have
\begin{multline}\label{E: Before Step 2}
\limsup_{R\to\infty} \sup_{\substack{\norm{f_{\ell_0+1}}_\infty, \ldots, \norm{f_\ell}_\infty\leq 1,\\ \CD_1, \ldots, \CD_J\in\FD_d}}  \E\limits_{N\in[R]} \E_{r\in [H(N)]}
\sup_{|c_n|\leq 1}
\left|\!\left|\E_{n\in[\tilde{L}(N)]}c_n \cdot T^{\floor{\ba_{1N}(\floor{H(N)}n + r)}} F_{RN}\right.\right.\\ \left.\left. \prod_{j = \ell_0+1}^\ell T^{\floor{\ba_{jN}(\floor{H(N)}n + r)}}f_j \cdot  \prod_{j\in[J]} \CD_j(\floor{b_j(N+\floor{H(N)}n + r)})\right|\!\right|_{L^2(\mu)}\gg \veps.
\end{multline}

\smallskip
 \textbf{Step 2:  PET in $n$ on short intervals.}
\smallskip 

We move on now to apply our PET bounds from Section \ref{S: polynomial PET}.
Since $\ba_{1N}, \ba_{(\ell_0+1)N}, \ldots, \ba_{\ell N}$ are polynomials in $n$, and $\ba_{1N}$ is nonconstant, \eqref{E: Before Step 2} is amenable to Proposition~\ref{P: PET for polynomials twisted by duals}. Indeed, 
by Proposition~\ref{P: PET for polynomials twisted by duals}, there exist positive integers $s_1, s_2=O(1)$ (independent of $k$) and nonzero multilinear polynomials $\bc_{1N}, \ldots, \bc_{s_2 N}\in\R^k[\uh]$ satisfying 
\begin{multline}\label{E: before removing H}
    \limsup_{R\to\infty}   \E\limits_{N\in[R]} \E\limits_{\uh\in[\pm H]^{s_1}} \E_{\um, \um'\in[\pm M]^{s_2}}\\
  \abs{\int \Delta_{\bc_{1N}(\uh), \ldots, \bc_{s_2 N}(\uh); (\um,\um')}F_{RN}\, d\mu}
  +\frac{1}{H}+\frac{1}{M}\gg \veps^{O(1)}.    
\end{multline}
The coefficients of the polynomials $\bc_{1N}, \ldots, \bc_{s_2 N}$ come from the leading coefficients of the polynomials
\begin{align}\label{E: leading coeffs of ajN}
    \ba_{1N}(N+\floor{H(N)}n + r) - \ba_{jN}(N+\floor{H(N)}n + r)
\end{align}
for $j\in\{0\}\cup [\ell_0+1, \ldots, \ell]$.
An analysis of the leading coefficients of \eqref{E: leading coeffs of ajN} shows that we can split $\bc_{l N} = \bc^{\pol}_{l N} + \bc^{\nonpol}_{l N}$ into its ``polynomial'' and ``nonpolynomial'' parts as
\begin{align}\label{E: c_l pol}
\bc^{\pol}_{l N}(\uh) &= \sum_{\uu\in \CA_l} (|\uu|+1)! (\bbeta_{1(|\uu|+1)} - \bbeta_{w_{l\uu}(|\uu|+1)})\floor{H(N)}^{|\uu|+1} \uh^\uu,\\ 
\label{E: c_l nonpol}
\bc^{\nonpol}_{l N}(\uh) &= \sum_{\uu\notin\CA_l} \brac{\sum_{i\in\CB_{|\uu|+1}} (\balpha_{1i}-\balpha_{w'_{l\uu} i}) g_i^{(|\uu|+1)}(N)} \floor{H(N)}^{|\uu|+1} \uh^\uu, 
\end{align}
for some set of multiindices $\CA_l\subseteq\{0,1\}^{s_1}$ and indices $w_{l\uu}, w'_{l\uu}\in\{0\}\cup \{\ell_0+1, \ldots, \ell\}$ satisfying
\begin{align}\label{E: degrees of differences poly}
d_{1w_{l\uu}} = \deg(\ba_{1N} - \ba_{w_{l\uu}N}) = \deg(\bp_1 - \bp_{w_{l\uu}}) = |\uu|+1 \quad \textrm{for}\quad \uu\in\CA_l,
\end{align} 
and
\begin{align}\label{E: degrees of differences nonpoly}
d''_{1 w'_{l\uu}}  = \deg(\ba_{1N} - \ba_{w'_{l\uu}N}) = d_{g_i} = |\uu|+1
\end{align}
for $\uu\notin\CA_l$ and $i\in \CB_{|\uu|+1}$. By comparing the degrees of $g_{iN}$ and $\bp_{j}$, we make the following important observation.
\begin{align}\label{E: description of A_l}
    \uu\in \CA_l \quad \Longleftrightarrow \quad \ba_1 - \ba_{w_{l\uu}}\quad \textrm{has\; no\; contribution\; from}\quad g_i \quad\textrm{for}\quad i\in[m_2+1, m].
\end{align}

\smallskip
 \textbf{Step 3:  Removing $\floor{H(N)}$.}
\smallskip 

The formulas \eqref{E: c_l pol} and \eqref{E: c_l nonpol} for $\bc_{lN}$ somewhat inconveniently contain the expressions $\floor{H(N)}$, and our next goal is to replace all their instances by $N$ using the lemma below.
    \begin{lemma}[{\cite[Lemma~5.1]{Fr10}}]\label{L: removing H}
        Let $(A_R(N))_{N, R\in\N}$ be a 1-bounded, two-parameter sequence of vectors in a normed space and let $H\in\mathcal{H}$ satisfy the growth condition $t^{\delta}\prec H(t)\prec t$ for some $\delta>0$. Then, we have 
        \begin{equation*}
	       \limsup\limits_{R\to \infty} \norm{\E\limits_{N\in[R]} A_R(\floor{H(N)})}\ll_H  \limsup\limits_{R\to \infty} \norm{ \E\limits_{N\in[R]} A_R(N)}.
	\end{equation*}
    \end{lemma}
In order to apply Lemma \ref{L: removing H}, we need to rephrase the integral in \eqref{E: before removing H} so that it only depends on $\floor{H(N)}$, not $N$.  This requires us to deal with the term $g_i^{(d_{g_i})}(N)$   appearing in the definition of $\bc^{\nonpol}_{l N}$ (for $i\in [m_2+1, m]$) and in the formula \eqref{E: F_RN} for $F_{RN}$ (for $i\in[m_2]$).
To this end, 
we set
\begin{align*}
\chi_{g_i}(t) = g_i^{(d_{g_i})}\circ H\inv(t) \cdot t^{d_{g_i}}\quad \textrm{so that}\quad \chi_{g_i}(H(N)) = g_i^{(d_{g_i})}(N) H(N)^{d_{g_i}}    
\end{align*}
are the expressions appearing in the formula for strongly nonpolynomial functions. In particular, Proposition \ref{P: Taylor expansion ultimate} gives the following corollary.

\begin{corollary}[Properties of the new coefficients]\label{C: properties of chi}
    The functions $\chi_{g_i}$ enjoy the following properties:
    \begin{enumerate}
        \item if $i\in[m_2]$, then $\chi_{g_i} = g_i\circ H\inv$; in particular, $\fracdeg \chi_{g_i} = 0$;
        \item\label{i: chi sublinear} for $i\in [m_2+1, m]$, we have $\fracdeg \chi_{g_i}\in (0,1)$; in particular, they are sublinear and have positive fractional degrees;
        \item\label{i: chi distinct fracdeg} if $i, j\in[m]$ are distinct, then $\chi_{g_i}\not\asymp\chi_{g_j}$.
    \end{enumerate}
\end{corollary}

\begin{proof}
By Proposition \ref{P: Taylor expansion ultimate}\eqref{i: arbitrarily large}, $d_{g_i} = 0$ whenever $i\in[m_2]$, and so $\chi_{g_i} = g_i\circ H\inv$ by definition. Then $\fracdeg \chi_{g_i} = \fracdeg g_i/\fracdeg{H} = 0$ by Lemma~\ref{L: properties of fracdeg}\eqref{i: fracdeg product}, and the first property follows. The last property follows immediately from Proposition~\ref{P: Taylor expansion ultimate}\eqref{i: property implying different chi}. The second property is vacuously true whenever $m_2 = m$, i.e., all the $g_i$'s are subfractional. Otherwise we have $H(N) = |g_m^{(d_{g_m})}(N)|^{-1/d_{g_m}} N^\eta$ for some $\eta>0$, and hence by Lemma~\ref{L: properties of fracdeg}, we have
\begin{align*}
    \fracdeg \chi_{g_i} = \frac{\fracdeg g_i^{(d_{g_i})}}{\fracdeg H} + d_{g_i} = \frac{c_i - d_{g_i}}{1-\frac{c_m}{d_{g_m}}+\eta}+d_{g_i} = \frac{c_i d_{g_m} - d_{g_i} c_m + \eta d_{g_i}d_{g_m}}{d_{g_m}-c_m + \eta d_{g_m}}
\end{align*}
where $c_i = \fracdeg g_i>0$. Because of Proposition \ref{P: Taylor expansion ultimate}\eqref{i: arbitrarily large} and our choice of $q$, we have
\begin{align}\label{E: d_m}
    d_{g_m} \geq q\geq 2d\geq 2c_m    
\end{align}
(recall that $d$ upper bounds the fractional degrees of $g_i$'s while $q$ lower bounds the degrees of Taylor approximations), and hence the denominator in the formula for $\fracdeg \chi_{g_i}$ above is positive.
The positivity of the denominator in turn implies that $\fracdeg \chi_{g_i} > 0$ iff 
\begin{align}\label{E: proxy bound with eta}
    \frac{c_i}{d_{g_i}}-\frac{c_m}{d_{g_m}}+\eta > 0,
\end{align}
which in turn follows from the positivity of $\eta$ and Proposition \ref{P: Taylor expansion ultimate}\eqref{i: comparing frac degs}. Thus, $\fracdeg \chi_{g_i} > 0$. To prove that $\fracdeg \chi_{g_i} < 1$, it suffices to show in the light of the formula for $\fracdeg \chi_{g_i}$ that
\begin{align*}
    c_i d_{g_m} - d_{g_i} c_m + \eta d_{g_i}d_{g_m} < d_{g_m}-c_m + \eta d_{g_m}.
\end{align*}
This in turn will follow from
\begin{align}\label{E: fracdeg < 1}
    (c_i-1)d_{g_m} < c_m (d_{g_i} - 1)
\end{align}
and taking sufficiently small $\eta>0$. In order to show \eqref{E: fracdeg < 1}, we first use the bound
\begin{align*}
        \bigabs{g_m^{(d_{g_m}) }(N) }^{-\frac{1}{d_{g_m}}} \lll   \bigabs{g_i^{(d_{g_i}+1) }(N)}^{-\frac{1}{d_{g_i}+1}}
\end{align*}
from \eqref{E: growth conditions 1}, which can be restated as $c_i d_{g_m} < c_m (d_{g_i} + 1)$. It then follows that
\begin{align*}
    (c_i - 1) d_{g_m} < c_m(d_{g_i} + 1) - d_{g_m} = c_m(d_{g_i} - 1) + 2 c_m - d_{g_m} \leq c_m(d_{g_i} - 1)
\end{align*}
as long as $d_{g_m} \geq 2 c_m$, which holds by \eqref{E: d_m}. 
\end{proof}

The next lemma is crucial in rephrasing the integral in \eqref{E: before removing H} in terms of $\floor{H(N)}$. 
		\begin{lemma}\label{Claim 2}
			For every $i\in [m]$, we have \begin{equation*}
				\lim\limits_{t\to\infty} \bigabs{g_i^{(d_{g_i})}\circ H^{-1}(t) \cdot \floor{t}^{d_{g_i}}-g_i^{(d_{g_i})}\circ H^{-1}(\floor{t})\cdot\floor{t}^{d_{g_i}}}=0.
			\end{equation*}
		\end{lemma}
		\begin{proof}
  By the mean value theorem, Lemma~\ref{L: Frantzikinakis growth inequalities}, and the fact that $\brac{g_i^{(d_{g_i})}\circ H^{-1}}'$ is decreasing, we have 
  \begin{align*}
      \bigabs{g_i^{(d_{g_i})}\circ H^{-1}(t)-g_i^{(d_{g_i})}\circ H^{-1}(\floor{t})} &\leq \{t\} \cdot \brac{g_i^{(d_{g_i})}\circ H^{-1}}'(\floor{t})\\
      &\ll \frac{g_i^{(d_{g_i})}\circ H^{-1}(\floor{t})}{\floor{t}}.
  \end{align*}
Replacing $\floor{t}$ by $t$ for simplicity, it therefore suffices to show that
\begin{equation*}
				g_i^{(d_{g_i})}\circ H^{-1}(t) \cdot t^{d_{g_i}}\prec t.
			\end{equation*}
        But the left hand side is just $\chi_{g_i}(t)$, and so the property follows immediately from Corollary~\ref{C: properties of chi}(ii), and specifically from the sublinearity of $\chi$.
\end{proof}
  
Recall the formula for $\bc^{\nonpol}_{l N}$ from \eqref{E: c_l nonpol}. By Lemma \ref{Claim 2} applied with $t = {H(N)}$,  we have 
\begin{align*}
\lim_{N\to\infty}\abs{\bc^{\nonpol}_{l N}(\uh) -\sum_{\uu\notin\CA_l} \brac{\sum_{i\in\CB_{|\uu|+1}} (\balpha_{1i}-\balpha_{w'_{l\uu} i}) \chi_{g_i}(\floor{H(N)})} \uh^\uu} = 0
\end{align*}
for every $l$ and $\uh$, and so we can replace $\bc^{\nonpol}_{l N}(\uh)$ in \eqref{E: before removing H} with the second expression above. We also set 
\begin{align*}
    \tilde{F}_{RN} = \prod_{j\in[\ell_0]} T^{\floor{\sum_{i\in[m_2]} (\balpha_{ji}-\balpha_{1i})g_i\circ H\inv(N)}} f_{jR} = \prod_{j\in[\ell_0]} T^{\floor{\sum_{i\in[m_2]} (\balpha_{ji}-\balpha_{1i})\chi_{g_i}(N)}} f_{jR},
\end{align*}
so that $\tilde{F}_{RH(N)} = F_{RN}$. Lemma \ref{Claim 2} applied with $t = {H(N)}$ similarly gives
\begin{align*}
    \lim_{N\to\infty}\abs{\sum_{i\in[m_2]} (\balpha_{ji}-\balpha_{1i})g_i(N) - \sum_{i\in[m_2]} (\balpha_{ji}-\balpha_{1i})g_i\circ H\inv(\floor{H(N)})} = 0,
\end{align*}
which allows us to replace $F_{RN}$ with $\tilde{F}_{R\floor{H(N)}}$. After all this massaging, \eqref{E: before removing H} becomes
\begin{multline}\label{E: before removing H 2}
        \limsup_{R\to\infty}   \E\limits_{N\in[R]} \E\limits_{\uh\in[\pm H]^{s_1}} \E_{\um, \um'\in[\pm M]^{s_2}}\\
  \abs{\int \Delta_{\tilde{\bc}_{1\floor{H(N)}}(\uh), \ldots, \tilde{\bc}_{s_2 \floor{H(N)}}(\uh); (\um, \um')}\tilde{F}_{R\floor{H(N)}}\, d\mu}
  +\frac{1}{H}+\frac{1}{M}\gg \veps^{O(1)} 
\end{multline}
for $\tilde{\bc}_{l N} = \tilde{\bc}^{\pol}_{l N} + \tilde{\bc}^{\nonpol}_{l N}$ given by
\begin{align*}
\tilde{\bc}^{\pol}_{l N}(\uh) &= \sum_{\uu\in \CA_l} (|\uu|+1)! (\bbeta_{1(|\uu|+1)} - \bbeta_{w_{l\uu}(|\uu|+1)})N^{|\uu|+1} \uh^\uu\\ 
\tilde{\bc}^{\nonpol}_{l N}(\uh) &= \sum_{\uu\notin\CA_l} \brac{\sum_{i\in\CB_{|\uu|+1}} (\balpha_{1i}-\balpha_{w'_{l\uu} i}) \chi_{g_i}(N)} \uh^\uu. 
\end{align*}
The new polynomials $\tilde{\bc}_{l N}\in \R^k[\uh]$ are also nonzero and multilinear.

The integral in \eqref{E: before removing H 2} is now phrased purely in terms of $\floor{H(N)}$ rather than $N$.  By Lemma~\ref{L: removing H}, we get\footnote{In reality, to apply Lemma~\ref{L: removing H}, we need the average over $N$ to be inside the integral. This minor technicality can be dealt with in an entirely standard way:  we apply the Cauchy-Schwarz inequality to \eqref{E: before removing H 2}, pass to the product system, incorporate the average over $N$ inside the integral, apply the Cauchy-Schwarz inequality, and only then apply Lemma \ref{L: removing H}. The desired conclusion will follow once we pass back to the original system.}
\begin{multline}\label{E: tilde c}
 \E\limits_{\uh\in[\pm H]^{s_1}} \E_{\um, \um'\in[\pm M]^{s_2}}\limsup_{R\to\infty}   \E\limits_{N\in[R]}
  \abs{\int \Delta_{\tilde{\bc}_{1N}(\uh), \ldots, \tilde{\bc}_{s_2 N}(\uh); (\um, \um')}\tilde{F}_{RN}\, d\mu}+\frac{1}{H}+\frac{1}{M}\gg\veps^{O(1)}.
\end{multline}

\smallskip
 \textbf{Step 4: Applying seminorm estimates for the polynomial + sublinear case.}
\smallskip

Our next goal is to analyze the average over $N$. The key observation, already made in \cite{Ts22}, is that the iterates in \eqref{E: tilde c} are sums of polynomial and sublinear terms when reinterpreted as sequences in $N$; hence \eqref{E: tilde c} is amenable to Proposition~\ref{P: sublinear + polynomial}. 
In order to construe the polynomials $\tilde{\bc}_{l N}$ as functions of $N$ rather than $\uh$, we define 
\begin{align*}
    \bb_{j\ueps\um\um'\uh}(N) &=  \sum_{i\in[m_2]} (\balpha_{ji}-\balpha_{1i})\chi_{g_i}(N) + \sum_{l\in[s_2]} m_l^{\eps_l}\tilde{\bc}_{l N}(\uh)\\
    &= \sum_{l\in[s_2]} m_l^{\eps_l} \tilde{\bc}^{\pol}_{l N}(\uh) + \sum_{i\in[m_2]} (\balpha_{ji}-\balpha_{1i})\chi_{g_i}(N)\\
    &+ \sum_{i\in[m_2+1, m]} \Bigbrac{\sum_{l\in[s_2]}\; \sum_{\substack{\uu\notin\CA_l:\\ |\uu|+1 = d_{g_i}}}\;(\balpha_{1i}-\balpha_{w'_{l\uu}i})m_l^{\eps_l} \uh^\uu}\chi_{g_i}(N),
\end{align*}
recalling the convention \eqref{E: h^eps},
so that
\begin{align*}
    \Delta_{\tilde{\bc}_{1N}(\uh), \ldots, \tilde{\bc}_{s_2 N}(\uh); (\um, \um')}\tilde{F}_{RN} = \prod_{j\in[\ell_0]}\prod_{\ueps\in\{0,1\}^{s_2}}\CC^{|\ueps|} T^{\floor{\bb_{j\ueps\um\um'\uh}(N)}+\br_{j\veps N\um\um'\uh}}f_{jR}
\end{align*}
for some $O(1)$ error terms $\br_{j\veps N\um\um'\uh}\in\Z^k$
that come from comparing a sum of integer parts to the integer part of the sum.
Combining this identity with \eqref{E: tilde c} and removing the error terms using Lemma \ref{L: errors}, we deduce that
\begin{multline*}
 \E\limits_{\uh\in[\pm H]^{s_1}} \E_{\um, \um'\in[\pm M]^{s_2}}\sup_{\norm{f_2}_\infty, \ldots, \norm{f_{\ell_0}}_\infty\leq 1}\\
 \limsup_{R\to\infty}\E\limits_{N\in[R]}\sup_{|c_N|\leq 1}
  \abs{\int c_N\cdot \prod_{j\in[\ell_0]}\prod_{\ueps\in\{0,1\}^{s_2}}\CC^{|\ueps|} T^{\floor{\bb_{j\ueps\um\um'\uh}(N)}}f_j\, d\mu}+\frac{1}{H}+\frac{1}{M}\gg\veps^{O(1)}.
\end{multline*}
Passing to the product system and applying the Cauchy-Schwarz inequality, we get
\begin{multline*}
 \E\limits_{\uh\in[\pm H]^{s_1}} \E_{\um, \um'\in[\pm M]^{s_2}} \sup_{\norm{f_2}_\infty, \ldots, \norm{f_{\ell_0}}_\infty\leq 1}\\
 \limsup_{R\to\infty}
 \norm{\E\limits_{N\in[R]} \prod_{j\in[\ell_0]}\prod_{\ueps\in\{0,1\}^{s_2}}\CC^{|\ueps|}(T\times T)^{\floor{\bb_{j\ueps\um\um'\uh}(N)}}f_j\otimes \overline{f_j}}_{L^2(\mu\times\mu)}+\frac{1}{H}+\frac{1}{M}\gg\veps^{O(1)}.
\end{multline*}

In order to apply Proposition \ref{P: sublinear + polynomial}, we need to understand the differences $\bb_{1\underline{1}\um\um'\uh} -\bb_{j\ueps\um\um'\uh}$ for $(j,\ueps)\in([0,\ell_0]\times\{0,1\}^{s_2})\setminus(1,\underline{1})$ which take the form
    \begin{multline*}
    \bb_{1\underline{1}\um\um'\uh}(N) -\bb_{j\ueps\um\um'\uh}(N) = \sum_{l\in[s_2]} (m_l' - m_l^{\eps_l}) \tilde{\bc}^{\pol}_{l N}(\uh) + \sum_{i\in[m_2]} (\balpha_{1i}-\balpha_{ji})\chi_{g_i}(N)\\
+ \sum_{i\in [m_2+1, m]} \Bigbrac{\sum_{l\in[s_2]}\; \sum_{\substack{\uu\notin\CA_l:\\ |\uu|+1 = d_{g_i}}}\;(\balpha_{1i}-\balpha_{w'_{l\uu}i})(m_l' - m_l^{\eps_l}) \uh^\uu}\chi_{g_i}(N);
\end{multline*}
specifically, we need to show the following.
\begin{lemma}\label{L: coordinates of b}
    For each $(j,\ueps)\in([0,\ell_0]\times\{0,1\}^{s})\setminus(1,\underline{1})$, the difference function $\bb_{1\underline{1}\um\um'\uh} -\bb_{j\ueps\um\um'\uh}: (t_0, \infty)\to\R^k$ has a coordinate function growing faster than $\log$ for all but $O(M^{2s_1 - 1} H^{s_1-1})$ elements $$(\um, \um', \uh)\in [\pm M]^{s_2}\times [\pm M]^{s_2}\times [\pm H]^{s_1}.$$
\end{lemma}

\begin{proof}
First, we observe that the term $\sum_{l\in[s_2]} m_l^{\eps_l}\tilde{\bc}^{\pol}_{l N}(\uh)$ is a polynomial in $N$ while the other terms in $\bb_{1\underline{1}\um\um'\uh} -\bb_{j\ueps\um\um'\uh}$ are sublinear in $N$ due to Corollary \ref{C: properties of chi}\eqref{i: chi sublinear}. Moreover, thanks to Corollary \ref{C: properties of chi}\eqref{i: chi distinct fracdeg}, for all distinct $i,i'\in [m]$, the functions $\chi_{g_i}, \chi_{g_{i'}}$ have distinct growth rates. 

If $j\neq 1$, then $\balpha_{1i}\neq\balpha_{ji}$ for some $i\in [m_1+1, m_2]$ due to the assumption that $\ba_1 - \ba_j$ differs by a subfractional term faster than log, and since the coefficients $\chi_{g_i}$ all have distinct growth and grow sublinearly (so more slowly than the polynomial $\tilde{\bc}^{\pol}_{l N}(\uh)$ for generic $\uh$), it follows that $\bb_{1\underline{1}\um\um'\uh} -\bb_{j\ueps\um\um'\uh}$ has a coordinate growing faster than log. If $\ueps = \underline{1}$, then necessarily $j\neq 1$, and so we land in the previous case.

If $j=1$, $\ueps\neq \underline{1}$, then
pick $l\in[s_2]$ for which $\eps_l =0$. Then $m_l \neq m_l'$ for all but $O(M)$ choices $(m_l, m'_l)\in[\pm M]^2$.
Since $\bc_{lN}$ is nonzero, one of the two things has to happen:
\begin{itemize}
    \item ${\bc}^{\pol}_{l N}$ is nonzero, in which case $\tilde{\bc}^{\pol}_{l N}(\uh)$ is a nonconstant polynomial in $N$ for all but $O(H^{s_1-1})$ values $\uh\in[\pm H]^{s_1}$;
    \item or ${\bc}^{\nonpol}_{l N}$ is nonzero, so that $\balpha_{1i}\neq\balpha_{w'_{l\uu}i}$ for some $i\in [m_2+1,m]$. Then $\tilde{\bc}^{\nonpol}_{l N}(\uh)$ has a coordinate growing faster than $\log$ for all but $O(H^{s_1-1})$ values $\uh\in[\pm H]^{s_1}$  because the functions $\chi_{g_i}$ have distinct growth rates and satisfy $\chi_{g_i}\succ \log$ by Corollary \ref{C: properties of chi}\eqref{i: chi sublinear}.
\end{itemize}
\end{proof}

By Lemma \ref{L: coordinates of b}, the function $\bb_{1\underline{1}\um\um'\uh} -\bb_{j\ueps\um\um'\uh}$ is nonconstant for a generic choice of $(\um, \um', \uh)$. Therefore, it follows from Proposition \ref{P: sublinear + polynomial} that for all $(\um,\um',\uh)$ but the exceptional set from Lemma~\ref{L: coordinates of b}, we have that \footnote{This is the point of the argument where we need the bounds in Proposition \ref{P: sublinear + polynomial} to be quantitative, or at least uniform in $(\um,\um',\uh)$, since otherwise we cannot guarantee the positivity of the average in $(\um, \um', \uh)$ as we take $M, H$ to infinity. }
\begin{align}\label{E: final concat}
    \E\limits_{\uh\in[\pm H]^{s_1}}\E_{\um, \um'\in[\pm M]^{s_2}}\nnorm{f_1}_{\bgamma_1(\um, \um', \uh), \ldots, \bgamma_{s_3}(\um, \um', \uh)}^+ +\frac{1}{H}+\frac{1}{M}\gg  \veps^{O(1)}
\end{align}
for a positive integer $s_3 = O(1)$ and nonzero vectors $\bgamma_1(\um, \um', \uh), \ldots, \bgamma_{s_3}(\um, \um', \uh)$ belonging to the set
\begin{gather*}
    \Bigl\{D!\cdot\sum_{l\in[s_2]}(m_l'-m_l^{\eps_l})\sum_{\substack{\uu\in\CA_l:\\ |\uu|+1
    =D}}(\bbeta_{1D}-\bbeta_{w_{l\uu}D})\uh^\uu:\; D\in[d],\; \ueps\in\{0,1\}^{s_2}\setminus\{\underline{1}\}\Bigr\}\\
     \cup\; \Bigl\{\sum_{l\in[s_2]}\; \sum_{\substack{\uu\notin\CA_l:\\ |\uu|+1
    = d_{g_i}}}\;(\balpha_{1i}-\balpha_{w'_{l\uu}i})(m_l'-m_l^{\eps_l}) \uh^\uu:\; i\in [m_2+1, m],\; \ueps\in\{0,1\}^{s_2}\setminus\{\underline{1}\}\Bigr\}\\
    \cup\; \{\balpha_{1i}-\balpha_{ji}:\; i\in [m_1+1, m_2],\; j\in[0,\ell],\; d_{1j} = 0,\; d'_{1j} = i\}.
\end{gather*}

The hierarchy of the three sets above is then as follows. If $\bb_{1\underline{1}\um\um'\uh}(N) -\bb_{j\ueps\um\um'\uh}(N)$ contains a nonzero element from the first set, then the vectors $\bgamma_i(\um, \um', \uh)$ that it contributes will come from the first set. Otherwise, if it contains a nonzero element from the second set, then it will contribute vectors from the second set. If the contributions of the first two sets to $\bb_{1\underline{1}\um\um'\uh}(N) -\bb_{j\ueps\um\um'\uh}(N)$ are trivial, then only it contributes vectors from the third set. 

It is important to emphasize that nowhere above can one find {vectors} $\balpha_{1i}-\balpha_{ji}$ with $i\in[m_1]$, i.e., those corresponding to $g_i\ll \log$; this is a direct consequence of our assumption that the differences $\ba_1 - \ba_i$ for $i\in[0,\ell]\setminus\{j\}$ have a coordinate growing faster than $\log$.

\smallskip
 \textbf{Step 5: A final concatenation argument.}
\smallskip

Recalling \eqref{E: degrees of differences poly} and \eqref{E: degrees of differences nonpoly}, we deduce that the coefficients of $\bgamma_1, \ldots, \bgamma_{s_3}$ interpreted as polynomials in $\um, \um', \uh$ belong to the set  of nonzero vectors 
\begin{align*}
    &\{d_{1j}!\cdot(\bbeta_{1d_{1j}} - \bbeta_{jd_{1j}}):\; j\in[0,\ell]\setminus\{1\},\; d'_{1j}\leq m_2,\; d_{1j}>0\}\\
    \cup &\;     \{\balpha_{1d'_{1j}} - \balpha_{jd'_{1j}}:\; j\in[0,\ell]\setminus\{1\},\; m_1 < d'_{1j}\leq m_2,\; d_{1j} = 0\}\\ 
    \cup &\; \{\balpha_{1d'_{1j}} - \balpha_{jd'_{1j}}:\; j\in[0,\ell]\setminus\{1\},\; d'_{1j}> m_2\},
\end{align*}
which is precisely the set \eqref{E: set of coefficients} with the first two components swapped. The first subset of the union above comes from the leading coefficients of $\bp_1-\bp_j$ for those $j$ for which $\ba_1-\ba_j$ is a sum of a polynomial and a sublinear term; that these are the only coefficients that arise follows from the observation \eqref{E: description of A_l}. The second subset corresponds to the case when $\ba_1 - \ba_j$ has a nonpolynomial term of positive fractional degree, in which case $\balpha_{1d'_{1j}} - \balpha_{jd'_{1j}}$ is the leading coefficient of the nonpolynomial part of $\ba_1 - \ba_j$. Lastly, the third subset consists of the leading coefficients of $\ba_1 - \ba_j$ when $\ba_1 - \ba_j$ is sublinear. Taking the limsup as $M\to\infty$ in \eqref{E: final concat} followed by limsup as $H\to\infty$, and applying Proposition \ref{P: polynomial concatenation}, we deduce that
\begin{align*}
    \nnorm{f_1}_{\bgamma_1, \ldots, \bgamma_{s_4}}^+ \gg \veps^{O(1)}
\end{align*}
for some $s_{4}=O(1)$ and nonzero $\bgamma_1, \ldots, \bgamma_{s_4}$ in \eqref{E: set of coefficients}. Using Lemma \ref{L: seminorms of subgroups}, we can replace $d_{1j}!\cdot(\bbeta_{1d_{1j}} - \bbeta_{jd_{1j}})$ with simply $\bbeta_{1d_{1j}} - \bbeta_{jd_{1j}}$. Theorem \ref{T: box seminorm bound} follows by taking $s=s_4$.

\section{Upgrading to Host-Kra seminorm estimates}\label{S: HK control}

The next objective is to upgrade Theorem \ref{T: box seminorm bound} to Theorem \ref{T: HK control}, i.e. replace generalized box seminorm estimates by proper Host-Kra seminorm estimates for averages \eqref{E: general average 2} along pairwise independent Hardy sequences. We will prove this by a variant of the seminorm smoothing argument developed in \cite{FrKu22a}. The argument does not really use any particular properties of Hardy sequences other than the estimates provided by Theorem~\ref{T: box seminorm bound}, so we will carry it out in a more abstract form that makes the computations cleaner and may find further applications in the future. The setting in which we end up working is captured by the definition below.

\begin{definition}
    Let $\CQ =\{q_1, \ldots, q_m\}$ be an ordered collection of sequences $q_i: \N\to\R$ whose germs are linearly independent,\footnote{That means that for every  $N\in\N$, the restrictions of $q_1, \ldots, q_m$ to $[N, \infty)$ are $\R$-independent.} and fix an integer $0\leq m_1<m$.

    Let $a_1, \ldots, a_\ell\in \Span_\R\CQ$ be given by $a_j(n) = \sum_{i=1}^m \alpha_{ji} q_i(n)$. We define the \textit{degree} of $a_j$ via $$\deg a_j := \max\{i\in[m]:\; \alpha_{ji}\neq 0\},$$ and its \textit{complexity} by $\max\limits_{i\in[m]}\{|\alpha_{ji}|,\; 1/|\alpha_{ji}|:\; \alpha_{ji}\neq 0\}.$ We call the sequences $a_1, \ldots, a_\ell$ \textit{pairwise independent} if $(\alpha_{j(m_1+1)}, \ldots, \alpha_{jm})$ and $(\alpha_{j'(m_1+1)}, \ldots, \alpha_{j'm})$ are $\R$-linearly independent for all distinct $j,j'\in[0,\ell]$. We also denote $d_{jj'} = \max\{i\in[m]:\; \alpha_{ji}\neq\alpha_{j'i}\}$ to be the \textit{degree} of $a_j-a_{j'}$, setting $d_{jj'} = 0$ if $a_j = a_{j'}$.
    
    We call $\CQ$ \textit{good for smoothing} if for all $d,\ell, M\in\N$, $J\in\N_0$, sequences $a_1, \ldots, a_\ell$, $b_1, \ldots, b_J\in \Span_\R\CQ$ with $a_1, \ldots, a_\ell$ being pairwise independent and of complexity at most $M$, and $a_1$ having maximum degree among $a_1, \ldots, a_\ell$, systems $(X, \CX, \mu, T_1, \ldots, T_k)$, indexing tuples $\eta\in[k]^\ell$, and 1-bounded functions $f_1, \ldots, f_\ell\in L^\infty(\mu)$, we have
    \begin{multline}\label{E: gen box bound}
         \limsup_{N\to\infty}\sup_{\substack{\CD_1, \ldots, \CD_J \in \FD_d}} \sup_{|c_n|\leq 1}\norm{\E_{n\in [N]} c_{n}\cdot \prod_{j\in[\ell]}  T_{\eta_j}^{\floor{a_j(n)}} f_j \cdot \prod_{j\in[J]}\CD_j(\floor{b_j(n)})}_{L^2(\mu)}^{{O_{d,J,\ell, \CQ}(1)}}\\ 
         \ll_{d, J, k, \ell, M, \CQ} \nnorm{f_1}_{(\alpha_{1d_{10}}\be_{\eta_1})^s,\; (\alpha_{1d_{12}}\be_{\eta_1}-\alpha_{2d_{12}}\be_{\eta_2})^s,\; \ldots,\;  (\alpha_{1d_{1\ell}}\be_{\eta_1}-\alpha_{\ell d_{1\ell}}\be_{\eta_\ell})^s}^+
    \end{multline}
    for some $s = O_{d,J,\ell, \CQ}(1)$. 
\end{definition}

Whilst it is not necessary for the seminorm estimates above to be of polynomial shape (any uniform bound would do), this is what we have in our case, and so there is no harm in specifying the shape of seminorm control in our definition.

    Note that the definition above is sensitive to the ordering of $q_1, \ldots, q_m$ and the choice of $m_1.$ In our case, the ordered basis $q_1, \ldots, q_m$ of sequences will consist of 
   \begin{align}\label{E: good for smoothing family}
       g_1(n), \ldots, g_{m_2}(n), n, n^2, \ldots, n^d, g_{m_2+1}(n), \ldots, g_m(n)
   \end{align} 
   satisfying the assumptions of Theorem \ref{T: box seminorm bound}, 
   and pairwise independence of two sequences in the real span of \eqref{E: good for smoothing family} amounts to saying that any nontrivial real combination of $g_{m_1+1}(n), \ldots, g_m(n), n, \ldots, n^d$ grows faster than log. 

       Lastly, we remark that in contrast to the arguments from the previous sections, here we do not allow $f_2, \ldots, f_\ell$ to depend on $N$. This will be important later on, as in the proof of Theorem \ref{T: HK control}, we will transfer the control over the average between the functions $f_1, \ldots, f_\ell$, and for this to make sense, we need them to be fixed.
   
    Theorem \ref{T: box seminorm bound} and Proposition \ref{P: sublinear} then imply the following.
   \begin{theorem}[Hardy sequences are good for smoothing]\label{T: box seminorm bound rephrased}
			Let $d, m\in\N$, $0\leq m_1\leq m_2\leq m$, and  $g_1,...,g_m\in \mathcal{H}$ be Hardy functions satisfying the following growth conditions:
   \begin{enumerate}
       \item $1 \prec g_1(t) \cdots \prec g_m(t)\ll t^d;$ 
       \item $g_i(t)\ll\log t$ iff $i\in[m_1]$;
       \item $g_i(t)\prec t^\delta$ for every $\delta>0$ (i.e. $\fracdeg g_i = 0$) iff $i\in[m_2]$.
   \end{enumerate}
   Then the ordered family \eqref{E: good for smoothing family} is good for smoothing.
   \end{theorem}
   \begin{proof}
       Let $a_1, \ldots, a_\ell$ be linear combinations of sequences in \eqref{E: good for smoothing family}, and suppose that $a_1$ has maximum degree among $a_1, \ldots, a_\ell$ with respect to the ordering of \eqref{E: good for smoothing family}. We want to establish that the estimate \eqref{E: gen box bound} always holds. If $\fracdeg a_1 > 0$, then \eqref{E: gen box bound} holds by Theorem \ref{T: box seminorm bound}. Otherwise all of $a_1, \ldots, a_\ell$ are subfractional by the assumption of the maximum degree of $a_1$, and so the desired estimate follows from Proposition \ref{P: sublinear}.
   \end{proof}

    The purpose of this section is to prove the following Host-Kra seminorm control that implies Theorem \ref{T: HK control} as a special case.
    \begin{proposition}\label{P: smoothing}
        Let $d, k, \ell, m, M\in\N$, $J\in\N_0$ and $\eta\in[k]^\ell$. Suppose that the family of sequences $\CQ$
        is {good for smoothing}, and the sequences $a_1, \ldots, a_\ell\in\Span_\R\CQ$ are pairwise independent and of complexity at most $M$. Then there exists a positive integer $s=O_{d,J,\ell, \CQ}(1)$ such that for all systems $(X, \CX, \mu, T_1, \ldots, T_k)$, 1-bounded functions $f_1, \ldots, f_\ell\in L^\infty(\mu)$ and sequences $b_1, \ldots, b_J\in\Span_\R\CQ$, we have 
\begin{multline}\label{E: HK bound}
         \limsup_{N\to\infty}\sup_{\substack{
         \CD_1, \ldots, \CD_J \in \FD_d}} \sup_{|c_n|\leq 1}\norm{\E_{n\in [N]} c_{n}\cdot \prod_{j\in[\ell]}  T_{\eta_j}^{\floor{a_j(n)}} f_j \cdot \prod_{j\in[J]}\CD_j(\floor{b_j(n)})}_{L^2(\mu)}^{{O_{d,J,\ell, \CQ}(1)}}\\ 
         \ll_{d, J, k, \ell, M, \CQ} \nnorm{f_1}_{s, T_{\eta_1}}.
    \end{multline}
    \end{proposition}

    In proving Proposition \ref{P: smoothing}, we will use the formalism developed in \cite{FrKu22a}. Consider the average
\begin{align}\label{general average}
	\E_{n\in [N]} c_n\cdot \prod_{j\in[\ell]}T_{\eta_j}^{\floor{a_j(n)}}f_j \cdot \prod_{j\in[J]}\CD_{j}(\floor{b_j(n)}).
\end{align}
We let $\ell$ be its \emph{length}, $m^* = \max_j\deg a_j$ be its \textit{degree}, and $\eta=(\eta_1,\ldots, \eta_\ell)\in[k]^\ell$ be its \emph{indexing tuple}.
Furthermore, we define
\begin{align*}
	\FL := \{j\in[\ell]\colon \ \deg a_j = m^*\}
\end{align*}
to be the set of indices corresponding to the sequences $a_1, \ldots, a_m$ of maximum degree.
The \emph{type} of \eqref{general average} is the tuple $w := (w_1, \ldots, w_k)\in[0, \ell]^k$ in which 
\begin{align*}
	w_l := |\{j\in \FL\colon \ \eta_j = l\}| = |\{j\in[\ell]\colon \ \eta_j = l,\ \deg a_j = m^*\}|,
\end{align*}
and it represents the number of times the transformation $T_l$ appears in the average with a sequence of maximum degree. For instance, given the ordered family $\CQ = \{n, n^{1/2}, n^{3/2}, n^{3/2}\log n\}$ (which respects the ordering that we have in our applications since we consider nonpolynomial functions of positive fractional degree before polynomials), the average
\begin{align*}
	\E_{n\in[N]}\, T_1^{\floor{\sqrt{2} n^{3/2}}}f_1\cdot T_1^{\floor{\sqrt{3}n^{1/2}}}f_2 \cdot T_2^{\floor{n^{3/2} + n}}f_3\cdot T_1^n f_4 \cdot T_2^{\floor{n^{3/2} + \sqrt{5} n}}f_5 \cdot \CD_{1}\Bigbrac{\floor{n^{3/2}\log n}}
\end{align*}
has length 5, degree 3, indexing tuple (1, 1, 2, 1, 2)  and type $(1,2)$ because the transformation $T_1$ appears only once with the sequence $n^{3/2}$ of maximum degree (at the index $j=1$) while $T_2$ has an iterate of maximum degree twice, at $j=3, 5$. It does not matter that the sequence $n^{3/2}\log n$ in the dual term has higher degree since our notion of type completely ignores how many dual terms we have and what sequences they involve.

The complexity of an average is measured by its length and type. Types $w$ with a fixed length $|w|:= w_1 + \cdots + w_{k}$ (which coincides with the length of an average of type $w$) are ordered by their variance: we set $w'<w$ if $w'$ has strictly higher variance than $w$. We remark here that comparing variances is the same as comparing second moments since we only compare tuples of the same lengths; hence equivalently $w'<w$ if $${w'_1}^2 + \cdots + {w'_k}^2 > w_1^2 + \cdots + w_k^2.$$  However, it is the variance rather than second moment that carries the intuition for why we take this particular ordering. Indeed, we think of the simplest averages as those in which all the leading terms correspond to the same transformation because for those averages, we can obtain a Host-Kra seminorm control immediately from the definition of being good for smoothing (see Proposition \ref{P: bounds for basic type} below). It is an easy exercise to show that the types of such averages maximize variance. By contrast, the most difficult averages to deal with are those with minimum variance, i.e. those in which each highest degree term corresponds to a different transformation.

In the induction procedure, the type of an average will be lowered through the following \textit{type operation}. For distinct integers $i_1,i_2\in [\ell]$ with $i_1\in\supp(w)$ (which we define to be $\supp(w) = \{l\in[\ell]:\ w_l>0\}$), we define 
\begin{align*}
	(\sigma_{i_1 i_2}w)_l := \begin{cases} w_l,\; &l \neq i_1, i_2\\
		w_{i_1}-1,\; &l = i_1\\
		w_{i_2} + 1,\; &l = i_2.
	\end{cases}.
\end{align*}
For instance, $\sigma_{12}(2, 3, 7) = (1, 4, 7)$. A simple computation shows that if $w_{i_1}\leq w_{i_2}$, then $\sigma_{i_1 i_2}w > w$, in which case $\sigma_{i_1 i_2}$ \textit{lowers the type} of $w$. Using this operation, we can for instance obtain the following chains of types:
\begin{align*}
	(4, 0, 0) < (3,1,0) < (2, 2, 0) < (2, 1, 1) \quad \textrm{and}\quad
	(0, 4, 0) < (1, 3, 0) < (2, 2, 0) < (2, 1, 1).
\end{align*}
We remark that this is not how the ordering on types of fixed length is defined in \cite{FrKu22a}, but the notion of order from \cite{FrKu22a} is a suborder of the current one in that $w'<w$ whenever $w'\prec w$, where $\prec$ is the order from \cite{FrKu22a}. Taking the (slightly more complicated) notion of type from \cite{FrKu22a} would thus work equally well.

The ordering $<$ on types organizes the way in which we pass from averages of indexing tuple $\eta$ to averages of indexing tuple $\eta'$. After applying the type operations finitely many types to a type $w'$, we arrive at a type $w$ with the property that $w_l = 0$ for all but one value $l\in\FL$; we call these types \textit{basic}.  For instance, $(4, 0, 0)$ and $(0, 4, 0)$ in the example above are basic types but $(3, 1, 0), (1, 3, 0), (2,2,0), (2,1,1)$ are not. Basic types are named so because they will serve as the basis for our induction procedure in that whenever the average has a basic type and $a_1$ is a sequence of maximum degree, we immediately obtain the claimed result.
\begin{proposition}\label{P: bounds for basic type}
    Let $d, k, \ell, m, M\in\N$ and $J\in\N_0$. Suppose that the family of sequences $\CQ$ is {good for smoothing}, and the sequences $a_1, \ldots, a_\ell\in \Span_\R\CQ$ are pairwise independent and of complexity at most $M$. Suppose that $a_1$ is a sequence of maximum degree and the tuple $\eta\in[k]^\ell$ is basic.\footnote{By abuse of notation, we say that the indexing tuple of an average is basic if the average itself is also basic. This is an abuse of notation because the type depends on $\eta|_\FL$ rather than $\eta$, and to know $\FL$, we also need to know the sequences $a_1, \ldots, a_\ell$.} Then there exists a positive integer $s=O_{d,J,\ell, \CQ}(1)$ such that for all systems $(X, \CX, \mu, T_1, \ldots, T_k)$,  1-bounded functions $f_1, \ldots, f_\ell\in L^\infty(\mu)$ and sequences $b_1, \ldots, b_J\in\Span_\R\CQ$, the bound \eqref{E: HK bound} holds. 
\end{proposition}
\begin{proof}
Let $a_j(n) = \sum_{i=1}^m \alpha_{ji} q_i(n)$ for each $j\in[\ell]$.
    From the property of being good for smoothing and the assumption of maximum degree of $a_1$, we immediately get the bound \eqref{E: gen box bound}. It remains to show that each vector
    \begin{align}\label{E: coefficientvectors}
        \alpha_{1d_{10}}\be_{\eta_1},\; \alpha_{1d_{12}}\be_{\eta_1}-\alpha_{2d_{12}}\be_{\eta_2},\; \ldots,\;  \alpha_{1d_{1\ell}}\be_{\eta_1}-\alpha_{\ell d_{1\ell}}\be_{\eta_\ell}
    \end{align}
    appearing in the generalized box seminorm from \eqref{E: gen box bound} is a nonzero real multiple of $\be_{\eta_1}$; if that is the case, then the result will follow from Lemma \ref{L: dilating seminorms}. Since the average is of basic type and $a_1$ has maximum degree, one of the two things can happen for each $j\in[0,\ell]\setminus\{1\}$: either $\deg a_j < \deg a_1$, in which case $$\alpha_{1d_{1j}} \be_{\eta_1} - \alpha_{jd_{1j}} \be_{\eta_j} = \alpha_{1d_{1j}} \be_{\eta_1},$$ or $\deg a_j = \deg a_1$, in which case $\eta_j = \eta_1$, so that  $$\alpha_{1d_{1j}} \be_{\eta_1} - \alpha_{jd_{1j}} \be_{\eta_j} = (\alpha_{1d_{1j}} - \alpha_{jd_{1j}}) \be_{\eta_1}.$$ One way or another, we end up with a nonzero real multiple of $\be_{\eta_1}$.
\end{proof}

When the type of an average is not basic, we will prove Proposition \ref{P: smoothing} by inducting on the type and length of the average, reducing both until we reach an average of basic type. The length of the average in the induction process will be reduced using the proposition below. 
\begin{proposition}\label{P: iterated smoothing pairwise}
        Let $d, k, \ell, m, M\in\N$ and $J\in\N_0$. Suppose that the family of sequences $\CQ$ is {good for smoothing}, and the sequences $a_1, \ldots, a_\ell\in\Span_\R\CQ$ are pairwise independent and of complexity at most $M$. Suppose that $a_1$ is a sequence of maximum degree, $\eta\in[k]^\ell$ is a tuple of type $w$, and $w_{\eta_1}$ is minimal among the nonzero coordinates of $w$. Then there exists a positive integer $s=O_{d,J,\ell, \CQ}(1)$ such that for all systems $(X, \CX, \mu, T_1, \ldots, T_k)$, 1-bounded functions $f_1, \ldots, f_\ell\in L^\infty(\mu)$ and sequences $b_1, \ldots, b_J\in\Span_\R\CQ$, the bound \eqref{E: HK bound} holds. 
\end{proposition}

Proposition \ref{P: iterated smoothing pairwise} will be derived from an iterated application of the result below. 
\begin{proposition}  \label{P: smoothing pairwise}
        Let $d, k, \ell, m, M\in\N$ and $J\in\N_0$. Suppose that the family of sequences $\CQ = \{q_1, \ldots, q_m\}$ is {good for smoothing}, and the sequences $a_1, \ldots, a_\ell\in\Span_\R\CQ$ given by $a_j(n) = \sum_{i\in[m]}\alpha_{ji}q_i(n)$ 
    are pairwise independent and of complexity at most $M$. Suppose that $a_1$ is a sequence of maximum degree, $\eta\in[k]^\ell$ is a tuple of type $w$, and $w_{\eta_1}$ is minimal among the nonzero coordinates of $w$. Then for all vectors  $\bv_1, \ldots, \bv_{s+1}$ in \eqref{E: coefficientvectors} there exists a positive integer $s'=O_{d,J,\ell, s, \CQ}(1)$ with the following property: if $(X, \CX, \mu, T_1, \ldots, T_k)$ is a system and 
        \begin{multline}\label{E: input smoothing bound}
         \limsup_{N\to\infty}\sup_{\substack{\CD_1, \ldots, \CD_J \in \FD_d}} \sup_{|c_n|\leq 1}\norm{\E_{n\in [N]} c_{n}\cdot \prod_{j\in[\ell]}  T_{\eta_j}^{\floor{a_j(n)}} f_j \cdot \prod_{j\in[J]}\CD_j(\floor{b_j(n)})}_{L^2(\mu)}^{{O_{d,J,\ell, s, \CQ}(1)}}\\ 
         \ll_{d, J, k, \ell, M, s, \CQ} \nnorm{f_1}_{\bv_1, \ldots, \bv_{s+1}}
    \end{multline}
    holds for all 1-bounded functions $f_1, \ldots, f_\ell\in L^\infty(\mu)$ and sequences $b_1, \ldots, b_J\in\Span_\R\CQ$, then 
            \begin{multline}\label{E: output smoothing bound}
         \limsup_{N\to\infty}\sup_{\substack{\CD_1, \ldots, \CD_J \in \FD_d}} \sup_{|c_n|\leq 1}\norm{\E_{n\in [N]} c_{n}\cdot \prod_{j\in[\ell]}  T_{\eta_j}^{\floor{a_j(n)}} f_j \cdot \prod_{j\in[J]}\CD_j(\floor{b_j(n)})}_{L^2(\mu)}^{{O_{d,J,\ell, s, \CQ}(1)}}\\ 
         \ll_{d, J, k, \ell, M, s, \CQ} \nnorm{f_1}_{\bv_1, \ldots, \bv_{s}, \be_{\eta_1}^{s'}}
    \end{multline}
    also holds for all 1-bounded functions $f_1, \ldots, f_\ell\in L^\infty(\mu)$ and sequences $b_1, \ldots, b_J\in\Span_\R\CQ$.
\end{proposition}
\begin{proof}[Proof of Proposition \ref{P: iterated smoothing pairwise} using Proposition \ref{P: smoothing pairwise}]
    By Lemma~\ref{L: plus vs normal seminorm} and the definition of being good for smoothing, we have the bound \eqref{E: gen box bound} (Lemma \ref{L: plus vs normal seminorm} allows us to use ``unplused'' seminorms rather than the ``plused'' ones at the cost of increasing the degree by 1). The bound \eqref{E: HK bound} then follows by iteratively applying Proposition \ref{P: smoothing pairwise}, at each step replacing one vector in the seminorm in \eqref{E: gen box bound} different from $\be_{\eta_1}$ by a bounded number of copies of $\be_{\eta_1}$.
\end{proof}

We are now in the position to set out the main steps in the proof of Proposition~\ref{P: smoothing} for an average of length $\ell$ and type $w$:
\begin{enumerate}
    \item We find an index $i\in[\ell]$ so that $a_i$ is a sequence of maximal degree and $w_{\eta_i}$ minimizes the nonzero coordinates of $w$. Applying Proposition \ref{P: iterated smoothing pairwise} with the role of 1 played by $i$, we obtain Host-Kra seminorm control on index $i$. By standard maneuvers, this allows us to replace $f_i$ by a bounded degree dual function of $T_{\eta_i}$, and thus we reduce to an average of length $\ell-1$. Then Proposition \ref{P: smoothing} for the average of length $\ell$ and type $w$ follows by invoking Proposition \ref{P: smoothing} for averages of length $\ell-1$. 
    \item To prove Proposition \ref{P: smoothing pairwise} (from which Proposition \ref{P: iterated smoothing pairwise} follows by the argument above), we argue in two steps, which in \cite{FrKu22a} have been called \textit{ping} and \textit{pong}. In the \textit{ping} step, we invoke Proposition \ref{P: smoothing} for an average of length $\ell$ and type $w'<w$; in the \textit{pong} step, we invoke Proposition \ref{P: smoothing} for an average of length $\ell-1$.
\end{enumerate}
Thus, in the inductive step we invoke Proposition \ref{P: iterated smoothing pairwise} for averages of either smaller length or the same length but lower type. After finitely many iterations, we therefore land in the case of averages of length 1 or basic type. Since averages of length 1 are automatically of basic type, the base case is fully covered by Proposition \ref{P: bounds for basic type}.

The proofs of  Propositions \ref{P: smoothing} and \ref{P: smoothing pairwise} resemble closely the proofs from \cite{FrKu22a} at the level of both the general setup and specific maneuvers. The main challenge in adapting the arguments from \cite{FrKu22a} is the inadequacy of certain technical tools from \cite{FrKu22a} for our setting. Specifically, the proofs from \cite{FrKu22a} utilize crucially the fact that for $\alpha_j, \alpha_{j'}\in\Z$, a function has a large box seminorm $\nnorm{\cdot}_{\alpha_j \be_j - \alpha_{j'}\be_{j'}}$ if and only if it correlates with a $T_j^{\alpha_j}T_{j'}^{-\alpha_{j'}}$-invariant function. This statement does not make sense if $\alpha_j$ or $\alpha_{j'}$ are not integers, so we need instead to develop an approximate version of the inverse theorem for degree-1 generalized box seminorms (Lemma \ref{L: U^1 inverse}) and a way of applying it to the problem (Lemma~\ref{L: approximate invariance}).

\subsection{Preliminary lemmas}
Before we prove Propositions \ref{P: smoothing} and \ref{P: smoothing pairwise}, we gather several lemmas needed in the proofs. The first one is a straightforward modification of \cite[Proposition 4.2]{Fr21}, phrased in a way that allows an immediate application. It says that at the cost of a small loss in the lower bound, we can replace an arbitrary function $f_1$ by one that is structured in an appropriate sense. 
\begin{lemma}\label{L: dual replacement}
	Let $d,\ell\in\N$, $J\in\N_0$ be numbers, $a_1,\ldots, a_\ell, b_1, \ldots, b_J\colon \N\to \R$ be sequences, $(X, \CX, \mu,T_1,\ldots, T_k)$ be a system, $\eta\in[k]^\ell$ be an indexing tuple and  $f_1, \ldots, f_\ell\in L^\infty(\mu)$ be $1$-bounded functions. For every $\delta>0$ there exists an increasing sequence of integers $(N_i)_{i\in\N}$, 1-bounded weights $(c'_n)_{n\in\N}$, $1$-bounded functions $(g_i)_{i\in\N}\subseteq L^\infty(\mu)$, and $\CD_1', \ldots, \CD_J'\in\FD_d$ such that 
 \begin{multline*}
     	\limsup_{N\to\infty}\sup_{\substack{\CD_1, \ldots, \CD_J \in \FD_d}} \sup_{|c_n|\leq 1}\norm{\E_{n\in [N]} c_n\cdot \prod_{j\in[\ell]}T_{\eta_j}^{\floor{a_j(n)}}f_j \cdot \prod_{j\in[J]}\CD_{j}(\floor{b_j(n)})}_{L^2(\mu)}^4-\delta\\
      \leq\limsup_{N\to\infty}\sup_{\substack{\CD_1, \ldots, \CD_J \in \FD_d}} \sup_{|c_n|\leq 1}\norm{\E_{n\in [N]} c_n\cdot T_{\eta_1}^{\floor{a_1(n)}}\tilde{f}_1\cdot \prod_{j\in[2,\ell]}T_{\eta_j}^{\floor{a_j(n)}}f_j \cdot \prod_{j\in[J]}\CD_{j}(\floor{b_j(n)})}_{L^2(\mu)},
 \end{multline*}
 where $\tilde{f}_1$ is the weak limit
 \begin{multline}\label{E: tilded}
     \tilde{f}_1 = \lim_{i\to\infty} \E_{n\in [N_i]} \overline{c'_n}\cdot T_{\eta_1}^{-\floor{a_1(n)}}g_i\\ \prod_{j\in[2, \ell]}T_{\eta_1}^{-\floor{a_1(n)}} T_{\eta_j}^{\floor{a_j(n)}}\overline{f_j} \cdot \prod_{j\in[J]}T_{\eta_1}^{-\floor{a_1(n)}} \overline{\CD'_{j}(\floor{b_j(n)})}.
 \end{multline}
\end{lemma}
\begin{proof}
    Let
    \begin{align*}
        \veps = \limsup_{N\to\infty}\sup_{\substack{\CD_1, \ldots, \CD_J \in \FD_d}} \sup_{|c_n|\leq 1}\norm{\E_{n\in [N]} c_n\cdot \prod_{j\in[\ell]}T_{\eta_j}^{\floor{a_j(n)}}f_j \cdot \prod_{j\in[J]}\CD_{j}(\floor{b_j(n)})}_{L^2(\mu)},
    \end{align*}
    and fix $\delta>0$. There exist $\CD_1', \ldots, \CD_J', (c'_n)$ such that 
        \begin{align*}
        \limsup_{N\to\infty}\norm{\E_{n\in [N]} c_n'\cdot \prod_{j\in[\ell]}T_{\eta_j}^{\floor{a_j(n)}}f_j \cdot \prod_{j\in[J]}\CD'_{j}(\floor{b_j(n)})}_{L^2(\mu)}^2 \geq \veps^2 - \delta.
    \end{align*}
    Defining
    \begin{align*}
        F_N = \E_{n\in [N]} c'_n\cdot \prod_{j\in[\ell]}T_{\eta_j}^{\floor{a_j(n)}}f_j \cdot \prod_{j\in[J]}\CD_j'(\floor{b_j(n)}),
    \end{align*}
    expanding the inner product and composing the integral with $T_{\eta_1}^{-\floor{a_1(n)}}$, we get that 
    \begin{multline*}
        \norm{\E_{n\in [N]} c_n'\cdot \prod_{j\in[\ell]}T_{\eta_j}^{\floor{a_j(n)}}f_j \cdot \prod_{j\in[J]}\CD'_{j}(\floor{b_j(n)})}_{L^2(\mu)}^2\\ = \int \overline{f_1} \cdot \E_{n\in [N]} \overline{c'_n}\cdot T_{\eta_1}^{-\floor{a_1(n)}}F_N\cdot \prod_{j\in[2, \ell]}T_{\eta_1}^{-\floor{a_1(n)}} T_{\eta_j}^{\floor{a_j(n)}}\overline{f_j} \cdot \prod_{j\in[J]}T_{\eta_1}^{-\floor{a_1(n)}} \overline{\CD'_{j}(\floor{b_j(n)})}\; d\mu.
    \end{multline*}
    By weak compactness, there exists a sequence $N_i\to\infty$ along which the average over $n\in[N_i]$ converges weakly; we call the weak limit $\tilde{f}_1$ and set $g_i = F_{N_i}$. Then $\int f_1 \cdot \tilde{f}_1\; d\mu \geq \veps^2-\delta$. 
    Applying the Cauchy-Schwarz inequality, using the 1-boundedness of $f_1$, and redefining $\delta$ appropriately, we get $\int \tilde{f}_1 \cdot \overline{\tilde{f}_1}\; d\mu \geq \veps^4-\delta$. The result follows on expanding the second $\tilde{f}_1$, composing back with $T_{\eta_1}^{\floor{a_1(n)}}$, applying the Cauchy-Schwarz inequality and taking the sups and limsup. 
\end{proof}

\begin{definition}
    Given a system $(X, \CX, \mu, T_1, \ldots, T_k)$ and a finitely generated subgroup $G\subseteq \R^k$, we call a function $u\in L^\infty(\mu)$ \textit{approximately $G$-invariant} if it takes the form
    \begin{align*}
        u = \E_{\bm\in G}T^{\floor{\bm}}f.   
    \end{align*}
    for some $f\in L^\infty(\mu)$. If $G = \langle \bg\rangle$, then we also call $u$ \emph{approximately $\bg$-invariant}.
\end{definition}

Such functions naturally appear in inverse theorems for degree-1 generalized box seminorms.
\begin{lemma}[Degree-1 inverse theorem]\label{L: U^1 inverse}
Let $(X, \CX, \mu, T_1, \ldots, T_k)$ be a system  and $G\subseteq \R^k$ be a finitely generated subgroup. For every $f\in L^\infty(\mu)$ there exists an approximately $G$-invariant function $u\in L^\infty(\mu)$ satisfying
    \begin{align}
        \nnorm{f}_{G}^2 = \int_X f \cdot u\, d\mu.
    \end{align}
\end{lemma}
\begin{proof}
By the definition of degree 1-box norms, we have
        \begin{align*}
        \nnorm{f}_{G}^2 &= \int_X \E_{\bm,\bm'\in G}T^{\floor{\bm}}f\cdot T^{\floor{\bm'}}\overline{f}\, d\mu\\
        &= \int_X f\cdot \E_{\bm'\in G}T^{\floor{\bm'}} \Bigbrac{\E_{\bm\in G}T^{-\floor{\bm}}\overline{f}}\, d\mu,\\
    \end{align*}
    where in passing to the second line, we compose the integral with $T^{-\floor{\bm}}$ and split the limit using Lemma \ref{L: convergence}. The result follows immediately.
\end{proof}

The following lemma explains the origin of the name for approximately invariant functions and illustrates the way in which such functions will be useful for us.
\begin{lemma}[Approximate invariance]\label{L: approximate invariance}
    Let $k,s\in\N$, $(X, \CX, \mu, T_1, \ldots, T_k)$ be a system, $j,j'\in[0,k]$ be indices, and $\alpha_j, \alpha_{j'}\in\R$. If $u\in L^\infty(\mu)$ is a product of $s$ 1-bounded approximately $(\alpha_j \be_j - \alpha_{j'} \be_{j'})$-invariant functions, then there exists 1-bounded $u'\in L^\infty(\mu)$ such that for every $N\in\N$, functions $A_1, \ldots, A_N\in L^\infty(\mu)$ and sequence $a:\N\to\R$, we have
    \begin{align*} 
        \sup_{|c_n|\leq 1}\norm{\E_{n\in[N]}c_n \cdot A_n \cdot T_j^{\floor{\alpha_j a(n)}}u}_{L^2(\mu)}\ll_{s} \sup_{|c_n|\leq 1}\norm{\E_{n\in[N]}c_n \cdot A_n \cdot T_{j'}^{\floor{\alpha_{j'} a(n)}}u'}_{L^2(\mu)}
    \end{align*}
\end{lemma}
\begin{proof}
Let
\begin{align*}
    u = \prod_{i\in[s]}\lim_{M\to\infty} \E_{m\in[\pm M]}T_j^{\floor{\alpha_j m}}T_{j'}^{\floor{-\alpha_{j'} m}}f_i.
\end{align*}
Then 
\begin{align*}
    T_j^{\floor{\alpha_j a(n)}}u &= \prod_{i\in[s]} T_j^{\floor{\alpha_j a(n)}} \lim_{M\to\infty} \E_{m\in[\pm M]}T_j^{\floor{\alpha_j m}}T_{j'}^{\floor{-\alpha_{j'} m}}f_i\\
    &= \prod_{i\in[s]}\lim_{M\to\infty} \E_{m\in[\pm M]}T_j^{\floor{\alpha_j (m+\floor{a(n)})}+r_{m,n}}T_{j'}^{\floor{-\alpha_{j'} m}}f_i\\
    &= \prod_{i\in[s]}\lim_{M\to\infty} \E_{m\in[\pm M]}T_j^{\floor{\alpha_j m}+r'_{m,n}}T_{j'}^{\floor{-\alpha_{j'} (m-\floor{a(n)})}}f_i\\
    &= T_{j'}^{\floor{\alpha_{j'} a(n)}}\brac{\prod_{i\in[s]}\lim_{M\to\infty} \E_{m\in[\pm M]}T_j^{\floor{\alpha_j m}+r'_{m,n}}T_{j'}^{\floor{-\alpha_{j'} m}+r''_{m,n}}f_i}
\end{align*}
for some integers $r_{m,n}, r'_{m,n}, r''_{m,n} = O(1)$ that come from commuting sums and integer parts. Plugging this formula into the original $L^2(\mu)$ norm, applying the triangle inequality and using Lemma \ref{L: errors} to make $r_{m,n}, r'_{m,n}, r''_{m,n}$ independent of $n$, we get
    \begin{multline*}
        \sup_{|c_n|\leq 1}\norm{\E_{n\in[N]}c_n \cdot A_n \cdot T_j^{\floor{\alpha_j a(n)}}u}_{L^2(\mu)}
        \ll_s \limsup_{M\to\infty}\E_{m_1, \ldots, m_s\in[\pm M]}\\
        \sup_{|c_n|\leq 1}\norm{\E_{n\in[N]}c_n \cdot A_n \cdot T_{j'}^{\floor{\alpha_{j'} a(n)}}\brac{\prod_{i\in[s]} T_j^{\floor{\alpha_j m_i}+r'_{m_i}}T_{j'}^{\floor{-\alpha_{j'} m_i}+r''_{m_i}}f_i}}_{L^2(\mu)}
    \end{multline*}
    for some $r'_{m_i}, r''_{m_i}\in\Z$. Upon pigeonholing in $m_1, \ldots, m_s$, the result follows with $$u' := \prod_{i\in[s]} T_j^{\floor{\alpha_j m_i}+r'_{m_i}}T_{j'}^{\floor{-\alpha_{j'} m_i}+r''_{m_i}}f_i.$$
\end{proof}

\begin{lemma}[Dual-difference interchange]\label{L: dual-difference interchange}
	Let $(X, \CX, \mu,T_1, \ldots, T_k)$ be a system,  $s,$ $s'\in \N$, $\bv_1, \ldots, \bv_{s+1}\in\R^k$ be vectors,
	$(f_{n,i})_{n,i\in\N}\subseteq L^\infty(\mu)$ be 1-bounded, and $\tilde{f}\in L^\infty(\mu)$ be defined by
	\begin{align*}
        \tilde{f}:=\, \lim_{i\to\infty}\E_{n\in [N_i]}\, f_{n,i},
	\end{align*}
	for some $N_i\to\infty$, where the average is assumed to converge weakly. Then 
	\begin{multline*}
            \fnnorm{\tilde{f}}_{\bv_1, \ldots, \bv_s, \bv_{s+1}^{s'}}^{O_{s,s'}(1)}
            \leq \liminf_{M\to\infty}\E_{\um, \um',\um'',\um'''\in [\pm M]^s}\limsup_{i\to\infty} \E_{n\in[N_i]}\\
             \int\Delta_{\substack{(\floor{\bv_1 m_1'} - \floor{\bv_1 m_1}, \floor{\bv_1 m_1'''} - \floor{\bv_1 m_1''}), \ldots,\\ (\floor{\bv_s m_s'} - \floor{\bv_s m_s}, \floor{\bv_s m_s'''} - \floor{\bv_s m_s''})}}\; f_{n,i} \cdot u_{\um\um'\um''\um'''} \, d\mu
    \end{multline*}
    for 1-bounded functions $u_{\um\um'\um''\um'''}\in L^\infty(\mu)$ which can be either of the following:
	\begin{enumerate}
            \item\label{it: ddi s'} if $s'\geq 1$, then  $u_{\um\um'\um''\um'''}$ is a product of $2^{s}$ 1-bounded dual functions of level $s'$ with respect to ${\bv_{s+1}}$;\
	    \item\label{it: ddi 1} if $s'=1$, then we can alternatively take $u_{\um\um'\um''\um'''}$ to be a product of $2^s$ 1-bounded approximately ${\bv_{s+1}}$-invariant functions.
	\end{enumerate}
\end{lemma}
We note two alternative conclusions in the case $s'=1$.

In proving Lemma \ref{L: dual-difference interchange}, we will use the following version of the Gowers-Cauchy-Schwarz inequality.
\begin{lemma}[Twisted Gowers-Cauchy-Schwarz inequality]\label{L: GCS twisted} Let $(X, \CX, \mu,T_1, \ldots, T_k)$ be a system, $s\in \N$, $\bv_1, \ldots, \bv_s\in\R^k$ be vectors, and  $(f_\ueps)_{\ueps\in\{0,1\}^s},$ $ (u_{\um\um'})_{\um, \um'\in\N^s}$  be tuples of 1-bounded functions in $L^\infty(\mu)$. Then for every $M\in \N$, we have
\begin{align*}
&\abs{\E_{\um,\um'\in [\pm M]^s}\,  \int \prod_{\ueps\in \{0,1\}^s} T^{\floor{\bv_1 m_1^{\eps_1}} + \cdots + \floor{\bv_s m_s^{\eps_s}}} f_\ueps\cdot u_{\um\um'}\, d\mu}^{2^s}\\
&\leq \E_{\um, \um',\um'',\um'''\in [\pm M]^s}\,  \int \Delta_{\substack{(\floor{\bv_1 m_1'} - \floor{\bv_1 m_1}, \floor{\bv_1 m_1'''} - \floor{\bv_1 m_1''}), \ldots,\\ (\floor{\bv_s m_s'} - \floor{\bv_s m_s}, \floor{\bv_s m_s'''} - \floor{\bv_s m_s''})}}\; f_{\underline{1}} \cdot u_{\um\um'\um''\um'''}\, d\mu,
\end{align*}
where $u_{\um\um'\um''\um'''}$ is a product of $2^s$ functions from the set $$\{T^\bn \CC^{\eps} u_{\uk, \uk'}:\; k_i\in \{m_i, m_i''\},\; k'_i\in \{m_i', m_i'''\},\; \eps\in\{0,1\},\; \bn\in\Z^k\}.$$
\end{lemma}
\begin{proof}
    We prove the case $s=2$, and the general case follows similarly. Upon composing with $T^{-\floor{\bv_1 m_1}}$, the expression inside the absolute value equals
    \begin{multline*}
        \E_{m_2,m_2'\in[\pm M]}\int  T^{\floor{\bv_2 m_2}}f_{00}\cdot T^{\floor{\bv_2 m_2'}}f_{01}\\
         \E_{m_1, m_1'\in[\pm M]} T^{\floor{\bv_1 m_1'}-\floor{\bv_1 m_1}}\brac{ T^{\floor{\bv_2 m_2}}f_{10}\cdot T^{\floor{\bv_2 m_2'}}f_{11}}\cdot T^{-\floor{\bv_1 m_1}}u_{m_1m_2m_1'm_2'}\; d\mu.
    \end{multline*}
    By the Cauchy-Schwarz inequality, its square is bounded by 
    \begin{multline*}
        \E_{m_2,m_2'\in[\pm M]}\int\E_{m_1, m_1', m_1'', m_1'''\in[\pm M]} \\
        \Delta_{(\floor{\bv_1 m_1'}-\floor{\bv_1 m_1}, \floor{\bv_1 m_1'''}-\floor{\bv_1 m_1''})}\brac{ T^{\floor{\bv_2 m_2}}f_{10}\cdot T^{\floor{\bv_2 m_2'}}f_{11}}\\
         T^{-\floor{\bv_1 m_1}}u_{m_1m_2m_1'm_2'}\cdot T^{-\floor{\bv_1 m_1''}}\overline{u_{m_1''m_2m_1'''m_2'}}\; d\mu.
    \end{multline*}
    Composing with $T^{-\floor{\bv_2 m_2}}$ and changing the order of summations, this equals
    \begin{multline*}
        \E_{m_1, m_1', m_1'', m_1'''\in[\pm M]}\int  \Delta_{(\floor{\bv_1 m_1'}-\floor{\bv_1 m_1}, \floor{\bv_1 m_1'''}-\floor{\bv_1 m_1''})}f_{10}\\ 
        \E_{m_2,m_2'\in[\pm M]} \Delta_{(\floor{\bv_1 m_1'}-\floor{\bv_1 m_1}, \floor{\bv_1 m_1'''}-\floor{\bv_1 m_1''})} T^{\floor{\bv_2 m_2'}-\floor{\bv_2 m_2}}f_{11}\\
         T^{-\floor{\bv_1 m_1}-\floor{\bv_2 m_2}}u_{m_1m_2m_1'm_2'}\cdot T^{-\floor{\bv_1 m_1''}-\floor{\bv_2 m_2}}\overline{u_{m_1''m_2m_1'''m_2'}}\; d\mu.
    \end{multline*}
    An application of the Cauchy-Schwarz inequality allows us to bound the square of the above by
    \begin{multline*}
        \E_{m_1, m_1', m_1'', m_1'''\in[\pm M]}\int \left|\E_{m_2,m_2'\in[\pm M]} \Delta_{(\floor{\bv_1 m_1'}-\floor{\bv_1 m_1}, \floor{\bv_1 m_1'''}-\floor{\bv_1 m_1''})} T^{\floor{\bv_2 m_2'}-\floor{\bv_2 m_2}}f_{11}\right.\\
        \left. T^{-\floor{\bv_1 m_1}-\floor{\bv_2 m_2}}u_{m_1m_2m_1'm_2'}\cdot T^{-\floor{\bv_1 m_1''}-\floor{\bv_2 m_2}}\overline{u_{m_1''m_2m_1'''m_2'}} \vphantom{\E_{m_2,m_2'\in[\pm M]}}\right|^2\; d\mu.
    \end{multline*}
    The result follows with
    \begin{multline*}
        u_{\substack{m_1m_1'm_1''m_1'''\\ m_2m_2'm_2''m_2'''}} = T^{-\floor{\bv_1 m_1}-\floor{\bv_2 m_2}}u_{m_1m_2m_1'm_2'}\cdot T^{-\floor{\bv_1 m_1''}-\floor{\bv_2 m_2}}\overline{u_{m_1''m_2m_1'''m_2'}}\\
        T^{-\floor{\bv_1 m_1}-\floor{\bv_2 m_2''}}\overline{u_{m_1m_2''m_1'm_2'''}}\cdot T^{-\floor{\bv_1 m_1''}-\floor{\bv_2 m_2''}}{u_{m_1''m_2''m_1'''m_2'''}}
    \end{multline*}
    upon expanding the square.
\end{proof}

\begin{proof}[Proof of Lemma \ref{L: dual-difference interchange}]
Let 
\begin{align*}
    \veps = \fnnorm{\tilde{f}}_{\bv_1, \ldots, \bv_s, \bv_{s+1}^{s'}}. 
\end{align*}

    Using the inductive formula for generalized box norms \eqref{E: inductive formula}, we get that
    \begin{align*}
        \lim_{M\to\infty}\E_{\um, \um'\in[\pm M]^s}\int \Delta_{\bv_1, \ldots, \bv_s; (\um, \um')}\tilde{f} \cdot u_{\um\um'}\; d\mu = \veps^{2^{s+s'}},
    \end{align*}
    where we can take $u_{\um\um'}$ to be any of the following depending on which conclusion we want to obtain:
    \begin{enumerate}
        \item if $s'\geq 1$: a dual function of level $s'$ with respect to $\bv_{s+1}$ (follows from the definition of dual functions);
        \item if $s'=1$: an approximately $\bv_{s+1}$-invariant function (follows from Lemma \ref{L: U^1 inverse}).
    \end{enumerate}
    We expand the expression above further as
    \begin{multline*}
        \lim_{M\to\infty}\lim_{i\to\infty}\E_{n\in[N_i]}\E_{\um, \um'\in[\pm M]^s}\\
        \int \brac{\prod_{\ueps\in\{0,1\}^s_*}\CC^{|\ueps|} T^{\floor{\bv_1 m_1^{\eps_1}}+\cdots + \floor{\bv_s m_s^{\eps_s}}}\tilde{f}} \cdot T^{\floor{\bv_1 m_1'}+\cdots + \floor{\bv_s m_s'}}f_{n,i} \cdot u_{\um\um'}\; d\mu;
    \end{multline*}    
    the reason why we expand only one copy of $\tilde{f}$ is to take the average over $n$ and the limit $N\to\infty$ out of the integral using weak convergence. The result follows from an application of Lemma \ref{L: GCS twisted} with $f_{\underline{1}} = f_{n,i}$.
\end{proof}

\subsection{Proof of Proposition \ref{P: smoothing pairwise}}

We now prove Proposition \ref{P: smoothing pairwise} for averages of length $\ell$ and type $w$ assuming Proposition \ref{P: smoothing} for averages of length $\ell-1$ as well as length $\ell$ and types $w'<w$.  If the type $w$ is basic, then the claim follows immediately from Proposition~\ref{P: bounds for basic type}, so we assume otherwise.  

    Let $\bv_{s+1} = \alpha_{1d_{1l}}\be_{\eta_1} - \alpha_{ld_{1l}}\be_{\eta_l}$ for some $l\in[0,\ell]\setminus\{1\}$; it is nonzero since it is the leading coefficient of $a_1\be_{\eta_1} - a_l\be_{\eta_l}$, and $a_1, a_l$ are essentially distinct (which follows from them being pairwise independent). We first sort out the easy cases. If $\eta_l = \eta_1$, then $\bv_{s+1}$ is a nonzero multiple of $\be_{\eta_1}$, and the claim follows from Lemmas~\ref{L: dilating seminorms} and \ref{L: plus vs normal seminorm}.\footnote{Specifically, we can use Lemma~\ref{L: plus vs normal seminorm} to add $+$ to our seminorm along $\bv_1, \ldots, \bv_{s+1}$, then apply Lemma~\ref{L: dilating seminorms} to replace $\bv_{s+1}$ with $\be_{\eta_1}$, and then apply the other part of Lemma~\ref{L: plus vs normal seminorm} to ``unplus'' the seminorm at the cost of adding one more $\be_{\eta_1}$. This gives the result with $s'=2$.} If $\deg a_l < \deg a_1$, then $\alpha_{ld_{1l}} = 0$, and the claim once again follows from Lemmas~\ref{L: dilating seminorms} and \ref{L: plus vs normal seminorm}. 
    
    Suppose therefore that $\eta_l\neq \eta_1$ and $\deg a_l = \deg a_1$ ($\deg a_l > \deg a_1$ is not allowed since $a_1$ has the maximal degree by assumption), in which case $\alpha_{ld_{1l}} \neq 0$. The proof of	Proposition~\ref{P: smoothing pairwise} in this case follows the same two-step strategy as in the analogous result in \cite{FrKu22a}. First, we establish auxiliary control of our average by $\nnorm{f_l}_{{\bv_1}, \ldots, {\bv_{s}}, \be_{\eta_l}^{s_1}}$ for some bounded $s_1\in\N$, and then we use this ancillary bound to get the claimed result. Like in \cite{FrKu22a}, we refer to these two steps as \textit{ping} and \textit{pong} since we are transferring the control from $f_1$ to $f_l$ and then back to $f_1$ much like a {table} tennis player passes the ball to his opponent, only to receive it back a second later.

    Throughout the argument, we let all the constant depend on $d, J, k, \ell, M, s, \CQ$, noting however that $s_1, s'$ as well as the powers of $\veps$ will not depend on $k, M$.
	
	\smallskip
	
	\textbf{Step 1 (\textit{ping}): Obtaining auxiliary control by a seminorm of $f_l$.}
	\smallskip
	
	Let
        \begin{align*}
           \veps = \limsup_{N\to\infty}\sup_{\substack{\CD_1, \ldots, \CD_J \in \FD_d}} \sup_{|c_n|\leq 1}\norm{\E_{n\in [N]} c_{n}\cdot \prod_{j\in[\ell]}  T_{\eta_j}^{\floor{a_j(n)}} f_j \cdot \prod_{j\in[J]}\CD_j(\floor{b_j(n)})}_{L^2(\mu)}.
        \end{align*} 
	Then Proposition~\ref{L: dual replacement} gives us a 1-bounded function $\tilde{f}_1$ as in \eqref{E: tilded} for which 
        \begin{multline*}
            \limsup_{N\to\infty}\sup_{\substack{\CD_1, \ldots, \CD_J \in \FD_d}} \sup_{|c_n|\leq 1}\\
            \norm{\E_{n\in [N]} c_{n}\cdot T_{\eta_1}^{\floor{a_1(n)}} \tilde{f}_1 \cdot \prod_{j\in[2, \ell]}  T_{\eta_j}^{\floor{a_j(n)}} f_j \cdot \prod_{j\in[J]}\CD_j(\floor{b_j(n)})}_{L^2(\mu)} \gg \veps^{O(1)}.
        \end{multline*}
	By our assumption, we have 
	$$
	\fnnorm{\tilde{f}_1}_{{\bv_1}, \ldots, {\bv_{s+1}}}\gg \veps^{O(1)}.
	$$
        We then unpack this seminorm using Lemma \ref{L: dual-difference interchange} and the formula \eqref{E: tilded} for $\tilde{f}_1$; after composing the resulting expression with $T_{\eta_1}^{\floor{a_1(n)}}$, and applying the Cauchy-Schwarz inequality, we obtain
        \begin{multline*}
            \liminf_{M\to\infty}\E_{\um, \um',\um'',\um'''\in [\pm M]^s}\limsup_{N\to\infty} \sup_{|c_n|\leq 1}\left|\!\left|\vphantom{\prod_{j\in[J]}}\E_{n\in[N]}c_n\cdot T_{\eta_1}^{\floor{a_1(n)}} u_{\um\um'\um''\um'''}\right.\right.\\
            \left.\left.\prod_{j\in[2, \ell]}T_{\eta_j}^{\floor{a_j(n)}}f_{j\um\um'\um''\um'''} \cdot \prod_{j\in[J]} \CD'_{j\um\um'\um''\um'''}(\floor{b_j(n)})\right|\!\right|_{L^2(\mu)}\gg \veps^{O(1)},
        \end{multline*}
        where
        \begin{align*}
            f_{j\um\um'\um''\um'''} &= \Delta_{\substack{(\floor{\bv_1 m_1'} - \floor{\bv_1 m_1}, \floor{\bv_1 m_1'''} - \floor{\bv_1 m_1''}), \ldots,\\ (\floor{\bv_s m_s'} - \floor{\bv_s m_s}, \floor{\bv_s m_s'''} - \floor{\bv_s m_s''})}}\overline{f_j}\\
            \textrm{and}\quad\CD'_{j\um\um'\um''\um'''} &= \Delta_{\substack{(\floor{\bv_1 m_1'} - \floor{\bv_1 m_1}, \floor{\bv_1 m_1'''} - \floor{\bv_1 m_1''}), \ldots,\\ (\floor{\bv_s m_s'} - \floor{\bv_s m_s}, \floor{\bv_s m_s'''} - \floor{\bv_s m_s''})}}\overline{\CD'_{j}}
        \end{align*}
        for some $\CD'_1, \ldots, \CD'_J\in \FD_d$ given by Lemma \ref{L: dual replacement}, and each $u_{\um\um'\um''\um'''}$ is a product of $2^s$ 1-bounded approximately $(\alpha_{1d_{1l}}\be_{\eta_1} - \alpha_{ld_{1l}}\be_{\eta_l})$-invariant functions.

        Thus, at the cost of having a family of averages to deal with rather than a single one, we have replaced the arbitrary function $f_1$ by one which is ``structured'' in the sense of being a product of approximately invariant functions. We now use this structure to lower the type of the average. By Lemma \ref{L: approximate invariance}, we can find 1-bounded functions $u'_{\um\um'\um''\um'''}$ such that
        \begin{multline*}
            \liminf_{M\to\infty}\E_{\um, \um',\um'',\um'''\in [\pm M]^s}\limsup_{N\to\infty} \sup_{|c_n|\leq 1}\left|\!\left|\vphantom{\prod_{j\in[J]}}\E_{n\in[N]}c_n\cdot T_{\eta'_1}^{\floor{a'_1(n)}} u'_{\um\um'\um''\um'''}\right.\right.\\ 
            \left.\left.\prod_{j\in[2, \ell]}T_{\eta'_j}^{\floor{a'_j(n)}}f_{j\um\um'\um''\um'''} \cdot \prod_{j\in[J]} \CD'_{j\um\um'\um''\um'''}(\floor{b_j(n)})\right|\!\right|_{L^2(\mu)}\gg \veps^{O(1)},
        \end{multline*}
	where 
		\begin{gather*} 
  a'_j(n):= \begin{cases} \frac{\alpha_{ld_{1l}}}{\alpha_{1d_{1l}}}a_1(n),\; &j = 1\\
                                    a_j,\; &j\in[2, \ell]
		\end{cases}, \quad\textrm{and}\quad
            \eta'_j := \begin{cases}
            \eta_l,\; &j = 1\\
            \eta_j,\; &j\in[2, \ell]
		\end{cases}.
  \end{gather*}
        Since $a_1, \ldots, a_\ell$ are pairwise independent,\footnote{All we need is that the sequences that we are dealing with at any given step of the argument are essentially distinct. However, the fact that we rescale sequences at subsequent steps means that essential distinctness need not be preserved from one step to another. By contrast, pairwise independence is preserved under scaling, and therefore it is the right assumption to make on our sequences.} $a'_1, \ldots, a'_\ell$ are as well. Moreover, the complexity of $a'_j$'s is at most the square of the complexity of $a_j$'s.

     By assumption, $w_{\eta_1}$ minimizes the nonzero coordinates of $w$. Since $\eta_l\neq\eta_1$, we deduce that the type $w'=\sigma_{\eta_1\eta_l}w$ of the new averages indexed by $(\um, \um',\um'',\um''')\in\Z^{4s}$ is strictly lower than $w$. We therefore apply inductively Proposition \ref{P: smoothing} for averages of length $\ell$ and type $w'<w$ to find a positive integer $s_1=O(1)$ such that for a set $B$ of $(\um, \um',\um'',\um''')\in\Z^{4s}$ of lower density $\Omega(\veps^{O(1)})$, we have
     \begin{align*}
         \nnorm{f_{l, \um\um'\um''\um'''}}_{s_1, T_{\eta_l}} = \nnorm{\Delta_{\substack{(\floor{\bv_1 m_1'} - \floor{\bv_1 m_1}, \floor{\bv_1 m_1'''} - \floor{\bv_1 m_1''}), \ldots,\\ (\floor{\bv_s m_s'} - \floor{\bv_s m_s}, \floor{\bv_s m_s'''} - \floor{\bv_s m_s''})}}\overline{f_l}}_{s_1, T_{\eta_l}}\gg \veps^{O(1)}.
     \end{align*}
    Averaging over $B$, extending the average to all of $\Z^{4s}$ by nonnegativity, raising it to the power $2^{s_1}$ using the H{\"o}lder inequality and applying the inductive formula for generalized box seminorms \eqref{E: inductive formula}, we get the auxiliary norm control\footnote{We are glossing over the fact that the inductive formula works for multiplicative derivatives of the form $\Delta_{(\floor{\bv_1 m_1'}, \floor{\bv_1 m_1})}$ rather than $\Delta_{(\floor{\bv_1 m_1'} - \floor{\bv_1 m_1}, \floor{\bv_1 m_1'''} - \floor{\bv_1 m_1''})}$. One way to fix it is to pigeonhole in $(m_1, m_1'')$ to find a single pair for which the appropriate fiber in $B$ has a lower density $\Omega(\veps^{O(1)})$ in $\Z^{2s}$, and then use the Gowers-Cauchy-Schwarz inequality to remove the shifts by these fixed pairs.}
    \begin{align*}
        \nnorm{f_l}_{\bv_1, \ldots, \bv_s, \be_{\eta_l}^{s_1}}\gg \veps^{O(1)}.
    \end{align*}

	\smallskip
	
	\textbf{Step 2 (pong): Obtaining control by a seminorm of $f_1$.}
	
	\smallskip
	
	To get the claim that $\nnorm{f_1}_{{\bv_1}, \ldots, {\bv_s}, \be_{\eta_1}^{s'}}$ controls the average for some bounded $s'\in\N$, we repeat the procedure once more with $f_l$ in place of $f_1$. From Proposition~\ref{L: dual replacement} 
	it follows that
         \begin{multline*}
            \limsup_{N\to\infty}\sup_{\substack{\CD_1, \ldots, \CD_J \in \FD_d}} \sup_{|c_n|\leq 1}\\
            \norm{\E_{n\in [N]} c_{n}\cdot T_{\eta_l}^{\floor{a_l(n)}} \tilde{f}_l \cdot \prod_{j\in[\ell], j\neq l}  T_{\eta_j}^{\floor{a_j(n)}} f_j \cdot \prod_{j\in[J]}\CD_j(\floor{b_j(n)})}_{L^2(\mu)} \gg \veps^{O(1)}.
        \end{multline*}
	for some function
          \begin{multline*}
            \tilde{f}_l = \lim_{i\to\infty} \E_{n\in [N_i]} \overline{c'_n}\cdot T_{\eta_l}^{-\floor{a_l(n)}}g_i\cdot\prod_{j\in[\ell], j\neq l}T_{\eta_l}^{-\floor{a_l(n)}} T_{\eta_j}^{\floor{a_j(n)}}\overline{f_j} \cdot \prod_{j\in[J]}T_{\eta_l}^{-\floor{a_l(n)}} \overline{\CD'_{j}(\floor{b_j(n)})}.
        \end{multline*}
	where $g_i$'s are 1-bounded functions, $c'_n$'s are 1-bounded weights, $\CD'_1, \ldots, \CD'_J\in \FD_d$ are dual sequences, and the limit is a weak limit. Then the auxiliary control obtained in the \textit{ping} step gives 
	$$
	\fnnorm{\tilde{f}_l}_{{\bv_1}, \ldots, {\bv_s}, \be_{\eta_l}^{s_1}}\gg \veps^{O(1)}.
	$$
	Lemma \ref{L: dual-difference interchange} implies that
 \begin{multline*}
            \liminf_{H\to\infty}\E_{\um, \um',\um'',\um'''\in [\pm M]^s}\limsup_{N\to\infty} \sup_{|c_n|\leq 1}\\
            \norm{\E_{n\in[N]}c_n\cdot \prod_{j\in[\ell], j\neq l}T_{\eta_j}^{\floor{a_j(n)}}f_{j\um\um'\um''\um'''} \cdot \prod_{j\in[J+1]} \CD'_{j\um\um'\um''\um'''}(\floor{b_j(n)})}_{L^2(\mu)}\gg \veps^{O(1)},
        \end{multline*}
        where
        \begin{align*}
            f_{j\um\um'\um''\um'''} &= \Delta_{\substack{(\floor{\bv_1 m_1'} - \floor{\bv_1 m_1}, \floor{\bv_1 m_1'''} - \floor{\bv_1 m_1''}), \ldots,\\ (\floor{\bv_s m_s'} - \floor{\bv_s m_s}, \floor{\bv_s m_s'''} - \floor{\bv_s m_s''})}}\overline{f_j}\\
        \end{align*}
        and
        \begin{align*}
            \CD'_{j\um\um'\um''\um'''}(n) = \begin{cases}\Delta_{\substack{(\floor{\bv_1 m_1'} - \floor{\bv_1 m_1}, \floor{\bv_1 m_1'''} - \floor{\bv_1 m_1''}), \ldots,\\ (\floor{\bv_s m_s'} - \floor{\bv_s m_s}, \floor{\bv_s m_s'''} - \floor{\bv_s m_s''})}}\overline{\CD'_{j}}(n),\; &j\in[J]\\
            T_{\eta_l}^{n} u_{\um\um'\um''\um'''},\; &j = J+1
            \end{cases},
        \end{align*}
        for $b_{J+1} = a_l$, some $\CD'_1, \ldots, \CD'_J\in \FD_d$ given by Lemma \ref{L: dual replacement}, and 1-bounded functions $u_{\um\um'\um''\um'''}$, each of which is a product of $2^s$ level-$s_1$ dual functions of $T_{\eta_l}$.

        Each of the new averages indexed by $(\um, \um',\um'',\um''')\in\Z^{4s}$ has length $\ell-1$. Hence, applying Proposition \ref{P: smoothing} for averages of length $\ell-1$, we obtain a positive integer $s'=O(1)$ such that for a set $B'$ of $(\um, \um',\um'',\um''')\in\Z^{4s}$ of lower density $\Omega(\veps^{O(1)})$, we have
        \begin{align*}
            \nnorm{\Delta_{\substack{(\floor{\bv_1 m_1'} - \floor{\bv_1 m_1}, \floor{\bv_1 m_1'''} - \floor{\bv_1 m_1''}), \ldots,\\ (\floor{\bv_s m_s'} - \floor{\bv_s m_s}, \floor{\bv_s m_s'''} - \floor{\bv_s m_s''})}} f_1}_{s', T_{\eta_1}} \gg \veps^{O(1)}.
        \end{align*}
        Averaging over $B'$, extending the average to all of $\Z^4$ by nonnegativity, raising it to the power $2^{s'}$ using the H{\"o}lder inequality and applying the inductive formula for generalized box seminorms \eqref{E: inductive formula}, we get        
        \begin{align*}
            \nnorm{f_1}_{\bv_1, \ldots, \bv_s, \be_{\eta_1}^{s'}} \gg \veps^{O(1)},
        \end{align*}
        as claimed.

\subsection{Proof of Proposition \ref{P: smoothing}}
Finally, we prove Proposition \ref{P: smoothing} for averages of length $\ell$ and type $w$ by assuming Proposition \ref{P: iterated smoothing pairwise} for averages of length $\ell$ and type $w$ as well as Proposition \ref{P: smoothing} for averages of length $\ell-1$.	In the base case $\ell=1$, Proposition \ref{P: smoothing} follows directly from Proposition \ref{P: bounds for basic type}. We assume therefore that $\ell>1$. 

 We let all the constants depend on $d, J, k, \ell, M, s, \CQ$, noting however that $s'$ as well as the powers of $\veps$ will not depend on $k, M$.
 
 Let 
 \begin{align*}
     \veps = \limsup_{N\to\infty}\sup_{\substack{\CD_1, \ldots, \CD_J \in \FD_d}} \sup_{|c_n|\leq 1}\norm{\E_{n\in [N]} c_{n}\cdot \prod_{j\in[\ell]}  T_{\eta_j}^{\floor{a_j(n)}} f_j \cdot \prod_{j\in[J]}\CD_j(\floor{b_j(n)})}_{L^2(\mu)}.
 \end{align*}
 Fix $i\in[\ell]$ for which $a_i$ is a sequence of maximum degree, and $w_{\eta_i}$ is minimal among the nonzero coordinates of $w$. By Proposition \ref{P: iterated smoothing pairwise} with the role of 1 played by $i$, there exists a positive integer $s'=O(1)$ for which 
 \begin{align*}
    \nnorm{f_i}_{s', T_{\eta_i}} \gg \veps^{O(1)}.
 \end{align*}
This gives the claimed Host-Kra seminorm control on $f_i$. If $i=1$, then we are done. Otherwise we also need to obtain a Host-Kra seminorm estimate in terms of other functions. To this end, we use Lemma \ref{L: dual replacement} to replace $f_i$ with $\tilde{f}_i$ (constructed analogously to $\tilde{f}_1$ from \eqref{E: tilded}), and then apply Proposition \ref{P: iterated smoothing pairwise} to get
 \begin{align*}
    \fnnorm{\tilde{f}_i}_{s', T_{\eta_i}} \gg \veps^{O(1)}.
 \end{align*}
 Using the identity \eqref{dual identity} and expanding the definition of $\tilde{f}_i$, we get
 \begin{multline*}
     \limsup_{N\to\infty}\sup_{\substack{\CD_1, \ldots, \CD_J \in \FD_d}} \sup_{|c_n|\leq 1}\norm{\E_{n\in [N]} c_{n}\cdot \prod_{j\in[\ell], j\neq i}  T_{\eta_j}^{\floor{a_j(n)}} f_j \cdot \prod_{j\in[J+1]}\CD_j(\floor{b_j(n)})}_{L^2(\mu)} \gg \veps^{O(1)},
 \end{multline*}
 where $b_{J+1} = a_i$ and 
 \begin{align*}
     \CD_{J+1}(n) =  T_{\eta_i}^{n}\CD_{s', T_{\eta_i}}\tilde{f}_i.
 \end{align*}
This average has length $\ell-1$, and so the claimed Host-Kra seminorm control of functions $f_j$ with $j\neq i$ follows by invoking Proposition \ref{P: smoothing} for averages of length $\ell-1$.

\section{Joint ergodicity for pairwise independent Hardy sequences}\label{S: joint ergodicity proofs}
The purpose of this section is to derive the joint ergodicity results from Section \ref{SS: joint ergodicity results}.
We start with a few relevant definitions.

\begin{definition}[Nonergodic eigenfunction]
    Let $(X, \CX, \mu, T)$ be a system. A function $\chi\in L^\infty(\mu)$ is a \textit{nonergodic eigenfunction} of $T$ if it satisfies the following two conditions:
    \begin{enumerate}
        \item there exist a $T$-invariant set $E\in \CX$ and a $T$-invariant function $\phi: X\to\T$ s.t.
        \begin{align*}
            T\chi(x) = 1_E(x)e(\phi(x))\chi(x)\quad \textrm{for}\quad \mu\textrm{-a.e.}\quad x\in X;   
        \end{align*}
        \item $|\chi(x)|\in\{0,1\}$ for $\mu$-a.e. $x\in X$.
    \end{enumerate}

    Similarly, we call $\chi\in L^\infty(\mu)$ an \textit{eigenfunction} of $T$ with \textit{eigenvalue} $\lambda\in[0,1)$ if $T\chi(x) = e(\lambda)\chi(x)$ holds for $\mu$-a.e. $x\in X$, and $|\chi|=1$ $\mu$-a.e.
\end{definition}

Classically, eigenfunctions of $T$ generate the Kronecker factor $\CK(T)$ whereas eigenfunctions with rational eigenvalues span the rational Kronecker factor $\Krat(T)$. Moreover, nonergodic eigenfunctions are known to generate the Host-Kra factor $\CZ_1(T)$ (see \cite[Theorem~5.2]{FH18} or \cite[Proposition~4.3]{FrKu22a}).  If $T$ is ergodic, then $\CZ_1(T)$ equals $\CK(T)$, but this fails for nonergodic systems in general (e.g., consider $T(x,y) = (x, y+x)$ on $X=\T^2$).

\begin{definition}[Good properties]\label{D: good properties}
    Let $a_1, \ldots, a_\ell:\N\to\R$ be sequences and let $(X, \CX, \mu, T_1, \ldots, T_\ell)$ be a system. The sequences $a_1, \ldots, a_\ell$ are:
    \begin{enumerate}
        \item \textit{good for seminorm control} for $(X, \CX, \mu, T_1, \ldots, T_\ell)$ if there exists $s\in\N$ such that for all 1-bounded $f_1, \ldots, f_\ell\in L^\infty(\mu)$, we have
    \begin{align}\label{E: convergence to 0}
        \lim_{N\to\infty}\norm{\E_{n\in[N]} T_1^{\floor{a_1(n)}}f_1 \cdots T_\ell^{\floor{a_\ell(n)}}f_\ell}_{L^2(\mu)} = 0
    \end{align}
    whenever $\nnorm{f_j}_{s, T_j} = 0$ for some $j\in[\ell]$;
        \item \textit{controlled by the Kronecker factor} for $(X, \CX, \mu, T_1, \ldots, T_\ell)$ if \eqref{E: convergence to 0} holds whenever $\E(f_j|\CK(T_j)) = 0$ for some $j\in[\ell]$;
        \item \textit{controlled by the rational Kronecker factor} on $(X, \CX, \mu, T_1, \ldots, T_\ell)$ if \eqref{E: convergence to 0} holds whenever $\E(f_j|\Krat(T_j)) = 0$ for some $j\in[\ell]$;
        \item \textit{good for equidistribution} for $(X, \CX, \mu, T_1, \ldots, T_\ell)$ if for eigenvalues $\lambda_1, \ldots, \lambda_\ell$ of $T_1, \ldots, T_\ell$ respectively such that $\lambda_1, \ldots, \lambda_\ell$ are not all 0, we have
        \begin{align}\label{E: good for equidistribution}
            \lim_{N\to\infty} \E_{n\in[N]}e(\lambda_1 \floor{a_1(n)} + \cdots + \lambda_\ell \floor{a_\ell(n)}) = 0;
        \end{align}
        \item \textit{good for strong irrational equidistribution} for $(X, \CX, \mu, T_1, \ldots, T_\ell)$ if \eqref{E: good for equidistribution} holds for eigenvalues $\lambda_1, \ldots, \lambda_\ell$ of $T_1, \ldots, T_\ell$ respectively such that $\lambda_1, \ldots, \lambda_\ell$ are not all rational;
        \item \textit{good for irrational equidistribution} for $(X, \CX, \mu, T_1, \ldots, T_\ell)$ if \eqref{E: good for equidistribution} holds for eigenvalues $\lambda_1, \ldots, \lambda_\ell$ of $T_1, \ldots, T_\ell$ respectively such that $\lambda_1, \ldots, \lambda_\ell\in \mathbb{R}\setminus\mathbb{Q}_\ast$ not all $0$;
        \item \textit{good for equidistribution} if \eqref{E: good for equidistribution} holds for all $\lambda_1, \ldots, \lambda_\ell\in[0,1)$ not all 0.
    \end{enumerate}
\end{definition}

The following characterizations have been proved in \cite{FrKu22a}.
\begin{theorem}\label{T: joint ergodicity criteria}
    Let $a_1, \ldots, a_\ell:\N\to\R$ be sequences and $(X, \CX, \mu, T_1, \ldots, T_\ell)$ be a system. Then 
    \begin{align}\label{E: weak joint ergodicity}
                \lim_{N\to\infty}\norm{\E_{n\in[N]}\prod_{j\in[\ell]} T_j^{\floor{a_{j}(n)}}f_j - \prod_{j\in[\ell]} \E(f_j|\CI(T_j))}_{L^2(\mu)} = 0
    \end{align}
    holds for all $f_1, \ldots, f_\ell\in L^\infty(\mu)$
    if and only if 
    \begin{enumerate}
        \item the sequences are good for seminorm control for this system, and
        \item the limiting formula \eqref{E: weak joint ergodicity} holds whenever $f_1, \ldots, f_\ell$ are nonergodic eigenfunctions of $T_1, \ldots, T_\ell$ respectively.
    \end{enumerate}

\smallskip

    Similarly,
    \begin{align}\label{E: Krat control}
                \lim_{N\to\infty}\norm{\E_{n\in[N]}\prod_{j\in[\ell]} T_j^{\floor{a_{j}(n)}}f_j - \E_{n\in[N]}\prod_{j\in[\ell]} T_j^{\floor{a_{j}(n)}}\E(f_j|\Krat(T_j))}_{L^2(\mu)} = 0
    \end{align}
    holds for all $f_1, \ldots, f_\ell\in L^\infty(\mu)$
    if and only if 
    \begin{enumerate}
        \item the sequences are good for seminorm control for this system, and
        \item the limiting formula \eqref{E: Krat control} holds whenever $f_1, \ldots, f_\ell$ are nonergodic eigenfunctions of $T_1, \ldots, T_\ell$ respectively.
    \end{enumerate}
\end{theorem}
When $T_1, \ldots, T_\ell$ are ergodic, then in the first part of Theorem \ref{T: joint ergodicity criteria}, the second product in \eqref{E: weak joint ergodicity} simplifies to $\prod_{j\in[\ell]}\int f_j\; d\mu$, while in the second conditions of both statements we just need to consider eigenfunctions. 

\subsection{Implications for general sequences}

We start with several rudimentary lemmas that show logical relationships between various properties from Definitions~\ref{D: good properties} and conditions from Problem~\ref{Pr: joint ergodicity problem}. A standard result that we use many times concerns the equivalent conditions for the good equidistribution property.

\begin{lemma}\label{L: equivalent form of good eq prop}
    Let $a_1, \ldots, a_\ell:\N\to\R$ be sequences and $(X, \CX, \mu, T_1, \ldots, T_\ell)$ be a system. Then the following are equivalent: 
    \begin{enumerate}
        \item the sequences are good for equidistribution for the system;
        \item the identity \eqref{E: joint ergodicity} holds for all eigenfunctions $f_1, \ldots, f_\ell$ of $T_1, \ldots, T_\ell$ respectively;
        \item the identity 
        \begin{align}\label{E: eigenfunctions on product system}
                    \lim_{N\to\infty}\norm{\E_{n\in[N]}\bigotimes_{j\in[\ell]} T_j^{\floor{a_{j}(n)}}f_j - \prod_{j\in[\ell]} \int f_j\, d\mu}_{L^2(\mu^{\ell})} = 0
        \end{align}
         holds for all eigenfunctions $f_1, \ldots, f_\ell$ of $T_1, \ldots, T_\ell$ respectively.
         \end{enumerate}
\end{lemma}
\begin{proof} 
    We prove the equivalence of (i) and (iii); the equivalence of (i) and (ii) follows similarly. Let $\lambda_1, \ldots, \lambda_\ell$ be the eigenvalues of $f_{1},\dots,f_{\ell}$ with respect to $T_1, \ldots, T_\ell$ respectively. If $\lambda_1 = \cdots = \lambda_\ell = 0$, then there is nothing to check regarding equidistribution, and the identity \eqref{E: eigenfunctions on product system} holds with both expressions {on the left hand side} being 1. Suppose therefore that $\lambda_j \neq 0$ for some $j\in[\ell]$. Then $\int f_j\; d\mu = 0$, and so the Cartesian product in \eqref{E: eigenfunctions on product system} vanishes. The tensor product in \eqref{E: eigenfunctions on product system} equals 
 \begin{align*}
     \lim_{N\to\infty}\E_{n\in[N]} e(\lambda_1 \floor{a_1(n)}+\cdots + \lambda_\ell \floor{a_\ell(n)})\cdot \bigotimes_{j\in[\ell]}f_{j},
 \end{align*}
 and since $\bigotimes_{j\in[\ell]}f_{j}\neq 0$ $\mu$-a.e., it vanishes if and only if \eqref{E: good for equidistribution} holds. Hence being good for equidistribution is equivalent to condition (iii).
\end{proof}

Lemma \ref{L: equivalent form of good eq prop} gives the following immediate corollary.
\begin{corollary}\label{C: equivalent form of good eq prop}
    Let $a_1, \ldots, a_\ell:\N\to\R$ be sequences and $(X, \CX, \mu, T_1, \ldots, T_\ell)$ be a system. Suppose that one of the conditions holds:
    \begin{enumerate}
        \item the identity \eqref{E: weak joint ergodicity} holds for all $f_1, \ldots, f_\ell\in L^\infty(\mu)$; or
        \item $(T_1^{\floor{a_1(n)}}\times \cdots \times T_\ell^{\floor{a_\ell(n)}})_n$ is ergodic on the product space $(X^\ell, \CX^{\otimes \ell}, \mu^\ell)$. 
    \end{enumerate}
    Then the sequences are good for equidistribution for the system.
\end{corollary}
\begin{proof}
    If the condition (i) (resp. (ii)) holds, then the condition (ii) (resp. (iii)) of Lemma \ref{L: equivalent form of good eq prop} holds, and so the sequences are good for equidistribution for the system.
\end{proof}

The next result gives conditions under which joint ergodicity implies the product ergodicity condition from Problem~\ref{Pr: joint ergodicity problem}.
\begin{lemma}\label{L: weak (ii)}
    Let $a_1, \ldots, a_\ell:\N\to\R$ be sequences and $(X, \CX, \mu, T_1, \ldots, T_\ell)$ be a system. Suppose that the following two conditions hold:
    \begin{enumerate}
        \item $a_1, \ldots, a_\ell$ are jointly ergodic for $(X, \CX, \mu, T_1, \ldots, T_\ell)$,
        \item each $a_j$ is controlled by the Kronecker factor on the system $(X\times X, \CX\otimes \CX, \mu\times\mu,$ $T_j\times T_j)$.
    \end{enumerate}
    Then $(T_1^{\floor{a_1(n)}}\times \cdots \times T_\ell^{\floor{a_\ell(n)}})_n$ is ergodic on the product space $(X^\ell, \CX^{\otimes \ell}, \mu^{\ell})$.
\end{lemma}

Corollary \ref{C: equivalent form of good eq prop} and Lemma \ref{L: weak (ii)} jointly imply that under the additional property (ii) from Lemma \ref{L: weak (ii)}, the difference and product ergodicity conditions from Problem~\ref{Pr: joint ergodicity problem} are necessary for joint ergodicity on a given system.
 \begin{proof}
     We want to show that 
     \begin{align}\label{E: convergence on product system}
         \lim_{N\to\infty}\norm{\E_{n\in[N]}T_1^{\floor{a_1(n)}}\times \cdots \times T_\ell^{\floor{a_\ell(n)}} f - \int f\; d\mu^{\ell}}_{L^2(\mu^{\ell})} = 0
     \end{align}
     for all $f\in L^\infty(\mu^{\ell})$. By an $L^2(\mu^{\ell})$-approximation argument, it suffices to consider 1-bounded functions of the form $f_1\otimes \cdots \otimes f_\ell$.  For such functions, we have 
     \begin{align*}
         &\limsup_{N\to\infty}\norm{\E_{n\in[N]}T_1^{\floor{a_1(n)}}\times \cdots \times T_\ell^{\floor{a_\ell(n)}} (f_1\otimes \cdots \otimes f_\ell)}_{L^2(\mu^{\ell})}^2\\
         &\qquad\qquad\qquad =\limsup_{N\to\infty} \E_{n, m\in[N]}\prod_{j\in[\ell]}\int T_j^{\floor{a_j(n)}}f_j\cdot T_j^{\floor{a_j(m)}}\overline{f_j}\; d\mu\\
         &\qquad\qquad\qquad \leq\limsup_{N\to\infty} \E_{n, m\in[N]}\abs{\int T_j^{\floor{a_j(n)}}f_j\cdot T_j^{\floor{a_j(m)}}\overline{f_j}\; d\mu}
     \end{align*}
     for each fixed $j\in[\ell]$. Applying the Cauchy-Schwarz inequality and passing to the product system, we get  
     \begin{align*}
         &\limsup_{N\to\infty}\norm{\E_{n\in[N]}T_1^{\floor{a_1(n)}}\times \cdots \times T_\ell^{\floor{a_\ell(n)}} (f_1\otimes \cdots \otimes f_\ell)}_{L^2(\mu^{\ell})}^4\\
         &\qquad\qquad\qquad \leq\limsup_{N\to\infty} \norm{\E_{n\in[N]}(T_j\times T_j)^{\floor{a_j(n)}}(f_j\otimes\overline{f_j})}_{L^2(\mu\times\mu)}^2,
     \end{align*}     
     and by the second condition, the latter vanishes whenever $\E(f_j\otimes\overline{f_j}|\CK(T_j\times T_j)) = 0$. Since $\CK(T_j\times T_j) = \CK(T_j)\otimes \CK(T_j)$, this last condition holds whenever $\E(f_j|\CK(T_j)) = 0$. 

    Using the well-known fact that $\CK(T_1 \times \cdots \times T_\ell) = \CK(T_1)\otimes \cdots \otimes \CK(T_\ell)$ \cite[Lemma~4.18]{Fu81} and an $L^2(\mu^{\ell})$ approximation argument, it therefore 
 suffices to show that (\ref{E: convergence on product system}) holds if $f_{1},\dots,f_{\ell}$ are eigenfunctions of $T_1, \ldots, T_\ell$ respectively. Since joint ergodicity implies being good for equidistribution by Corollary \ref{C: equivalent form of good eq prop}, the veracity of \eqref{E: convergence on product system} for eigenfunctions follows from Lemma \ref{L: equivalent form of good eq prop}.
 \end{proof}

Being good for equidistribution for a specific system does not imply that \eqref{E: weak joint ergodicity} holds for all \textit{nonergodic} eigenfunctions, as illustrated by the following example.
\begin{example}
Let $Y=\{y_0, y_1, y_2\}$, $Z = \{z_0, z_1\}$, $X$ be the disjoint union of $Y$ and $Z$, and $T:X\to X$ be defined by setting $T|_Y, T|_Z$ to be nontrivial permutations of $Y, Z$ respectively. Then $T$ admits no nontrivial eigenvalues, and so any sequence $a: \N\to\R$ is vacuously good for equidistribution. At the same time, the identity \eqref{E: weak joint ergodicity} fails for $a(n) = 6n$ since then
\begin{align*}
    \lim_{N\to\infty}\E_{n\in[N]} T^{6n}f = f
\end{align*}
is different from $\E(f|\CI(T))$ whenever $f|_Y$ or $f|_Z$ is not constant. Since every $f:X\to\C$ is a linear combination of nonergodic eigenfunctions, this implies that the identity \eqref{E: weak joint ergodicity} fails for some nonergodic eigenfunctions.
\end{example}
However, one can check that if a family of sequences is good for equidistribution, then \eqref{E: weak joint ergodicity} holds for all systems and all nonergodic eigenfunctions.

Lastly, we show how seminorm estimates aid in obtaining the sufficient direction in Problem~\ref{Pr: joint ergodicity problem}.
 \begin{lemma}[Sufficient direction for joint ergodicity]\label{L: sufficient direction}
     Let $a_1, \ldots, a_\ell:\N\to\R$ be sequences and $(X, \CX, \mu, T_1, \ldots, T_\ell)$ be a system with $T_1, \ldots, T_\ell$ ergodic. Suppose that:
     \begin{enumerate}
         \item $a_1, \ldots, a_\ell$ are {good for seminorm control} for $(X, \CX, \mu, T_1, \ldots, T_\ell)$; and
         \item $(T_1^{\floor{a_1(n)}}\times \cdots \times T_\ell^{\floor{a_\ell(n)}})_n$ is ergodic on the product space $(X^\ell, \CX^{\otimes \ell}, \mu^\ell)$. 
     \end{enumerate}
    Then $a_1, \ldots, a_\ell$ are jointly ergodic for $(X, \CX, \mu, T_1, \ldots, T_\ell)$.
 \end{lemma}
 \begin{proof}
     In the light of Theorem \ref{T: joint ergodicity criteria}, it suffices to show that \eqref{E: joint ergodicity} holds when $f_1, \ldots, f_\ell$ are eigenfunctions of $T_1, \ldots, T_\ell$ respectively. The second condition and Corollary \ref{C: equivalent form of good eq prop} imply that $a_1, \ldots, a_\ell$ are good for equidistribution for the system, and then Lemma \ref{L: equivalent form of good eq prop} implies that \eqref{E: joint ergodicity} holds for eigenfunctions.
 \end{proof}

 \subsection{Specializing to pairwise independent Hardy sequences}
We now specialize to the case of pairwise independent Hardy sequences, and use the general results from the previous section in conjunction with specific properties of Hardy sequences to prove results from Section \ref{SS: joint ergodicity results}. We start with the easy proofs.
\begin{proof}[Proof of Theorem \ref{T: Joint ergodicity for weakly mixing transformations}]
    If $T_1, \ldots, T_\ell$ are weakly mixing, then $\nnorm{f_j}_{s, T_j} = \abs{\int f_j\; d\mu}$ for every $j\in[\ell]$ and $f_j\in L^\infty(\mu)$, and so the result follows immediately from Theorem~\ref{T: HK control} and a telescoping identity.
\end{proof}

 For the proof of Theorem \ref{T: Krat control}, we need the following lemma:

\begin{lemma}\label{2to3}
Let  $\ell\in\N$ and $a_1, \ldots, a_\ell\in \CH$ be irrationally independent. Then the sequences $\floor{a_{1}(n)},\dots,\floor{a_{\ell}(n)}$ are good for irrational equidistribution for any system $(X, \CX, \mu, T_1, \ldots, T_\ell),$ where the $T_1,\ldots, T_\ell$ are totally ergodic.
\end{lemma}

\begin{proof} The proof is similar to \cite[Proposition~2.8]{KoSu23} and so it is omitted.
    \end{proof}

\begin{proof}[Proofs of Theorems \ref{T: joint ergodicity for strongly independent Hardy sequences} and \ref{T: Krat control}]
If $a_1, \ldots, a_\ell$ are strongly independent (resp. strongly irrationally independent), then by Theorem~\ref{T: Boshernitzan} as well as \cite[Lemma~6.2]{Fr21} (resp. \cite[Lemma~6.4]{Fr21}), the sequences $a_1, \ldots, a_\ell$ are good for equidistribution (resp. strong irrational equidistribution). This, the seminorm estimates from Theorem~\ref{T: HK control}, and Theorem~\ref{T: joint ergodicity criteria} give the results, where for the irrationally independent case and a system with totally ergodic transformations, we use Lemma~\ref{2to3} to get the corresponding equidistribution property. 
\end{proof}


\begin{proof}[Proof of Theorem \ref{T: joint ergodicity for strongly independent Hardy sequences along prime numbers}]
    The result follows from \cite[Theorem 1.2]{KoTs23}, Theorem~\ref{T: Joint ergodicity for weakly mixing transformations} (resp. Theorem \ref{T: joint ergodicity for strongly independent Hardy sequences}), and the fact that if $a_1,\ldots, a_{\ell}$ are strongly independent (resp. pairwise independent), then the same holds for  $a_1(W\cdot+b), \ldots, a_\ell(W\cdot+b)$ for all $W,b\in \N$.
\end{proof}

To prove Theorem \ref{T: joint ergodicity}, we need a few more auxiliary results. The first one concerns the nonvanishing of exponential sums along pairwise independent Hardy sequences.

 \begin{lemma}\label{L: countably many exceptions}
     Let $a_1, a_2\in\CH$ be pairwise independent. Then there exists at most countably many $(\lambda_1, \lambda_2)\in\R^2$ for which
     \begin{align*}
         \E_{n\in[N]}e(\lambda_1 a_1(n) + \lambda_2 a_2(n))
     \end{align*}
     does not converge to 0 as $N\to\infty$.
 \end{lemma}
 Lemma \ref{L: countably many exceptions} clearly fails without the pairwise independence assumption. For instance, if $a_1 = a_2$, then the conclusion fails for any $(\lambda_1, \lambda_2) = (\lambda, -\lambda)$ with $\lambda\in\R$.
 \begin{proof}
     By Theorem~\ref{T: Boshernitzan}, a sequence $a\in\CH$ is equidistributed if and only if stays logarithmically away from rational polynomials.\footnote{We recall that this means that $a - p \succ \log$ for any $p\in\Q[t]$; otherwise we say that $a$ is \textit{logarithmically close to a rational polynomial}.} Our objective therefore is to show that $\lambda_1 a_1 + \lambda_2 a_2$ is logarithmically away from rational polynomials for all but countably many bad pairs $(\lambda_1, \lambda_2)\in\R^2$. 
     
     If $a_1, a_2$ are strongly independent, then the only bad pair is $(0,0)$, and we are done. Otherwise there exists some $(\lambda_1, \lambda_2)\neq(0,0)$ for which $\lambda_1 a_1 + \lambda_2 a_2$ is logarithmically close to a rational polynomial. A rescaling $c\lambda_1 a_1 + c\lambda_2 a_2$ is also logarithmically close to a rational polynomial if and only if $c\in\Q$. If $\{(c\lambda_1, c\lambda_2):\; c\in\Q\}$ are the only bad pairs, then this gives us countably many bad pairs, and the claim holds. 
     
     Otherwise there exists $(\lambda_1', \lambda_2')\neq(0,0)$ such that $(\lambda_1, \lambda_2), (\lambda_1', \lambda_2')$ are $\R$-linearly independent, and $\lambda'_1 a_1 + \lambda'_2 a_2$ is also logarithmically close to a rational polynomial. 
    Then 
    \begin{align*}
        (\lambda_1 \lambda_2' - \lambda'_1 \lambda_2) a_1 = \lambda_2'(\lambda_1 a_1 + \lambda_2 a_2) - \lambda_2(\lambda'_1 a_1 + \lambda'_2 a_2)
    \end{align*}
    is also logarithmically close to a rational polynomial, and similarly for $(\lambda_1 \lambda_2' - \lambda'_1 \lambda_2) a_2$. The $\R$-linear independence of $(\lambda_1, \lambda_2), (\lambda_1', \lambda_2')$ implies that $\lambda_1 \lambda_2' - \lambda'_1 \lambda_2 \neq 0$, and hence
    \begin{align*}
        a_1 = p_1 + g_1 \qquad \textrm{and}\qquad a_2 = p_2 + g_2,
    \end{align*}
    where $p_1, p_2$ are real multiples of some rational polynomials and $g_1, g_2\ll \log$. In other words, the abundance of bad pairs $(\lambda_1, \lambda_2)$ imposes a special structure on the functions $a_1, a_2$. 
    
    So far, we have made no use of the pairwise independence of $a_1, a_2$; this is the time to utilize it. The pairwise independence of $a_1, a_2$ implies that $p_1, p_2$ are pairwise independent, and we claim that $\lambda''_1 p_1 + \lambda''_2 p_2$ can be a rational polynomial for only countably many pairs $(\lambda_1'', \lambda_2'')\in\R^2$. Letting $p_1(t) = \sum_{i=1}^d \beta_{1i}t^i$, $p_2(t) = \sum_{i=1}^d \beta_{2i}t^i$ (we can absorb constant terms into $g_1, g_2$), we can then find indices $1 \leq i < j\leq d$ such that $(\beta_{1i}, \beta_{1j})$ and $(\beta_{2i}, \beta_{2j})$ are $\R$-linearly independent. If $\lambda''_1 p_1 + \lambda''_2 p_2\in\Q[t]$, then 
    \begin{align*}
        \lambda_1'' \beta_{1i} + \lambda_2'' \beta_{2i} &= q_i\\
        \lambda_1'' \beta_{1j} + \lambda_2'' \beta_{2j} &= q_j
    \end{align*}    
    for some $q_1, q_j\in\Q$. The $\R$-linear independence of $(\beta_{1i}, \beta_{1j})$ and $(\beta_{2i}, \beta_{2j})$ then implies that $(q_i, q_j)$ uniquely determines $(\lambda_1'', \lambda_2'')$. Since there are countably many choices of $(q_i, q_j)\in\Q^2$, we also have countably many choices of $(\lambda_1'', \lambda_2'')\in\R^2$ that make $\lambda''_1 p_1 + \lambda''_2 p_2$ into a rational polynomial, and we are done.
 \end{proof}

    The next corollary shows that the conclusion of Lemma \ref{L: countably many exceptions} is upheld when we take integer parts of $a_1, a_2$.
 \begin{corollary}\label{C: countably many exceptions}
     Let $a_1, a_2\in\CH$ be independent. Then there exists at most countably many $(\lambda_1, \lambda_2)\in\R^2$ for which
     \begin{align*}
         \E_{n\in[N]}e(\lambda_1 \floor{a_1(n)} + \lambda_2 \floor{a_2(n)})
     \end{align*}
     does not converge to 0 as $N\to\infty$.
 \end{corollary}
\begin{proof}
    We follow the strategy of \cite[Lemma 6.2]{Fr21}.
    We begin by rephrasing our exponential sum as
         \begin{align*}
         \E_{n\in[N]}e(\lambda_1 \floor{a_1(n)} + \lambda_2 \floor{a_2(n)}) = \E_{n\in[N]}e(\lambda_1 {a_1(n)} + \lambda_2 {a_2(n)}-\lambda_1 \rem{a_1(n)} - \lambda_2 \rem{a_2(n)}).
     \end{align*}
     The function $G(x,y) = e(-\lambda_1 x - \lambda_2 y)$ on $[0,1)^2$ is Riemann integrable, and so we can split it into real and imaginary parts and approximate each from below and above by trigonometric polynomials. It therefore suffices to show that
    \begin{align*}
         \E_{n\in[N]}e((\lambda_1 + k_1) {a_1(n)} + (\lambda_2 + k_2){a_2(n)})
     \end{align*}
     fails to converge to 0 as $N\to\infty$ for countably many pairs $(\lambda_1, \lambda_2, k_1, k_2)\in\R^2\times\Z^2$. This follows from Lemma \ref{L: countably many exceptions}. 
\end{proof}

 The next lemma shows that joint ergodicity for pairwise independent Hardy sequences implies the difference ergodicity condition from Problem~\ref{Pr: joint ergodicity problem}.
\begin{lemma}\label{L: (i)}
    Let $a_1, \ldots, a_\ell\in\CH$ be pairwise independent, and suppose that they are jointly ergodic for the system $(X, \CX, \mu, T_1, \ldots, T_\ell)$. Then $(T_i^{\floor{a_i(n)}}T_j^{-\floor{a_j(n)}})_n$ is ergodic for $(X, \CX, \mu)$ for all distinct $i,j\in[\ell]$.
\end{lemma}
 For general sequences $a_1, \ldots, a_\ell:\Z\to\R$, a much simpler argument shows that joint ergodicity implies the average along $(T_i^{\floor{a_i(n)}}T_j^{-\floor{a_j(n)}})_n$  to converge to an integral \textit{weakly} (see e.g. \cite[Proposition 5.3]{DKS22}). However, for independent Hardy sequences $a_i, a_j$, the $L^2(\mu)$ limit along $(T_i^{\floor{a_i(n)}}T_j^{-\floor{a_j(n)}})_n$ may a priori not exist, and so we need a different approach than the one used in previous works on joint ergodicity \cite{DFKS22, DKS23, DKS22,  FrKu22a, FrKu22b} in which the $L^2(\mu)$ convergence was already known. 

\begin{proof}
    Let $f\in L^\infty(\mu)$, and suppose that $\int f\; d\mu = 0$, so that the claim reduces to showing that
    \begin{align*}
        \lim_{N\to\infty}\norm{\E_{n\in[N]}T_i^{\floor{a_i(n)}}T_j^{-\floor{a_j(n)}}f}_{L^2(\mu)} = 0.
    \end{align*}
    By the classical Herglotz-Bochner theorem, there exists a finite Borel measure $\nu$ on $\T^2$ such that
    \begin{align*}
        \norm{\E_{n\in[N]}T_i^{\floor{a_i(n)}}T_j^{-\floor{a_j(n)}}f}_{L^2(\mu)} = \norm{\E_{n\in[N]}e\brac{\lambda_i\floor{a_i(n)}-\lambda_j\floor{a_j(n)}}}_{L^2(\nu(\lambda_1, \lambda_2))}
    \end{align*}
    for every $N\in\N$; since $\int f\; d\mu = 0$, we have no point mass on the origin. By Corollary~\ref{C: countably many exceptions}, the exponential sum vanishes in the limit for all but countably many $(\lambda_i, \lambda_j)\in\T^2$. By the joint ergodicity, the exponential sum also converges to 0 whenever $(\lambda_i, -\lambda_j)$ is an eigenvalue of the joint action of $(T_i, T_j)$. Hence the set $(\lambda_i, \lambda_j)\in\T^2$ for which the exponential sum does not vanish has no point masses, and since it is countable, it is $\nu$-measure must be 0. The claim follows from the bounded convergence theorem.
\end{proof}

\begin{lemma}\label{L: Kronecker for single average}
    Let $a\in\CH$ with $a\succ \log$. Then $a$ is controlled by the Kronecker factor on any system $(X, \CX, \mu, T)$.
\end{lemma}

Lemma~\ref{L: Kronecker for single average} is fairly classical, and we refrain from providing a proof since it would descend into a tedious analysis of several cases. If $a$ is a polynomial of degree $d$, then it can be proved by $d-1$ applications of the van der Corput inequality and the known fact that weighted averages along $\floor{\alpha n}$ are controlled by the Kronecker factor. If $a = p + g$ for $p\in\R[t]$ and strongly nonpolynomial $g\prec p$, then after $d-1$ applications of the van der Corput inequality, we additionally need to pass to short intervals on which $g$ is constant and then apply the results for the linear case. If $a = p + g$ for $p\in\R[t]$ and strongly nonpolynomial $g\succ \log$ (which overlaps with the previous case), then we can pass to short intervals, Taylor expand $g$ to a polynomial of degree exceeding the degree of $p$, apply the van der Corput inequality to kill $p$, and then apply the usual machinery for polynomials.

We can now deduce Theorem~\ref{T: joint ergodicity} from the following more precise statement that illustrates exactly how the implications in the joint ergodicity classification problem for pairwise independent Hardy sequences go.
\begin{theorem}\label{T: joint ergodicity 2}
    Let $\ell\in\N$, $(X, \CX, \mu, T_1,$ $\ldots, T_\ell)$ be a system, and $a_1, \ldots, a_\ell\in \CH$ be pairwise independent. Then the following implications hold:
     \begin{enumerate}
         \item the joint ergodicity of $a_1, \ldots, a_\ell$ for $(X, \CX, \mu, T_1, \ldots, T_\ell)$ is equivalent to the ergodicity of $(T_1^{\floor{a_1(n)}}\times \cdots \times T_\ell^{\floor{a_\ell(n)}})_n$ on the product system $(X^\ell, \CX^{\otimes \ell}, \mu^\ell)$;
         \item the joint ergodicity of $a_1, \ldots, a_\ell$ for $(X, \CX, \mu, T_1, \ldots, T_\ell)$ implies that the sequence $(T_i^{\floor{a_i(n)}}T_j^{-\floor{a_j(n)}})_n$ is ergodic on $(X, \CX, \mu)$ for all distinct $i, j\in[\ell]$.
     \end{enumerate}
\end{theorem}
 \begin{proof}
In the first statement, the forward implication follows from Lemmas \ref{L: weak (ii)} and \ref{L: Kronecker for single average}, and the converse implication is covered by Lemma \ref{L: sufficient direction} and Theorem \ref{T: HK control}. The second statement follows from Lemma \ref{L: (i)}.
\end{proof}

\begin{appendix}

\section{Properties of Hardy functions}\label{A: Hardy properties}
In this appendix, we gather properties of Hardy sequences that are relevant in our proofs. Most of the lemmas exist in the literature already, so we omit their proofs.

The first lemma appears in \cite{Fr09} and allows us to compare the growth rate of a Hardy field function to the growth rate of its derivative.
 \begin{lemma}[{\cite[Lemma 2.1]{Fr09}}]\label{L: Frantzikinakis growth inequalities}
    Suppose that $a\in \mathcal{H}$ satisfies $t^{-m}\prec a(t)\prec t^m$ for some  $m
    \in\N$ and $a(t)$ does not converge to a nonzero constant as $t\to\infty$. Then  \begin{equation*}
        \frac{a(t)}{t\log^2 t}\prec a'(t)\ll \frac{a(t)}{t}.
    \end{equation*}
    In addition, if $a(t)\gg t^{\veps}$ or $a(t)\ll t^{-\veps}$ for some $\veps>0$, we also have \begin{equation*}
        a'(t)\sim \frac{a(t)}{t}.
   \end{equation*} 
\end{lemma}

In the case of strongly nonpolynomial functions that dominate $\log$, we can show that the derivatives always satisfy the assumptions of the hypothesis, so that this lemma can be iterated arbitrarily many times. The proof of this assertion appears in \cite[Lemma~A.2]{Ts22}. 

\begin{lemma}[{\cite[Lemma~A.2]{Ts22}}]\label{L: degree inequality}
Suppose that $a\in \mathcal{H}$ satisfies $a(t) \succ \log t$ and $t^{d}\prec a(t)\prec t^{d+1}$ for some $d\in\N_0$. Then, for any $m\geq d+1,$ we have  \begin{equation*}
        1\prec |a^{(m)}(t)|^{-\frac{1}{m}}\prec |a^{(m+1)}(t)|^{-\frac{1}{m+1}}\prec t.
    \end{equation*}
    Additionally, the leftmost relation holds under the weaker assumption $a(t) \succ 1$.
\end{lemma}
By considering $a(t) = \log t$ and its derivatives, it is immediate that the assumption $a(t) \succ \log t$ is necessary for the second and third relations to hold.

We turn to presenting various properties of our novel notion of fractional degree. First, we show that the fractional degree of a function in $\CH$ is always a real number.
\begin{lemma}\label{L:Relation of growth rates}
Let $a\in \CH$. We have 
\[\lim_{t\to\infty}\frac{ta'(t)}{a(t)}\in \R.\]
\end{lemma}

\begin{proof}
Since Hardy fields are closed under multiplication and differentiation, and Hardy field functions are eventually monotone,\footnote{The closure of the Hardy field under compositional inverses implies that each Hardy field function has eventually constant sign, for otherwise it would not be possible to define an inverse. Applying this observation to $a'$, we deduce that $a$ is eventually monotone.} the limit  $\lim\limits_{t\to\infty}\frac{ta'(t)}{a(t)}$ exists in the extended real line. 

 We will show that it cannot equal to $+\infty$ (the case $-\infty$ is completely analogous).

Indeed, assuming the contrary, for every $M>0$ there exists $N>0$ so that for all $x>N,$ we have $\frac{ta'(t)}{a(t)}>M.$ Then
\[\log\left(\left(\frac{x}{N}\right)^{M}\right)=\int_N^x \frac{M}{t}\;dt\leq\int_{N}^x \frac{a'(t)}{a(t)}\;dt\leq \log\frac{|a(x)|}{|a(N)|},\]
from where we get
\[|a(N)| (x/N)^{M}\leq |a(x)|,\] a contradiction to the fact that $a$ is of polynomial growth.
\end{proof}

The next lemma lists various straightforward properties of the notion of fractional degree.

\begin{lemma}[Properties of fractional degrees]\label{L: properties of fracdeg}
    Let $a, b\in\CH$. Then
\begin{enumerate}
    \item\label{i: fracdeg derivative} $\fracdeg{a^{(k)}} = \fracdeg{a} - k$ for any $k\in\N$ whenever $a$ is strongly nonpolynomial;
    \item\label{i: fracdeg product} $\fracdeg (ab) = \fracdeg a+\fracdeg b$;
    \item\label{i: fracdeg composition} $\fracdeg{a\circ b} = \fracdeg{a}\cdot\fracdeg{b}$ whenever  $\lim\limits_{t\to\infty}b(t)=\infty$; and
    \item\label{i: fracdeg power} $\fracdeg{a^k} = k\cdot \fracdeg{a}$ for any $k\in\R$ for which $a^k\in\CH$.
\end{enumerate}
\end{lemma}
We note that the first property fails for polynomials since all sufficiently large derivatives of a particular polynomial are zero.
\begin{proof}
To prove the first property, we first note from Lemma \ref{L:Relation of growth rates} that there exists $c\in\R$ for which
\[\lim_{t\to\infty}\frac{ta'(t)}{a(t)}=c.\]
By a repeated application of L'H\^opital's rule, we deduce that for every $k\in \N,$ we have 
\[\lim_{t\to\infty}\frac{ta^{(k+1)}(t)}{a^{(k)}(t)}=\lim_{t\to\infty}\frac{ta^{(k)}(t)}{a^{(k-1)}(t)}-1=\ldots=\lim_{t\to\infty}\frac{ta'(t)}{a(t)}-k=c-k.\] 

The second property is a simple application of Leibniz' rule. For the third property, we compute that
\begin{align*}
    \lim_{t\to\infty}\frac{t (a\circ b)'(t)}{a\circ b(t)} = \lim_{t\to\infty}\frac{t a'(b(t))b'(t)}{a(b(t))} = \lim_{t\to\infty}\frac{tb'(t)}{b(t)}\cdot\frac{b(t) a'(b(t))}{a(b(t))}.
\end{align*}
Since $b(t)\to\infty$ as $t\to\infty$, we have
\begin{align*}
    \lim_{t\to\infty}\frac{b(t) a'(b(t))}{a(b(t))} = \lim_{t\to\infty}\frac{t a'(t)}{a(t)} = \fracdeg a,
\end{align*}
and so the result follows from the product rule for limits.

Lastly, since 
\begin{align*}
     \lim_{t\to\infty}\frac{t (a^k)'(t)}{a^k(t)} = k \cdot \lim_{t\to\infty}\frac{t a(t)^{k-1}a'(t)}{a^k(t)} = k\cdot \fracdeg a,
\end{align*} we have property \eqref{i: fracdeg power}.
\end{proof}

 \section{Simultaneous approximations in short intervals}\label{A: approximations}
 In this part, we provide a proof for Proposition \ref{P: Taylor expansion ultimate}. This necessitates the study of the growth rates of Hardy field functions and their interaction with differentiation. A large part of this analysis has been carried out originally in \cite{Ts22}. However, the information provided by the work in \cite{Ts22} is insufficient for our purposes.

Firstly, we establish a combinatorial lemma, which virtually covers the case of Proposition \ref{P: Taylor expansion ultimate} when all the implicit Hardy field functions are distinct fractional powers.

\begin{lemma}\label{L: fractional powers common expansion}
	    Let $\theta_1,\theta_2,\dots,\theta_{\ell}$ be distinct positive real numbers, all smaller than a real number $u$. Then {for any $q\in\N$, there exist integers $d_{\theta_1},\dots,d_{\theta_{\ell}},d_{u}\geq q$ such that the following conditions are satisfied:}
\begin{enumerate}
    \item\label{i: Kftp1} we have \begin{equation}\label{E: non-empty intersection of intervals}
      \Big(\frac{u}{d_u+1},\frac{u}{d_u} \Big] \cap\brac{\bigcap_{i=1}^{\ell}\Big(\frac{\theta_i}{d_{\theta_i}+1},\frac{\theta_i}{d_{\theta_i}}\Big]} \neq \emptyset;
    \end{equation}
    \item\label{i: Kftp2} for any $i\in [\ell]$, we have \begin{equation}\label{E: first bound on d_u}
        \frac{u}{d_{u}}<\frac{\theta_{i}}{d_{\theta_{i}}}.
    \end{equation}In particular, $d_u>d_{\theta_i}$;
    \item\label{i: Kftp3} for any $i, j\in [\ell]$ with $i\neq j$, we have \begin{equation}\label{E: coefficients have distinct growth}
        \frac{u}{d_{u}}<\frac{\theta_{i}-\theta_{j}}{d_{\theta_{i}}-d_{\theta_{j}}}.
    \end{equation}
\end{enumerate}
         \end{lemma}

The reader will observe that no Hardy sequences appear in our statement, so we make some remarks on the connection between this lemma and Proposition \ref{P: Taylor expansion ultimate}.  
Suppose we are given the collection $g_i(t)=t^{\theta_i}$ and $g(t)=t^u$ with the notation of Lemma \ref{L: fractional powers common expansion}.
The computation of the derivatives of fractional powers is straightforward, so that
unpacking the complicated conditions of Proposition \ref{P: Taylor expansion ultimate} leads us to some inequalities between the integers $\theta_i,u$ and $d_{\theta_i},d_u$, which will correspond to the degrees of the Taylor expansions.
For instance, the fact that the intersection \eqref{E: non-empty intersection of intervals} is nonempty corresponds to condition \eqref{i: property implying common Taylor expansion} of Proposition \ref{P: Taylor expansion ultimate}, the inequality \eqref{E: first bound on d_u} corresponds to condition \eqref{i: comparing frac degs} while the inequality \eqref{E: coefficients have distinct growth} corresponds to condition \eqref{i: property implying different chi}. Then, our main task is to demonstrate that this system of inequalities possesses infinitely many solutions $(d_{\theta_1},\dots,d_{\theta_{\ell}},d_{u})\in\N^{\ell+1}$. We will establish this by constructing such a family of solutions, which will be produced through the set of the return times of a (multidimensional) irrational rotation.

We will use the following lemma, which is a simple consequence of Weyl's equidistribution theorem.
\begin{lemma}\label{L: nested integer parts}
	    Let $k,Q\in\N$ and $\veps,x_1,\dots, x_k>0$ be such that $1,\; x_1,\; x_1x_2,\; \ldots,\; x_1\cdots x_k$ are $\mathbb{Q}$-independent. Then there exist infinitely many $\theta \in \N$ with $Q$ dividing $\theta$ such that the following hold:
     \begin{enumerate}
         \item we have the inequality \begin{equation}\label{E: inequality with integer parts0}
	        \floor{x_k\floor{x_{k-1}\floor{\dots \floor{x_2\floor{x_1\theta}}\dots}   }}\geq x_k\ldots x_1\theta-\veps;
	    \end{equation}
        \item for every $1\leq i\leq k$, $Q$ divides $ \floor{x_i\floor{x_{i-1}\floor{\dots \floor{x_2\floor{x_1\theta}}\dots}   }}$.
     \end{enumerate}
	\end{lemma}
    \begin{proof}
        Since $1,x_1,x_1x_2,\ldots,x_1x_2\ldots x_k$ are $\mathbb{Q}$-independent, there exist infinitely many $\theta\in \N$, such that the following collection of inequalities is satisfied \begin{equation*}
              \begin{cases}
        \frac{\veps}{Q}\leq \frac{\{x_1\theta/Q\}}{x_1}\leq \frac{2\veps}{Q}\\
        \frac{2\veps}{Q} \leq \frac{\{x_1x_2\theta/Q\}}{x_1x_2}\leq \frac{4\veps}{Q} \\
      \vdots \\
        \frac{2^{k-2}\veps}{Q} \leq \frac{\{x_1\cdots x_{k-1}\theta/Q\}}{x_1\cdots x_{k-1}}\leq \frac{2^{k-1}\veps}{Q} \\
       \frac{2^{k-1} \veps}{Q}\leq \frac{\{x_1\cdots x_k\theta/Q\}}{x_1\cdots x_k}\leq \frac{2^k\veps}{Q}
    \end{cases},
        \end{equation*}where we assume (as we may) that $\veps$ is very small  in terms of $Q/(2^k x_1\cdots x_k)$. Furthermore, we can choose $\theta$ to be a multiple of $Q$ (by applying the previous argument to $Qx_1,\; Qx_1x_2,\; \ldots,\; Qx_1\cdots x_k$).
        
Observe that the previous inequalities immediately imply the following weaker inequalities \begin{equation}\label{E: inequalities without Q}
    2^{i-1}\veps\leq \frac{\{x_1\cdots x_i\theta\}}{x_1\cdots x_i}\leq {2^i\veps}.
\end{equation}for every $1\leq i\leq k$.   This is an immediate consequence of the observation that if 
$\{a\}$ is very small in terms of $1/Q$, then $\{Qa\}=Q\{a\}$ (which does not hold, in general, if $Q$ is not an integer).

        We use induction to show that for all $1\leq i \leq k$ and any $\theta$ satisfying the previous inequalities, we have \begin{equation*}
            x_i\cdots x_1\theta\geq \floor{x_i\floor{x_{i-1}\floor{\dots \floor{x_2\floor{x_1\theta}}\dots}   }}\geq    x_i\cdots x_1\theta-2^{i}\veps x_i\dots x_{1}
        \end{equation*}
 For $i=1$, the inequality follows from \eqref{E: inequalities without Q}. Now, assume that the claim has been proven for $i$. Then, we want to show that \begin{equation*}
              x_{i+1}\cdots x_1\theta\geq \floor{x_{i+1}\floor{x_{i}\floor{\dots \floor{x_2\floor{x_1\theta}}\dots}   }}\geq    x_{i+1}\cdots x_1\theta-2^{i+1}\veps x_{i+1}\cdots x_1.
        \end{equation*}The left inequality is always valid, as dropping the integer parts increases the quantity in the middle. In addition, the induction hypothesis implies that \begin{equation}\label{E: inequalities following from induction hypothesis}
             x_{i+1}\cdots x_1\theta\geq x_{i+1}\floor{x_{i}\floor{\dots \floor{x_2\floor{x_1\theta}}\dots}   }\geq    x_{i+1}\cdots x_1\theta-2^{i}\veps x_{i+1}\cdots x_1
        \end{equation}Since $$2^i\veps x_{i+1}\cdots x_1\leq \{   x_{i+1}\cdots x_1\theta\}\leq 2^{i+1}\veps x_{i+1}\cdots x_1 $$ by \eqref{E: inequalities without Q}, we infer that
        \begin{equation}\label{E: equality of integer parts}
            \floor{ x_{i+1}\cdots x_1\theta} =\floor{x_{i+1}\floor{x_{i}\floor{\dots \floor{x_2\floor{x_1\theta}}\dots}   }}=\floor{ x_{i+1}\cdots x_1\theta-2^i\veps x_{i+1}\cdots x_1}.
        \end{equation}
        Consequently, using the fact that the fractional part of $x_{i+1}\cdots x_1\theta$ is smaller than $2^{i+1}\veps x_{i+1}\cdots x_1$, we infer that 
        \begin{equation*} 
             \floor{x_{i+1}\floor{x_{i}\floor{\dots \floor{x_2\floor{x_1\theta}}\dots}   }}\geq x_{i+1}\cdots x_1\theta-2^{i+1}\veps x_{i+1}\cdots x_1
        \end{equation*}and the induction is complete.

Now, we show that $Q$ divides each one of the integers $\floor{x_{i}\floor{\dots \floor{x_2\floor{x_1\theta}}\dots}   }$. 

We recall \eqref{E: equality of integer parts} which readily implies that
\begin{equation*}
             \floor{x_i\floor{x_{i-1}\floor{\dots \floor{x_2\floor{x_1\theta}}\dots}   }}=\floor{x_i\cdots x_1\theta}.
        \end{equation*}
 Therefore, it suffices to show that $Q|\floor{x_i\cdots x_1\theta}$, which follows easily from the fact that \begin{equation*}
     \rem{\frac{x_i\cdots x_1\theta}{Q}}\leq \frac{2^i\veps x_i\cdots x_1}{Q}< \frac{1}{Q},
 \end{equation*}where the last inequality holds if we pick $\veps$ to be sufficiently small. 
        \end{proof}

We are ready to prove Lemma \ref{L: fractional powers common expansion}. We will apply Lemma \ref{L: nested integer parts} to a $\Q$-independent subset of $\theta_i$'s. The $d_{\theta_i}$'s will be constructed inductively using iterated integer parts for the subset of the $\theta_i$ and then the other $d_{\theta_i}$ will be defined using linearity.
The assumption that $Q$ divides the expressions with the integer parts in the previous lemma is necessary to ensure that certain fractions in the definition of the $d_i$ below will be integer numbers.
    
\begin{proof}[Proof of Lemma \ref{L: fractional powers common expansion}]
    We choose a subset $\{c_0,\ldots, c_k\}$ of $\{\theta_1,\ldots, \theta_{\ell}\}$ with maximal cardinality, and such that $c_0,\ldots, c_k$ are $\mathbb{Q}$-independent. Reordering if necessary, we will assume that $c_0<c_1<\ldots<c_k$. Observe that for any $\theta_i$, we can write \begin{equation}\label{E: linear combination of theta_i}
        \theta_i=p_{i0}c_0+\dots+p_{ik}c_k
    \end{equation} for some $p_{ij}\in\mathbb{Q}$.
    Let $M\in\N$ be larger than the height\footnote{The height of a rational number $p/q$ written in lowest terms is the quantity $\max\{|p|,|q|\}.$} of all $p_{ij}$. In particular, this implies $|p_{ij}|<M$ for all admissible indices $i,j$.

We consider the set \begin{equation*}
    \mathcal{A}=\Bigl\{\beta\in (0,u)\colon \beta=\sum_{i=0}^{k}\frac{p_i}{q_i}c_i, \ \text{where } p_i,q_i\in \Z, |q_i|<M^2, |p_i|<M^3\Bigr\}
\end{equation*}
Namely, $\mathcal{A}$ contains those linear combinations of $c_0,\ldots, c_k$ with bounded rational coefficients, whose denominators are also bounded, and which also lie in the open interval $(0,u)$.

Observe that all of the initial numbers $\theta_1,\dots,\theta_{\ell}$ belong to $\mathcal{A}$. Also, we observe that $\mathcal{A}$ is finite and therefore bounded away from zero and $u$. Namely, there exists $\delta>0$, such that 
\begin{equation}\label{E: bounded away}
    \delta<s<u
\end{equation}
for all $s\in \mathcal{A}$.

We define $d_0,\dots, d_k$ recursively by \begin{equation}\label{E: recursion for degrees d_i}
	        d_0=\theta, \ \ d_{i+1}=\floor{\frac{c_{i+1}}{c_i}d_i   },\ \text{ for every } 0\leq i\leq k-1.  
	    \end{equation}where $\theta\in \N$ is chosen (arbitrarily large) so that the following relations hold:
\begin{enumerate}
    \item we have that \begin{equation}\label{E: M factorial divides d_i}
   M!\big|d_0\ \text{and } M!\big|d_i= \floor{\frac{c_i}{c_{i-1}} \floor{\dots \floor{\frac{c_2}{c_1}\floor{\frac{c_1}{c_0} d_0}  }   \dots}  }
\end{equation}for every $1\leq i\leq k$;
    \item we have the inequality \begin{equation}\label{E: inequality with nested integer parts}
	         \floor{\frac{c_k}{c_{k-1}} \floor{\dots \floor{\frac{c_2}{c_1}\floor{\frac{c_1}{c_0} d_0}  }   \dots}  }\geq \frac{d_0c_k}{c_0}-c_k\veps,
	    \end{equation}
     where $\veps$ is very small in terms of the fixed numbers $\delta, M,u$ and $c_0$.
\end{enumerate}
      We can choose $d_0$ in such a way by applying Lemma \ref{L: nested integer parts} (with $\theta = d_0, Q = M!, x_1 = \frac{c_1}{c_0}, \ldots, x_k = \frac{c_k}{c_{k-1}}$), since the numbers $1,\frac{c_1}{c_0}.\ldots,\frac{c_k}{c_0}$ are $\mathbb{Q}$-independent by assumption. By the definition of $d_0, \ldots, d_k$, we have the following chain of inequalities \begin{equation}
         \frac{d_0}{c_0}\geq \frac{d_1}{c_1}\geq \ldots\geq \frac{d_k}{c_k}\geq \frac{d_0}{c_0}-\veps.
     \end{equation}
     This implies that \begin{equation}\label{E: distance of d_i/c_i}
         \Bigabs{\frac{d_i}{c_i}-\frac{d_j}{c_j}}\leq \veps\implies |d_ic_j-d_jc_i|\leq \veps|c_i||c_j|<\veps u^2.
     \end{equation}

     Now, for any $s\in \mathcal{A}$ of the form $s=\sum_{i=0}^{k} \frac{p_i}{q_i}c_i, $we define $$D_s=\sum_{i=0}^{k} \frac{p_i}{q_i}d_i.$$
In particular, if $s$ is one of the original numbers $\theta_i$, we will ultimately pick $d_{\theta_i}=D_{\theta_i}$. Observe that there is no ambiguity in this definition, since the numbers $1,\frac{c_1}{c_0},\ldots, \frac{c_k}{c_0}$ are  $\mathbb{Q}$-independent, so that every $s\in A$ has a unique representation as a rational combination of these numbers.

Firstly, we prove that $D_{\theta_i}$ are  integers. 
We observe that is true since the denominators of the rationals ${p_{i0}},\dots, p_{ik}$ in \eqref{E: linear combination of theta_i} are all smaller than $M$ (by the initial choice of $M$) and that $M!$ divides $d_0,\dots, d_{k}$ by \eqref{E: M factorial divides d_i}.

 Next we establish a bound for the difference \begin{equation}
    \Bigabs{\frac{D_s}{s}-\frac{d_0}{c_0}}
\end{equation}for any $s\in \mathcal{A}$.
     Indeed, this may be rewritten as \begin{equation*}
         \abs{
        \dfrac{\sum_{i=0}^{k} \frac{p_i}{q_i} d_i}{\sum_{i=0}^{k} \frac{p_i}{q_i} c_i}-\frac{d_0}{c_0}
         }=\abs{\dfrac{\sum_{i=1}^{k} \dfrac{p_i}{q_i} \big( d_ic_0-d_0c_i\big)}{c_0\big(\sum_{i=0}^{k} \frac{p_i}{q_i} c_i\big)}}\leq \dfrac{\sum_{i=1}^{k} \abs{\frac{p_i}{q_i}} \big| d_ic_0-d_0c_i\big|}{|c_0||s|}<
         \veps \frac{M^3ku^2}{|c_0||\delta|},
      \end{equation*}where we used \eqref{E: bounded away}, \eqref{E: distance of d_i/c_i}, and the fact that $|p_i|<M^3$ by the definition of $\mathcal{A}$.
This also implies that \begin{equation}\label{E: distances of D_s/s}
    \abs{\frac{D_s}{s}-\frac{D_t}{t}}<2\veps \frac{M^3ku^2}{|c_0||\delta|}
\end{equation}for any $s,t\in \mathcal{A}$.

\begin{claim}\label{C: claim 1}
    If $\veps$ is sufficiently small, then the following hold:
\begin{enumerate}
    \item for each $s\in \mathcal{A}$, we have $D_s>0$;
    \item the intersection  \begin{equation*}
    \bigcap_{s\in \mathcal{A}}\Big(\frac{s}{D_{s}+1},\frac{s}{D_{s}}\Big]
\end{equation*}is non-empty.
\end{enumerate}
\end{claim}

For the first part of the claim, note that $s$ is positive by the definition of the set $\mathcal{A}$. So
\begin{equation}\label{E: positivity of D}
    \frac{D_s}{s}>\frac{d_0}{c_0}-\veps\frac{M^3ku^2}{|c_0||\delta|}\geq\frac{1}{c_0}-\veps\frac{M^3ku^2}{|c_0||\delta|}>0
\end{equation} if we pick $\veps$ sufficiently small. 

The second part of the claim amounts to showing that \begin{equation*}
     \frac{s}{D_s}>\frac{t}{D_t+1}
 \end{equation*}for any $s,t\in \mathcal{A}$, which can be rewritten\footnote{Recall that $D_s>0$ for all $s\in \mathcal{A}$.} as \begin{equation*}
     \frac{D_s}{s}-\frac{D_t}{t}<\frac{1}{t}.
 \end{equation*}If the left-hand side is negative, then the claim is obvious. Otherwise, our claim will follow if we prove that \begin{equation*}
     \abs{\frac{D_s}{s}-\frac{D_t}{t}}<\frac{1}{u},
 \end{equation*}since all elements $t\in \mathcal{A}$ are smaller than $u$. However, we have already established that the left-hand side is smaller than \begin{equation*}
     \frac{2\veps M^3ku^2}{|c_0||\delta|}
 \end{equation*}which can be made smaller than $1/u$ if we pick $\veps$ sufficiently small.

 \

\begin{claim}\label{C: claim 2}
    If $\veps$ is sufficiently small, then there exists  $d_u\in\N$, such that the following two conditions are satisfied:
\begin{enumerate}
    \item We have \begin{equation}\label{E: non-empty intersection mathcalA}
       \Big(\frac{u}{d_u+1},\frac{u}{d_u} \Big] \cap\brac{  \bigcap_{s\in \mathcal{A}}\Big(\frac{s}{D_{s}+1},\frac{s}{D_{s}}\Big]}\neq \emptyset.
\end{equation}
    \item For all $s\in \mathcal{A}$, we have \begin{equation}\label{E: d_u/u dominates}
    \frac{u}{d_u}<\frac{s}{D_s}.
\end{equation}
\end{enumerate}
\end{claim}

If Claim \ref{C: claim 2} holds, then the first two conditions of the statement will be proven (with $d_{\theta_i}=D_{\theta_i}$)\footnote{We will also need to show that $D_{\theta_i}$ are integers.} by restricting \eqref{E: non-empty intersection mathcalA} and \eqref{E: d_u/u dominates} from the more general set $\mathcal{A}$ to its subset $\{\theta_1,\ldots, \theta_{\ell}\}$.

 To prove Claim \ref{C: claim 2}, it suffices to show that there exists $d_u\in\N$ such that
 \begin{equation}\label{E: goal inequalities for d_u}
     \frac{D_s}{s} < \frac{d_u}{u}<\frac{D_s+1}{s}
  \end{equation}for all $s\in \mathcal{A}$. 
By multiplying both sides with $u$, we can rephrase these inequalities as 
\begin{equation*}
    \max_{s\in \mathcal{A}}\Bigl\{u\cdot \frac{D_s}{s}\Bigr\}<d_u<\min_{s\in \mathcal{A}}\Bigl\{u\cdot \frac{D_s}{s}+\frac{u}{s}\Bigr\}.
\end{equation*}
So it suffices to show that 
\begin{equation}\label{qrqwrq}
    \min_{s\in \mathcal{A}}\Bigl\{u\cdot \frac{D_s}{s}+\frac{u}{s}\Bigr\}-\max_{s\in \mathcal{A}}\Bigl\{u\cdot \frac{D_s}{s}\Bigr\}>1.
\end{equation}

We remark that the desired $d_{u}$ guaranteed by (\ref{qrqwrq})  is not equal to either $u\cdot \frac{D_t}{s}$ and $u\cdot \frac{D_s}{s}+\frac{u}{s}$ for any $s\in \mathcal{A}$, because the inequality in (\ref{qrqwrq}) is strict.

Note that (\ref{qrqwrq}) holds if we can show that 
\begin{equation*}
 u\cdot \frac{D_s}{s}+\frac{u}{s}  - u\cdot \frac{D_t}{t}>1
\end{equation*}for all $s,t\in \mathcal{A}$, which can be rewritten 
in the form \begin{equation*}
    \frac{D_t}{t}-\frac{D_s}{s}<\frac{1}{s}-\frac{1}{u}.
\end{equation*} Observe that the right-hand side is positive, since all elements of $\mathcal{A}$
 are smaller than $u$. On the other hand, the left-hand side is bounded in magnitude by the quantity in \eqref{E: distances of D_s/s}. If we pick $\veps$ sufficiently small, we can ensure that \begin{equation*}
      \frac{2\veps M^3ku^2}{|c_0||\delta|}<\min_{s\in \mathcal{A}}\Bigabs{{\frac{1}{s}-\frac{1}{u}}},
 \end{equation*}which implies the desired result.

\

In addition to Claim \ref{C: claim 2}, we shall consider the positive integers $D_{\theta_1},\ldots, D_{\theta_{\ell}}$ provided by the previous construction and which correspond to our original real numbers $\theta_1,\ldots, \theta_{\ell}$.
We have already proven that for this choice of $d_u$ and $d_{\theta_i}=D_{\theta_i}$, conditions \eqref{i: Kftp1} and \eqref{i: Kftp2} of our statement hold and we will show that the same extends to condition \eqref{i: Kftp3}.
We also remind the reader that we have already established that $D_{\theta_i}$ are positive integers, so the choice of $d_{\theta_i}=D_{\theta_i}$ is indeed sensible.

Now, we show that condition \eqref{i: Kftp3} in our statement is satisfied.
We have to restrict to the set $\{\theta_1,\dots,\theta_{\ell}\}$ instead of the whole set $\mathcal{A}$.
The main point is that the difference $\theta_{i}-\theta_j$ belongs to the set $\mathcal{A}$ (when it is a positive number), while the difference $s-t$ might not. Indeed, if $p_i/q_i$ and $p'_i/q'_i$ are the coefficients of $s,t$ in their linear combinations with respect to $c_0,\dots, c_k$, then the difference $s-t$ is a linear combination of $c_0,\dots,c_k$. However, the rational number $p_i/q_i-p'_i/q'_i$ can have a denominator comparable to $q_iq'_i$, which can be as large as $M^4$.

To conclude the proof, we show that \begin{equation}\label{feqf}
    \frac{u}{d_u}<\frac{\theta_i-\theta_j}{D_{\theta_i}-D_{\theta_j}}
\end{equation}for all distinct $i,j\in [\ell]$. Assume without loss of generality that $\theta_i>\theta_j$. If $\theta_i-\theta_j\in \mathcal{A}$, then $D_{\theta_i-\theta_j}=D_{\theta_i}-D_{\theta_j}$ by the linearity in the definition of $D_s$, and so \eqref{E: d_u/u dominates} implies (\ref{feqf}). So it suffices to show that $\theta_i-\theta_j\in \mathcal{A}$.

We remark the obvious bounds \begin{equation*}
    0<\theta_i-\theta_j<u.
\end{equation*}In addition, if we write $$\theta_i=\sum_{h=0}^{k}p_{ih}c_{h}$$
and similarly for $\theta_j$, then the rational numbers $p_{ih}$ are bounded by $M$ and have height at most $M$, by the initial choice of the positive integer $M$. Therefore, we have \begin{equation*}
\theta_i-\theta_j=\sum_{h=0}^{k}(p_{ih}-p_{jh})c_h,
 \end{equation*}and the rational number $p_{ih}-p_{jh}$ is bounded in magnitude by $2M<M^3$ (so that its numerator is bounded by $M^3$ as well) and has denominator at most $2M^2$, which are both acceptable.

This shows that $\theta_i-\theta_j$ is an element of $\mathcal{A}$ and we are done.
\end{proof}

We are now ready to prove Proposition \ref{P: Taylor expansion ultimate}, which is a consequence of the following proposition.

\begin{proposition}\label{P: Taylor expansion ultimate in the positive fracdeg case}
     Let $m, q\in\N$ and $ g_1\prec \dots\prec g_m\in \mathcal{H}$ be strongly nonpolynomial functions satisfying $c_i := \fracdeg g_i > 0$.
      Then there exist $d_{g_1},\dots,d_{g_m}\in\N$ such that the following hold:
     \begin{enumerate}
         \item\label{i: large enough degrees} for all $i\in [m]$, $d_{g_i}\geq q$;
         \item\label{i: condition for equal degrees} for $i,j\in[m]$, we have $d_{g_i} = d_{g_j}$ iff $c_i = c_j$;
         \item\label{i: relations on different endpoints} for all $i,j\in [m]$,  we have
\begin{equation}\label{E: growth conditions for different endpoints}
    \bigabs{g_i^{(d_{g_i}) }(N) }^{-\frac{1}{d_{g_i}}}\lll   \bigabs{g_j^{(d_{g_j}+1) }(N)}^{-\frac{1}{d_{g_j}+1}};
   \end{equation}
   \item\label{i: property implying sublinearity for super-fractional} for all $i\in[m]$, we have \begin{equation}\label{E: growth conditions that imply sublinearity and super-fractionality on coefficients}
             \fracdeg \bigabs{g_i^{(d_{g_i}) }(N) }^{-\frac{1}{d_{g_i}}}\leq \fracdeg \bigabs{g_m^{(d_{g_m}) }(N) }^{-\frac{1}{d_{g_m}}},
\end{equation}
         with equality if and only if $\fracdeg g_i = \fracdeg g_m$;

\item\label{i: property implying different chi for super-fractional case} for all distinct $i,j\in[m]$ and all $\eta >0$, we have 
\begin{equation*}
    g_i^{(d_{g_i})}(N)\bigabs{g_{m}^{(d_{g_m}) }(N)}^{-\frac{d_{g_i}}{d_{g_m}}}N^{\eta d_{g_i}}\not\asymp    g_j^{(d_{g_j})}(N)\bigabs{g_{m}^{(d_{g_m}) }(N)}^{-\frac{d_{g_j}}{d_{g_m}}}N^{\eta d_{g_j}}.
\end{equation*}
     \end{enumerate}
 \end{proposition}
 
\begin{proof}
{Throughout the proof, we will make use of Lemma \ref{L: properties of fracdeg} without citations.}
We may write $g_i(t)=t^{c_i}b_i(t)$, where $b_i(t)\in \CH$ is a function of fractional degree $0$.
     We partition the set $[m]$ into disjoint, consecutive intervals $B_1,\dots, B_{\ell}$, where $g_i$ and $g_j$ belong to the same set $B_{r}$ if and only if $c_i=c_j$. 
     For $r\in[\ell]$, we shall denote by $C_{r}$ the common fractional power corresponding to the set $B_{r}$.
     Let $D_1,\ldots, D_{\ell-1}$ and $D_\ell$ be the integers $d_{\theta_1}, \ldots, d_{\theta_\ell}$ and $d_u$ provided by Lemma \ref{L: fractional powers common expansion}. Then we set
     \begin{align*}
         d_{g_i} = D_r \quad \Longleftrightarrow \quad i \in B_r.
     \end{align*}
    Namely, the choice of $d_{g_i}$ is the one provided by Lemma \ref{L: fractional powers common expansion} if we took $g_i$ to be fractional powers. 
     For reference, we will rewrite the main bounds from Lemma \ref{L: fractional powers common expansion} that $D_1,\ldots, D_{\ell}$ obey:\begin{enumerate}
         \item we have\begin{equation}\label{E: finite intersection non-empty}  
               \bigcap_{i=1}^{\ell}\Big(\frac{C_i}{D_{i}+1},\frac{C_i}{D_{i}}\Big] \neq \emptyset;
         \end{equation}
         \item for any $1\leq i\leq \ell-1$, we have \begin{equation}\label{E: D_l/C_l is big}
             \frac{C_{\ell}}{D_{\ell}}<\frac{C_i}{D_i};
         \end{equation}
         \item for any $1\leq i< j\leq \ell-1$ we have \begin{equation}\label{E: D_l/C_l is big v2}
             \frac{C_{\ell}}{D_{\ell}}<\frac{C_i-C_j}{D_i-D_j}.
         \end{equation}
     \end{enumerate}

Firstly, we observe that Lemma \ref{L: fractional powers common expansion} allows us to choose all  the numbers $d_{g_i}$ to be larger than $q$. Thus, condition \eqref{i: large enough degrees} is satisfied.  Secondly, condition \eqref{i: condition for equal degrees} follows by construction. Noting that
\begin{align*}
    \fracdeg \bigabs{g_i^{(d) }}^{-\frac{1}{d}} = 1- \frac{c_i}{d}
\end{align*} 
for any $d\in\N$, condition \eqref{i: relations on different endpoints} is equivalent to the inequality $\frac{c_j}{d_{g_j}+1} < \frac{c_i}{d_{g_i}}$; when $c_i = c_j$, this is immediate from the positivity of fractional degrees, otherwise it follows from \eqref{E: finite intersection non-empty}. Similarly, condition \eqref{i: property implying sublinearity for super-fractional} is equivalent to $\frac{c_m}{d_{g_m}}\leq \frac{c_i}{d_{g_i}}$, with equality only if and only if $c_i = c_m$, and this is an immediate consequence of \eqref{E: D_l/C_l is big}. To prove the last condition, we split into the cases $c_i = c_j$ and $c_i \neq c_j$. If $c_i = c_j$, then $d_{g_i} = d_{g_j}$ as well, and then result follows since $g_i, g_j$ have distinct growth. For the case $c_i \neq c_j$, we have
\begin{align*}
    g_i^{(d_{g_i})}(N)\bigabs{g_{m}^{(d_{g_m}) }(N)}^{-\frac{d_{g_i}}{d_{g_m}}}N^{\eta d_{g_i}} \asymp g_j^{(d_{g_j})}(N)\bigabs{g_{m}^{(d_{g_m}) }(N)}^{-\frac{d_{g_j}}{d_{g_m}}}N^{\eta d_{g_j}}
\end{align*}
only if both sides have the same fractional degrees. We will show that this is not the case. Since
\begin{align*}
    \fracdeg \abs{ g_i^{(d_{g_i})}(N)\bigabs{g_{m}^{(d_{g_m}) }(N)}^{-\frac{d_{g_i}}{d_{g_m}}}N^{\eta d_{g_i}}} = \frac{c_i d_{g_m} - c_m d_{g_i} + \eta d_{g_i}d_{g_{m}}}{d_{g_{m}}},
\end{align*}
both functions above have the same fractional degrees if and only if 
\begin{align}\label{E: to fail}
    \frac{c_i - c_j}{d_{g_i}-d_{g_j}} = \frac{c_m}{d_{g_m}} - \eta.
\end{align}

Since 
\begin{align*}
    \frac{c_m}{d_{g_m}} - \eta < \frac{c_m}{d_{g_m}} < \frac{c_i - c_j}{d_{g_i}-d_{g_j}}
\end{align*}
for every $\eta>0$ by \eqref{E: D_l/C_l is big v2}, the equality \eqref{E: to fail} will fail for all $\eta$ in this range, giving the result.
\end{proof}
We are ready to conclude the proof of Proposition \ref{P: Taylor expansion ultimate}.
\begin{proof}[Proof of Proposition \ref{P: Taylor expansion ultimate}]
    We use Proposition \ref{P: Taylor expansion ultimate in the positive fracdeg case} for functions $g_{m_2 + 1}, \ldots, g_m$ to find degrees $d_{g_{m_2+1}}, \ldots, d_{g_m}$ satisfying the conditions of Proposition \ref{P: Taylor expansion ultimate in the positive fracdeg case}, and we keep using the notation $c_i = \fracdeg g_i$. If $d_{g_m}  = 0$, i.e. all $g_1, \ldots, g_m$ are subfractional, we let $\eta_0 = 1$, otherwise take any $\eta_0$ satisfying
    \begin{align}\label{E: eta_0}
        0<\eta_0 <\frac{c_m}{d_{g_m}} -\max_{i\in[m_2+1, m]} \frac{c_i}{d_{g_i}+1} < c_m.
    \end{align}
     By  \eqref{E: growth conditions for different endpoints} applied with $i= m, j = i$, we have
    \begin{align*}
        \frac{c_m}{d_{g_m}} > \max_{i\in[m_2+1, m]} \frac{c_i}{d_{g_i}+1}, 
    \end{align*}
    and so the range of $\eta_0$ is nonempty.
    We then set $H(N) = N^{\eta}$ if $d_{g_m} = 0$ and
    $$H(N) = |g_m^{(d_{g_m})}(N)|^{-1/d_{g_m}}N^\eta$$
    otherwise for an arbitrary choice of $\eta\in (0, \eta_0)$. By \eqref{E: eta_0} and Proposition \ref{P: Taylor expansion ultimate in the positive fracdeg case}\eqref{i: property implying sublinearity for super-fractional}, we have
    \begin{align*}
            \bigabs{g_i^{(d_{g_i}) }(N) }^{-\frac{1}{d_{g_i}}}\lll H(N) \lll   \bigabs{g_j^{(d_{g_j}+1) }(N)}^{-\frac{1}{d_{g_j}+1}}
    \end{align*}
    for all $i,j\in[m_2+1, m]$,
    proving the first part of \eqref{E: growth conditions 1} (this holds vacuously if $d_{g_m} = 0$). We can then take $L$ to be any positive Hardy function satisfying 
    \begin{align*}
        H(N) \lll L(N)\lll  \bigabs{g_j^{(d_{g_j}+1) }(N)}^{-\frac{1}{d_{g_j}+1}}
    \end{align*}
    for all $j\in[m_2+1, m]$ (if $d_{g_m} = 0$, then this condition again holds vacuously, so we can take $L(N)=N^{\eta'}$ for $\eta<\eta'<1$). This completes the proof of \eqref{E: growth conditions 1}, and hence the condition \eqref{i: property implying common Taylor expansion} of Proposition~\ref{P: Taylor expansion ultimate}. So chosen functions $H, L$ have fractional degrees between 0 and 1 because the functions $\bigabs{g_i^{(d_{g_i})}}^{-\frac{1}{d_{g_i}}}, \bigabs{g_i^{(d_{g_i}+1)}}^{-\frac{1}{d_{g_i}+1}}$ do. It follows from this and the argument above Proposition \ref{P: Taylor expansion ultimate} that the degree of the Taylor expansions of $g_1, \ldots, g_{m_2}$ is 0. Together with Proposition \ref{P: Taylor expansion ultimate in the positive fracdeg case}, this completes the proof of the conditions \eqref{i: arbitrarily large} and \eqref{i: equal fractional degrees} of Proposition \ref{P: Taylor expansion ultimate}. Property \eqref{i: comparing frac degs} follows immediately from the analogous property in Proposition \ref{P: Taylor expansion ultimate in the positive fracdeg case}, and so it remains to show the property \eqref{i: property implying different chi}. For $i\in[m_2]$, the fact that $d_{g_i} = 0$ implies that $g_i^{(d_{g_i})}(N) H(N)^{d_i} =g_i(N)$, and so the functions $g_i^{(d_{g_i})}(N) H(N)^{d_i}, g_j^{(d_{g_j})}(N) H(N)^{d_j}$ have distinct growth for $i,j \in[m_2]$ because $g_i, g_j$ have distinct growth. For $i,j\in [m_2+1, m]$, property \eqref{i: property implying different chi} follows from Proposition~\ref{P: Taylor expansion ultimate in the positive fracdeg case}\eqref{i: property implying different chi for super-fractional case}. Lastly, we have $\fracdeg g_i^{(d_{g_i})}(N) H(N)^{d_i} = 0$ iff $i\in[m_2]$, and so property  \eqref{i: property implying different chi} clearly holds when exactly one $i,j$ is in $[m_2]$. 
\end{proof}

\end{appendix}

\end{document}